\documentclass[11pt,reqno]{amsart}
\usepackage{amssymb,amsmath,amsfonts,amsthm}
\usepackage[foot]{amsaddr}
\usepackage{enumitem}
\usepackage{calc}
\usepackage{multirow}
\usepackage[margin=1in]{geometry}
\usepackage[all]{xy}
\usepackage{hyperref}
\usepackage{mathtools}
\usepackage{fancyhdr}
\usepackage{url}
\usepackage{xcolor}
\usepackage{dsfont}
\usepackage{MnSymbol}
\usepackage{graphicx}
\usepackage{caption}
\usepackage{pgfplots}
\pgfplotsset{compat=1.14}

\newtheorem{thm}{Theorem}[section]
\newtheorem{theorem}{Theorem}

\newtheorem{lem}[thm]{Lemma}  
\newtheorem{prop}[thm]{Proposition}
\newtheorem{cor}[thm]{Corollary}

\newtheorem{definition}[thm]{Definition}
\newtheorem{conj}{Conjecture} 
\newtheorem{remark}[thm]{Remark}
\newtheorem{claim}[thm]{Claim}

\newcommand{\PP}{\mathbb{P}}

\newcommand{\re}{Re}
\newcommand{\im}{Im}
\newcommand{\res}{Res}
\newcommand{\var}{\mathrm{Var}}
\newcommand{\poisson}{\mathrm{Poisson}}

\mathtoolsset{showonlyrefs}

\numberwithin{equation}{section}

\title{The critical one-dimensional multi-particle DLA}

\author{Dor Elboim}
\author{Danny Nam}
\author{Allan Sly}

\address{Department of Mathematics, Princeton University, Princeton, NJ 08544}
\email{delboim@princeton.edu, dhnam@princeton.edu, allansly@princeton.edu}

\begin{document}

\begin{abstract}
 We study  one-dimensional multi-particle Diffusion Limited Aggregation (MDLA) at its critical density $\lambda=1$. Previous works have verified that the size of the aggregate $X_t$ at time $t$ is $t^{1/2}$ in the subcritical regime and linear in the supercritical regime.  This paper establishes the conjecture that the growth rate at criticiality is $t^{2/3}$. Moreover, we derive the scaling limit proving that
    $$\{ t^{-2/3}X_{st} \}_{s\ge 0} \overset{\textnormal{d}}{\to} \Big\{ \int_0^s Z_u du \Big\}_{s\ge 0},  $$
    where the speed process $\{Z_t\}$ is a $(-1/3)$-self-similar diffusion given by $Z_t = (3V_t)^{-2/3}$, where~$V_t$ is the $\frac{8}{3}$-Bessel process.
    
    The proof shows that locally the speed process can be well approximated by a stochastic integral representation which itself can be approximated by a critical branching process with continuous edge lengths.  From these representations, we determine its infinitesimal drift and variance to show that the speed asymptotically satisfies the SDE $dZ_t = 2Z_t^{5/2}dB_t$.  To make these approximations, regularity properties of the process are established inductively via a multiscale argument.
\end{abstract}
\maketitle 

\setcounter{tocdepth}{1}
	\tableofcontents
\newpage
\section{Introduction}

\begin{figure}
	\includegraphics[width=\textwidth]{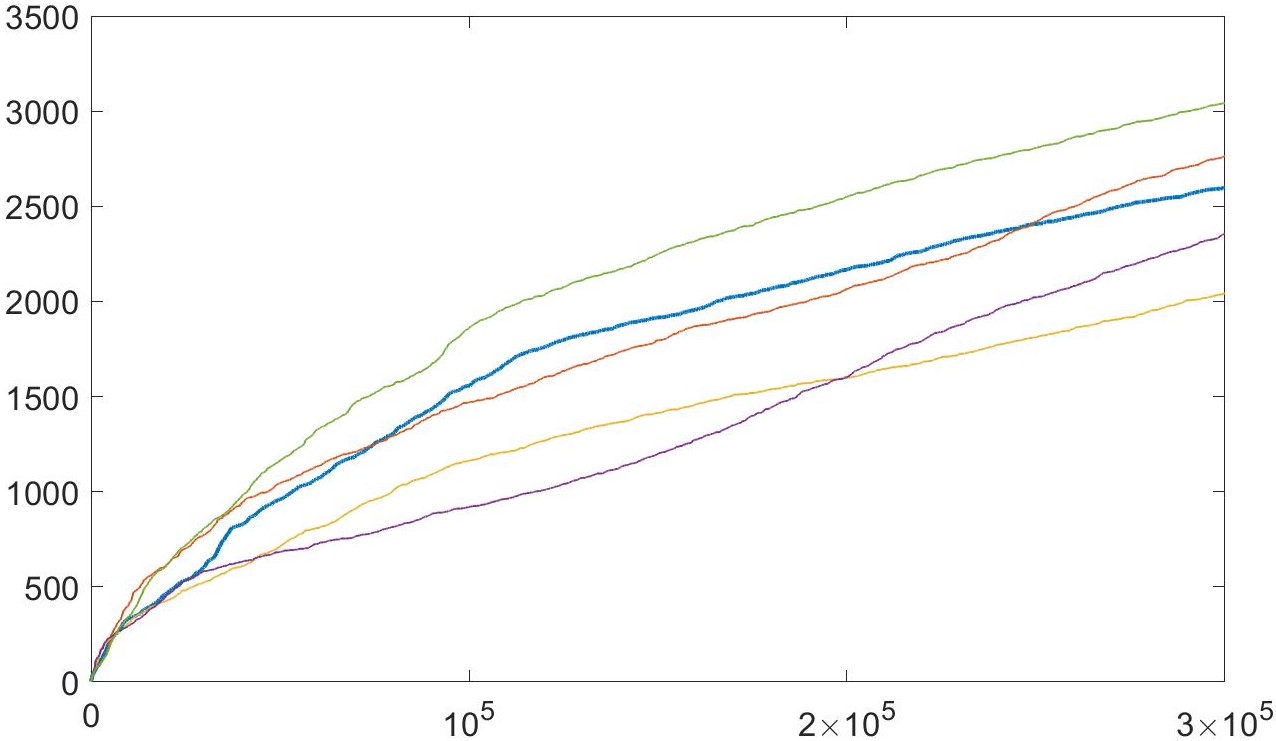}
	\hbox{\hspace{0.035em} \includegraphics[width=0.9925\textwidth]{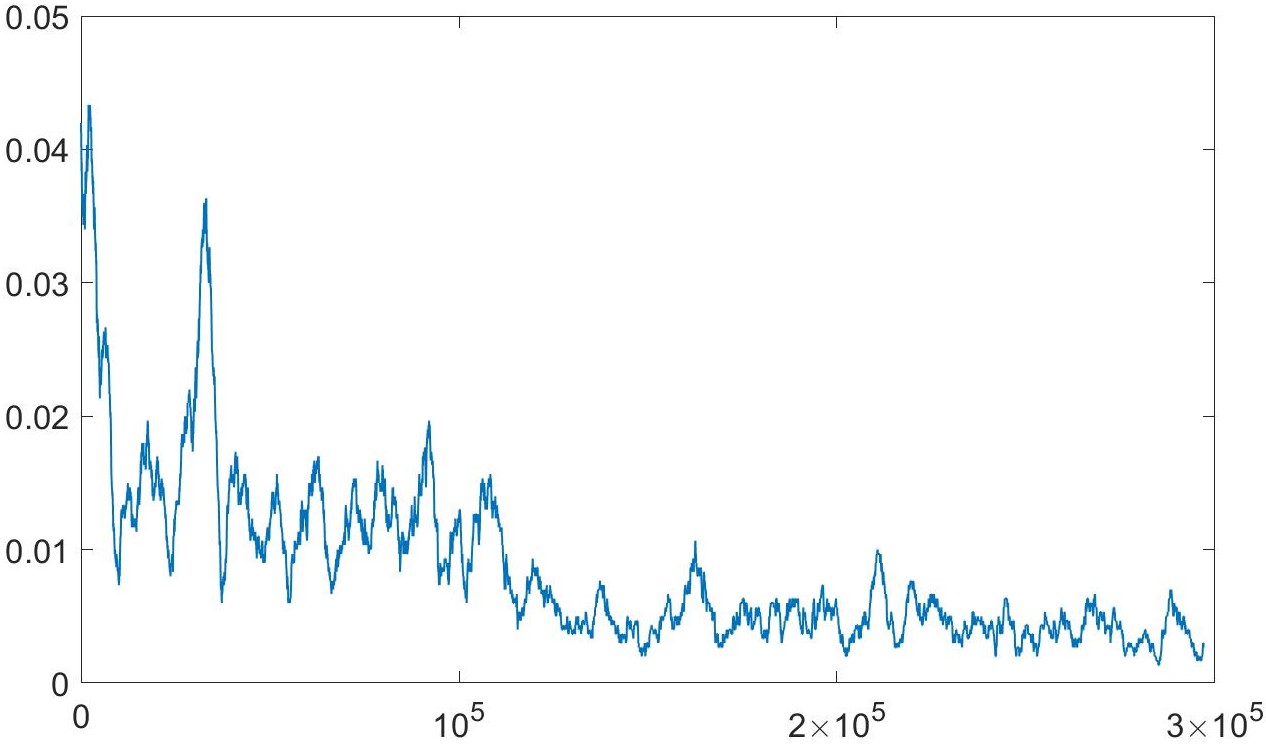}}
	\caption{We obtained the following graphs from a computer simulation of the critical $1$-d MDLA. The first picture shows the size as a function of time of 5 independent aggregates running for time $300,000$. The second picture shows the speed of growth of the blue aggregate in the first picture. Theorem~\ref{thm:main theorem 2} states that the speed converges to the solution of $dZ_t=2Z_t^{5/2}dB_t$ with initial condition $Z_0=\infty $.}\label{fig:simulation}
\end{figure}

We study a variant of the DLA model called the multi-particle Diffusion Limited Aggregation (MDLA), where a cloud of particles diffuse simultaneously before adhering to a growing aggregate. This model was first studied in the physics \cite{voss84mdfa, voss1984multiparticle, meakin98book,meakin1988multiparticle,meakin83} and later in the mathematics literature \cite{sidoravicius2019multi,sidoravicius2017one,kesten2008problem,sly2016one,dembo2019criticality,kesten2008positive}.
The model in dimension $d \ge 1$ and density $\lambda >0$ is defined as follows. For each time $t>0$ the aggregate is a set of vertices $\mathcal A _t \subseteq \mathbb Z ^d$.  Initially, $\mathcal A _0=\{0\}$ and on each vertex $v \neq 0$ there is a random number of particles distributed as $ \poisson (\lambda )$, independently of the other vertices. At time $t=0$ each particle starts to move according to a simple, continuous time random walk, independently of other particles. The aggregate $\mathcal A _t$ grows according to the following rule: If at time $t$ one of the particles at $v \notin \mathcal A_{t-} $ attempts to jump into the aggregate, it freezes in place, together with all the other particles at $v$ and the aggregate grows by $\mathcal A _t =\mathcal A_{t-}\cup \{v\}$. Frozen particles do not move for the rest of the process and cannot make the aggregate grow.

Perhaps surprisingly at first sight, only the one dimensional model is expected to exhibit a phase transition, while in higher dimensions the aggregate is expected to grow linearly for all initial densities $\lambda>0$ which has been confirmed only for large enough $\lambda$ by Sidoravicius and Stauffer~\cite{sidoravicius2019multi} and also in~\cite{sly2016one}.  From now on we focus only on the one dimensional case. In this case the aggregate is simply a line segment and the processes on the positive and negative axes are independent so we simply restrict our attention to the rightmost position of the aggregate at time $t$ which we denote $X_t$. If at time $t$ a particle at $X_{t^-} + 1$ attempts to take a step to the left, the aggregate grows by one: $X_t=X_{t^-}+1$.

Kesten and Sidoravicious \cite{kesten2008problem} proved that in the subcritical regime when $\lambda <1$, $X_t=\Theta (\sqrt{t})$ with high probability. Moreover, using the results of Dembo and Tsai \cite{dembo2019criticality}, one can show that in fact $X_t =(c_-(\lambda) +o(1))\sqrt{t}$ for an explicit $c_-(\lambda) >0$. In the supercritical regime, the third author proved the existence of a phase transition~\cite{sly2016one} by showing that when $\lambda >1$, $X_t$ grows linearly and moreover that $t^{-1}X_t \to c_+(\lambda) >0$ almost surely.

It was widely conjectured that in the critical case, when $\lambda =1$, the aggregate grows like $t^{2/3}$. Sidoravicious and Rath~\cite{sidoravicius2017one} gave a heuristic explanation for this prediction via PDEs. In \cite{dembo2019criticality}, Dembo and Tsai proved an upper bound of $ O(t^{2/3})$ by studying a modified ``frictionless'' variant of the model which stochastically dominates the aggregate (see Section~\ref{subsec:intro:relatedworks} for details). We verify the conjectured $t^{2/3}$ growth in the following theorem and determine the scaling limit.

\begin{theorem}\label{thm:main theorem 2}
	Let $V_t$ be the Bessel process with dimension $\frac{8}{3}$ given by
	\begin{equation}
	d V_t=\frac{5}{6} \frac{dt}{V_t} + d B _t, \quad V_0 = 0,
	\end{equation}
	and let $Z_t = (3V_t)^{-2/3}$. Then,
	\begin{equation}
	\big\{ t^{-\frac{2}{3}}X_{st} \big\} _{s>0} \overset{d}{\longrightarrow } \Big\{ \int _0^s Z_x dx \Big\} _{s>0},\quad t \to \infty.
	\end{equation}
\end{theorem}

In particular $ t^{-2/3}X_{t}$ has a limiting distribution which is positive almost surely.  The limiting speed process $Z_t$ is a local martingale which is $(-1/3)$-self-similar which means $\{Z_{st}\}_{s\geq 0} \stackrel{d}{=} \{t^{1/3} Z_s\}$.  While this representation is perhaps surprising, any $(-1/3)$-self-similar diffusion must be a $(-2/3)$ power of a Bessel process.  An alternative representation, which can easily be verified by It\^o's Formula, is that 
\begin{equation}\label{eq:SDE equation in main theorem}
    d Z_t = 2Z_t^{5/2} dB_t,
\end{equation}
with $Z_0=\infty.$
The initial condition $Z_0= \infty$ can be understood by the limit of the diffusion satisfying~\eqref{eq:SDE equation in main theorem} with finite initial condition $Z_0$ sent to infinity. See Sections \ref{subsubsec:outline:scaling limit} and  \ref{sec:scalinglimit} for details.

\subsection{Related models}\label{subsec:intro:relatedworks}
The MDLA model originated from the study of Diffusion Limited Aggregation (DLA). In the DLA model,  particles come from infinity one by one by a diffusion and adhere to the aggregate. The fundamental problem is determining the asymptotic growth and shape of the aggregate. The model was introduced by Witten and Sanders~\cite{witten1981diffusion}, and describes physical phenomena such as  mineral growths and electrodeposition. Although it has been studied extensively (e.g., \cite{witten83dla,meakin83,vicsek84,kesten87hitprob, kesten1987long, kesten1990upper, benjamini17}), there has been limited progress in obtaining rigorous results.

The frictionless model, which is a variant of MDLA, was introduced and successfully studied by Dembo and Tsai~\cite{dembo2019criticality}. The model is defined in one dimension as follows. Like the MDLA, there is a growing aggregate on the set of positive integers and a cloud of particles that perform a continuous time random walk but the aggregate grows by the number of particles it absorbs rather than 1.  In the supercritical case it explodes in finite time.  
The size of the aggregate is exactly equal to the number of particles absorbed, which is called in \cite{dembo2019criticality} the ``flux condition" allowed Dembo and Tsai to derive the asymptotic behavior of the model based on PDE techniques, deriving a discontinuous scaling limit. We note that the flux condition does not hold in the MDLA model but the front stochastically dominates it and thus gives an upper bound.

Another way to interpret the MDLA is to view it as a special case of a two-type particle system (so-called the $A$-$B$ model), which is sometimes used to model the spread of a rumor or a disease in a network. In the $A$-$B$ model, there are two types of particles, the $A$-type (``healthy individuals") and the $B$-type (``infected"). All the particles perform a continuous time random walk on $\mathbb Z ^d $ in which the $A$ particles jump with rate $D_A\ge 0$ and the $B$ particles jump with rate $D_B \ge 0$. When $A$ and $B$ particles coincide, the $A$ particle turns into a $B$ particle. 
Theses models were studied extensively and we refer to \cite{kesten2012asymptotic,berard2010large,comets2009fluctuations,kesten2008shape,richardson1973random,kesten2005spread}.

In the case $D_B=0$ and $D_A>0$ (that is, only the $A$ particles move) the transition rule is modified such that when an $A$ particle tries to jump into a vertex with a $B$ particle, the jump is suppressed and the $A$ particle, together with all other $A$ particles on the same vertex turn into $B$ particles. It is clear that when $D_A=1$ and when initially there is one $B$ particle in the origin this corresponds to the MDLA.

The opposite case $D_A=0$ and $D_B>0$ is sometimes referred to in the physics literature as the Stochastic combustion process and it is used to model the burning of propellant material \cite{ramirez2004asymptotic,comets2007fluctuations}. In the mathematics literature it is sometimes called the frog model \cite{alves2002phase,alves2002shape}.

\subsection{Further directions and open problems}
\subsubsection{Near critical behavior}
For the one-dimensional MDLA with initial density $\lambda $ very close to $1$, Sidoravicius and Rath \cite{sidoravicius2017one} asked to determine the critical exponent of the aggregate size as $\lambda \to 1$. To be precise, when $\lambda >1$, they predicted that the constant $c_+(\lambda) := \lim_{t\to\infty} t^{-1} X_t$  is linear in $(\lambda -1)$ as $\lambda\searrow 1$, and conjectured that $c_+(\lambda) \sim \frac{1}{2}(\lambda-1)$ from a heuristic argument counting the sites in the aggregate with more than one particle. On the other hand, for the subcritical case, they expected that the constant $c_-(\lambda):= \lim_{t \to \infty} t^{-1/2}X_t$ satisfies $ c_-(\lambda) \sim (1-\lambda)^{-1/2}.$

In a forthcoming companion paper~\cite{ens20nearcrit} building on the techniques we develop here, we give an affirmative answer to the latter conjecture, but disprove the first one by establishing
\begin{equation}\label{eq:nearcrit:super}
    \lim_{\lambda\searrow 1} \frac{c_+(\lambda)}{\lambda-1} = \frac{1}{3}, \quad \quad \lim_{\lambda\nearrow 1} (1-\lambda)^{1/2} c_-(\lambda) = 1.
\end{equation}
The heuristic calculation from \cite{sidoravicius2017one} for the slightly supercritical regime leads to a wrong answer as it assumed that the speed of the growth is concentrated as $t\to \infty$. Instead, it converges to a nondegenerate diffusion upon appropriate rescaling as $\lambda \searrow 1$. 
We remark that the second identity of \eqref{eq:nearcrit:super} can be obtained by the results of \cite{dembo2019criticality} from a completely different approach.

When $\lambda$ is very close to $1$, it is intuitively clear that when $t$ is not very large, the aggregate will look the same as the critical case $\lambda =1$. This leads to the question of how large must $t$ be (in terms of $\lambda $) in order for the model to ``feel" that it is slightly supercritical or subcritical, and is related to the rescaling of the process mentioned in the previous paragraph. In \cite{ens20nearcrit}, we show that the answer to this question is $t\asymp |\lambda -1|^{-3} $, and give a more precise description on the scaling limit of $X_t$ as follows: Let $U_t$ and $R_t$ be the solutions to the SDEs
\begin{equation}
    dU_t = 2U_t^3 dt + 2U_t^{5/2} dB_t,\quad dR_t = -2R_t^3 dt + 2R_t^{5/2} dB_t,
\end{equation}
with initial conditions $U_0 = R_0 = \infty.$ The methods from the current paper can be extended to prove that
 the aggregate size $X_t = X_t(\lambda)$ satisfies
\begin{equation}
    \big\{ (\lambda -1)^2 X_{s(\lambda -1)^{-3}}  \big\}_{s>0} \overset{d}{\longrightarrow } \Big\{ \int _0^s U_x dx \Big\}_{s>0}, \quad \lambda \searrow 1.
\end{equation}
Similarly, in the slightly subcritical case we have 
\begin{equation}
    \big\{ (1-\lambda )^2 X_{s(1-\lambda )^{-3}}  \big\}_{s> 0} \overset{d}{\longrightarrow } \Big\{ \int _0^s R_x dx \Big\}_{s> 0}, \quad \lambda \nearrow 1. 
\end{equation}

One can see that $U_t$ converges to a nondegenerate stationary process as $t\to \infty$, and a standard calculation tells us that its stationary measure is an inverse gamma distribution whose mean is $\frac{1}{3}$, implying \eqref{eq:nearcrit:super}.
On the other hand, in the subcritical case, we have almost surely that $2 \sqrt{t} R(t)\to 1$ as $t\to \infty $, since the diffusive term of the SDE becomes negligible compared to the drift as $t$ increases. This  leads to the second equation of \eqref{eq:nearcrit:super}.

\subsubsection{Different initial distribution}
One can consider the one-dimensional MDLA with an initial distribution that is not Poisson. Suppose that at time $0$, the number of particles on vertex $i$ is $Z_i$ where \begin{equation}\label{eq:initial condition}
    Z_1,Z_2,\dots \ \text{are i.i.d.}, \quad \mathbb E Z_1=\lambda , \quad \var  (Z_1) =\sigma ^2. 
\end{equation}
In case where all the particles just perform independent random walks on $\mathbb Z$ without a growing aggregate, it is clear that for all $\lambda >0$, i.i.d.~Poisson$(\lambda)$ on each point of $\mathbb Z$ is a stationary distribution for the dynamics. Thus, one might expect that the aggregate with the initial condition \eqref{eq:initial condition} grows exactly like the original MDLA with the Poisson initial profile. However, we expect that in the critical case, the aggregate grows faster than the time it takes for large intervals to approach stationarity. Therefore, the scaling limit in the critical case is expected to depend on the initial distribution of particles. More precisely, we conjecture formulate the following conjecture.

\begin{conj}
Let $X_t$ be the size of the aggregate with initial condition as in \eqref{eq:initial condition}. Then, with high probability
\begin{enumerate}
    \item 
    If $\lambda <1$, then $X_t=\Theta (\sqrt{t} )$.
    \item 
    If $\lambda >1$, then $X_t =\Theta (t)$.
    \item 
    If $\lambda =1$, then $X_t=\Theta ( t^{2/3} )$. Morover, we have
    \begin{equation}
        \big\{ t^{-\frac{2}{3}} X_{st} \big\} _{s>0} \overset{d}{\longrightarrow } \Big\{ \int _0^s Z_xdx   \Big\}_{s>0} ,\quad t \to \infty,
    \end{equation}
    where $Z_t$ is the solution to the equation
    \begin{equation}\label{eq:sde general initial condition}
        dZ_t=(4 \sigma ^2 -4)Z_t^4dt +2 \sigma  Z_t^{5/2}dB_t,
    \end{equation}
    with $Z_0=\infty $.
\end{enumerate}
\end{conj}

\begin{remark}
We elaborate several aspects of the conjecture as follows.
  \begin{enumerate}
      \item 
      Theorem~\ref{thm:main theorem 2} is compatible with this conjecture. Indeed, in the Poisson case with $\lambda =1$, we have that $\sigma ^2 =1$ and therefore the drift term vanishes.
     \item 
     For the deterministic initial condition where each vertex has exactly one particle in the beginning, we see that the diffusive term vanishes and we get a deterministic scaling limit of the form  $X_t=(c_0+o(1)) t^{2/3}$ with high probability.
     \item 
     Note that $Z_t$ from \eqref{eq:sde general initial condition} is $(-\frac{1}{3})$-self-similar and hence is a power of a Bessel process. Also, recall that the Bessel process with dimension $n \le 2$ is recurrent. Applying It\^o's formula, 
     it turns out that the corresponding Bessel process for $Z_t$ has dimension $(\frac{4}{3} + \frac{4}{3\sigma^2})$. Therefore,  $Z_t$ explodes in finite time if $\sigma ^2 \ge 2$ while it tends to $0$ if $\sigma ^2 <2$. Thus, we expect that the scaling limit of $X_t$ is differentiable when $\sigma ^2 <2$ and is not differentiable when $\sigma ^2  \ge 2$.
     \end{enumerate}  
\end{remark}

\section{Approximations of the speed and proof outline}\label{sec:outline}

In this section, we define  the speed of the process, which is our main object of study. As the speed is a complicated process by itself, we give several approximations of the speed by more tractable processes. Based on those expressions, we also give an overview of the proof of Theorem \ref{thm:main theorem 2}. 

\begin{definition}
	The speed $S(t)$ of the aggregate at time $t$ is $\frac{1}{2}$ times the conditional expectation of the number of particles at position $X_t +1$, given $\{X_s\}_{0\le s \le t}$.
\end{definition}

This corresponds to the infinitesimal jump rate of $X_t$; note that the particles at position $X_t+1$ are equally likely to jump to $X_t$ and $X_t+2$, hence the multiplication by $\frac{1}{2}$. To give a more explicit expression of $S(t)$, we follow the approach of \cite{sly2016one} and define
\begin{equation}\label{eq:def:Yt}
Y_t(s):= \begin{cases}
X_t - X_{t-s} & \textnormal{if } 0\le s \le t;\\
\infty & \textnormal{if } s>t.
\end{cases}
\end{equation}
The following observation was first made in \cite{sly2016one}, and this is the starting point of our work as well. Here, we consider a general MDLA with initial density $\lambda>0$, i.e., having i.i.d. Poisson$(\lambda)$ number of  particles at each position in the beginning.

\begin{prop}\label{prop:speed:basicdef}
	Let $\{X_t\}_{t\ge 0}$ be the 1D MDLA with density $\lambda>0$ and let $W(s)$ denote an independent continuous-time simple random walk starting at $0$. Then, we have
	\begin{equation}\label{eq:speed:basic expression}
	S(t) = \frac{\lambda}{2}\  \PP \left( \left.\sup_{s\ge 0 } \{ W(s)-Y_t(s)\} \le 0 \ 
 \right| \ Y_t\right) .
	\end{equation}
	Moreover, $\{X_t\}_{t\ge 0}$ is a Poisson process with rate $\{S(t) \}_{t\ge 0}.$
\end{prop}

In \cite{sly2016one}, \eqref{eq:speed:basic expression} was used to prove that $X_t$ grows linearly in $t$ if $\lambda>1$. The approach was to show that the speed has the property of mean reversion.  If the growth rate of $X_t$ becomes too small, say, $\alpha$, for a period of time, then  $S(t) \approx\lambda \alpha$, corresponding to a faster growth rate. However, at criticality, the same analysis would give $S(t) = \alpha(1+o(1))$, and thus understanding the lower order terms becomes necessary. Indeed, our ultimate goal is to deduce the scaling limit of the \textit{smoothed} speed by  conducting a more refined and quantitative analysis on such error terms.

In Section \ref{subsec:speed:speedproc}, we introduce the notion of the \textit{speed process} which generalizes \eqref{eq:speed:basic expression} and explain how they are approximated by the objects we can quantitatively deal with. Then, we see how the speed of the critical aggregate is described in Section \ref{subsec:speed:speedagg}. Furthermore, in Section \ref{subsec:speed:Lt}, we give an overview on how to obtain its scaling limit by using a \textit{smoothed} version. In Section \ref{subsec:speed:outline}, we describe an outline of the proof. In the final subsection, Section \ref{subsec:Kestim:intro}, we illustrate the necessary estimates on the important quantities that appear frequently in the article.

\subsection{The speed process and its approximation}\label{subsec:speed:speedproc}

In the remainder of this section, $Y$ denotes a \textit{unit-step function} which is defined as follows.

\begin{definition}[Unit-step function and its speed process]\label{def:unitstepfunc}
	$Y:\mathbb{R}_{\ge 0} \to \mathbb{N}\cup \{\infty\}$ is called a \textit{unit-step function} if there exist either $x_1 < x_2 < \ldots $ such that
	\begin{equation}
	Y(s) = \sum_{i\ge 1} I\{x_i < s\},
	\end{equation}
	or $n \in \mathbb{N}$ and $x_1<x_2<\ldots < x_n < x$ such that
	\begin{equation}
	Y(s) = \begin{cases}
	\sum_{i=1}^n I\{ x_i < s\} & \textnormal{if } s\le x;\\
	\infty & \textnormal{if } s>x.
	\end{cases}
	\end{equation}
	Note that in both cases $Y(0)=0$,  $Y$ is increasing and  left-continuous. Furthermore, the \textit{speed process} of $Y$ is defined as
	\begin{equation}\label{eq:def:speedproc}
	S(Y) : = \frac{1}{2} \ \PP\left(\left. \sup_{s\ge 0} \{W(s)-Y(s)\} \le 0 \ \right| \ Y \right),
	\end{equation}
	where $W(s)$ denotes the continuous-time simple random walk starting at $0$.
\end{definition}

We let $\PP_\alpha$ denote the probability with respect to $Y$ being a rate-$\alpha$ Poisson process.  We will proceed by considering the formula $S(Y)$ under $\PP_\alpha$.  In doing so we obtain a formula for $S(Y)$ described in Proposition~\ref{prop:speed:Dt0thorder} below which enables us to conduct a refined quantitative study on the speed process.  While our starting point is a Poisson $Y$, the formula will hold for all unit-step functions $Y$.

We define the stopping time $T$ as
\begin{equation}\label{eq:def:T:basic form}
T = T(Y) := \inf \{s>0: W(s) > Y(s) \}.
\end{equation}
and set 
\begin{equation}
U(s) := Y(s) - W(s) +1.
\end{equation} 
which is a continuous time random walk with drift under $\PP_\alpha$. Setting $\xi = (1+2\alpha)^{-1}$ we have that $\xi^{U_s}$ is a martingale, and hence 
\begin{equation}\label{eq:speed:Ut mg conv}
\xi^{U_{s\wedge T}} \overset{\textnormal{in }L^2 }{\underset{s\to \infty}{\longrightarrow}} \ \ I\{T<\infty \}.
\end{equation}

Thus, an application of the optimal stopping theorem gives
\begin{equation}
\PP_\alpha ( T<\infty) = \xi.
\end{equation}
Thus, we obtain that
\begin{equation}
\mathbb{E}_\alpha [ S(Y)] = \frac{\alpha}{1+2\alpha}.
\end{equation}

Coming back to the general case, let $Y$ be a given unit-step function,  let $\alpha>0$ and set $\xi = \frac{1}{1+2\alpha}$. We define $D_t$ to be
\begin{equation}
D_t := \mathbb{E}_\alpha \left[\left.\xi^{U_{t\wedge T}} \right| Y_{\le t}\right],
\end{equation}
where $\mathbb{E}_\alpha [ \ \cdot \ | Y_{\le t}]$ denotes the expectation over $W$ and the unit-step function obtained by the following procedure:  take the profile of $Y$ in $[0,t]$ and generate the rest beyond $t$ by a rate-$\alpha$ Poisson process.  In such a case, $\{\xi^{U_{t' \wedge T}} \}_{t'\ge t}$ forms a martingale in $t'$ which converges to $I\{T<\infty \}$ as $t'\to \infty$. Thus, we also know that
\begin{equation}\label{eq:speed:Dt is prob T finite}
D_t = \PP_\alpha ( T<\infty \ | \ Y_{\le t}).
\end{equation}
The proposition below tells us a way to understand $D_t$ from its infinitesimal differences. In what follows,  $dY$ denote the additive combination of the Dirac measures at the points of discontinuity, that is, under the notation of Definition \ref{def:unitstepfunc}, we define
\begin{equation}
dY(s) := \sum_{x_i} \delta_{x_i}(s).
\end{equation}

\begin{prop}\label{prop:speed:Dt0thorder}
	Let $\alpha, t>0$, and $ Y$ be a unit-step function that is finite on $[0,t]$. Denoting $d\widehat{Y}(s):= dY(s) - \alpha ds$, we have
	\begin{equation}\label{eq:speed:Dt0thorder}
	D_t = \frac{1}{1+2\alpha} - \frac{2\alpha}{1+2\alpha} \intop_0^t 
H_s \ d\widehat{Y}(s),
	\end{equation}
	where $H_s=H_s(Y_{\le s}, \alpha):=	\PP_\alpha ( s<T<\infty \ | \ Y_{\le s}).$
\end{prop}

\begin{proof} 
	Set $\xi = \frac{1}{1+2\alpha}$ as before. For $Y$ which is a unit-step function, we write
	$$\frac{dY(s)}{ds} = Y(s+)-Y(s). $$
	Then,  observe that
	\begin{equation}
	\begin{split}
	\lim_{\Delta \searrow 0}  \frac{1}{\Delta} \Big\{\mathbb{E}\left[D_{s+\Delta} | \mathcal{F}_s \right]
	-\mathbb{E}\left[D_s | \mathcal{F}_s \right]
	\Big\}
	&=
	\mathbb{E}_\alpha \left[\left.\xi^{U_{s\wedge T}} I\{s<T \}\left\{ (\xi-1)\frac{dY(s)}{ds} + \left(\frac{\xi+\xi^{-1}}{2}-1 \right)  \right\} \ \right| \ Y_{\le s} \right]\\
	&=
	-\frac{2\alpha}{1+2\alpha}  \left(\frac{dY(s)}{ds} - \alpha \right) \mathbb{E}_\alpha \left[\left.\xi^{U_{s\wedge T}} I\{s<T \}\ \right| \ Y_{\le s}  \right]\\
	&=
	-\frac{2\alpha}{1+2\alpha}  \left(\frac{dY(s)}{ds} - \alpha \right)
	\PP_\alpha (s<T<\infty \ | \ Y_{\le s}),
	\end{split}
	\end{equation}
	where the last identity can be deduced similarly as \eqref{eq:speed:Dt is prob T finite}: for a fixed $s>0$, 
	\begin{equation}
	\left\{ \xi^{U_{s'\wedge T}} I\{s<T \}  \right\}_{s'\ge s} 
	\end{equation}
	is a martingale in $s'$, considering $Y$ as a rate-$\alpha$ Poisson process on $[s,\infty)$. Since this converges to $I\{s<T <\infty \}$ as $s'\to \infty$, the optimal stopping theorem tells us that
	\begin{equation}
	\mathbb{E}_\alpha \left[\left.\xi^{U_{s\wedge T}} I\{s<T \}\ \right| \ Y_{\le s}  \right]
	=
	\PP_\alpha (s<T<\infty \ | \ Y_{\le s}).
	\end{equation}
	Then, we can conclude the proof by observing that $D_0 = \xi.$
\end{proof}

From Proposition \ref{prop:speed:Dt0thorder}, we can also deduce the value of $\intop_0^\infty \mathbb{E}_\alpha H_s ds$ as follows.

\begin{cor}\label{cor:speed:integralofH}
	Let $\alpha>0$. For $H_s=H_s(Y_{\le s}, \alpha)$ defined as above, we have
	\begin{equation}
	\intop_0^\infty \mathbb{E}_\alpha [H_s] ds = \frac{1}{\alpha(1+2\alpha)}.
	\end{equation} 
\end{cor}
 
 \begin{proof}
 	Let $\alpha'>0$ (different from $\alpha$), and let $Y$ be the rate-$\alpha'$ Poisson process. Recall the expression for $D_\infty(Y)$ from Proposition \ref{prop:speed:Dt0thorder} and take the expectation over $Y$:
 	\begin{equation}
 	\frac{1}{1+2\alpha'}=\mathbb{E}_{\alpha'} D_\infty(Y) = \frac{1}{1+2\alpha} - \frac{2\alpha}{1+2\alpha} \mathbb{E}_{\alpha'} \left[\intop_0^\infty H_s(Y_{\le s},\alpha) d\widehat{Y}(s)\right].
 	\end{equation}
 	Here, note that $d\widehat{Y}(s)=dY(s)-\alpha ds$ does not depend on $\alpha'$.
 	Thus, differentiating each side with $\alpha'$ and then plugging in $\alpha'=\alpha$ gives
 	\begin{equation}
 	-\frac{2}{(1+2\alpha)^2} = -\frac{2\alpha}{1+2\alpha} \intop_0^\infty \mathbb{E}_\alpha[ H_s(Y_{\le s},\alpha) ]ds,
 	\end{equation}
 	which gives the conclusion after rearranging.
 \end{proof}
 
 The above corollary motivates us to define the \textit{branching density} $K_\alpha$, which is one of the main objects that we need to understand very precisely.
 \begin{definition}[The branching density]\label{def:Kalpha:main}
 	For each $\alpha>0$, the function $K_\alpha : \mathbb{R}_{\ge 0} \to \mathbb{R}_{\ge 0}$, which we call the \textit{branching denstiy}, is defined as
 	\begin{equation}
 	K_\alpha (s ):= \alpha(1+2\alpha)\mathbb{E}_\alpha[H_s] = \alpha(1+2\alpha) \PP_\alpha (s<T<\infty).
  	\end{equation}
 \end{definition}
 
 Corollary \ref{cor:speed:integralofH} tells us that $\intop_0^\infty K_\alpha(s)ds = 1$, and hence $K_\alpha$ can be understood as a probability density function. It turns out that $K_\alpha$ describes the  density of the branch lengths of a certain branching process which locally approximates $X_t$ (see Remark \ref{rmk:speed:branchingdensity}), and hence we call it  the branching density.

Next, we introduce a similar decomposition as \eqref{eq:speed:Dt0thorder} that works for $H_s$, in order to approximate $H_s$ by $\mathbb{E}_\alpha[H_s]$, or $K_\alpha(s)$. Here, we stress that the expectation $\mathbb{E}_\alpha[H_s] = \mathbb{E}_\alpha[H_s(Y_{\le s},{ \alpha})]$ is taken  over the randomness of $Y_{\le s}$ as the rate-$\alpha$ Poisson process. For each $u>0$, define the stopping time $T_u$ to be
\begin{equation}\label{eq:def:Tu}
T_u= T_u(Y) := \inf \{s>0: Y(s)-W(s)+ I\{s>u\} <0 \}.
\end{equation} 
In other words, it is the first time when $W$ exceeds $Y$, with an additional unit jump added to $Y$ at $u$. Then, the decomposition for $H_s$ is given by the following proposition.

\begin{prop}\label{prop:speed:Dt1storder}
	Suppose that $Y$ is finite in $[0,t]$ and let $d\widehat{Y}$, $H_s$ and $T_u$ be as above. Define 
	\begin{equation}\label{eq:def:J:basic form}
	{\textnormal{\textbf{J}}}_{u,s}= {\textnormal{\textbf{J}}}_{u,s}(Y; \alpha) := \PP_\alpha (s<T_u<\infty \ | \ Y_{\le u}) - \PP_\alpha(s<T<\infty \ | \ Y_{\le u}).
	\end{equation}
	Then, we have
	\begin{equation}
	H_s = \mathbb{E}_\alpha[H_s] + \intop_0^s {\textnormal{\textbf{J}}}_{u,s} \ d\widehat{Y}(u).
	\end{equation}
	In particular, we can rewrite the formula \eqref{eq:speed:Dt0thorder} as
	\begin{equation}\label{eq:Dt:1storderexpansion}
	D_t = \frac{1}{1+2\alpha} -\frac{1}{(1+2\alpha)^2} \intop_0^t K_\alpha(s) \ d\widehat{Y}(s) - \frac{2\alpha}{1+2\alpha} \intop_0^t \intop_0^s {\textnormal{\textbf{J}}}_{u,s} \ d\widehat{Y}(u)d\widehat{Y}(s).
	\end{equation}
\end{prop}

\begin{remark}
	We use bold alphabet ${\textnormal{\textbf{J}}}_{u,s}$ to define \eqref{eq:def:J:basic form} to emphasize  its stochastic nature, since it depends on  $Y_{\le u}$ which is going to be random. This will prevent confusion from its ``averaged'' form introduced in \eqref{eq:def:J:expected val}.
\end{remark}

\begin{proof}[Proof of Proposition \ref{prop:speed:Dt1storder}]
	For $u<s$, observe that we have the following expression for an infinitesimal difference of $H_s$:
	\begin{equation}
	\lim_{\Delta \searrow 0}\frac{1}{\Delta} \Big\{ \mathbb{E}_\alpha \left[ H_s \ | \ Y_{\le u+\Delta} \right]-\mathbb{E}_\alpha \left[ H_s \ | \ Y_{\le u} \right] \Big\}
	= {\textnormal{\textbf{J}}}_{u,s} \left(\frac{d{Y}(u)}{du} - \alpha \right).
	\end{equation}
	Thus, we directly obtain the conclusion by noticing that  $H_s = \mathbb{E}_\alpha[H_s \ | \ Y_{\le s}]$.
\end{proof}

Combining \eqref{eq:speed:Dt is prob T finite}, Corollary \ref{cor:speed:integralofH} and Proposition \ref{prop:speed:Dt1storder}, we obtain the following expression for the speed process:
\begin{equation}\label{eq:speedproc:1storder main}
\begin{split}
\mathbb{E}_\alpha&[S(Y) \ | \ Y_{\le t}]
= \PP_\alpha ( T =\infty \ | \ Y_{\le t})
\\
&=\frac{\alpha}{1+2\alpha} + \frac{1}{(1+2\alpha)^2} \intop_0^t K_\alpha(s) \left\{dY(s) - \alpha ds \right\} + \frac{\alpha}{1+2\alpha}  \intop_0^t \intop_0^s {\textnormal{\textbf{J}}}_{u,s} \ d\widehat{Y}(u)d\widehat{Y}(s) \\
&= \frac{2\alpha^2}{(1+2\alpha)^2} + \frac{1}{(1+2\alpha)^2}\intop_0^t K_\alpha(s) dY(s) 
+
\frac{\alpha}{1+2\alpha}  \intop_0^t \intop_0^s {\textnormal{\textbf{J}}}_{u,s} \ d\widehat{Y}(u)d\widehat{Y}(s).
\end{split}
\end{equation}
This formula is one of the two primary formulas which lies at the heart of our work, and sometimes we refer to this as \textit{the first-order expansion of the speed process}.
 In the next subsection, we discuss in more detail that how we express the actual speed of the aggregate in the form of \eqref{eq:speedproc:1storder main}.

\begin{remark}\label{rmk:speed:branchingdensity}
	We will later see that $\intop_0^t K_\alpha(s) dY(s)$ is the leading order in the expression \eqref{eq:speedproc:1storder main}, which is of order $\alpha$ while other terms are of smaller order. Moreover, this term can be understood as a branching process, where each jump in Y at distance $s$ results in a new jump to branch out  at rate $K_\alpha(s)$. More detailed illustrations are given in the following subsections (see, e.g.,  \eqref{eq:def:Rt:basic form} and the discussion below).
\end{remark}

The second main formula gives an additional order of approximation of the speed process. In \eqref{eq:speedproc:1storder main}, it turns out that the double integral of ${\textnormal{\textbf{J}}}_{u,s}$ has a non-negligible order in the derivation of the scaling limit of (the smoothed version of) the speed. Thus, we need another layer of approximation to control its size more accurately. To this end, we first define the stopping time $T_{v,u}$ for each $u>v>0$ similarly as \eqref{eq:def:Tu}.
\begin{equation}
T_{v,u} = T_{v,u}(Y) := \inf \{s>0: Y(s)-W(s)+I\{s>v\} + I\{s>u\} < 0 \}.
\end{equation}
In other words, it is the first time when $W$ exceeds $Y$, with two additional unit jumps added to $Y$ at $v$ and $u$.

\begin{prop}\label{prop:speed:Dt2ndorder}
	Suppose that $Y$ is finite on $[0,t]$ and let $d\widehat{Y}$, ${\textnormal{\textbf{J}}}_{u,s}$, $T_{v,u}$ and $T_u$ be as above. Define
	\begin{equation}\label{eq:def:Q:basic form}
	\begin{split}
	{\textnormal{\textbf{Q}}}_{v,u,s}= {\textnormal{\textbf{Q}}}_{v,u,s}(Y;\alpha):=& \PP_\alpha(s<T_{v,u}<\infty \ | \ Y_{\le v}) -\PP_\alpha(s<T_{v}<\infty \ | \ Y_{\le v})\\
	&-\PP_\alpha(s<T_{u}<\infty \ | \ Y_{\le v})+\PP_\alpha(s<T<\infty \ | \ Y_{\le v}).
	\end{split}
	\end{equation}
	Then, we have
	\begin{equation}
	{\textnormal{\textbf{J}}}_{u,s} = \intop_0^u {\textnormal{\textbf{Q}}}_{v,u,s} \ d\widehat{Y}(v).
	\end{equation}
	In particular, the speed process admits the following expression:
	\begin{equation}\label{eq:speedproc:2ndorder main}
	\begin{split}
	\mathbb{E}_\alpha[S(Y) |  Y_{\le t}]
	=&
	 \frac{2\alpha^2}{(1+2\alpha)^2} + \frac{1}{(1+2\alpha)^2}\intop_0^t K_\alpha(s)dY(s) \\
	 &+
	 \frac{\alpha}{1+2\alpha} \intop_0^t \intop_0^s J_{u,s}^{(\alpha)}  d\widehat{Y}(u)d\widehat{Y}(s)+ \frac{\alpha}{1+2\alpha}\mathcal{Q}_t ,
	\end{split}
	\end{equation}
	where $J_{u,s}^{(\alpha)} $ is defined as 
	\begin{equation}\label{eq:def:J:expected val}
	J_{u,s}^{(\alpha)}  := \mathbb{E}_\alpha [{\textnormal{\textbf{J}}}_{u,s}],
	\end{equation}
	and $\mathcal{Q}_t$ denotes
	\begin{equation}
	\begin{split}
	\mathcal{Q}_t = \mathcal{Q}_t(Y;\alpha):=  \intop_0^t \intop_0^s \intop_0^u {\textnormal{\textbf{Q}}}_{v,u,s} \ d\widehat{Y}(v)d\widehat{Y}(u)d\widehat{Y}(s).
	\end{split}
	\end{equation}
\end{prop}

\begin{proof}
	Proof follows analogously as Proposition \ref{prop:speed:Dt1storder}, by observing that
	\begin{equation}
	\lim_{\Delta \searrow 0} \frac{d}{d\Delta} \mathbb{E}_\alpha [{\textnormal{\textbf{J}}}_{u,s} \ | \ Y_{\le v+\Delta}] = {\textnormal{\textbf{Q}}}_{v,u,s} \left(\frac{d{Y}(v)}{dv} - \alpha \right).
	\end{equation}
\end{proof}

We sometimes refer to the formula \eqref{eq:speedproc:2ndorder main} as \textit{the second-order expansion of the speed process}. Later, it turns out that the triple integral of ${\textnormal{\textbf{Q}}}_{v,u,s}$ is negligible for the purpose of obtaining the scaling limit. Moreover, the terms $K_\alpha(s)$ and $J_{u,s}^{(\alpha)} $ no longer possess dependence on $Y$ and in fact they can be studied very explicitly. These facts eventually allow us to analyze \eqref{eq:speedproc:2ndorder main} very precisely, which leads to the derivation of its scaling limit.

\subsection{The speed of the aggregate}\label{subsec:speed:speedagg}

The goal of this section is to see how the actual speed of the aggregate is described in terms of the first-  and second-order expansion. Recall the definition of $Y_t$ \eqref{eq:def:Yt} and the formula \eqref{eq:speed:basic expression} for $S(t)$. In the aggregate, a new particle is added at rate $S(t)$. In other words, if we let $\Pi$  be the Poisson point process on the plane with the standard Lebesgue intensity, we can write $X_t = |\Pi_S [0,t]|$, where
\begin{equation}
\Pi_S[0,t] = \{x \in[0,t] : \ (x,y)\in \Pi \textnormal{ for some } y \in[0, S(x)] \}.
\end{equation}
Based on this notation, we can write $dY_t(s)$ as 
\begin{equation}
dY_t(s) = d\Pi_S(t-s),
\end{equation}
which is the notation we mainly use throughout the rest of the paper. From this identity, we can recover $\{Y_t(s)\}_{s\ge 0 }$ from $\Pi_S[0,t]$ as
\begin{equation}\label{eq:def:Y:from S}
Y_t(s)= |\Pi_S[t-s,t]|.
\end{equation}

In principle, $S(t)$ can be written as \eqref{eq:speedproc:1storder main} and \eqref{eq:speedproc:2ndorder main} by plugging in $Y_t$ in the place of $Y$. However, there are several aspects that we need to keep in mind when making such a substitution.
\begin{itemize}
	\item[(a)] $Y_t$ reads off the aggregate backwards in time, and hence as $t$ increases, the jump is added at the origin as a particle sticks to the aggregate.
	
	\item[(b)]  $Y_t(s)$ is infinite for $s>t$, and on this regime Propositions \ref{prop:speed:Dt0thorder}, \ref{prop:speed:Dt1storder} and \ref{prop:speed:Dt2ndorder} are not applicable.

\end{itemize}
Furthermore, in order to give valid approximations on the speed,  we introduce several technical assumptions on the aggregate that will be justified in the proof section.
\begin{itemize}
	\item[(A1)]  Let $t_0>0$ be a large enough number. In some time interval $[t_0, t_0+\Delta]$, $S(t)$ stays ``close'' to some small enough constant $\alpha>0$. Here, one may understand $\alpha$ as the average of $S(t)$ over $t\in[t_0,t_0+\Delta]$.
	
	\item[(A2)] There exists $0<t_0^-<t_0$ such that $S(t)$ is ``sufficiently close'' to
	\begin{equation}\label{eq:def:Sprime:basic form}
	S'(t)= S'(t;t_0^-,\alpha) := \mathbb{E}_\alpha [ S(Y_t) \ | \ (Y_t)_{\le t-t_0^-} ],
	\end{equation}
 	uniformly on $t\in[t_0,t_0+\Delta],$ where we denote $(Y_t)_{\le t-t_0^-} = \{Y_t(s)\}_{s\le t-t_0^-}$ as before.
\end{itemize}
Note that we can avoid the issue raised in (b) if we work with $S'(t)$ defined in (A2). Moreover, (A1) will tell us that the first- and second-order approximations on $S'(t)$ are valid. Both assumptions (A1) and (A2) will be a part of the notion that we call \textit{regularity} of the speed. This will be discussed in more detail in Section \ref{subsec:speed:outline}.

Having the issue (a) in mind, we can express $S'(t)$ as \eqref{eq:speedproc:1storder main} and \eqref{eq:speedproc:2ndorder main} by plugging in $Y_t$, or $\Pi_S$, in the place of $Y$. We first define 
\begin{equation}\label{eq:def:S1:basic form}
S_1(t) = S_1(t;t_0^-,\alpha) := \intop_{t_0^-}^t K_\alpha(t-s) d\Pi_S (s),
\end{equation}
 Then, the first-order expansion of the speed for $t\in[t_0, t_0+\Delta]$ is given as
\begin{equation}\label{eq:speed:1storder main}
\begin{split}
S'(t) = \frac{2\alpha^2}{(1+2\alpha)^2} +
\frac{S_1(t)}{(1+2\alpha)^2}+\frac{\alpha}{1+2\alpha} \intop_{t_0^-}^t \intop_{t_0^-}^s  
{\textnormal{\textbf{J}}}_{t-s,t-u; t}^{\Pi_S} \ d\widehat{\Pi}_S(u) d\widehat{\Pi}_S(s),
\end{split}
\end{equation}
where $d\widehat{\Pi}_S(s) := d\Pi_S(s) - \alpha ds$. We also call $S_1$ \eqref{eq:def:S1:basic form} the \textbf{first-order approximation of the speed}. One important change we want to emphasize is the definition of ${\textnormal{\textbf{J}}}_{t-s,t-u;t}^{\Pi_S}$ which is slightly different from \eqref{eq:def:J:basic form}: due to the time-reversed nature of $Y_t$, it is defined as
\begin{equation}\label{eq:def:J:agg}
\begin{split}
 {\textnormal{\textbf{J}}}_{t-s,t-u;t}^{\Pi_S} &= {\textnormal{\textbf{J}}}_{t-s,t-u;t}^{\Pi_S}(\alpha)\\
 &:= \PP_\alpha(t-u<T_{t-s}<\infty \ | \ \Pi_S[s,t])- \PP_\alpha(t-u<T<\infty \ | \ \Pi_S[s,t]).
\end{split}
\end{equation} 
From \eqref{eq:def:Y:from S}, note that conditioning on $\Pi_S[s,t]$ is equivalent to conditioning on $\{ Y_t(s')\}_{0 \le s'\le t-s}$.

Further, the second order expansion is written as
\begin{equation}\label{eq:speed:2ndorder main}
\begin{split}
S'(t) = \frac{2\alpha^2}{(1+2\alpha)^2} +
\frac{S_1(t)}{(1+2\alpha)^2}+\frac{\alpha}{1+2\alpha} \intop_{t_0^-}^t \intop_{t_0^-}^s  
J_{t-s,t-u}^{(\alpha)}  \ d\widehat{\Pi}_S(u) d\widehat{\Pi}_S(s) + \frac{\alpha}{1+2\alpha} \mathcal{Q}_t,
\end{split}
\end{equation}
where we define $J_{t-s, t-u}^{(\alpha)} $ as \eqref{eq:def:J:expected val} and  $\mathcal{Q}_t$ as
\begin{equation}\label{eq:def:Qt:agg}
	\begin{split}
	\mathcal{Q}_t = \mathcal{Q}_t(t_0^-, \alpha):=  \intop_{t_0^-}^t \intop_{t_0^-}^s \intop_{t_0^-}^u {\textnormal{\textbf{Q}}}_{t-s,t-u,t-v;t}^{\Pi_S} \ d\widehat{\Pi}_S(v)d\widehat{\Pi}_S(u)d\widehat{\Pi}_S(s),
	\end{split}
\end{equation}
with ${\textnormal{\textbf{Q}}}_{t-s,t-u,t-v;t}^{\Pi_S}$ given by
\begin{equation}\label{eq:def:Q:agg}
\begin{split}
{\textnormal{\textbf{Q}}}_{t-s,t-u,t-v;t}^{\Pi_S}=&{\textnormal{\textbf{Q}}}_{t-s,t-u,t-v;t}^{\Pi_S}(\alpha)
\\
:=& \PP_\alpha(t-v<T_{t-s,t-u}<\infty \ | \ \Pi_S[s,t])- \PP_\alpha(t-v<T_{t-s}<\infty \ | \ \Pi_S[s,t])\\
&- \PP_\alpha(t-v<T_{t-u}<\infty \ | \ \Pi_S[s,t])+\PP_\alpha(t-v<T<\infty \ | \ \Pi_S[s,t]).
\end{split}
\end{equation}
Further, we define the \textbf{second-order approximation of the speed} as
\begin{equation}\label{eq:def:S2:basic form}
S_2(t) = S_2(t;t_0^-,\alpha) := \frac{2\alpha^2}{(1+2\alpha)^2} + \frac{S_1(t)}{(1+2\alpha)^2} + \frac{\alpha}{1+2\alpha}\intop_{t_0^-}^t \intop_{t_0^-}^s J_{t-s,t-u}^{(\alpha)}  d\widehat{\Pi}_S(u) d\widehat{\Pi}_S(s).
\end{equation}

 \begin{remark}
 	We will need both the first- and second-order approximations \eqref{eq:def:S1:basic form}, \eqref{eq:def:S2:basic form} in order to understand the speed precisely enough for our purpose. Although the first-order approximation  often turns out to be enough in many cases, the second-order approximation is essential in the derivation of the scaling limit of the speed. See Sections \ref{subsec:speed:Lt}, \ref{subsubsec:outline:reg} and \ref{subsubsec:outline:EJ} for more discussion.
 \end{remark}

Suppose that we define a process $R(t)$ by 
\begin{equation}\label{eq:def:Rt:basic form}
R(t) =R(t;t_0^-,t_0,\alpha):= \begin{cases}
\alpha & \textnormal{if } t < t_0,\\
\intop_{t_0^-}^t K_\alpha(t-s) d\Pi_{R}(s ) & \textnormal{if } t \ge t_0.
\end{cases}
\end{equation}
We can interpret $R(t)$ as an age-dependent critical branching process: the average number of offspring that a particle at $t>t_0$ produces is $\intop_t^{\infty} K_\alpha(t'-t) dt = 1$. In this branching process, a particle at $t>t_0$ generates Pois$(1)$ number of offspring, and each of them will be located at $t+s$ with probability $K_\alpha(s)ds$ independently. 

The critical branching process $R(t)$ is easier to understand than $S(t)$, and our  study on $S(t)$ largely depends on  comprehending $R(t)$. A more detailed outline of the plan is discussed in Section \ref{subsec:speed:outline}.

\subsection{Smoothing the speed and the scaling limit}\label{subsec:speed:Lt}

In this subsection, we give the heuristics on how the speed converges to its scaling limit. Recall the critical branching process $R(t)$ defined in the previous subsection. Since its branching is age-dependent, $R(t)$ itself is not a martingale in $t\ge t_0$ (unlike the discrete time critical Galton-Watson branching process where the number of particles at each generation forms a martingale). To build a martingale based on $R(t)$, we first define the following \textit{renewal process} of $K_\alpha$:
\begin{equation}\label{eq:def:Kstar}
K_\alpha^*(s) = \sum_{j\ge 1 } K_\alpha^{*j}(s),
\end{equation}
where $K_\alpha^{*j}$ is the convolution of $K_\alpha$ with itself taken $j$ times. At the moment, assume that $K_\alpha^*$ is well-defined, and moreover  that the limit
\begin{equation}
K_\alpha^* = \lim_{s\to \infty} K_\alpha^*(s)
\end{equation}
exists. These facts are explained in Section \ref{subsec:Kestim:intro}.
Then, we consider a  ``smoothed'' version of the speed, which is defined for $t \ge t_0$ as
\begin{equation}\label{eq:def:Lt:basic form}
L(t) = L(t;t_0^-,\alpha) := \intop_{t_0^-}^t \intop_t^\infty K_\alpha^* \cdot K_\alpha(x-s) dx d\Pi_S(s).
\end{equation}
In words, $L(t)$ is the expected branching rate at infinity when we start the critical branching process from time $t$ with initial points given by $\Pi_S[t_0^-,t]$. For details  on this process, see Section \ref{sec:criticalbranching}.

We choose to study the scaling limit of $L(t)$ rather than $S(t)$, since it has a more tractable structure described as a martingale added with some perturbation. On the other hand, the behavior of $S(t)$ has a very spiky nature: $S(t)$ jumps significantly whenever a new point arrives at $\Pi_S$. Furthermore, our ultimate interest is to obtain $X_t$, which is closely related with the integral of $S(t)$. From this perspective, we would not lose anything from studying $L(t)$ since it is essentially a smoothed version of $S(t)$.

In order to obtain the scaling limit of $L(t)$, we need to understand the mean and variance of the increment $L(t_0+\Delta)-L(t_0)$ (conditioned on $\mathcal{F}_{t_0})$. Observe that
\begin{equation}\label{eq:Lt diff:basic}
dL(t) = K_\alpha^* d\Pi_S(t) - \intop_{t_0^-}^t K_\alpha^* \cdot K_\alpha(t-s) d\Pi_S(s) dt = K_\alpha^* \left\{d\Pi_S(t) - S_1(t)dt \right\}.
\end{equation}
which gives
\begin{equation}\label{eq:Lt:increment}
\begin{split}
L({t_0+\Delta}) - L({t_0}) = 
 K_\alpha^* \intop_{t_0}^{t_0+\Delta} \big(d\Pi_S(t) - S(t)dt\big) + K_\alpha^* \intop_{t_0}^{t_0+\Delta} \big(S(t)-S_1(t)\big)dt.
\end{split}
\end{equation}
As seen in the formula, we can decompose this increment into a ``martingale part'' and a ``drift part''. 
Computations of its mean and variance is  the most important ingredient in establishing the main theorem. To this end, we need to assume that the processes $S(t)$ and $S_1(t)$ behaves nicely enough, otherwise the terms such as $S(t)-S_1(t)$ might become too large and the increment would not scale as expected. At the moment, we state this assumption on \textit{regularity} informally as follows.
\begin{itemize}
	\item[(A3)] In addition to (A1) from Section \ref{subsec:speed:speedagg}, the processes $S(t)$ and $S_1(t)$ both stay ``close'' to $\alpha$ and to each other in the interval $[t_0, t_0+\Delta]$.
\end{itemize}
Formulating  (A1), (A2) and (A3) into a rigorous, quantitative condition of \textit{regularity} and proving it  are two of the core endeavours of this article, and the methods are  outlined in the next subsection. 
Under these assumptions, we can state the estimates on the increment informally as follows.

\begin{thm}[Informal]\label{thm:speed:increment:informal}
	Suppose that (A1), (A2) and (A3) are satisfied in the interval $[t_0, t_0+\Delta]$ for a sufficiently small $\alpha>0$. Then, we have
	\begin{equation}\label{eq:increment:informal}
	\begin{split}
	\mathbb{E} [L(t_0+\Delta)- L(t_0) \ | \ \mathcal{F}_{t_0}] &= o(\alpha^4)\Delta;\\
	\var[L(t_0+\Delta)- L(t_0) \ | \ \mathcal{F}_{t_0}]
	&= (1+ o(1))4\alpha^5 \Delta.
	\end{split}
	\end{equation}
\end{thm}

\begin{proof}[Sketch of proof]
	Recall the expression \eqref{eq:Lt:increment}. We can first see that
	\begin{equation}
	\mathbb{E} [L(t_0+\Delta)-L(t_0) \ | \ \mathcal{F}_{t_0}] = \mathbb{E} \left[\left.K_\alpha^* \intop_{t_0}^{t_0+\Delta} (S(t)-S_1(t)) dt\ \right| \ \mathcal{F}_{t_0} \right].
	\end{equation}
	Later, we will see that $\int_{t_0}^{t_0+\Delta}|S(t)-S_1(t)|dt = o(\alpha^2\Delta)$ when $S(t)$ is \textit{regular}. Moreover, it turns out that (see Section \ref{subsec:Kestim:intro})
	\begin{equation}
	K_\alpha^* = \left(\intop_0^\infty xK_\alpha(x)dx \right)^{-1} = \frac{2\alpha^2}{1+2\alpha}.
	\end{equation}
	Thus we obtain the first equation of \eqref{eq:increment:informal}. For the second one, we will show that
	\begin{equation}
	\begin{split}
	\var[L(t_0+\Delta)-L(t_0)\ | \ \mathcal{F}_{t_0}] &\approx \var\left[\left.
	K_\alpha^* \intop_{t_0}^{t_0+\Delta} (d\Pi_S(t)- S(t)dt)
	\ \right| \ \mathcal{F}_{t_0} \right]\\
	&=
	(K_\alpha^*)^2\mathbb{E}\left[\left. \intop_{t_0}^{t_0+\Delta} S(t)dt \ \right| \ \mathcal{F}_{t_0} \right] = (1+o(1))4\alpha^5 \Delta.
	\end{split}
	\end{equation}
 Note that the last equality will follow from the assumption (A1).
\end{proof}

To obtain the scaling limit of $L(t)$ from the above, we need to rescale $L(t)$ by $\alpha^{-1}$, and then we can see that the time $\Delta$ should be rescaled by $\alpha^{-3}$. In fact, roughly speaking, we will later see that
\begin{equation}\label{eq:conv to sde:informal}
\left\{M^{1/3} L(t_0 + Ms)\right\}_{s\ge 0} \underset{M\to\infty}{\longrightarrow} \{Z_s\}_{s\ge 0},
\end{equation}
where $\{Z_s\}_{s\ge 0}$ satisfies the following stochastic differential equation:
\begin{equation}\label{eq:sde:main}
dZ_s = 2Z_s^{5/2} dB_s.
\end{equation}
Here, $t_0$ can be understood as a large time parameter where we start to see $S(t)$ being \textit{regular} with high enough probability. A more detailed description is given in Section \ref{subsubsec:outline:scaling limit}.

Note that $Z_t$ given by \eqref{eq:sde:main} is $(-\frac{1}{3})$-self-similar: $Z_t' := s^{1/3} Z_{st}$ has the same law as $Z_t$ for any $s>0$. This suggests that the decay of $Z_t$ (and hence, $S(t)$) in a larger time scale is asymptotically $t^{-1/3}$, and consequently the asymptotic growth of $X_t$ should be of order $t^{2/3}$.

\begin{remark}\label{remark:on bessel}
	Since the only $\frac{1}{2}$-self-similar processes are Bessel processes, $(Z_t)^{-3/2}$ should be a constant multiple of a Bessel process. One can see by applying It\^o's formula that $Z_t = (3V_t)^{-2/3}$, where $V_t$ is the $\frac{8}{3}$-Bessel process. 
\end{remark}

\subsection{Proof outline}\label{subsec:speed:outline}

Before delving into the actual proof, we devote this subsection to  outlining the proof step by step. The following describes the main steps that we go through in order.
\begin{itemize}
	\item  $O(t^{3/4+\epsilon})$ upper bound and reduction to the fixed rate processes.

	\item Analysis of the fixed rate process and its perturbation.
	
	\item The age-dependent critical branching process.
	
	\item Regularity of the speed.
	
	\item Refined computation of the lower order terms.

	\item Scaling limit of the speed. 
\end{itemize}

\noindent In the following subsections, we explain the purpose of each step and the main ideas to achieve it.

\subsubsection{$O(t^{3/4+\epsilon})$-upper bound and reduction to the fixed rate processes}\label{subsubsec:outline:simpler upper bound}
All the discussions from the previous subsections assumes $\alpha>0$ to be small enough. In order to apply such arguments, we at least need to know that the speed tends arbitrarily close to $0$ in finite time. 

We obtain this by a simpler but weaker argument that gives an $O(t^{3/4+\epsilon})$-upper bound on $X_t$. This is done by counting the number of \textit{excess particles}, which is the difference between the total number of particles attached to the aggregate up to time $t$ and the size $X_t$ of the aggregate. 

Then, we relate $S(t)$ with the fixed rate processes as follows: since $S(t)$ drops to an arbitrary low scale in finite time, we show that $S(t)$ can be bounded above and below by arbitrary small constants for a while. Thus, our later analysis focuses on the study of the speed process generate by those fixed constants, which will converge to the same scaling limit.

 \subsubsection{Analysis of the fixed rate process and its perturbation}\label{subsubsec:outline:fixedrate}
 In studying \eqref{eq:speed:1storder main} and \eqref{eq:speed:2ndorder main}, bounding the sizes of double and triple integrals in the expressions 
 turns out to be essential, because they are the error terms of the first- and second-order approximations (equations \eqref{eq:def:S1:basic form} and~\eqref{eq:def:S2:basic form}), respectively. However, these integrals are in terms of $d\widehat{\Pi}_S$, which is difficult to understand since we do not have a good understanding on $S(t)$. Furthermore, for instance the term ${\textnormal{\textbf{J}}}_{t-u, t-s; t}^{\Pi_S}$ in \eqref{eq:speed:1storder main} is only measurable with respect to $\mathcal{F}_t$, meaning that the double  integral in \eqref{eq:speed:1storder main} would not have a martingale structure.  Instead, we choose to work with the integrals over the fixed rate Poisson process, namely,
 \begin{equation}\label{eq:speed:fixedrate:med}
  \intop_{0}^t \intop_{0}^s  
  {\textnormal{\textbf{J}}}_{u,s} \ d\widehat{\Pi}_\alpha(u) d\widehat{\Pi}_\alpha(s),
 \end{equation}
 upon reversing the time axis, and the corresponding analogue for $\mathcal{Q}_t$. Here, $\Pi_\alpha$ denotes the rate-$\alpha$ Poisson process. A precise form of \eqref{eq:speed:fixedrate:med} will be given in Section \ref{sec:fixedrate}. These integrals can be understood using Azuma-type concentration inequalities for martingales. 
 In addition, we also need an appropriate estimate on ${\textnormal{\textbf{J}}}_{u,s}$ (and ${\textnormal{\textbf{Q}}}_{v,u,s}$). Recalling its definition in \eqref{eq:def:J:basic form}, we couple the events $\{s<T_u<\infty \}$ and $\{s<T<\infty \}$ and rely on fundamental hitting time estimates of random walk. The same but more complicated coupling argument can be carried out for ${\textnormal{\textbf{Q}}}_{v,u,s}$ as well.
 
 To translate the estimates for the fixed rate process to those for $S(t)$, we analyze the Radon-Nykodym derivative of $d\Pi_S$ with respect to $d\Pi_\alpha$. We show that it does not become too big with high probability under an appropriately mild assumption on $S(t)$ which is included in the definition of regularity. On the other hand, the double integral over $d\widehat{\Pi}_S(u) d\widehat{\Pi}_S(s)$ would stay small with high probability if the Radon-Nykodym derivative is smaller compared to the reciprocal of the probability that \eqref{eq:speed:fixedrate:med} becomes too large.  Details will be discussed in Section \ref{sec:fixedrate}.

\subsubsection{The age-dependent critical branching process}\label{subsubsec:critical branching} 
As discussed in Section \ref{subsec:speed:speedagg}, studying the speed $S(t)$ directly is rather complicated. Thus, we work with the age-dependent critical branching process $R(t)$  \eqref{eq:def:Rt:basic form} in Section \ref{sec:criticalbranching} as an intermediate step to understand the speed, which is particularly useful in understanding the evolution of the fixed rate process from Section \ref{subsubsec:outline:simpler upper bound}. 

The main property of $R(t)$ we deduce here is that it stays close to $\alpha$ by the right distance, roughly of order $\alpha^{\frac{3}{2}}$. We obtain this by applications of martingale concentration inequalities and relating it with the standard Galton-Watson critical branching processes.

\subsubsection{Regularity of the speed}\label{subsubsec:outline:reg}

As discussed in the previous subsection, a notion of regularity is introduced to deduce Theorem \ref{thm:speed:increment:informal} and thus to obtain the scaling limit of the speed. Its definition provides a quantitative description of the assumptions (A1), (A2) and (A3) from Sections \ref{subsec:speed:speedagg} and~\ref{subsec:speed:Lt}, and is given by various stopping times: Each stopping time will denote the first time when a certain desired property is violated, and $S(t)$ will be called regular at time $t$ if the minimum of all those stopping times is larger than $t$ by a certain amount. For instance, the first times when $S(t)$, $|S(t)-S_1(t)|$ become larger than they are supposed to be are included in the definition. In Section \ref{sec:reg:intro}, we give a thorough overview on the notion of regularity and explain why such conditions need to be understood.

The main purpose of defining the regularity using stopping times is to effectively understand how all the conditions we impose are intertwined with each other. If $\tau = \min_i \tau_i < t'$, then there should exist a $j$ such that $\tau_j = \tau <t'$, meaning that the property described by $\tau_j$ is violated for the first time among all others. Thus, we can show that $\tau\ge t'$, that is, the speed is regular with high probability, by proving $\tau_i=\tau<t'$ happens with very small probability for every $i$. In Section \ref{sec:reg:conti of reg}, we explain the details of this argument.

Another major issue is the variation of the frame of reference $\alpha$, to complete building the inductive framework of the argument. After we prove that the speed is regular from time $t_0$ to $t_1=t_0+\Delta$ with respect to $\alpha$, we want to argue that the speed continues to be regular in $[t_1,t_1+\Delta']$ with high probability, with respect to a different value $\alpha'$ to take account of the change of $L(t)$ during the previous time interval. This means that we generate the processes $S_1(t)$ and $S_2(t)$ with $\alpha'$, and show that they still satisfy the conditions of regularity as long as the change from $\alpha$ to $\alpha'$ is small. To this end, we require a refined understanding  on analytic properties of $K_\alpha(s)$, $K_\alpha^*(s)$ and $J_{s,t}^{(\alpha)}$, as well as a combination of various tools mentioned in the previous subsections. 

The computations to establish the regularity turn out to be long and complicated, and hence we divide their details into three sections, Sections \ref{sec:reg:intro}, \ref{sec:reg:conti of reg} and \ref{sec:reg:next step}.

\subsubsection{Refined computation of the lower order terms}\label{subsubsec:outline:EJ}
In the computation of the mean of the increment in Theorem \ref{thm:speed:increment:informal}, we need to understand the mean of $S(t)-S_1(t)$, which can be expressed, by using \eqref{eq:speed:2ndorder main}, as
\begin{equation}\label{eq:speed:SminusS1:2nd order:basic}
\begin{split}
S(t)-S_1(t) =& 
\frac{2\alpha^2}{(1+2\alpha)^2} +
\frac{4(\alpha+\alpha^2)}{(1+2\alpha)^2}S_1(t)\\
&+\frac{\alpha}{1+2\alpha} \intop_{t_0^-}^t \intop_{t_0^-}^s  
J_{t-s,t-u}^{(\alpha)}  \ d\widehat{\Pi}_S(u) d\widehat{\Pi}_S(s) + \frac{\alpha\mathcal{Q}_t}{1+2\alpha} .
\end{split}
\end{equation}
Although we have derived an appropriate bound on the double integral part which is essentially sharp, it is an estimate that holds with high probability that does not give the precise information on the mean
\begin{equation}\label{eq:speed:mean of doubleint}
\mathbb{E} \left[ \intop_{t_0^-}^t \intop_{t_0^-}^s  
J_{t-s,t-u}^{(\alpha)}  \ d\widehat{\Pi}_S(u) d\widehat{\Pi}_S(s)  \right].
\end{equation}

To precisely caculate \eqref{eq:speed:mean of doubleint}, we switch the  integral over $d\widehat{\Pi}_S(u)d\widehat{\Pi}_S(s)$ with simpler ones that only involves $K_\alpha^*(x)$ and $\Pi_\alpha$,
and show that the error arising from the change is of negligible compared to the whole integral. After such modification, we carry out an analytic study to compute the precise expected value.

Combining with the previous analysis on regularity, we also discuss the formal version of Theorem \ref{thm:speed:increment:informal}.

\subsubsection{Obtaining the scaling limit}\label{subsubsec:outline:scaling limit}

Given all the works mentioned in Sections \ref{subsubsec:outline:simpler upper bound}--\ref{subsubsec:outline:EJ}, we are almost ready to deduce the scaling limit of $L(t)$ based on Theorem \ref{thm:speed:increment:informal}. However, when we consider the process $M^{1/3} L(t_0 + sM)$, the initial condition at $s=0$ blows up as $M$ tends to infinity, if $t_0$ is a fixed value. The classical results on the convergence to the limiting diffusion (which we mention in the next subsection) would not be applicable in such a case.

Thus, our choice is to make $t_0$ be a random time $\tau(M)$ depending on $M$, which is roughly
\begin{equation}
\tau(M):= \inf\{t>0: L(t) \le CM^{-1/3} \textnormal{ and }L(t) \textnormal{ is regular}  \},
\end{equation}
where $C$ is a fixed, large enough constant. Then, we need to show that $\tau(M) \le \epsilon M$ with high probability, with a constant $\epsilon>0 $ that can be arbitrarily small by setting $C$ to be large, and this is equivalent to show that $X_t = O(t^{2/3})$ with high probability.

This is done by a multi-scale analysis on $L(t)$ based on Theorem \ref{thm:speed:increment:informal}. To ensure it is regular, we first drop $L(t)$ to a small enough scale using the results from Section \ref{subsubsec:outline:simpler upper bound}. From then, we show that  $L(t)$ is more likely to halve than to double its size. Further, we argue that either event is going to happen after an appropriate amount of time 
 that makes $L(t)$ scale like $t^{-1/3}$. The main technical tools we use to achieve this goal are the martingale concentration inequalities such as the Azuma-Bernstein inequality.

When the above issue on the initial regularity  is resolved, we can appeal to the classical result of Helland \cite{helland1981minimal} and Theorem \ref{thm:speed:increment:informal} to show that
\begin{equation}\label{eq:speed:conv to sde:outline}
\left\{ M^{1/3} L(\tau(M) + sM) \right\}_{s\ge 0} \underset{M\to \infty}{\longrightarrow} \{A_s\}_{s\ge 0},
\end{equation}
where $\{A_s\}_{s\ge 0}$ is the solution of 
\begin{equation}
dA_s = 2A_s^{5/2} dB_s, \quad A_0 = C.
\end{equation}
The convergence  \eqref{eq:speed:conv to sde:outline} is in Stone topology, which is good enough to ensure the convergence of the integrated process
\begin{equation}
\left\{ M^{-2/3} \intop_{0}^{Mt} L(\tau(M)+s)ds\right\}_{t\ge 0}
\underset{M\to \infty}{\longrightarrow} \left\{ \intop_0^t A_s ds \right\}.
\end{equation}
Finally, the main result is obtained by increasing the level of $C$ to infinity, and then observing that $\intop_0^{Mt} L(s)ds$ is close enough to $X_{Mt}$ based on the conditions of regularity.

\subsection{Estimates on fundamental quantities}\label{subsec:Kestim:intro}

As mentioned before, understanding the analytic properties of the branching density $K_\alpha(s)$ and the renewal process $K_\alpha^*(s)$ will be crucial throughout the paper. In this subsection, we review the main estimates we need, while their proofs are deferred to Appendix \ref{sec:fourier and renewal}. They can be obtained by studying  the moment generating function of $T$ via Fourier analysis.

To begin with, the branching density $K_\alpha(s)$ (Definition \ref{def:Kalpha:main}) satisfies the following property.

\begin{lem}\label{lem:estimate on K:intro}
	There exist  absolute constants $c, C>0$ such that for any $t>0$ and $0<\alpha<\frac{1}{2}$,
	\begin{equation}
	K(t) \le \frac{C\alpha }{\sqrt{t+1}}e^{-c \alpha ^2 t}.
	\end{equation}
	Furthermore, its derivative in $t$ satisfies
	\begin{equation}
	|K_\alpha'(t)| \le C\alpha (t+1)^{-\frac{3}{2}} e^{-c\alpha^2 t}.
	\end{equation}
\end{lem}

Moreover, the properties of $K_\alpha^*(s)$ \eqref{eq:def:Kstar} are summarized by the following lemma.
\begin{lem}\label{lem:estimat for K tilde:intro}
	Let $K^*_\alpha:= \frac{2\alpha^2}{1+2\alpha}$. There exist  absolute constants $c, C>0$ such that for any $t>0$ and $0<\alpha<\frac{1}{2}$, 
	\begin{equation}
	\left|K^*_\alpha (t) - K^*_\alpha \right| \le \frac{C\alpha}{ \sqrt{t+1}}e^{-c\alpha^2 t}.
	\end{equation}
	Furthermore, its derivative in $t$ satisfies
	\begin{equation}
	|(K_\alpha^*)'(t)| \le C\alpha (t+1)^{-\frac{3}{2}} e^{-c\alpha^2 t}.
	\end{equation}
\end{lem}

Finally, we state a similar bound on $J_{s,t}^{(\alpha)}$ \eqref{eq:def:J:expected val}.
\begin{lem}\label{lem:bound on deterministic J}
	There exist absolute constants $c,C>0$ such that for any $t>s>0$ and $0<\alpha<\frac{1}{2}$, 
	\begin{equation}
	\left|J_{s,t}^{(\alpha)} \right| \le \frac{Ce^{-c\alpha^2 t}}{\sqrt{(s+1)(t+1)}}.
	\end{equation}
\end{lem}

\section{An a priori upper bound and the induction base}\label{sec:inductionbase}

The first step of our proof is establishing that the speed drops down to arbitrary small level in finite time. In Section \ref{subsec:simpleupbd}, we achieve this by showing an $t^{\frac{3}{4}}$polylog$(t)$-bound on $X_t$ which follows from counting the excessive particles in the aggregate. Based on this result,  we prove in Sections \ref{subsec:sandwich1}, \ref{sec:Sandwich} that when $t$ gets sufficiently large (but finite), the speed stays between arbitrary small constants for a long enough time. Such a sandwiching  argument enables us to analyze the speed process that evolves from particles given by a fixed rate process in the later sections.

\subsection{Simple upper bound of $t^{\frac{3}{4}}$}\label{subsec:simpleupbd}

\begin{lem}\label{lem:three quarter upper bound}
	For any $t>0$ we have that
	\begin{equation}
	\mathbb P (X_t\ge t^{\frac{3}{4}}\log ^{2}t) \le C t^{-10}
	\end{equation}
\end{lem}

\begin{proof}
	Denote by $N_t$ the total number of particles that were absorbed into the aggregate by time $t$. We start by bounding $N_t$. To this end, note that the event $\mathcal A:=\{  X_t \le t\}$ satisfies $\mathbb P (\mathcal A )\ge 1-e^{-ct}$. Indeed, we have with probability one that $ S(x)\le 1/2$ for all $x>0$ and therefore $ X_t\precsim \text{Poisson}(t/2)$. 
	
	Next, for any $l \ge 0$, define the random variable $Y_l$ to be the number of particles that at some point, before time $t$,  were inside the interval $[0,l]$. Define the event $\mathcal B _l :=\{ N_l \le  l+\sqrt{t}\log ^2 t  \}$. It is clear that on the event 
	\begin{equation}
	    \mathcal B := \mathcal A \cap \bigcap _{l=1}^{\lceil t \rceil }  \mathcal B _l
	\end{equation}
     we have that 
     \begin{equation}\label{eq:bound on N}
         N_t \le X_t + \sqrt{t} \log ^ 2 t.
     \end{equation} 
     Indeed, a particle that was absorbed into the aggregate before time $t$ was, at some point, inside the interval $[0,X_t]$.
     
     We turn to show that $\mathcal B $ happens with high probability. For any $
     l \le t $, let $\mathcal D _l $ be the event that any particle which was initially outside the interval $[0, l +\frac{1}{2}\sqrt{t} \log ^2 t  ]$ did not reach the interval $[0,l ]$ and let $\mathcal E _l $ be the event that the number of particles initially in the interval $[0, l + \frac{1}{2} \sqrt{t} \log ^2 t  ]$ is at most $l + \sqrt{t}\log ^2 t $. It is clear that $\mathcal D _l \cap \mathcal E _l \subseteq \mathcal B _ l  $ and that $\mathbb P (\mathcal D _l )\ge 1-Ce^{-c \log ^2 t }$. Since the number of particles that were initially in $[0, l + \frac{1}{2} \sqrt{t} \log ^2 t  ]$ is distributed by $\text{Poisson} (l + \lfloor \frac{1}{2} \sqrt{t} \log ^2 t \rfloor )$ we have for any  $l \le t$ that $\mathbb P ( \mathcal E _l) \ge 1- C e^{-c \log ^2 t }$. Thus we have that $\mathbb P ( \mathcal B _l) \ge 1- C e^{-c \log ^2 t }$ and therefore, from a union bound we get $\mathbb P ( \mathcal B ) \ge 1- C e^{-c \log ^2 t }$.
     
     Next, let $Z_t $ be the number of particles at $X_t+1$ at time $t$, and recall that the conditional distribution of $Z_t$ given $\{X_s, \ s\le t \}$ is $ \text{Poisson}( 2S(t) )$.
     
     The aggregate size, $X_t$, increases by $1$ with rate $S(t)$ and therefore $X_t -\int _0^t S(x)ds $ is a martingale. Similarly, for each $j\ge 1 $, $N_t $ increases by $j$ with rate $ \frac{1}{2} j \cdot \mathbb P (Z_t =j \ | \ X_s, \ s\le t )$ and therefore 
	\begin{equation}
	    N_t -\frac{1}{2}  \intop _0^t \sum _{j=1}^{\infty } j^2 \cdot \mathbb P (Z_x =j \ | \ X_s, \ s\le x )dx= N_t-\frac{1}{2}\intop _0^t \mathbb E \left[ Z_x^2 \ | \ X_s , \ s\le x \right] dx= N_t-\intop _0^t  S(x)+2S(x)^2 dx
	\end{equation}
	is a martingale. By Corollary~\ref{cor:concentration of integral} with $M=t$ and $f=1$ (using the bounds $S,S+2S^2 \le 1$) we have probability at least $1-C e^{-c \log ^2 t }$ that 
	\begin{equation}\label{eq:azuma for X and N}
	\left| X_t-\intop _0^t S(x)dx \right| \le \sqrt{t} \log ^2 t,\quad  \left| N_t-\intop _0^t S(x)+2 S(x)^2dx \right| \le \sqrt{t} \log ^2t.
	\end{equation}
	Denote by $\mathcal C$ the event where these inequalities hold. On $\mathcal B \cap \mathcal C $ we have that 
	\begin{equation}
	    2\intop _0^t S(x)^2 dx \le -X_t +\sqrt{t} \log ^2 t +\intop _0^t S(x)+2 S(x)^2dx \le N_t-X_t +2 \sqrt{t}\log ^2 t \le 3 \sqrt{t} \log ^2 t,
	\end{equation}
	where in the first two inequalities we used \eqref{eq:azuma for X and N} and in the last inequality we used \eqref{eq:bound on N}. Thus, on $\mathcal B \cap \mathcal C $, by Cauchy Schwartz inequality we have 
	\begin{equation}
	X_t- \sqrt{t}\log ^2t \le \intop _0^t S(x)dx\le \sqrt{t} \cdot \Big( \intop _0^t S(x)^2dx\Big) ^{\frac{1}{2}} \le  2 t^{\frac{3}{4}}\log  t.
	\end{equation}
	This finishes the proof of the lemma.
\end{proof}

The next corollary follows immediately from Lemma~\ref{lem:three quarter upper bound}.

\begin{cor}
We have that
\begin{enumerate}
    \item 
	 $\mathbb P (X_t =O(t^{\frac{3}{4}}\log ^2t) )=1$.
	\item
	For all $t\ge 2$,  $\mathbb E X_t \le C t^{\frac{3}{4}} \log ^2 t$.
\end{enumerate}
\end{cor}

\begin{proof}
	The first part follows immediately from Lemma~\ref{lem:three quarter upper bound} and the Borel-Cantelli lemma. We turn to prove the second statement. By Lemma~\ref{lem:three quarter upper bound} and the fact that $X_t\precsim \text{Poisson}(t/2)$ we have for sufficiently large $t$,
	\begin{equation}
	\mathbb E X_t \le t^{\frac{3}{4}} \log ^2t +t\cdot \mathbb P ( t^{\frac{3}{4}} \log ^2t \le X_t \le t) +\sum _{k=1}^{\infty } (k+1)t \cdot \mathbb P (X_t\ge kt) \le Ct^{\frac{3}{4}} \log ^2t,
	\end{equation}
	as needed.
\end{proof}

\subsection{Stochastic domination and bound on rate}\label{subsec:sandwich1}

The process $Y_t$ is itself a function valued Markov chain and it will be useful at time to initialize it from different starting states. We define the generalized aggregate with initial condition $Y_{0}$ for some $t_0\ge 0$.

\begin{definition}\label{def:aggregate with initial condition}
	Let $Y_0:\mathbb R_+\to \mathbb N \cup \{\infty \}$ be a monotone random step function and let $\Pi $ be an independent $2$-D Poisson process. The rate of the aggregate with initial condition $(Y_0,t_0)$ is defined by 
	\begin{equation}
	S(t):=\frac{1}{2}\mathbb P \left( W(s) \le Y_t(s) \ | \ Y_t\right), \quad t \ge t_0,
	\end{equation}
	where $Y_t$ is defined by 
	\begin{equation}\label{eq:generalised Y}
	Y_t(s)=\begin{cases}
	\quad \quad X_t-X_{t-s}\quad \quad \quad \quad   s \le t-t_0  \\
	Y_t(t-t_0)+Y_0(s-(t-t_0)) \ \quad s >t-t_0
	\end{cases}
	\end{equation}
	As before, the size of the generalised aggregate will be $X_t:= \Pi _{S}[t_0,t]$. We say that $(Y_t ,S(t), X_t)$ is the aggregate with initial condition $(Y_0,t_0)$ driven by the Poisson process $\Pi $.
\end{definition}

The following claim follows immediately from the above definition.

\begin{claim}\label{claim:aggregate with initial condition}
	Let $Y_0^{(1)}, Y_0^{(2)}:\mathbb R_+\to \mathbb N \cup \{\infty \}$ be monotone random step functions and let $\Pi$ be an independent Poisson process. Let $(Y_t^{(1)},S^{(1)} (t), X_t^{(1)} )$ and $(Y_t^{(2)},S^{(2)} (t), X_t^{(2)} )$ be the aggregates with initial conditions $(Y_0^{(1)},0)$ and $(Y_0^{(2)},0)$ respectively, both driven by $\Pi $. We also let $(Y_t,S(t),X_t )$ be the usual aggregate.
	\begin{enumerate}
		\item 
		Suppose that $Y_0^{(1)}\equiv \infty $  with probability one. Then 
		\begin{equation}
		\{S^{(1)}(t)\} _{t \ge 0}\overset{d}{=} \{S(t)\} _{t \ge 0} \quad \{ X_t^{(1)}\} _{t \ge 0}\overset{d}{=} \{X_t\} _{t \ge 0}.
		\end{equation}
		\item 
		Let $t_0>0$ and suppose that $Y_0^{(1)}\overset{d}{=} Y_{t_0}$. Then 
		\begin{equation}
		\{S^{(1)}(t)\} _{t \ge 0}\overset{d}{=} \{S(t+t_0)\} _{t \ge 0} \quad \{ X_t^{(1)}\} _{t \ge 0}\overset{d}{=} \{X_{t+t_0}\} _{t \ge 0}.
		\end{equation}
		\item 
		On the event $\{ Y_0^{(1)}(s)\le Y_0^{(2)}(s) \text{ for all }s \}$ we have for all $t>0$
		\begin{equation}
		S^{(1)}(t) \le S^{(2)}(t),\quad X_t^{(1)} \le X_t^{(2)}, \quad \forall s>0,\ Y_t^{(1)}(s)\le Y_t^{(2)}(s).
		\end{equation}
	\end{enumerate}
\end{claim}

\begin{cor}\label{cor:stochastic domination}
	We have that $S(t)$ is stochastically decreasing in the sense that, if $t_1\le t_2$ then there is a coupling of the processes $\{S(t+t_1)\}_{ t \ge 0}$ and $\{ S(t+t_2)\} _{t\ge 0}$ such that the first process is larger than the second.
\end{cor}

\begin{proof}
	Let $Y_0^{(1)}\overset{d}{=}Y_{t_2-t_1}$ and $Y_0^{(2)}\equiv \infty$  and let $\Pi $ be a Poisson process independent of everything. By parts (1) and (2) of Claim~\ref{claim:aggregate with initial condition} we have that
	\begin{equation}
	\{S^{(1)}(t)\} _{t \ge 0}\overset{d}{=} \{S(t)\} _{t \ge 0} \quad \text{and} \quad   \{S^{(2)}(t)\}_{t \ge _0 }\overset{d}{=}\{S(t+t_2-t_1)\}_{t \ge 0}.
	\end{equation} 
	Since $Y_0^{(1)}(s)\le Y_0^{(2)}(s)$ for all $s$ almost surely, we have by part (3) of Claim~\ref{claim:aggregate with initial condition} that $S^{(1)}(t)\le S^{(2)}(t)$ for all $t>0$. Thus, the processes $\{S^{(1)}(t+t_1) \}_{t \ge 0}$ and $\{S^{(2)}(t+t_1) \}_{t \ge 0}$ give the desired coupeling between $\{S(t+t_1)\}_{t \ge 0}$ and $\{S(t+t_2)\}_{t \ge 0}$.
\end{proof}

In the following lemma we show that the rate of the aggregate gets arbitrarily close to $0$ for long periods of time.

\begin{lem}\label{lem:bound on rate}
	For any $\alpha >0$ sufficiently small the following holds. For all $t\ge \alpha ^{-20}$ we have that
	\begin{equation}
	\mathbb P \left( \sup _{t-\alpha ^{-3} \le s \le t } S(s) \le \alpha  \right)  \ge 1- \alpha .
	\end{equation}
\end{lem}

\begin{proof}
	Let $t\ge \alpha ^{-20}$ and let $I_k:=[t- k \alpha ^{-3}, t -(k-1) \alpha ^{-3}]$ for all $1\le k \le \lfloor \alpha ^3 t  \rfloor $.
	
	Define the random sets
	\begin{equation}
	\begin{split}
	A_1:=\left\{ 1\le k\le \lfloor\alpha ^3 t \rfloor-1  \text{ odd }: \substack{ \text{the aggregate grew in }\\ \text{the time interval } I_{k+1}\cup I_{k} } \right\}, \\
	A_2:=\left\{ 1 \le k\le \lfloor\alpha ^3 t \rfloor-1 \text{ even }: \substack{ \text{the aggregate grew in}\\ \text{the time interval } I_{k+1}\cup I_{k} } \right\}
	\end{split}
	\end{equation}
	and
	\begin{equation}
	B:= \left\{  1\le k\le \lfloor\alpha ^3 t \rfloor-1 \   : \  \sup _{s\in I_k}S(s)> \alpha  \right\}.
	\end{equation}
	It is clear that $|A_1|\le X_{t}$ and $|A_2|\le X_{t}$. Next, we claim that $B \subseteq A_1\cup A_2$. Indeed suppose that $k \notin A_1\cup A_2$ and let $s \in I_k$. By the definition of $A_1,A_2$, the aggregate did not grow in the time interval $[s-\alpha ^{-3} ,s]$ and therefore
	\begin{equation}
	S(s)=\frac{1}{2} \mathbb P \left(\forall x>0 , \ W(x) \le Y_s(x) \ | \ Y_s \right) \le \mathbb P \left( \forall x \le \alpha ^{-3}, \  W(x)\le 0 \right) \le C \alpha ^{\frac{3}{2}} \le \alpha ,
	\end{equation} 
	where the last two inequalities hold for sufficiently small $\alpha $. Since the last bound holds for all $s \in I_k$ we get that $k \notin B$.
	
	Finally, by Corollary~\ref{cor:stochastic domination} we have that $\mathbb P (k \in B) \ge \mathbb P (1 \in B)$ and therefore using Lemma~\ref{lem:three quarter upper bound} we obtain
	\begin{equation}
	(\lfloor\alpha ^3 t \rfloor-1 ) \mathbb P (1 \in B) \le \sum _{k=1}^{ \lfloor\alpha ^3 t \rfloor-1} \mathbb P (k\in B) =\mathbb E |B| \le \mathbb E |A_1|+\mathbb E |A_2| \le 2 \mathbb E X_{t} \le C t ^{\frac{3}{4}} \log ^5 t .
	\end{equation}
	Thus, using that $t \ge  \alpha ^{-20}$ and that $\alpha $ is sufficiently small we get
	\begin{equation}
	\mathbb P \left( \sup _{t-\alpha ^{-3} \le s \le t } S(s) \ge \alpha  \right)  =\mathbb P (1 \in B) \le \alpha 
	\end{equation}
	as needed.
\end{proof}

\subsection{Bounding with fixed rate processes}\label{sec:Sandwich}

In this section we show that the aggregate can be bounded from above and below by regular processes. Recall that $\Pi $ is the $2-D$ Poisson process that is driving the aggregate.

For $\alpha , t >0 $ and define the random functions 
\begin{equation}\label{eq:def:sandwichfixed}
\underline{Y} _{t,\alpha } (s):= \Pi _{\alpha  }[t -s , t ] ,\quad   \overline{Y} _{t,\alpha }(s):= 
\begin{cases}
\Pi _{ \alpha }[t-s,t], \quad   s \le \alpha ^{-3} \\
\quad \quad \infty,\quad \ \quad \quad  s > \alpha ^{-3}
\end{cases}\!\!\!\!
\end{equation}
for all $s>0$.  In terms of the speed process, we can write
\begin{equation}\label{eq:def:sandwich speed}
\underline{S}_{t,\alpha}(s) \equiv \alpha, \quad \overline{S}_{t,\alpha}(s) = \begin{cases}
\alpha, & t-\alpha^{-3}\le s < t;\\
\infty, &  s<t-\alpha^{-3}.
\end{cases}
\end{equation}

The following lemma follows immediately from Lemma~\ref{lem:bound on rate} and the fact that $S(t) \ge c t^{-\frac{1}{2}}$ for all $t>0$ almost surely.
\begin{lem}\label{lem:aggregate is bounded}
	Let $\alpha >0$ sufficiently small.
	\begin{enumerate}
		\item 
		For all $t \ge \alpha ^{-20}$  
		\begin{equation}
		\mathbb P \left( \forall s>0, \  Y_{t}(s) \le \overline{Y} _{t,\alpha }(s)  \right)  \ge 1-\alpha 
		\end{equation}
		\item 
		For all $t \le 100 \alpha ^{-2}$
		\begin{equation}
		\mathbb P \left( \forall s>0, \  Y_{t}(s) \ge \underline{Y} _{t,\alpha }(s)  \right)=1
		\end{equation}
	\end{enumerate}
\end{lem}

In the next sections we'll show that the processes with initial conditions $\overline{Y} _{t,\alpha }$ and $\underline{Y} _{t,\alpha }$ satisfy some regularity properties which will allow us to apply the inductive argument on them. We then conclude that the same results hold for the usual aggregate as these processes sandwich the aggregate by Lemma~\ref{lem:aggregate is bounded}

\section{Fixed rate process and its perturbation}\label{sec:fixedrate}

Recall the first- and second-order expansions of speed \eqref{eq:speed:1storder main}, \eqref{eq:speed:2ndorder main} and the definitions \eqref{eq:def:J:agg}, \eqref{eq:def:Qt:agg}. In this section, our objective is to understand the multiple integral 
\begin{equation}
\intop_{0}^t \intop_{0}^s  
{\textnormal{\textbf{J}}}_{t-s,t-u;t}^{\Pi_g} \ d\widehat{\Pi}_g(u) d\widehat{\Pi}_g(s),
\end{equation}
both when $g$ is a fixed rate $g\equiv \alpha$  and when it is slightly perturbed, and also the corresponding analogue for $\mathcal{Q}_t$. Consequently, we further derive estimates on the error of the first- and second-order approximations of the speed \eqref{eq:def:S1:basic form}, \eqref{eq:def:S2:basic form}. Throughout this section, $g(s)>0$ denotes a random function which represents the rate of the Poisson process at time $s$. 

To begin with, we introduce  some notations as follows. As before, $\Pi$ denotes the Poisson point process on the plane with the standard Lebesgue intensity,  and let $\Pi_g$ be the collection of the $x$-coordinate locations of points lying below $g$, that is,
\begin{equation}
\Pi_g[0,t] = \{x \in[0,t] : \ (x,y)\in \Pi \textnormal{ for some } y \in[0, g(x)] \},
\end{equation}
we define $\pi_i(t;g)$ to be the distance in $x$-axis from $t$ to the $i$-th closest point to $t$ in $\Pi_g[0,t]$. For instance, 
\begin{equation}\label{eq:def:pi closest points}
\pi_1(t;g) =\pi_1(t;\Pi_g) = t- \max\{x: x\in \Pi_g[0,t] \}.
\end{equation}
Moreover, for convenience, we denote
\begin{equation}\label{eq:def:sig closest points}
\sigma_i(t;g):= (\pi_i(t;g)+1)^{-\frac{1}{2}},
\end{equation}
and abbreviate the products among them by
\begin{equation}
\sigma_1\sigma_2(t;g):= \sigma_1(t;g)\sigma_2(t;g), \quad \textnormal{and} \quad \sigma_1\sigma_2\sigma_3(t;g):= \sigma_1(t;g)\sigma_2(t;g)\sigma_3(t;g).
\end{equation} 
Further, $\alpha,C>0$ denote  small enough and large enough constants, respectively, and  we set
 \begin{equation}\label{eq:def:horizon}
 \hat{h} = \hat{h}(\alpha,C):= \alpha^{-2} \log^C(1/\alpha).
 \end{equation}
Then, the goal of this section is to establish the following statements.

\begin{prop}\label{prop:fixed perturbed:double int}
	Let  $\epsilon>0$ be arbitrary, $\alpha>0$ be a  sufficiently small constant depending on $\epsilon,C
	$, and recall the definition of ${\textnormal{\textbf{J}}}$ \eqref{eq:def:J:agg}. Let $\tau$ be a stopping time, and $\{g(s)\}_{s \ge 0}$ be a positive stochastic process progressively measurable with respect to $\Pi_g$, and suppose that it satisfies the following three conditions almost surely:
	\begin{equation}\label{eq:fixed perturbed:assumption}
	\begin{split}
	\intop_0^{\hat{h}\wedge \tau} (g(s)-\alpha)^2 ds \le \alpha^{1-\frac{\epsilon}{400}},\quad 
	\intop_0^{\hat{h}\wedge \tau} (g(s)-\alpha)^2 g(s)ds \le \alpha^{2-\frac{\epsilon}{400}},\quad 
	\sup_{s\le \hat{h}\wedge \tau}g(s)  \le \alpha^{1-\frac{\epsilon}{400}}. 
	\end{split}
	\end{equation}
	Furthermore, denoting $d\widehat{\Pi}_g(s) = d\Pi_g(s) - \alpha ds$ as before, we define for $t'\le t$ that
	\begin{equation}\label{eq:def:Jint:general}
	\mathcal{J}[t;\Pi_g]:= \intop_{0}^{t} \intop_{0}^s  
	{\textnormal{\textbf{J}}}_{t-s,t-u;t}^{\Pi_g} \ d\widehat{\Pi}_g(u) d\widehat{\Pi}_g(s).
	\end{equation}
	Then, there exist $c_\epsilon,\alpha_0(\epsilon,C)>0$ such that for any $\alpha\in (0,\alpha_0)$, we have 
	\begin{equation}
	\PP \left(-\alpha^{\frac{1}{2}-\epsilon}\sigma_1(t;g)\le  \mathcal{J}[t;\Pi_g]  \le \alpha^{-\epsilon}\sigma_1\sigma_2(t;g),\ \forall t\le  \hat{h}(\alpha,C)\wedge \tau \right)
	\ge 1- 
	\exp\left(-\alpha^{-c_\epsilon} \right).
	\end{equation}
\end{prop}

\begin{prop}\label{prop:fixed perturbed:triple int}
	Recall the definition of $Q$ \eqref{eq:def:Q:agg}, and define  $\mathcal{Q}[t;\Pi_g]$  as 
	\begin{equation}\label{eq:def:Qint:general}
	\mathcal{Q}[t;\Pi_g] := \intop_{0}^{t} \intop_{0}^s  \intop_0^u
	{\textnormal{\textbf{Q}}}_{t-s,t-u,t-v;t}^{\Pi_g} \ d\widehat{\Pi}_g(v) d\widehat{\Pi}_g(u) d\widehat{\Pi}_g(s).
	\end{equation}
	 Then, under the same setting as Proposition \ref{prop:fixed perturbed:double int}, we have
	\begin{equation}
	\PP \left(\left| \mathcal{Q}[t;\Pi_g] \right| \le \alpha^{-\epsilon}\sigma_1\sigma_2\sigma_3(t;g), \ \forall t \le \hat{h}(\alpha,C)\wedge \tau \right) \ge 1- \exp\left(-\alpha^{-c_\epsilon}\right).
	\end{equation}
\end{prop}
Based on these results, we discuss the error estimates for the first- and second-order approximations. Formal statements and details will be given in Section \ref{subsec:fixed:error}.

\begin{remark}
	The assumption \eqref{eq:fixed perturbed:assumption} will be revisited in Section \ref{sec:reg:intro} when we discuss regularity. In fact, it will be shown that the speed $S(t)$ of the aggregate actually satisfies \eqref{eq:fixed perturbed:assumption} with high probability.
\end{remark}

\begin{remark}
    In Proposition \ref{prop:fixed perturbed:double int}, observe that the lower bound of $\mathcal{J}[t;\Pi_g]$ is stronger than the upper bound. This comes from the nature of $\textbf{J}$ who can be (positively) larger compared to the (negative) lower bound, which will be clear in Section \ref{subsec:fixed:JandQ}. Having a stronger lower bound on $\mathcal{J}$ will play an important role in the discussion of regularity, in  Sections \ref{subsec:reg:conseq} and \ref{sec:double int}.
\end{remark}

As mentioned in Section \ref{subsubsec:outline:fixedrate}, the proofs consist of four major steps which are discussed one by one in detail in the following subsections:
\begin{itemize}
	\item The Azuma-type martingale concentration lemmas: Section \ref{subsec:fixed:mgconcen}.
	
	\item Controlling the sizes of ${\textnormal{\textbf{J}}}_{u,s}$ and ${\textnormal{\textbf{Q}}}_{v,u,s}$: Section \ref{subsec:fixed:JandQ}.
	
	\item Proving the propositions with respect to the fixed rate process $\Pi_\alpha$: Section \ref{subsec:fixed:fixed}.
	
	\item Converting the results into a general form: Section \ref{subsec:fixed:perturbed}. 
	
	\item The error estimate on the first- and second-order approximations of the speed: Section \ref{subsec:fixed:error}.
\end{itemize}

\subsection{The martingale concentration lemmas}\label{subsec:fixed:mgconcen}

We begin with developing lemmas on martingale concentration used in the proof of Propositions \ref{prop:fixed perturbed:double int} and \ref{prop:fixed perturbed:triple int}. Moreover, the lemmas we announce here will appear again frequently in Sections \ref{sec:criticalbranching}--\ref{sec:reg:next step}. 

For the    Poisson point process $\Pi$ with standard Lebesgue intensity, let $\mathcal F _t$ denote the $\sigma$-algebra generated by $\Pi$ up to time $t$ (including $t$). For any positive stochastic process $g(t)$ which is predictable with respect to $\mathcal F _t$, we write $\Pi _g$ to denote the points in $\Pi[0,t]$ below the function $g$ as before.  In particular, for a constant $\alpha>0$, $\Pi _\alpha $ refers to the rate-$\alpha$ Poisson process. We note that for any such function $g$ the process $\widetilde{ \Pi}_ g(t):= |\Pi _g[0,t]| -\int_0^t g(s)ds $ is a martingale. 

We begin with the simplest form of concentration lemma for martingales. The statement includes a stopping time which is introduced for a later purpose.

\begin{lem}\label{lem:concentration of integral}
	Let $a,M,\lambda > 0$, let $g(t)>0$ and  $f(t)$ be stochastic processes predictable with respect to $\mathcal F _t  $, and let $\tau$ be a stopping time. Suppose that 
	\begin{equation}\label{eq:concentration:conditions:basic}
	\lambda |f(x\wedge \tau )|\le 1, \quad \text{and} \quad \intop _{0}^{t \wedge \tau }f(x)^2g(x)dx \le M, \quad a.s.
	\end{equation}
	Then
	\begin{equation}
	\mathbb P \left(\sup _{s \le t\wedge \tau  }  \left| \intop _0^s f(x) d \widetilde{\Pi }_g(x) \right|\ge a\sqrt{M}  \right) \le  Ce^{\lambda ^2 M-a\lambda \sqrt{M}}. 
	\end{equation}
\end{lem}

\begin{proof}
	For $0 \le s \le t $ let 
	\begin{equation}
	\begin{split}
	M_s&:=\exp \left( \lambda \intop _0^s  f(x) d \Pi _g (x)-\intop _0^s \left(e^{\lambda f(x)}-1\right) g(x)dx \right) \\
	L_s &:=\exp \left( \intop _0^s \left(e^{\lambda f(x)}-1\right) g(x)- \lambda f(x)g(x) dx \right).
	\end{split}
	\end{equation}
	Note that $M_s$ is a martingale and that, using the inequality $e^x\le 1+x+x^2$ for $|x|\le 1$ we get
	\begin{equation}
	L_{s\wedge \tau } \le \exp \left( \lambda ^2 \intop _0^{s\wedge \tau } f(x)^2 g(x)dx\right) \le \exp \left( \lambda ^2  \intop _0^{t\wedge \tau } f(x)^2 g(x)dx\right) \le e^{\lambda ^2 M}.
	\end{equation}
	Define the stopping time 
	\begin{equation}
	\tau _1:= \inf \left\{  s>0\  : \ \intop _0^s f(x) d \widetilde{\Pi }_g(x) \ge a\sqrt{M}  \right\},\quad \tau _2:= \inf \left\{  s>0\  : \ \intop _0^s f(x) d \widetilde{\Pi }_g(x) \le - a\sqrt{M}  \right\}
	\end{equation}
	and let $\tau ' =\tau \wedge \tau _1$. We have 
	\begin{equation}
	e^{\lambda a\sqrt{M}}\cdot \mathbb P \left( \tau _1\le \tau \right) \le  \mathbb E \left[ \lambda \exp \left( \intop _0^{t \wedge \tau '} f(x) d \widetilde{\Pi }_g(x) \right) \right] =\mathbb E \left[  M_{t\wedge \tau '}\cdot L_{t\wedge \tau '} \right] \le e^{\lambda ^2 M} \mathbb E [M_{t\wedge \tau '}]= e^{\lambda ^2 M}, 
	\end{equation}
	which gives the desired bound for $\tau _1$. In order to get the same bound for $\tau _2$ we swich $f$ by $-f$.
\end{proof}

In particular, taking $\lambda =1/ \sqrt{M}$ in Lemma \ref{lem:concentration of integral} gives the following corollary.

\begin{cor}\label{cor:concentration of integral}
	Let $a,M > 0$, let $g(t)>0$ and $f(t)$ be predictable processes with respect to $\mathcal F _t  $ and let $\tau$ be a stopping time. Suppose that 
	\begin{equation}
	|f(x\wedge \tau )|\le \sqrt{M} \quad \text{and} \quad \intop _{0}^{t \wedge \tau }f(x)^2g(x)dx \le M, \quad a.s.
	\end{equation}
	Then
	\begin{equation}
	\mathbb P \left(\sup _{s \le t\wedge \tau  }  \left| \intop _0^s f(x) d \widetilde{\Pi }_g(x) \right|\ge a\sqrt{M}  \right) \le  Ce^{ -a}. 
	\end{equation}
\end{cor}

The following corollary deals with the case $f\equiv 1$, which is useful when estimating the size of the Poisson process itself.

\begin{cor}\label{cor:concentration:numberofpts each interval}
	Let $a,h,\Delta > 0$, $M\geq 1$, let $g(t)>0$  a process with respect to $\mathcal F _t  $ and let $\tau$ be a stopping time. Suppose that 
	\begin{equation}
	\intop _{(t-\Delta)\wedge \tau}^{t \wedge \tau }g(x)dx \le M, \quad a.s.,
	\end{equation}
	and define the stopping time
	\begin{equation}
	\tau'=\inf\{t>0: |\Pi_g[(t-\Delta)\vee 0, t]|\geq 2M+2a\sqrt{M}+2 \}\wedge h.
	\end{equation}
	Then, we have $\PP(\tau'\le \tau)\le \left(\frac{h}{\Delta}\right)e^{-a} $.
\end{cor}

\begin{proof}
	For each $t\in [0,h],$ we apply Corollary \ref{cor:concentration of integral} to  the quantity
	\begin{equation}
	P(t):= |\Pi_g[t\wedge\tau-\Delta, t\wedge\tau]|,
	\end{equation}
	and obtain that
	\begin{equation}
	\PP \left(P(t)\geq M+a\sqrt{M} \right) \le e^{-a}.
	\end{equation}
	Then, we take a union bound over all $t\in[0,h]$ of the form $t=k \Delta $, $k\in \mathbb{Z}$: 
	\begin{equation}
	\PP \left(P(\Delta)\leq M+a\sqrt{M}, \ \ \forall k\Delta \in[0,h], \ k\in\mathbb{Z} \right) \ge 1-\left(\frac{h}{\Delta} \right)e^{-a}.
	\end{equation}
	Under the event described inside the above probability, the intervals $[t\wedge\tau-\Delta, t\wedge\tau]$ should contain at most $2M+2a\sqrt{M}$ points for all $t\in[0,h]$, implying that $\tau'> \tau$.
\end{proof}

 Unfortunately, Lemma \ref{lem:concentration of integral} often turns out to be insufficient due to several reasons, and we require  more involved versions of it.
 To motivate the formulation of the more complicated lemmas below, we briefly explain the setting. Suppose that $f_t(x)$, $g(x)$ are predictable processes with respect to $\mathcal{F}_x$, where $f_t$ is defined for each $t\in[0,h]$. We want to control the size of 
 $$\intop_{ \tau_-}^{t\wedge \tau} f_t(x)d\widetilde{\Pi}_g(x), $$
where $\tau_-,\tau\ge 0$ are two given  stopping times. Our desired estimate should take account of the following traits, which causes additional complication compared to the previous lemma.
\begin{itemize}
	\item We want our bound to hold for all $t\in[0,h]$, which would require a union bound followed by a continuity argument.
	
	\item We need to deal with the cases where $|f_t(x)|$ is not bounded uniformly in $x$ as in the first condition of \eqref{eq:concentration:conditions:basic}. 
	
	\item The integral starts from $\tau_-$, which potentially causes the second bound of \eqref{eq:concentration:conditions:basic} to be $\tau_-$-measurable (which is random).
\end{itemize}
Having these aspects in mind, the device we need is stated as follows. Due to its complicatedness, we suggest the reader to skip this lemma for the moment and come back later when it is actually applied (e.g., in Section \ref{subsec:fixed:fixed}). 

\begin{lem}\label{lem:concen of int:conti:forwardtime}
	Let $h, \Delta, D>0$ be given constants, and let $g(x)>0$ and $f_{t}(x)$ be stochastic processes predictable with respect to $\mathcal{F}_x$, where $f_{t}:[0,t] \to \mathbb{R}$ is defined for each $t\in [0, h]$. Suppose that there exist stopping times $\tau_-, \tau$, and  random variables ${\bf M},{\bf A}>0$ that satisfy the following conditions:
	\begin{itemize}
		\item ${\bf M},{\bf A}$  are $\tau_-$-measurable.

		\item $f_t$ and $g$ satisfy
		\begin{equation}\label{eq:concen:condition:forwardtime}
		|{f}_t(x)| \le \sqrt{{\bf M}} \ \forall x \ge \Delta, \ \ \intop_0^\Delta |{f}_t(x)|g(x) dx \le {\bf A}, \ \ \intop_0^{t\wedge \tau} {f}_t(x)^2 g(x)dx \le {\bf M}, \ \ a.s.  \ \forall t\in[0,h].
		\end{equation}
		
		\item With probability one, we have $|\partial_t f_t(x)| \vee |\partial_x f_t(x)| \le D$ for all $t\in[0,h]$, $x\le t$.
		
		\item We define $\bar{f}_t(x)\ge 0$ and $\underline{f}_t(x) \le 0$ to be  random functions that are decreasing and increasing, respectively, and satisfy the following conditions: 
		
		\begin{itemize}
		\item For all $x\le t$, $\bar{f}_t(x)$ and $\underline{f}_t(x)$ are $\tau_-$-measurable.
		
			\item For all $x\le t \le \tau$, $\bar{f}_t(x) \ge  \sup_{y\ge x} f_t(y)$ and $\underline{f}_t(x)\le \inf_{y\ge x} f_t(y)$ a.s..
		\end{itemize}

	\end{itemize} 
	Furthermore, for given constants $N,\eta>0$, define stopping times
	\begin{equation}\label{eq:def:tau:concenofint:conti}
	\begin{split}
	\tau' &:= \inf\{t>\Delta : |\Pi_g[t-\Delta,t]| \ge N \},\\
	\tau'' &:= \inf \{t>0: |g(t)| \ge \eta \},\\
	\tau_0&:= \tau \wedge \tau' \wedge \tau'',
	\end{split}
	\end{equation}
	and let $\delta>0$ be a constant satisfying a.s.~that
	\begin{equation}\label{eq:concen:condition:delta:forward}
	\delta \le \frac{{\bf A}}{D(2N\Delta+TN+T\eta \Delta)}.
	\end{equation}
	Then, denoting the closest point to $0$ in $\Pi_g[0,t]$ by $p_1$, we have for all $a>0$ that
	\begin{equation}
\begin{split}
	 	\PP \left( \sup_{0\le s \le t\wedge \tau} \intop_{\tau_-}^s {f}_t(x) d\widetilde{\Pi}_g(x) 
	 	\le 2N \bar{f}_t(p_1) + 3{\bf A} + a\sqrt{{\bf M}}, \ \forall t\in[0,h]
	 	\right)
	 	\ge
	 	1- \left(\frac{Ch}{\delta\Delta} \right)e^{-a};\\
	 		\PP \left( \inf_{0\le s \le t\wedge \tau} \intop_{\tau_-}^s {f}_t(x) d\widetilde{\Pi}_g(x) 
	 		\ge 2N \underline{f}_t(p_1) - 3{\bf A} - a\sqrt{{\bf M}}, \ \forall t\in[0,h]
	 		\right)
	 		\ge
	 		1- \left(\frac{Ch}{\delta\Delta} \right)e^{-a}.
\end{split}
	\end{equation}
	with an absolute constant $C>0$.
	In the integral, we regard $\intop_a^b f = 0$ if $a\ge b$.
\end{lem}
Its proof is based on a union bound applied to Lemma \ref{lem:concentration of integral} to cover a discretized points in the interval $[0,h]$ and a continuity argument to cover the whole interval. Distinction of upper and lower bounds on the quantity is necessary to deduce Proposition \ref{prop:fixed perturbed:double int}. Due to its technicality, the proof is deferred to Appendix \ref{subsec:app:mgconcen}.

The following corollary will not be used in Section \ref{sec:fixedrate}, but they will be applied frequently later in Sections \ref{sec:criticalbranching}--\ref{sec:reg:next step}. It is an analogue of Lemma \ref{lem:concen of int:conti:forwardtime}, but with \textit{deterministic} $f_t(x)$ which is \textit{increasing} in $x$, and also  simpler due to the absence of $\tau_-$, allowing ${\bf M}, {\bf A}$ above to be deterministic.

\begin{cor}\label{lem:concentrationofint:continuity}
	Let $h,M,\Delta,D,A>0$ be given, let  $g(x)>0$  be a predictable process  with respect to $\mathcal F _x  $ and let $\tau$ be a stopping time. For each $t\in[0,h]$, Let $f_t:(-\infty,t]\to \mathbb{R}_{\ge 0}$ be a deterministic, increasing function that satisfies
	\begin{equation}
	|f_t(x)|\leq \sqrt{M} \ \ \forall  x\leq t-\Delta,\quad \intop_{t- \Delta}^t |f_t(x)|g(x)dx \leq A,\quad |\partial_t f_t(x)| \vee |\partial_x f_t(x)|  \leq D \ \ \forall  0\le t\leq h, \ x<t.
	\end{equation}
	Suppose that for each $t\in[0,h]$,
	\begin{equation}
	\intop _{0}^{ t\wedge \tau }f_t(x)^2g(x)dx \le M, \quad a.s.
	\end{equation}
	Furthermore, for given constants $N,\eta>0$, let the stopping times $\tau',\tau'',\tau_0$ and the constant $\delta>0$ be as \eqref{eq:def:tau:concenofint:conti}, \eqref{eq:concen:condition:delta:forward}.
	Then, denoting the closest point to $t$ in $\Pi_g[0,t]$ by $p_1(t)$,  we have
	\begin{equation}
\begin{split}
	\mathbb P \left(  \intop _0^{t\wedge \tau_0} f_t(x) d \widetilde{\Pi }_g(x) \le  2Nf_t(p_1(t)) +3A+ a\sqrt{M}, \ \forall t\in[0,h]  \right) \ge 1- \left(\frac{Ch}{\delta \Delta} \right)e^{-a};\\
	 	\mathbb P \left(  \intop _0^{t\wedge \tau_0} f_t(x) d \widetilde{\Pi }_g(x) \ge-3A- a\sqrt{M}, \ \forall t\in[0,h]  \right) \ge 1- \left(\frac{Ch}{\delta \Delta} \right)e^{-a}.
\end{split}
	\end{equation}
\end{cor}
Note that the absence of the term $Nf_t(p_1(t))$ in the lower bound comes from the assumption that $f_t$ is nonnegative. It turns out that this can be proven analogously as Lemma \ref{lem:concen of int:conti:forwardtime}, and we omit the details (see Appendix \ref{subsec:app:mgconcen}).

\subsection{Estimating ${\textnormal{\textbf{J}}}_{u,s}$ and ${\textnormal{\textbf{Q}}}_{v,u,s}$}\label{subsec:fixed:JandQ}

Recall the definitions of ${\textnormal{\textbf{J}}}_{u,s}$ \eqref{eq:def:J:basic form} and ${\textnormal{\textbf{Q}}}_{v,u,s}$ \eqref{eq:def:Q:basic form}.  W study its quenched version ${\textnormal{\textbf{J}}}_{u,s}$ and show that it satisfies a similar but weaker bound than $J_{u,s}^{(\alpha)} $ from Section \ref{subsec:Kestim:intro}. Moreover, we derive an analogous result for ${\textnormal{\textbf{Q}}}_{v,u,s}$.
The goal is to establish the following estimates on ${\textnormal{\textbf{J}}}_{u,s}$ and ${\textnormal{\textbf{Q}}}_{v,u,s}$. 

Throughout this section we fix $C_0>0$ and let $\alpha >0$ sufficiently small depending on  $\epsilon $ and $C_0$. We also let $\hat{h}:= \alpha ^{-2} \log ^{C_0}(1/ \alpha )$. 

\begin{prop}\label{prop:Jbound:quenched}
	For all $\epsilon >0$ there exists $c_\epsilon >0$  such that 
	\begin{equation}
	\PP_\alpha \left(  -\frac{\alpha ^{1-\epsilon }}{\sqrt{s+1}}  \le  {\textnormal{\textbf{J}}}_{u,s}  \le \frac{\alpha^{-\epsilon}}{\sqrt{(u+1)(s+1)}}, \ \forall 0\le u<s\le \hat{h} \right) \ge 1- \exp\left(-\alpha^{-c_\epsilon } \right).
	\end{equation}
\end{prop}

\begin{prop}\label{prop:Qbound:quenched}
	For all $\epsilon >0$ there exists $c_\epsilon >0$ such that 
	\begin{equation}
	\PP_\alpha \left( |{\textnormal{\textbf{Q}}}_{v,u,s}| \le \frac{\alpha^{-\epsilon}}{\sqrt{(v+1)(u+1)(s+1)}}, \ \forall 0\le s\le \hat{h} \right) \ge 1- \exp\left(-\alpha^{-c_\epsilon } \right).
	\end{equation}
\end{prop}

Proposition \ref{prop:Jbound:quenched} can be established by a coupling argument of the events $\{ s<T_u<\infty \}$ and $\{s< T <\infty \}$, and using hitting time estimates of random walks. Proof of Proposition \ref{prop:Qbound:quenched} is based on the same idea, but the coupling for the events defining ${\textnormal{\textbf{Q}}}_{v,u,s}$ is more technical and we defer the proof to Appendix \ref{subsec:app:Qbound}.

As the first step in verifying Proposition \ref{prop:Jbound:quenched}, we address the following elementary lemma on simple random walk.

 \begin{lem}\label{lem:basic random walk estimates}
 	Let $W_t$ be a continuous time random walk with $W_0=0$. Define the maximum process $M_t:= \max _{s \le t} W_s$. For any $t >0$ we have:
 	\begin{enumerate}
 		\item 
 		for any $A\ge 1$
 		\begin{equation}
 		\mathbb P (M_t \le A) \le C \frac{A}{\sqrt{t+1}}
 		\end{equation}
 		\item 
 		For any integer $k \ge 0 $
 		\begin{equation}
 		\mathbb P \left( M_t \le 0, \    W_t=-k \right)\le \frac{C}{t+1}
 		\end{equation}
 	\end{enumerate} 	
 \end{lem}
 
 \begin{proof}
 	By the reflection principle we have 
 	\begin{equation}
 	\mathbb P (M_t \le A) \le \mathbb P (|W_t |\le A) \le C\frac{A} {\sqrt{t+1}}
 	\end{equation}
 	where in the last inequality we used that $\mathbb P (W_t =k) \le C/ \sqrt{t+1}$ for any $k \in \mathbb Z$ which follows from the local central limit theorem.
 	
 	The second part of the claim also follows from the reflection principle.
 \end{proof}

 Based on the above observation, we deduce the following estimate on $T$ defined in \eqref{eq:def:T:basic form}.
 \begin{lem}\label{lem:bound on H}
 For all $\epsilon >0$ there exists $\delta > 0$ such that on the event
 \begin{equation}\label{eq:def of C delta}
     \mathcal C =\mathcal C _\delta :=\{ \forall x>0, \ Y_x \le \alpha ^{-\delta } +2 \alpha x \}
 \end{equation}
 we have for all $s\le \hat{h}$
 \begin{equation}\label{eq:bound on C}
     \PP(T \ge s \: | \: Y) \le \frac{\alpha^{-\epsilon}}{\sqrt{s+1}}.
 \end{equation}
 	In particular there exist $c_\epsilon >0$ such that
 	\begin{equation}
 	\PP_\alpha \left[ \PP(T \ge s \: | \: Y) \le \frac{\alpha^{-\epsilon}}{\sqrt{s+1}}, \ \forall 0\le s \le \hat{h} \right] \ge 1- \exp\left(-\alpha^{-c_\epsilon } \right).
 	\end{equation}
 \end{lem}
 
 Note that the outer probability $\PP_\alpha$ considers the probability over the rate-$\alpha$ Poisson process $Y$, and the inner probability $\PP$ is taken over the simple random walk conditioned on $Y$.

 \begin{proof}
 	Let $\delta >0$ sufficiently small depending on $\epsilon $, and $\alpha >0$ sufficiently small depending on all other parameters. It is easy to check that $\mathbb P (\mathcal C ) \ge 1- \exp (-\alpha ^{-c_\delta })$ and therefore the second statement in the lemma follows from the first one.
 	
 	On $\mathcal C$ we have 
 	\begin{equation}
 	\mathbb P \left( T \ge s  \  \big|  \  Y \right) = \mathbb P \left(\forall x \le s, \  W_x \le Y_x \  \big|  \  Y \right)\le \mathbb P \left(\forall x \le s, \  W_x \le \alpha ^{-\delta  } +2 \alpha x\right)
 	\end{equation}
 	Denote the last probability by $p_s$. If $s \le  \alpha ^{-1-\delta }$ by the first part of Lemma~\ref{lem:basic random walk estimates} we have 
 	\begin{equation}
 	p_s \le \mathbb P \left(M_s \le 3 \alpha ^{-\delta } \right) \le C \frac{\alpha ^{-\delta  }}{\sqrt{s+1}}.
 	\end{equation}
 	If $\alpha ^{-1-\delta } \le s \le  \alpha ^{-\frac{3}{2}- \delta }$ we have 
 	\begin{equation}
 	\begin{split}
 	p_s &\le \mathbb P  \left( M_{\alpha ^{-1}} \le 2 \alpha ^{-\delta }, \ M_s \le 2 \alpha ^{-\frac{1}{2}- \delta } \right) \\
 	&\le \mathbb P \left(  M_{\alpha ^{-1}}\le 2 \alpha ^{-\delta } , \  W_{\alpha ^{-1}} \ge -\alpha ^{-\frac{1}{2}-\delta } , \ \max _{\alpha ^{-1} \le x \le s } W_x - W_{\alpha ^{-1}}   \le 3 \alpha ^{-\frac{1}{2}-\delta }, \right)+\mathbb P \left( W_{\alpha ^{-1}}\le -\alpha ^{-\frac{1}{2}-\delta } \right) \\
 	&\le \mathbb P \left(M_{\alpha ^{-1}} \le 2 \alpha ^{-\delta } \right) \mathbb P \left( M _{s - \alpha ^{-1 }} \le 3 \alpha ^{-\frac{1}{2}- \delta }\right) +C e^{-\alpha ^{-\delta }}\le C\frac{\alpha ^{-2 \delta }}{\sqrt{s+1}}
 	\end{split}
 	\end{equation}
 	if $\alpha ^{-\frac{3}{2}- \delta } \le s \le  \alpha ^{-\frac{7}{4}- \delta }$ by the same arguments we have $p_s \le C _{\delta }\alpha ^{-3 \delta } / \sqrt{s+1}$. We can repeat this process $C\log (1/\delta )$ times to obtain the bound $p_s \le C _{\delta }\alpha ^{-C\delta  \log (1/ \delta ) } / \sqrt{s+1}$ for any $s \le \alpha ^{-2-\frac{\delta }{2} } $. Taking $\delta >0$ sufficiently small depending on $\epsilon $ we get $p_s \le  \alpha ^{-\epsilon }/ \sqrt{s+1} $ for any $s\le \alpha ^{-2 -\frac{\delta }{2}}$.  
 \end{proof}

Let $f_Y$ be the probability density of $T $ given $Y$. That is 
\begin{equation}
    f_Y(s):= \frac{d}{ds} \mathbb P (T \le  s \ | \ Y).
\end{equation}
In the following lemma we give a bound on $f$ that folds with high probability.

\begin{lem}\label{lem:bound on density}
	. For all $\epsilon >0$ there exists $c_\epsilon >0$ such that 
	\begin{equation}\label{eq:the equation in the lemma on the density}
	\PP_\alpha\left[ f_Y(s)\le \frac{\alpha^{-\epsilon}}{(s+1)^{\frac{3}{2}}}, \ \forall 0\le s\le \hat{h} \right] \ge 1- \exp\left(-\alpha^{-c_\epsilon } \right).
	\end{equation}
\end{lem}

\begin{proof}
    Let $\delta >0$ sufficiently small such that Lemma~\ref{lem:bound on H} holds and recall the definition of $\mathcal C= \mathcal C _\delta $ in \eqref{eq:def of C delta}. It is clear that $f_Y$ is given by
    \begin{equation}
        f_Y(s)=\frac{1}{2} \cdot  \mathbb P \left( \forall x \le s , \ W_x \le Y_x, \ W_s=Y_s  \ | \ Y \right).
    \end{equation}
	Thus, on $\mathcal C$ we have
	\begin{equation}
	\begin{split}
	f_Y(s)&=\frac{1}{2} \sum _{k \le Y(s/2)}\mathbb P \left( \forall x \le s , \ W_x \le Y_x, \  W ( s/2 )  =k, \ W_s=Y_s  \ |  \ Y \right)\\
	&\le \sum _{k \le Y(s/2)}\mathbb P \left( \forall x \le s/2, \  W_x \le Y_x  \ | \ Y \right) \mathbb P \left(  W(s/2) =k \  \big| \ \forall x \le s/2, \  W_x \le Y_x , \ Y  \right) \\
	&\quad \quad \quad \quad \quad \quad \quad \quad \quad \quad \quad \quad \quad \quad \cdot \mathbb P \left( \forall s/2 \le x \le s, \   W_x \le Y_s, \  W_s=Y_s \ \big| \ W(s/2 )=k  , \ Y  \right)\\
	&\le \frac{C \alpha ^{-\epsilon }}{\sqrt{s+1}}\sum _{k \le Y(s/2)}\mathbb P \left(  W(s/2) =k \  \big| \ \forall x \le s/2, \  W_x \le Y_x , \ Y \right) \\
	&\quad \quad \quad \quad \quad \quad \quad \quad \quad \quad \quad \quad \quad \quad \cdot \mathbb P \left( \forall x\le s/ 2, \  W_x \le Y_s, \  W(s/2 )=k \big| \  W_0=Y_s , \ Y \right)\\
	&=  \frac{C \alpha ^{-\epsilon }}{\sqrt{s+1}}\sum _{k \le Y(s/2)}\mathbb P \left(  W(s/2) =k \  \big| \ \forall x \le s/2, \  W_x \le Y_x , \ Y \right) \\
	&\quad \quad \quad \quad \quad \quad \quad \quad \quad \quad \quad \quad \quad \quad \cdot \mathbb P \left( M_{\frac{s}{2}}=0, \  W(s/2 )=k-Y_s \ \big| \  W_0=0 , Y  \right) \\
	&\le  \frac{C \alpha ^{-\epsilon }}{(s+1)^{\frac{3}{2}}}\sum _{k \le Y(s/2)}\mathbb P \left(  W(s/2) =k \  \big| \ \forall x \le s/2, \  W_x \le Y_x \right) \le  \frac{C \alpha ^{-\epsilon }}{(s+1)^{\frac{3}{2}}}\le \frac{ \alpha ^{-2\epsilon }}{(s+1)^{\frac{3}{2}}},
	\end{split}
	\end{equation}
	where in the second inequality we used Lemma~\ref{lem:bound on H} in order to bound the first factor and we shifted and inverted time in the third factor. In the third inequality we used Lemma~\ref{lem:basic random walk estimates}. The last bound finishes the proof as $\epsilon $ is arbitrary.
\end{proof}

Let $f_u$ be the density of $T$ conditioned on $\mathcal F _u :=\sigma (Y_{\le u })=\sigma (Y_x , \ x\le u )$. That is 
\begin{equation}
    f_u(s):=\mathbb E [f_Y(s) \ | \ \mathcal F _u ]= \frac{d}{ds} \mathbb P (T \le s  \ | \ \mathcal F _u ).
\end{equation}

In the following corollary we show, using Lemma~\ref{lem:bound on density} and Doob's martingale inequality that a bound similar to \eqref{eq:the equation in the lemma on the density} holds for $f_u$ as well.

\begin{cor}\label{cor:conditioning on less information}
     For all $\epsilon >0$ there exists $c_\epsilon >0$ such that 
	\begin{equation}
	\PP_\alpha\left[ f_u(s)\le \frac{\alpha^{-\epsilon}}{(s+1)^{\frac{3}{2}}}, \ \forall 0\le u, s\le \hat{h} \right] \ge 1- \exp\left(-\alpha^{-c_\epsilon } \right).
	\end{equation}
\end{cor}

\begin{proof}
    Let $\mathcal D $ be the event inside the probability in \eqref{eq:the equation in the lemma on the density}. $\mathbb P (\mathcal D ^c \  | \ \mathcal F _u )$ is a martingale and therefore by Doob's martingale inequality we have
    \begin{equation}
      \mathbb P \Big(  \sup _{ u\le \hat{h} } \mathbb P (\mathcal D ^c \  | \ \mathcal F _u  ) \ge \alpha ^3   \Big)  \le \alpha ^{-3 } \mathbb P (\mathcal D ^c)   \le \exp (-\alpha ^{- c_\epsilon }).
    \end{equation}
    Next we claim that $f_Y(s)\le 1$. Indeed, on the event $\{ s <  T \le  s+\delta  \}$ we have that $\{ W_{s+\delta }\neq W_s \}$ which happens with probability  $1-e^{-\delta }\le \delta $ and therefore $f_Y(s) =\frac{d}{ds} \mathbb P (T\le s \ | \ Y) \le 1$. Thus, on the event $\{\sup _{ u\le s } \mathbb P (\mathcal D ^c \  | \ \mathcal F _u  ) \le \alpha ^3 \}$ we have for all $u,s\le \hat{h}$
    \begin{equation}
    \begin{split}
        f_u(s) = \mathbb E [f_Y (s) \ | \ \mathcal F _u ] &= \mathbb E [  f_Y (s) \cdot \mathds 1  _{\mathcal D } \ | \ \mathcal F _u ] +\mathbb E [f_Y (s) \cdot \mathds 1  _{\mathcal D^c } \ | \ \mathcal F _u ] \\
        &\le \frac{\alpha ^{-\epsilon}}{\sqrt{s+1}} +\mathbb P( \mathcal D ^c \ | \ \mathcal F _u)\le \frac{\alpha ^{-\epsilon}}{\sqrt{s+1}} +\alpha ^ 3  \le \frac{\alpha ^{-2\epsilon}}{\sqrt{s+1}},
        \end{split}
    \end{equation}
    where in the first inequality we used the definition of $\mathcal D $ and the fact that $f_Y (s) \le 1 $.  This finishes the proof as $\epsilon $ is arbitrary.
\end{proof}

Now we are ready to conclude the proof of Proposition \ref{prop:Jbound:quenched}.

\begin{proof}[Proof of Proposition~\ref{prop:Jbound:quenched}]
	Let $u<s<\hat{h}$. We have  
	\begin{equation}\label{eq:coupeling form of J}
	{\textnormal{\textbf{J}}}_{u,s}= \PP_\alpha (s<T_u<\infty \ | \ Y_{\le u}) - \PP_\alpha(s<T<\infty \ | \ Y_{\le u})   =  \mathbb P \left(  \mathcal A _{u,s}  \ | \ Y_{\le u} \right)-\mathbb P \left( \mathcal B _s \ | \ Y_{\le u} \right)
	\end{equation}
	where
	\begin{equation}
	\mathcal A _{u,s}:=\left\{ u \le  T <s,\ s\le T _u <\infty \right\},\quad \mathcal B _s :=\left\{s \le T  <\infty ,\  T _u =\infty  \right\}.
	\end{equation}
	
	Let $\delta >0$ sufficiently small such that Lemma~\ref{lem:bound on H} holds and let $\mathcal D $ be the event inside the probability in Lemma~\ref{lem:bound on density}. Define 
	\begin{equation}
	    \mathcal E := \mathcal D \cap \left\{  \forall u \le \alpha ^{-3} ,\ \forall x>u, \ Y_x- Y_u \le \alpha ^{-\delta } +2 \alpha (x-u) \right\}
	\end{equation}
	From Lemma~\ref{lem:bound on density} it is clear that $\mathbb P (\mathcal E ) \ge 1 - \exp (-\alpha ^{-c_\epsilon })$.
	We start by bounding the first term on the right hand side of \eqref{eq:coupeling form of J}. To this end we first condition of all of $Y$ and then use the same arguments as in Corollary~\ref{cor:conditioning on less information} to get the same bound when conditioning only on $Y _{\le u }$. By integrating over the different values $x$ that $T$ can take we get that on $\mathcal E$,
	\begin{equation}\label{eq:A_us probability}
	\begin{split}
	\mathbb P \left(\mathcal A_{u,s} \ \big|  \ Y \right) &\le  \intop _u^s f_Y (x) \cdot \mathbb P \left( \forall x \le y \le s, \ W_y \le Y_y +1 \ \big| \  W_x=Y_x+1 , \ Y \right) dx\\
	& \le \intop _u^s  \frac{\alpha ^{-\epsilon }}{(x+1)^{\frac{3}{2}}} \mathbb P \left( \forall  y \le s-x, \ W_y \le Y_{y+x}-Y_x  \ \big| \ W_0=0 , \ Y \right) dx\\
    &\le \intop _u^s  \frac{\alpha ^{-\epsilon }}{(x+1)^{\frac{3}{2}}} \frac{\alpha ^{-\epsilon }}{\sqrt{s-x+1}} dx\le \frac{C \alpha ^{-2 \epsilon }}{\sqrt{s+1}\sqrt{u+1}},
	\end{split}
	\end{equation}
	where the second inequality is by the definition of $\mathcal D $ and the third inequality is by Lemma~\ref{lem:bound on H}. Indeed, on $\mathcal E $ the process $Y'_y :=Y_{y+x}-Y_x$ satisfies $Y'_y\le \alpha ^{-\delta }+2\alpha y$ and therefore by Lemma~\ref{lem:bound on H} the bound \eqref{eq:bound on C} holds when we replace $Y$ by $Y'$. The last inequality is by Claim~\ref{claim:some claim}.

	Thus, as $\epsilon $ is arbitrary we get
	\begin{equation}
	\mathbb P \left(   \mathbb P \left( \mathcal A _{u,s} \  \big| \  Y \right) \ge \frac{\alpha ^{-\epsilon }}{ \sqrt{s+1} \sqrt{u+1}} , \  \forall 0 \le u<s \le \hat{h} \right) \ge 1-\exp ( -\alpha ^ {-c_\epsilon } ).
	\end{equation} 
	By the same arguments as in Corollary~\ref{cor:conditioning on less information} we get that 
	\begin{equation}\label{eq:bound on first term}
	\mathbb P \left(   \mathbb P \left( \mathcal A _{u,s} \  \big| \  Y_{\le u} \right) \ge \frac{\alpha ^{-\epsilon }}{ \sqrt{s+1} \sqrt{u+1}} , \  \forall 0 \le u<s \le \hat{h} \right) \ge 1-\exp ( -\alpha ^ {-c_\epsilon } ).
	\end{equation}
	
	We turn to bound the second term in \eqref{eq:coupeling form of J}. By the same arguments as in \eqref{eq:A_us probability} we have on $\mathcal E $,
	\begin{equation}
	\begin{split}
	\mathbb P \left(  \mathcal B _s \ \big| \  Y \right) &\le \intop _s^\infty  f_Y(x) \mathbb P \left( \forall  y \ge x,\  W_y \le Y_y+1  \ \big | \ W_x=Y_x+1, \ Y \right)dx \\
	 & \le \intop _s^\infty \frac{\alpha ^{-\epsilon }}{(x+1)^{\frac{3}{2}}}
	 \mathbb P \left( \forall  y \le \hat{h} ,\  W_y \le Y_{y+x}-Y_x  \ \big | \ W_0=0, \ Y \right)dx \\
	 &\le \intop _s^\infty  \frac{\alpha ^{1-2 \epsilon }}{(x+1)^ {\frac{3}{2}} } \le \frac{\alpha ^{1-3 \epsilon }}{\sqrt{s+1}} ,
	\end{split}
	\end{equation}
	where in the third inequality we used Lemma~\ref{lem:bound on H}. Since $\epsilon $ is arbitrary and by the same arguments as in Corollary~\ref{cor:conditioning on less information} we get that
	\begin{equation}\label{eq:bound on the second term}
	\mathbb P \left(   \mathbb P \left( \mathcal B _s \  \big| \  Y_{\le u} \right) \ge \frac{\alpha ^{1-\epsilon }}{ \sqrt{s+1}} , \  \forall 0 \le s \le \hat{h} \right) \ge 1-\exp ( -\alpha ^ {-c_\epsilon } ).
	\end{equation}
	Substituting \eqref{eq:bound on first term} and \eqref{eq:bound on the second term} into \eqref{eq:coupeling form of J} finishes the proof of the proposition.
\end{proof}

\subsection{Multiple integrals over the fixed rate process}\label{subsec:fixed:fixed}

Based on the results we obtained in the previous subsections, we prove Propositions \ref{prop:fixed perturbed:double int} and \ref{prop:fixed perturbed:triple int} in the case of $\Pi_\alpha$, the fixed rate Poisson process. Switching $\Pi_g$ into $\Pi_\alpha$ not only makes the underlying process be simpler, but more importantly, enables us to interpret the integral as martingales. Previously, this was impossible since $g$ is progressively measureable with respect to $\Pi$ and ${\textnormal{\textbf{J}}}_{t-s,t-u;t}^{\Pi_g}$ is measureble only in terms of $\Pi_g[s,t]$, which essentially requires to revealing the entire information of $\Pi_g$.

Let $\Pi_\alpha^{(t)}$ denote the backwards-time point process of $\Pi_\alpha$ with respect to $t$, defined as
\begin{equation}
\Pi_\alpha^{(t)}[0,s] := \Pi_\alpha[t-s,t], \ \textnormal{ with } d\widehat{\Pi}_\alpha^{(t)}(s ) = d\Pi_\alpha^{(t)}(s)-\alpha ds.
\end{equation}
Then, we can see that $\mathcal{J}$ defined in \eqref{eq:def:Jint:general} becomes
\begin{equation}\label{eq:def:Jint:rev time}
\mathcal{J}[t;\Pi_\alpha] = \intop_{0}^t \intop_{0}^s {\textnormal{\textbf{J}}}_{u,s}^{\Pi_\alpha^{(t)}} \ d\widehat{\Pi}_\alpha^{(t)}(u) d\widehat{\Pi}_\alpha^{(t)}(s),
\end{equation}
where $ {\textnormal{\textbf{J}}}_{u,s}^{\Pi^{(t)}_\alpha}:= \PP_\alpha(s<T_u<\infty\, | \, \Pi_\alpha^{(t)}[0,u] ) - \PP_\alpha(s<T<\infty\, |\, \Pi_\alpha^{(t)}[0,u])$ follows the previous definition \eqref{eq:def:J:basic form}, with respect to $Y$ generated by $\Pi_\alpha^{(t)}$. Similarly, we can write $\mathcal{Q}$ in \eqref{eq:def:Qint:general} by
\begin{equation}\label{eq:def:Qint:rev time}
	\mathcal{Q}[t;\Pi_\alpha] := \intop_{0}^{t} \intop_{0}^s  \intop_{0}^u
	{\textnormal{\textbf{Q}}}_{v,u,s}^{\Pi_\alpha^{(t)}} \ d\widehat{\Pi}_\alpha^{(t)}(v) d\widehat{\Pi}_\alpha^{(t)}(u) d\widehat{\Pi}_\alpha^{(t)}(s),
\end{equation}
where ${\textnormal{\textbf{Q}}}_{v,u,s}^{\Pi_\alpha^{(t)}}$ is defined as \eqref{eq:def:Q:basic form}, namely,
\begin{equation}\label{eq:def:Q:fixed rate}
\begin{split}
 {\textnormal{\textbf{Q}}}_{v,u,s}^{\Pi_\alpha^{(t)}} := & \PP_\alpha(s<T_{v,u}<\infty\ | \ \Pi_\alpha^{(t)}[0,v] ) - \PP_\alpha(s<T_v<\infty\ |\ \Pi_\alpha^{(t)}[0,v])\\
&- \PP_\alpha(s<T_u<\infty\ | \ \Pi_\alpha^{(t)}[0,v] ) + \PP_\alpha(s<T<\infty\ |\ \Pi_\alpha^{(t)}[0,v]).
\end{split}
\end{equation}
Furthermore, the points $\pi_i(t;\alpha)$ from \eqref{eq:def:pi closest points} can now be interpreted as the $i$-th closest point from the origin in $\Pi_\alpha^{(t)}$.
Based on these identities, our goal is to establish the following estimates.

\begin{prop}\label{prop:fixed perturbed:fixed rate}
	Let $\epsilon, C>0$ be given, and set $\hat{h}$ as \eqref{eq:def:horizon}. Then, there exists $\alpha_0=\alpha_0(\epsilon,C)>0$ such that for all $0<\alpha<\alpha_0$, we have
	\begin{equation}\label{eq:fixed perturbed:fixed result}
	\begin{split}
	&\PP \left(-\alpha^{\frac{1}{2}-\epsilon}\sigma_1(t)\le \mathcal{J}[t;\Pi_\alpha] \le \alpha^{-\epsilon} \sigma_1\sigma_2(t), \ \forall t \le \hat{h} \right) \ge 1- \exp\left(-\alpha^{-\frac{\epsilon}{200}} \right);\\
	&\PP \left(|\mathcal{Q}[t;\Pi_\alpha]| \le 
	\alpha^{-\epsilon} \sigma_1\sigma_2\sigma_3(t),\ \forall  t\le \hat{h} \right) \ge 1-\exp\left(-\alpha^{-\frac{\epsilon}{200}} \right).
	\end{split}
	\end{equation}
\end{prop}
Since everything dealt in this subsection is within the fixed rate Poisson process $\Pi_\alpha$, we use the simplified notation $\pi_i(t) = \pi_i(t;\Pi_\alpha)$ and $\sigma_i(t) = \sigma_i(t;\Pi_\alpha).$

We discuss the proof only for the second one, since the first one can be derived by analogous but simpler arguments; Although there is a difference that the first one has distinct upper and lower bound while the other is not, they can be obtained from the same proof as we will see in Remark \ref{rmk:fixed perturbed:Janalog}. The main idea is applying results from Section \ref{subsec:fixed:mgconcen} to control the triple integral \eqref{eq:def:Qint:rev time} based on the bound we got from Proposition \ref{prop:Qbound:quenched}. However, there are a couple of major differences we need to keep in mind:
\begin{enumerate}
	\item We want to estimate each integral in \eqref{eq:def:Qint:rev time} one by one, starting from the inner one. This entails understanding an appropriate continuity property (in $u,s,t',t$) of the inner integrals, which we will obtain using Lemma \ref{lem:derivative of Q and J}.
	
	\item ${\textnormal{\textbf{Q}}}_{v,u,s}^{\Pi_\alpha^{(t)}}$ is not small enough for small $v,u,s$, and we need to control such cases relying on Lemma \ref{lem:concentrationofint:continuity}, not on Lemma~\ref{lem:concentration of integral}. This is why an upper bound of $\alpha^{-\epsilon} \sigma_1\sigma_2\sigma_3$ is necessary.
	
\end{enumerate} 

The following estimate on the derivative of ${\textnormal{\textbf{Q}}}_{v,u,s}$ gives the desired continuity property later.

\begin{lem} \label{lem:derivative of Q and J}
	Let ${\textnormal{\textbf{Q}}}_{v,u,s}^{\Pi_\alpha^{(t)}}$ be defined as \eqref{eq:def:Q:fixed rate}. With probability $1$, we have
	\begin{equation}
	\begin{split}
	\left|\frac{\partial}{\partial u} {\textnormal{\textbf{Q}}}_{v,u,s}^{\Pi_\alpha^{(t)}} \right| \vee
	\left|\frac{\partial}{\partial s} {\textnormal{\textbf{Q}}}_{v,u,s}^{\Pi_\alpha^{(t)}} \right|
	\leq 4
	\end{split}
	\end{equation}
	for all $0\leq v <u<s\le  t.$
\end{lem}

\begin{proof}
	Let $\delta>0$ be a small number such that $s+\delta<t$ and recall the definition \eqref{eq:def:Q:fixed rate}. We investigate each of the four terms in \eqref{eq:def:Q:fixed rate} separately. To study the partial derivative of the first term, we observe that 
	\begin{equation} \label{eq:tau prob diff}
	\begin{split}
	& \left| \,\PP_\alpha (s\le T_{v,u+\delta}<\infty \,|\, \Pi_\alpha^{(t)}[0,v] )
	-\PP_\alpha (s\le T_{v,u}<\infty \,|\, \Pi_\alpha^{(t)}[0,v] )\,\right| \\
	&\quad =
	\PP( u\leq T_{v,u+\delta}<u+\delta,\; t \leq T_{v,u}<\infty \,|\,\Pi_\alpha^{(t)}[0,v] ),
	\end{split}
	\end{equation}
	since the only possible way to have $T_{v,u+\delta}\neq T_{v,u}$ is when $u\leq T_{v,u+\delta}<u+\delta \leq T_{v,u}$. Then, we  can write down a crude bound such that
	\begin{equation}
	\PP( s\leq T_{v,u+\delta}<u+\delta,\; t\leq T_{v,u}<\infty \,|\,\mathcal{F}_v)
	\leq \delta,
	\end{equation}
	since $s\leq T_{u,s+\delta} <s+\delta$ implies that the random walk $W$ has performed a jump between times $s$ and $s+\delta$. The same argument holds for $\delta<0$ as well, and hence we obtain that 
	\begin{equation}
	\left| \frac{\partial}{\partial u} \PP_\alpha (s\le T_{v,u}<\infty \,|\, \mathcal{F}_v ) \right| \leq 1.
	\end{equation}
	Applying this argument to the three other terms of \eqref{eq:def:Q:fixed rate}, we get the desired bound for the $s$-partial derivative of $\textbf{Q}_{v,u,s}$. 
	
	The derivative with respect to $s$ can be estimated analogously, by noticing that (for $\delta>0$)
	\begin{equation}
	|\,\PP_\alpha (s+\delta\le T_{v,u}<\infty \,|\, \mathcal{F}_v )-
	\PP_\alpha (s\le T_{v,u}<\infty \,|\, \mathcal{F}_v )\,|
	=
	\PP(s\leq T_{v,u} <s+\delta \,|\,\mathcal{F}_v).
	\end{equation}
\end{proof}

A straight-forward generalization gives the analogous estimate on $\partial_t J(s,t\,|\,\mathcal{F}_s)$, which we record in the following corollary.

\begin{cor} \label{cor:derivative of J}
	With probability $1$, we have
	\begin{equation}
	\left|\frac{\partial}{\partial s} {\textnormal{\textbf{J}}}_{u,s}^{\Pi_\alpha^{(t)}} \right|	\leq 2,
	\end{equation}
	for all $0\leq u< s\le t.$
\end{cor}

Next, we introduce stopping times that provides a fundamental advantage on dealing with the issue (2) mentioned above. 

\begin{equation} \label{eq:def:tau:Qint bound}
\begin{split}
&\tau_{\textnormal{f}1}^{\textnormal{rev}} (t)= \tau_{\textnormal{f}1}^{\textnormal{rev}}(t,\alpha) := \inf \left\{s>\alpha^{-1}: \left|\Pi_{\alpha}^{(t)}[s-\alpha^{-1},\, s]\right| \geq \alpha^{-\frac{\epsilon}{20}} \right\};\\
&\tau_{\textnormal{f}2}^{\textnormal{rev}} (t) =\tau_{\textnormal{f}2}^{\textnormal{rev}}(t,\alpha) := \inf\left\{{s>\alpha^{-1-\frac{\epsilon}{10}}: \left|\Pi_\alpha^{(t)}[s-\alpha^{-1-\frac{\epsilon}{20}},\,s] \right|=0 }\right\},\\
&\tau_{\textnormal{f}3}^{\textnormal{rev}}(t) =\tau_{\textnormal{f}3}^{\textnormal{rev}}(t,\alpha) := 
\inf \left\{s>0: \sup_{v,u: \, v<u<s} \left\{\left|{\textnormal{\textbf{Q}}}_{v,u,s}^{\Pi_\alpha^{(t)}}\right| - \frac{\alpha^{-\frac{\epsilon}{20}}}{\sqrt{(v+1)(u+1)(s+1)}}\right\} >0
\right\},\\
&\tau_{\textnormal{f}}^{\textnormal{rev}}(t) =\tau_{\textnormal{f}}^{\textnormal{rev}}(t,\alpha) := \tau_{\textnormal{f}1}^{\textnormal{rev}}(t)\wedge\tau_{\textnormal{f}2}^{\textnormal{rev}}(t)\wedge\tau_{\textnormal{f}3}^{\textnormal{rev}}(t).
\end{split}
\end{equation}

In the proof of Proposition \ref{prop:fixed perturbed:fixed rate}, $\tau_{\textnormal{f}1}^{\textnormal{rev}}$ provides fundamental control on the size of the Poisson process, and  $\tau_{\textnormal{f}2}^{\textnormal{rev}}$ is  needed to ensure that $\pi_1, \pi_2, \pi_3 \leq 4\alpha^{-1-\frac{\epsilon}{10}}$. The purpose of $\tau_{\textnormal{f}3}^{\textnormal{rev}}$ is obvious in controlling the integral \eqref{eq:def:Qint:rev time}.

\begin{lem} \label{lem:Qint:taubase}
	Let $\hat{h}$ be as \eqref{eq:def:horizon}. Under the above definition, there exists $\alpha_0 = \alpha_0(\epsilon,C)$ such that for all $0<\alpha<\alpha_0$,  $$\PP\left(\inf_{t\le \hat{h}} \left\{ \tau_{\textnormal{f}}^{\textnormal{rev}}(t,\alpha)\right\} \ge \hat{h} \right) \geq  1- \exp(-\alpha^{-\epsilon/30}).$$
\end{lem}

\begin{proof}
	To estimate $\tau_{\textnormal{f}1}^{\textnormal{rev}}$, we divide the interval $[0,\hat{h}]$ into subintervals of length $\alpha^{-1}$. Each interval $[k\alpha^{-1}, (k+1)\alpha^{-1}]$ contains more than $\alpha^{-\frac{\epsilon}{20}}/2$ points with probability less than $\exp(-\alpha^{-\frac{\epsilon}{25}})$, and we take union bound over $k=0,1,\ldots, \alpha^{-2}$. Note that if each interval $[k\alpha^{-1}, (k+1)\alpha^{-1}]$ contains at most $\alpha^{-\frac{\epsilon}{20}}/2$ points, then every interval $[t-\alpha^{-1},\,t]$ has at most $\alpha^{-\frac{\epsilon}{20}}$ points. 
	
	Controlling $\tau_{\textnormal{f}2}^{\textnormal{rev}}$ is similar, where we divide $[0,\hat{h}]$ into subintervals of length $\alpha^{-1-\frac{\epsilon}{20}}/4$. Each of these subintervals contains at least one point with probability $1-\exp(-\alpha^{-\frac{\epsilon}{25}} )$.
	
	Finally, we will see that $\tau_{\textnormal{f}3}^{\textnormal{rev}} \geq \hat{h}$ with very high probability from Proposition \ref{prop:Qbound:quenched}.
\end{proof}

\begin{lem} \label{lem:Qint 1st}
	Let $\hat{h}$ be as \eqref{eq:def:horizon}, and for each $t\in[0,\hat{h}]$ let $\tau_{\textnormal{f}}^{\textnormal{rev}}(t)$ be as \eqref{eq:def:tau:Qint bound}. Then, we have
	\begin{equation}\label{eq:Qint:result:1stint}
	\begin{split}
	\PP \left(\sup_{u'\le  u\wedge \tau_{\textnormal{f}}^{\textnormal{rev}}(t)} \left|\intop_0^{u' } {\textnormal{\textbf{Q}}}_{v,u,s}^{\Pi_\alpha^{(t)}} d\widehat{ \Pi}_\alpha^{(t)}(v) \right| \leq \frac{5\alpha^{-\frac{\epsilon}{8}}\sigma_1(t) }{\sqrt{(u+1)(s+1)}},\;\forall 0<u<s \le t\leq \hat{h} \right) \\\geq 1-\exp\left(-\alpha^{-\frac{\epsilon}{160}} \right).
	\end{split}
	\end{equation}
\end{lem}

\begin{remark} \label{rmk:Qint 1st}
	Another way to phrase Lemma \ref{lem:Qint 1st} is the following: for each $t\in[0,\hat{h}]$, define the stopping time $\tau_1^{\textnormal{int}}(t)$ as
	\begin{equation} \label{eq:def:sig1:Qint}
	\tau_1^{\textnormal{int}}(t) := \inf\left\{s>0: \exists u\leq s \textnormal{ s.t. }
	\sup_{u'\le u}\left|\intop_0^{u'} {\textnormal{\textbf{Q}}}_{v,u,s}^{\Pi_\alpha^{(t)}} d\widehat{ \Pi}_\alpha^{(t)}(v) \right| \geq \frac{5\alpha^{-\frac{\epsilon}{8}}\sigma_1(t)}{\sqrt{(u+1)(s+1)}}  \right\}.
	\end{equation}
	Then, Lemmas \ref{lem:Qint:taubase} and \ref{lem:Qint 1st} imply that \begin{equation}
	\PP_\alpha \left( \bigcap_{t\le \hat{h}}\left\{\tau_1^{\textnormal{int}}(t) \wedge \tau_{\textnormal{f}}^{\textnormal{rev}}(t) \ge t \right\} \right) \geq 1-2\exp\left(-\alpha^{-\frac{\epsilon}{160}} \right).
	\end{equation}
\end{remark}

\begin{proof}[Proof of Lemma \ref{lem:Qint 1st}]

	We first establish \eqref{eq:Qint:result:1stint} for fixed $s$ and $t$, then extend the result to the form of \eqref{eq:Qint:result:1stint}.
	
	Let $s<t\le \hat{h}$ be fixed, and set 
	\begin{equation}
	f_u(x) := {\textnormal{\textbf{Q}}}_{x,u,s}^{\Pi_\alpha^{(t)}},\quad \bar{f}_u(x):= \frac{\alpha^{-\frac{\epsilon}{20}}}{\sqrt{(x+1)(u+1)(s+1)}}, \quad g(x) \equiv \alpha.
	\end{equation}
	We apply Lemma \ref{lem:concen of int:conti:forwardtime} to these functions, setting the parameters in the lemma as follows:
	\begin{itemize}
		\item We set $\tau_- \equiv 0 $ be deterministic, and let $\tau= \tau(t) = \tau_{\textnormal{f}}^{\textnormal{rev}}(t)$.
		
		\item Set $h=\hat{h}$, $\eta = 2\alpha$, $\Delta=\alpha^{-1}$, and $N=\alpha^{-\frac{\epsilon}{20}}$. Under these values, we let $\tau' = \tau_{\textnormal{f}1}^{\textnormal{rev}}(t)$, and $\tau'' = \infty$.
		
		\item Letting  $$D = 4, \quad M = \frac{\alpha^{1-\frac{\epsilon}{8}}}{(u+1)(s+1)}, \quad A = \frac{3\alpha^{-\frac{1}{2}-\frac{\epsilon}{20}}}{\sqrt{(u+1)(s+1)}},$$ (where we set $M$ to be deterministic as $\tau_- \equiv 0$) we see that the conditions given in  \eqref{eq:concen:condition:forwardtime}  are satisfied.
		
		\item Set $\delta=\alpha^{10}$ which satisfies \eqref{eq:concen:condition:delta:forward}, and let $a = \alpha^{-\frac{\epsilon}{100}}$.
	\end{itemize}
	
	Applying Lemma \ref{lem:concen of int:conti:forwardtime} under this setting gives
	\begin{equation}
	\begin{split}
	\PP \left( \sup_{u'\le u\wedge \tau_{\textnormal{f}}^{\textnormal{rev}}(t)} 
	\left|
	\intop_0^{u'}  {\textnormal{\textbf{Q}}}_{v,u,s}^{\Pi_\alpha^{(t)}} d\widehat{\Pi}_\alpha^{(t)}(v)
	\right| \le \frac{\alpha^{-\frac{\epsilon}{8}}\sigma_1(t) }{\sqrt{(u+1)(s+1)}}
	, \ \forall u \in [0,s] 
	\right) \\
	 \ge 1- \exp\left(-\alpha^{-\frac{\epsilon}{150} }\right),
	\end{split}
	\end{equation}
	and this holds for all small enough $\alpha>0$. Note that we used $\pi_1(t) \le \alpha^{-1-\frac{\epsilon}{20}}$, which comes from $\tau_{\textnormal{f}2}^{\textnormal{rev}}(t)$, to simplify the bound inside the probability.
	
	What remains is to extend this bound to hold for all $s$ and $t$. To this end, we first take a union bound over $s<t$ in a discretized interval such that
	\begin{equation}
	s,t \in \mathcal{T}:= \{x\in[0,\hat{h}]: x= k\delta, \ k \in \mathbb{Z} \},
	\end{equation}
	where $\delta = \alpha^{10}$.
	This gives
	\begin{equation}
	\begin{split}
	\PP \left( \sup_{u'\le u\wedge \tau_{\textnormal{f}}^{\textnormal{rev}}(t)} 
	\left|
	\intop_0^{u'}  {\textnormal{\textbf{Q}}}_{v,u,s}^{\Pi_\alpha^{(t)}} d\widehat{\Pi}_\alpha^{(t)}(v)
	\right| \le \frac{\alpha^{-\frac{\epsilon}{8}}}{\sqrt{(\pi_1(t)+1)(u+1)(s+1)}}
	, \ \forall u \in [0,s], \ s,t\in \mathcal{T}, \ s\le  t 
	\right) \\
	\ge 1- \exp\left(-\alpha^{-\frac{\epsilon}{160} }\right).
	\end{split}
	\end{equation}
	
	We first extend this to all $s<t$ and $t\in \mathcal{T}$. For any $s<t$ with $t\in \mathcal{T}$, let $s_\delta \ge s$ such that $s_\delta \in \mathcal{T}$ and $s_\delta \le s+\delta$. Using the estimate
	\begin{equation}
	\left|	{\textnormal{\textbf{Q}}}_{v,u,s}^{\Pi_\alpha^{(t)}} -{\textnormal{\textbf{Q}}}_{v,u,s_\delta}^{\Pi_\alpha^{(t)}} \right| \le 4\delta= 4\alpha^{10}, 
	\end{equation}
	and the fact that $\int_0^{u\wedge{\tau_{\textnormal{f}}^{\textnormal{rev}}(t)}} \left| d\widehat{\Pi}_\alpha^{(t)} (v)\right| \le \alpha^{-1-\frac{\epsilon}{10}}$ which comes from the definition of $\tau_{\textnormal{f}1}^{\textnormal{rev}}(t)$, we obtain that
	\begin{equation}
	\begin{split}
	\PP \left( \sup_{u'\le u\wedge \tau_{\textnormal{f}}^{\textnormal{rev}}(t)} 
	\left|
	\intop_0^{u'}  {\textnormal{\textbf{Q}}}_{v,u,s}^{\Pi_\alpha^{(t)}} d\widehat{\Pi}_\alpha^{(t)}(v)
	\right| \le \frac{2\alpha^{-\frac{\epsilon}{8}}}{\sqrt{(\pi_1(t)+1)(u+1)(s+1)}}
	, \ \forall u <s \le t, \ t\in \mathcal{T} 
	\right) \\
	\ge 1- \exp\left(-\alpha^{-\frac{\epsilon}{160} }\right).
	\end{split}
	\end{equation}
	
	Finally, we extend this result to hold for all $t\le \hat{h}$. Let $t_\delta \ge t$ be $t_\delta \in \mathcal{T}$ and $\delta' := t_\delta  -t\le   \delta$. Here, we need to consider the difference between measure $\Pi_\alpha^{(t)}$ and $\Pi_\alpha^{(t_\delta)}$. Note that for any $u'\le u<s\le t$ such that $u'\le \tau_{\textnormal{f}}^{\textnormal{rev}}(t)$,
	\begin{equation}
	\intop_0^{u'} {\textnormal{\textbf{Q}}}_{v, u,s}^{\Pi_\alpha^{(t)}} d\widehat{\Pi}_\alpha^{(t)}(v)
	= \intop_{\delta'}^{u'+\delta'} {\textnormal{\textbf{Q}}}_{v-\delta',u,s}^{\Pi_\alpha^{(t_\delta)}} d\widehat{\Pi }_\alpha^{(t_\delta)}(v).
	\end{equation}
	We can then proceed by switching ${\textnormal{\textbf{Q}}}_{v-\delta',u,s}^{\Pi_\alpha^{(t_\delta)}}$ to ${\textnormal{\textbf{Q}}}_{v,u,s}^{(t_\delta)}$, and  noting that on the event $\inf_{t\le \hat{h}} \{\tau_{\textnormal{f}}^{\textnormal{rev}}(t) \} \ge \hat{h}$ which holds with high probability from Lemma \ref{lem:Qint:taubase}, we have
	\begin{equation}\label{eq:Qint:1st:conti:estim}
\begin{split}
&\left|	\intop_{\delta'}^{(u'+\delta') \wedge u} {\textnormal{\textbf{Q}}}_{v,u,s}^{\Pi_\alpha^{(t_\delta)}} d\widehat{\Pi }_\alpha^{(t_\delta)}(v)
	+ \intop_{u}^{u'+\delta'} {\textnormal{\textbf{Q}}}_{v-\delta',u,s}^{\Pi_\alpha^{(t_\delta)}} d\widehat{\Pi }_\alpha^{(t_\delta)}(v)\right|\\
	&\le
	\left|\intop_{0}^{(u'+\delta')\wedge u} {\textnormal{\textbf{Q}}}_{v,u,s}^{\Pi_\alpha^{(t_\delta)}} d\widehat{\Pi }_\alpha^{(t_\delta)}(v)\right|
	+
		\left|\intop_{0}^{\delta'} {\textnormal{\textbf{Q}}}_{v,u,s}^{\Pi_\alpha^{(t_\delta)}} d\widehat{\Pi }_\alpha^{(t_\delta)}(v)\right| + \alpha^{-\frac{\epsilon}{20}}\bar{f}_u(\pi_1(t)).
\end{split}
	\end{equation}
	Thus, we obtain 
	\begin{equation}
		\begin{split}
		\PP \left( \sup_{u'\le u\wedge \tau_{\textnormal{f}}^{\textnormal{rev}}(t)} 
		\left|
		\intop_0^{u'}  {\textnormal{\textbf{Q}}}_{v,u,s}^{\Pi_\alpha^{(t)}} d\widehat{\Pi}_\alpha^{(t)}(v)
		\right| \le \frac{5\alpha^{-\frac{\epsilon}{8}}}{\sqrt{(\pi_1(t)+1)(u+1)(s+1)}}
		, \ \forall u <s \le t \le \hat{h} 
		\right) \\
		\ge 1- \exp\left(-\alpha^{-\frac{\epsilon}{160} }\right),
		\end{split}
	\end{equation}
	concluding the proof.
\end{proof}	
	
	\begin{remark}\label{rmk:fixed perturbed:Janalog}
		We stress that the same method can be applied to deduce the corresponding bound for the integral of $\textnormal{\textbf{J}}.$ Namely, changing the definition of $\tau_{\textnormal{f}3}^{\textnormal{rev}}$ into the one that has the bound from Lemma \ref{prop:Jbound:quenched} instead of \ref{prop:Qbound:quenched}, we obtain
			\begin{equation}
			\begin{split}
			\PP \left(\sup_{s'\le  s\wedge \tau_{\textnormal{f}}^{\textnormal{rev}}(t)} \intop_0^{s' } {\textnormal{\textbf{J}}}_{u,s}^{\Pi_\alpha^{(t)}} d\widehat{ \Pi}_\alpha^{(t)}(u)  \leq \frac{\alpha^{-\frac{\epsilon}{8}}\sigma_1(t) }{\sqrt{(s+1)}},\;\forall 0<s \le t\leq \hat{h} \right) \geq 1-\exp\left(-\alpha^{-\frac{\epsilon}{160}} \right);\\
				\PP \left(\inf_{s'\le  s\wedge \tau_{\textnormal{f}}^{\textnormal{rev}}(t)} \intop_0^{s' } {\textnormal{\textbf{J}}}_{u,s}^{\Pi_\alpha^{(t)}} d\widehat{ \Pi}_\alpha^{(t)}(u)  \geq -{\alpha^{1-\frac{\epsilon}{8}}\sigma_1(t) }-\frac{\alpha^{\frac{1}{2} -\frac{\epsilon}{8} }}{\sqrt{s+1}} ,\;\forall 0<s \le t\leq \hat{h} \right) \geq 1-\exp\left(-\alpha^{-\frac{\epsilon}{160}} \right).
			\end{split}
			\end{equation}
			(Note that unlike the upper bound, the two terms in the lower bound cannot be unified together since one is not always larger than the other.)
			Based on this observation, we similarly obtain the analogue of Corollary \ref{cor:Qint 2nd} for $\textnormal{\textbf{J}}$ to deduce the first conclusion of Proposition \ref{prop:fixed perturbed:fixed rate}.
	\end{remark}
	
	Next step, we derive an analogous result after integrating once more. For convenience we define
	\begin{equation}
		F_1(u,s,t):= \intop_0^u {\textnormal{\textbf{Q}}}_{v,u,s}^{\Pi_\alpha^{(t)}} d\widehat{ \Pi}_\alpha(v).
	\end{equation}

\begin{cor}\label{cor:Qint 2nd}
	For each $t$, let $\tilde{ \tau}_1^{\textnormal{int}}(t):= \tau_1^{\textnormal{int}}(t) \wedge \tau_{\textnormal{f}}^{\textnormal{rev}}(t)$ with $\tau_1^{\textnormal{int}}(t) $ given in Remark \ref{rmk:Qint 1st}. Then, we have 
	\begin{equation}\label{eq:Qint 2nd:main ineq:inlem}
	\begin{split}
	\PP \left( \sup_{s'\le s\wedge \tilde{\tau}_1^{\textnormal{int}}(t) } \left|\intop_0^{s'} F_1(u,s,t) d\widehat{ \Pi}_\alpha^{(t)}(u) \right| \leq \frac{\alpha^{-\frac{\epsilon}{4}}\sigma_1\sigma_2(t)}{\sqrt{s+1}},\;\forall 0\le s\le t\leq\hat{h} \right) \\
	\geq 1-\exp\left(-\alpha^{-\frac{\epsilon}{160}} \right).
	\end{split}
	\end{equation}
\end{cor}

\begin{remark} \label{rmk:Qint 2nd}
	Similarly as Remark \ref{rmk:Qint 1st}, define
	\begin{equation} \label{eq:def:sig2:Qint}
	\tau_2^{\textnormal{int}}(t):= \inf \left\{s>0:
	\sup_{s'\le s}\left| \intop_0^{s'} F_1(u,s,t) d\widehat{ \Pi}_\alpha^{(t)}(u) \right|\geq 
	\frac{\alpha^{-8\epsilon}\sigma_1\sigma_2(t)}{\sqrt{s+1}}
	\right\}.
	\end{equation}
	Then, Lemma \ref{lem:Qint 1st} and Corollary \ref{cor:Qint 2nd} imply
	\begin{equation}
	\PP\left( \bigcap_{t\le \hat{h}} \left\{ \tau_2^{\textnormal{int}}(t) \wedge \tilde{\tau}_1^{\textnormal{int}}(t) \ge t \right\}  \right) \geq 1-3\exp\left(-\alpha^{-\frac{\epsilon}{160}} \right).
	\end{equation}
\end{remark}

\begin{proof}[Proof of Corolloary \ref{cor:Qint 2nd}]
	The proof goes similar to that of Lemma \ref{lem:Qint 1st}, and hence we describe the argument more concisely. We first work with fixed $t$, and then extend the result to the desired conclusion. For each $s\le t \le \hat{h}$, define
	\begin{equation}
	f_s(x):= F_1(x,s,t), \quad \bar{f}_s(x):= \frac{5\alpha^{-\frac{\epsilon}{8}}\sigma_1(t) }{\sqrt{(x+1)(s+1)}}, \quad g(x)\equiv \alpha.
	\end{equation}
	We again apply Lemma \ref{lem:concen of int:conti:forwardtime} to these functions, with the other parameters set to be as follows.
	\begin{itemize}
		\item Let $\tau_- = \tau_-(t):= \pi_1(t)$, and let $\tau = \tau(t) =  \tilde{\tau}_1^{\textnormal{int}}(t)$.
		
		\item Set $h=\bar{h}, \ \eta = 2\alpha, \ \Delta = \alpha^{-1},$ $N=\alpha^{-\frac{\epsilon}{20}}$,   $\tau' = \tau_{\textnormal{f}1}^{\textnormal{rev}}(t)$, and $\tau''=\infty$ as in the proof of Lemma \ref{lem:Qint 1st}.
		
		\item Similarly before, the parameters
		\begin{equation}
		D=\alpha^{-2},\quad M = \frac{\alpha^{1-\frac{\epsilon}{3}} \sigma_1(t)^2}{s+1}, \quad A = \frac{\alpha^{-\frac{1}{2}-\frac{\epsilon}{7}} \sigma_1(t)}{\sqrt{s+1}},
		\end{equation}
		satisfy the conditions in \eqref{eq:concen:condition:forwardtime}. Note that $D$, the parameter for the derivative can be controlled deterministically by
		\begin{equation}
		|\partial_u F_1(u,s,t)| \le {\textnormal{\textbf{Q}}}_{u,u,s}^{\Pi_\alpha^{(t)}} \cdot (1+\alpha) 
		+
		\intop_{0}^u \left|\partial_u {\textnormal{\textbf{Q}}}_{v,u,s}^{\Pi_\alpha^{(t)}} \right|\cdot \left| d\widehat{\Pi}_\alpha^{(t)} (v)\right| \le \alpha^{-2},
		\end{equation}
		for all $s\le \tau \le \tau_{\textnormal{f}1}^{\textnormal{rev}} $. $|\partial_s F_1(u,s,t)| $ can be estimated in the same but simpler way, since the regime of integral is independent of $s$. 
		
		\item Set $\delta = \alpha^{10}$ as before which satisfies \eqref{eq:concen:condition:delta:forward}, and let $a=\alpha^{\frac{\epsilon}{100}}.$
	\end{itemize}
	
	Then, Lemma \ref{lem:concen of int:conti:forwardtime} implies that for each $t\in[0,\hat{h}]$,
	\begin{equation}
	\begin{split}
	\PP \left( 
	\sup_{s'\le s\wedge \tilde{\tau}_{1}^{\textnormal{int}}} \left|
	\intop_{\tau_-}^{s'} F_1(u,s,t)d\widehat{\Pi}_\alpha^{(t)}(u)
	\right|
	\le
	\frac{\alpha^{-\frac{\epsilon}{4}}\sigma_1\sigma_2(t)}{\sqrt{s+1}}, \ \forall s\in [0,t]
	\right)\\
	\ge 1- \exp\left(-\alpha^{-\frac{\epsilon}{150}} \right),
	\end{split}
	\end{equation}
	which holds for all small enough $\alpha>0.$ Note that we used $\pi_2(t) \le 3 \alpha^{-1-\frac{\epsilon}{20}}$, which comes from $\tau_{\textnormal{f}2}^{\textnormal{rev}}(t)$, to simplify the bound inside the probability.
	
	Moreover, note that the definition $\tau_- = \pi_1(t)$ gives that with probability one,
	\begin{equation}
\begin{split}
	\left|\intop_{0}^{\tau_-} F_1(u,s,t) d\widehat{\Pi}_\alpha^{(t)}(u) \right|
	&\le
	\intop_0^{\alpha^{-1-\frac{\epsilon}{20}}}
	\intop_0^u \frac{\alpha^{-\frac{\epsilon}{20}} \alpha^2 dvdu}{\sqrt{(v+1)(u+1)(s+1)}} \\
	&= \frac{2\alpha^{1-\frac{\epsilon}{10}}}{\sqrt{s+1}}
	\le
	\frac{\alpha^{-\frac{\epsilon}{4}} \sigma_1\sigma_2(t)}{\sqrt{s+1}},
\end{split}
	\end{equation}
	where the first and the last inequality hold if $s\le \tau_{\textnormal{f}}^{\textnormal{rev}}(t)$. Thus, we combine the two to obtain
	\begin{equation}
	\begin{split}
	\PP \left( 
	\sup_{s'\le s\wedge \tilde{\tau}_{1}^{\textnormal{int}}} \left|
	\intop_{0}^{s'} F_1(u,s,t)d\widehat{\Pi}_\alpha^{(t)}(u)
	\right|
	\le
	\frac{2\alpha^{-\frac{\epsilon}{4}}\sigma_1\sigma_2(t)}{\sqrt{s+1}}, \ \forall s\in [0,t]
	\right)\\
	\ge 1- \exp\left(-\alpha^{-\frac{\epsilon}{150}} \right),
	\end{split}
	\end{equation}
	
Extending this to the desired result is analogous as the argument in the proof of Lemma \ref{lem:Qint 1st}. Namely we take a union bound over $t$ in the discretized interval $\mathcal{T}$, and deduce an estimate corresponding to \eqref{eq:Qint:1st:conti:estim} appealing to the fact that the derivatives of $F_1$ are bounded. The details are omitted and left for interested readers.
\end{proof}

Now we are only left with the outer integral. The corollary below follows analogously as the proof of Lemma \ref{lem:Qint 1st} and Corollary \ref{cor:Qint 2nd}. Thus, we state the result and omit the details of its proof.

\begin{cor} \label{cor:Qint 3rd}
	Let $\tilde{\tau}_2^{\textnormal{int}}(t)  := \tau_2^{\textnormal{int}}(t) \wedge \tilde{\tau}_1^{\textnormal{int}}(t)$ with $\tau_2^{\textnormal{int}}(t), \tilde{\tau}_1^{\textnormal{int}}(t)  $ given in Remark \ref{rmk:Qint 2nd}, Lemma \ref{cor:Qint 2nd}, respectively, and define
	\begin{equation}
	F_2(s,t):= \intop_0^s F_1(u,s,t) d\widehat{ \Pi}_\alpha^{(t)} (u).
	\end{equation}
	Then, we have
	\begin{equation}
	\begin{split}
	\PP \left( \sup_{t'\leq t \wedge \tilde{\tau}_2^{\textnormal{int}}(t) } \left| \intop_0^{t'} F_2(s,t) d\widehat{ \Pi}_\alpha^{(t)} (s) \right| \leq \alpha^{-\epsilon} \sigma_1\sigma_2\sigma_3(t), \ \forall t\le\hat{h} \right) \geq 1-\exp \left( -\alpha^{-\frac{\epsilon}{160}}\right).
	\end{split}
	\end{equation}
\end{cor}

We conclude Section \ref{subsec:fixed:fixed} by establishing Proposition \ref{prop:fixed perturbed:fixed rate}. 

\begin{proof}[Proof of Proposition \ref{prop:fixed perturbed:fixed rate}]
	Combining Remark \ref{rmk:Qint 2nd} and Corollary \ref{cor:Qint 3rd} implies that
	\begin{equation}
	\PP \left(|\mathcal{Q}[t;\Pi_\alpha]| \le 
	\alpha^{-\epsilon} \sigma_1\sigma_2\sigma_3(t),\ \forall  t\le \hat{h} \right) \ge 1-\exp\left(-\alpha^{-\frac{\epsilon}{200}} \right),
	\end{equation}
	proving the second statement of \eqref{eq:fixed perturbed:fixed result} in Proposition \ref{prop:fixed perturbed:fixed rate}. 
	The first inequality in \eqref{eq:fixed perturbed:fixed result} can be obtained analogously by developing Lemma \ref{lem:Qint 1st}, Corollary \ref{cor:Qint 2nd}, Remarks \ref{rmk:Qint 1st} and \ref{rmk:Qint 2nd}, based on the bound on ${\textnormal{\textbf{J}}}_{u,s}^{\Pi_\alpha^{(t)}}$ given in Proposition \ref{prop:Jbound:quenched}. We omit the details which are left for interested readers. 
\end{proof}

To conclude this subsection, we introduce the following result which controls the integral of the error bound given in Proposition \ref{prop:fixed perturbed:fixed rate}

\begin{lem}\label{lem:fixed perturbed:error int:fixed}
	let $\epsilon, C >0$ be given, and set $\hat{h}$ as \eqref{eq:def:horizon}. Then, there exists $\alpha_0=\alpha_0(\epsilon,C)>0$ such that for all $0<\alpha<\alpha_0$, we have 
	\begin{equation}
	\begin{split}
	\PP_\alpha \left( 	\left|\intop _{0}^{\hat{h}} \sigma_1\sigma_2(t )dt \right| \le  \alpha ^{-1-\epsilon }  \right) &\ge 1-\exp\left(-\alpha^{-\frac{\epsilon}{20}} \right);\\
	\PP_\alpha \left( 	\left|\intop _{0}^{\hat{h}} \sigma_1\sigma_2\sigma_3(t )dt \right| \le  \alpha ^{-\frac{1}{2}-\epsilon }  \right) &\ge 1-\exp\left(-\alpha^{-\frac{\epsilon}{20}} \right).
	\end{split}
	\end{equation}
\end{lem}

\begin{proof}
	Let $N:=\Pi _\alpha [0,\hat{h}]$ and let $\pi _1 \le \pi _2\le \cdots $ be the points of the process. For convinence we let $\pi _0:=0$ and $\pi _{N+1}=T$. Note that with very high probability $\alpha ^{-1+\epsilon }\le N\le \alpha ^{-1-\epsilon }$ and for any $0 \le i \le N$, $\pi _{i+1}-\pi _i\le \alpha ^{-1-\epsilon }$. We start with the first inequality. We have with very high probability  
	\begin{equation}
	\intop _{0}^{T} \frac{dt}{\sqrt{(\pi _1(t)+1)(\pi _2(t) +1)}} \le \intop _{0}^{T} \frac{dt}{\pi _1(t)+1} =\sum _{i=0}^{N} \intop _{\pi _i}^{\pi _{i+1}} \frac{dt }{t-\pi _i+1}\le C N \log(1/ \alpha ) \le \alpha ^{-1-\epsilon }  
	\end{equation}
	
	We turn to prove the second inequality. For $k \le \log _2 (1/\alpha )$ and $\alpha ^{-1+\epsilon } \le n\le \alpha ^{-1-\epsilon }$ define the sets
	\begin{equation}
	I_{k}^n:=\left\{ i\le n \  \big| \ \pi _{i}-\pi _{i-1}\in [2^k,2^{k+1}]\right\}
	\end{equation}
	and $I_{k}:=I_{k}^N$. We have that $\mathbb P (i \in I _{k}^n)\le \alpha 2^k$ and therefore with very high probability for all $\alpha ^{-1+\epsilon } \le n \le \alpha ^{-1-\epsilon }$ we have $|I_{k}^n|\le \alpha ^{-\epsilon }n\alpha 2^k \le \alpha ^{-2\epsilon }2^k$. Thus with very high probability $|I_{k}|\le  \alpha ^{-2\epsilon }2^k$.
	Thus, if we let $I':=[N]\setminus \bigcup _{k}I_{k}$ we have 
	\begin{equation}
	\begin{split}
	\intop _{0}^{T}& \frac{dt}{\sqrt{(\pi _1(t)+1)(\pi _2(t) +1)(\pi _3(t) +1)}} \le \intop _{0}^{T} \frac{dt}{\sqrt{(\pi _1(t)+1)}(\pi _2(t)+1)} \\
	&= \sum _{i = 0}^N \intop _{\pi _i}^{\pi _{i+1}} \frac{dt}{\sqrt{(t- \pi _i +1)}(\pi _2(t) +1)} \le C\alpha \sum _{i \in I'} \intop _{\pi _i}^{\pi _{i+1}} \frac{1}{\sqrt{t- \pi _i +1 }}+  C \sum _{k=1 }^{\log _2(1/\alpha )}  2^{-k} \sum _{i \in I _k} \intop _{\pi _i}^{\pi _{i+1}} \frac{1}{\sqrt{t-\pi _i +1}} \\
	&\le C \alpha ^{\frac{1}{2}-\epsilon }|I'| +  C\alpha ^{-\frac{1}{2}-\epsilon } \sum _{k=1}^{\log _2(1/\alpha )} 2^{-k}|I_k|  \le CN\alpha ^{\frac{1}{2}-\epsilon } +C\alpha ^{-\frac{1}{2}-3 \epsilon }\log (1/\alpha ) \le C \alpha ^{-\frac{1}{2}-4 \epsilon }.\qedhere 
	\end{split}
	\end{equation}
\end{proof}

\subsection{Perturbation of  the underlying Poisson processes}\label{subsec:fixed:perturbed}

In this subsection, we establish Propositions \ref{prop:fixed perturbed:double int} and \ref{prop:fixed perturbed:triple int}. As mentioned before, our method is to derive an estimate on the Radon-Nykodym derivative of $\Pi_g$ with respect to $\Pi_\alpha$ and use Proposition \ref{prop:fixed perturbed:fixed rate}.

Let $\mathcal{P}_g^t$ denote the law of $\Pi_g[0,t]$. Using this notation, we will write, for instance,
\begin{equation}
\PP \left( \mathcal{J}[t;\Pi_g] \in \mathcal{E}  \right)
 =\int_{\Pi} I\{\mathcal{J}[t;\Pi] \in \mathcal{E} \} \mathcal{P}_g^t(d\Pi).
 \end{equation}
Similarly, we let $\mathcal{P}_\alpha^t$ to be the law of $\Pi_\alpha[0,t]$. The main observation to establish Propositions \ref{prop:fixed perturbed:double int} and \ref{prop:fixed perturbed:triple int} is the following: for  a stopping time $\tau$ in terms of $g$ and $\Pi$, the Radon-Nykodym derivative of $\mathcal{P}_g^\tau$ with respect to $\mathcal{P}_\alpha^\tau$ can be written as
\begin{equation}\label{eq:def:radon deriv:fixed perturbed}
r_g^{\tau}( \Pi)  := \frac{d\mathcal{P}_g^{\tau} }{d\mathcal{P}_\alpha^{\tau} }(\Pi) = \exp\left( -\intop_0^{\tau} (g(y)-\alpha) dy \right) \prod_{x\in \Pi[0,{\tau}]} \frac{g(x)}{\alpha}.
\end{equation}
This comes from the fact that $g$ is predictable with respect to $\Pi$. 

From now on, let $\tau$ be the stopping time given in Proposition \ref{prop:fixed perturbed:double int}, and for convenience we set $${\hat{\tau}}:=\hat{h}\wedge \tau.$$ The next lemma shows how to control the size of $r_g^\tau(\Pi)$, when $\Pi$ is given by  $\Pi_g \sim \mathcal{P}_g^\tau$. 

\begin{lem}\label{lem:fixed perturbed:radon deriv bd}
	Under the setting of Proposition \ref{prop:fixed perturbed:double int}, we have
	\begin{equation}
	\PP_{\Pi \sim \mathcal{P}_g^{\hat{\tau}}}\left( r_g^{\hat{\tau}}(\Pi) \ge \exp\left(\alpha^{-\frac{\epsilon}{250}} \right) \right) \le \exp\left(-\alpha^{\frac{\epsilon}{2000}} \right).
	\end{equation}
\end{lem}

\begin{proof}
From \eqref{eq:def:radon deriv:fixed perturbed}, we use $1+x\le e^x$ to obtain that
	\begin{equation}
	\begin{split}
	r_g^{\hat{\tau}}(\Pi )\le \exp \left( \alpha ^{-1} \intop_{0}^{\hat{\tau}} (g(x)-\alpha ) \, d\Pi (x)-\intop_{0}^{\hat{\tau}} (g(x)-\alpha ) \, dx  \right).
	\end{split}
	\end{equation}
Thus, for $\lambda := \alpha^{\frac{\epsilon}{300}}$ we have $r_g^\tau(\Pi)^\lambda \le M_\tau\cdot L_\tau$ where
	\begin{equation}
	\begin{split}
	M_t:&=\exp \left( \lambda \alpha ^{-1} \intop_{0}^t (g(x)-\alpha ) \, d\Pi _g (x)-\intop_{0}^t g(x)(e^{\lambda \alpha ^{-1} (g(x)-\alpha )}-1) \, dx  \right), \\
	L_t:&=\exp \left(  \intop_{0}^t g(x)(e^{\lambda \alpha ^{-1} (g(x)-\alpha )}-1) - \lambda (g(x)-\alpha ) \, dx  \right).
	\end{split}
	\end{equation}
	
	We start by bounding  $L_{ {\hat{\tau}} }$. We have that 
	\begin{equation}
	\begin{split}
	L_{{\hat{\tau}} }&\le \exp \left(  \intop_{0}^{{\hat{\tau}} } \left\{g(x)\left( \lambda \alpha ^{-1} (g(x)-\alpha )+\lambda ^2  \alpha ^{-2} (g(x)-\alpha )^2 \right)  - \lambda (g(x)-\alpha ) \right\}\, dx  \right)\\
	& \le \exp \left( \intop_0^{{\hat{\tau}}} \left\{ \lambda \alpha^{-1} (g(x)-\alpha)^2 + \lambda^2\alpha^{-2} (g(x)-\alpha)^2g(x)  \right\} dx \right)\le 2
	\end{split}
	\end{equation}
	where in the first inequality we used  $\lambda \alpha ^{-1} |g(x)-\alpha |\le \lambda \alpha ^{-1}(g(x)+\alpha) \le 1$ from the third assumption of \eqref{eq:fixed perturbed:assumption}, along with the fact that $e^y\le 1+y+y^2$ for $y\le 1$. In the last inequlity we used  the first two assumptions of \eqref{eq:fixed perturbed:assumption}.
	Thus, we obtain that 
	\begin{equation}
		\begin{split}
		\exp\left(\lambda \alpha^{-\frac{\epsilon}{250}} \right)\cdot \PP_{\Pi \sim \mathcal{P}_g^{\hat{\tau}}}\left( r_g^{\hat{\tau}}(\Pi) \ge \exp\left(\alpha^{-\frac{\epsilon}{250}} \right) \right) &\le 
		\mathbb E_{\Pi\sim \mathcal{P}_g^{\hat{\tau}}} \left[ r_{g }^{{\hat{\tau}}}(\Pi)^\lambda \right]\\
		& \le  \mathbb E \left[ M_{ {\hat{\tau}} }\cdot L_{ {\hat{\tau}} } \right]\le 2 \mathbb E \left[ M_{ {\hat{\tau}} } \right]=2,
		\end{split}
	\end{equation}
	where in the last equality we used that $M_t$ and therefore $M_{t\wedge \tau }$ are martingales and that $M_{0}=1$. This concludes the proof of the lemma.
\end{proof}

Now, We are ready to  establish Propositions \ref{prop:fixed perturbed:double int} and \ref{prop:fixed perturbed:triple int}.

\begin{proof}[Proof of Propositions~\ref{prop:fixed perturbed:double int} and \ref{prop:fixed perturbed:triple int}]
	Define the events 
	\begin{eqnarray}
	\mathcal A_1(t)  :=\left\{ \Pi : -\alpha^{\frac{1}{2}-\epsilon}\sigma_1(t)\le  \mathcal{J}[t;\Pi ]  \le \alpha ^{-\epsilon }\sigma_1\sigma_2(t)  \right\}, &
	\mathcal A_2(t)  :=\left\{ \Pi : |\mathcal{Q}[t;\Pi ] | \le \alpha ^{-\epsilon }\sigma_1\sigma_2\sigma_3(t)  \right\}, \\
	\mathcal{A}(t):= \mathcal{A}_1(t)\bigcap \mathcal{A}_2(t) ,\qquad \qquad  \qquad  \qquad & 
	\mathcal B := \left\{ \Pi : r_g^{{\hat{\tau}}} (\Pi )\ge \exp\left(\alpha^{-\frac{\epsilon}{250}} \right)  \right\}. \quad 
	\end{eqnarray}
	Moreover, we set $\mathcal{A}_{\le t} := \cap_{s\le t} \mathcal{A}(s)$. Then, we have
	\begin{equation}\label{eq:fixed perturbed:radon dec}
	\PP_{\Pi \sim \mathcal{P}_g^{\hat{\tau}}} \left( \left(\mathcal{A}_{\le {\hat{\tau}}}\right)^c  \right)
	\le \PP_{\Pi \sim \mathcal{P}_g^{\hat{\tau}}} \left( \left(\mathcal{A}_{\le {\hat{\tau}}}\right)^c \bigcap  \mathcal{B}^c \right) + \PP_{\Pi \sim \mathcal{P}_g^{\hat{\tau}}} \left( \mathcal{B} \right).
	\end{equation}
	Lemma~\ref{lem:fixed perturbed:radon deriv bd} tells us that
	the second term in the RHS is bounded by $\exp(-\alpha^{-\frac{\epsilon}{2000}})$, and the first term can be estimated by
	\begin{equation}
\begin{split}
	\PP_{\Pi \sim \mathcal{P}_g^{\hat{\tau}}} \left( \left(\mathcal{A}_{\le {\hat{\tau}}}\right)^c \bigcap  \mathcal{B}^c \right)
	&\le
	\int_{\Pi} I \left\{\left(\mathcal{A}_{\le \hat{h}}\right)^c \bigcap \mathcal{B}^c \right\} r_g^{\hat{\tau}}(\Pi) \mathcal{P}_\alpha^{\hat{\tau}}(d\Pi)\\
	&\le
	\exp\left(\alpha^{-\frac{\epsilon}{250}} \right) \int_{\Pi}I\{\left(\mathcal{A}_{\le \hat{h}}\right)^c \} \mathcal{P}_\alpha^{{\hat{h}}}(d\Pi)\\
	&\le
	\exp\left(-\alpha^{-\frac{\epsilon}{250}} \right),
\end{split}
	\end{equation}
	where the last inequality comes from Proposition \ref{prop:fixed perturbed:fixed rate}. Plugging the two bounds into \eqref{eq:fixed perturbed:radon dec} deduces the desired results.
\end{proof}

The same proof as above can be applied to generalize Lemma \ref{lem:fixed perturbed:error int:fixed}. Due to its similarity, we omit the proof of the following result.

\begin{lem}\label{lem:fixed perturbed:error int:perturbed}
	Let $\epsilon>0$ be arbitrary, $\alpha>0$ be a sufficiently small constant depending on $\epsilon$, and $\hat{h}$ be as \eqref{eq:def:horizon}. Let $\tau$ be a stopping time, and $\{g(s)\}_{s\ge 0} $ be a positive stochastic process progressively measurable with respect to $\Pi_g$, and suppose that they satisfy \eqref{eq:fixed perturbed:assumption}. Then, we have
	\begin{equation}
	\begin{split}
\PP \left(\left|\intop_0^{\hat{h}} \sigma_1\sigma_2(t;g) dt\right| \le \alpha^{-1-\epsilon} \right)& \ge 1-\exp\left(-\alpha^{-\frac{\epsilon}{3000}} \right);\\
	\PP \left(\left|\intop_0^{\hat{h}} \sigma_1\sigma_2\sigma_3(t;g) dt\right| \le \alpha^{-\frac{1}{2}-\epsilon} \right) &\ge 1-\exp\left(-\alpha^{-\frac{\epsilon}{3000}} \right).
	\end{split}
	\end{equation}
\end{lem}

\subsection{The error estimates of the speed}\label{subsec:fixed:error}

In this subsection, we translate the main estimates into the error estimates of the first- and second-order approximations (\eqref{eq:def:S1:basic form} and \eqref{eq:def:S2:basic form}). The results are direct consequences of Propositions \ref{prop:fixed perturbed:double int} and \ref{prop:fixed perturbed:triple int}, and we can state them as follows.

\begin{prop}\label{prop:fixed perturbed:error}
		Let  $\epsilon,C>0$ be arbitrary, $\alpha>0$ be a  sufficiently small constant depending on $\epsilon,C
		$, and let $t^, \hat{t}$ be $t^-<\hat{t}$ satisfying  $\hat{t}-t^- \le \alpha^{-2} \log^C(1/\alpha)$. Define $S_1(s)=S_1(s;t^-,\alpha)$ and $S_2(s)=S_2(s;t^-,\alpha)$ as  \eqref{eq:def:S1:basic form} and \eqref{eq:def:S2:basic form}, respectively, and recall the notation $\sigma_i(s;S)$ from \eqref{eq:def:sig closest points}. Suppose that there is a stopping time $\tau$ satisfying the following conditions almost surely:
		\begin{equation}\label{eq:error:assumptions}
		\begin{split}
		&\intop_{t^-}^{\hat{t}\wedge \tau} (S(s)-\alpha)^2 ds \le \alpha^{1-\frac{\epsilon}{400}}; \qquad  \intop_{t^-}^{\hat{t}\wedge \tau} (S(s)-\alpha)^2 S(s)ds \le \alpha^{2-\frac{\epsilon}{400}};\\ & \sup_{t^-\le s\le \hat{t}\wedge\tau }\{ S(s) \vee S_1(s )\} \le \alpha^{1-\frac{\epsilon}{400}}. 
		\end{split}
		\end{equation}
		Further, let $S'(s):= S'(s;t^-,\alpha)$ be defined as \eqref{eq:def:Sprime:basic form}.
		Then, we have  that
		\begin{equation}
		\begin{split}
		\PP\left(\left. -\alpha^{\frac{3}{2}-\epsilon}\sigma_1(s;S)\le  S'(s) - S_1(s) \le \alpha^{1-\epsilon} \sigma_1\sigma_2(s;S), \ \forall s\in[t^-, \hat{t}\wedge \tau]\, \right| \, \mathcal{F}_{t^-} \right) &\ge 1-2\exp\left(-\alpha^{-\frac{\epsilon}{3000}} \right);\\
		\PP \left( \left.|S'(s)-S_2(s)| \le \alpha^{1-\epsilon}\sigma_1\sigma_2\sigma_3(s;S), \ \forall s\in [t^-, \hat{t}\wedge \tau]\,\right| \, \mathcal{F}_{t^-} \right)
		&\ge
		1-2\exp\left(-\alpha^{-\frac{\epsilon}{3000}} \right),
		\end{split}
		\end{equation}
		where this holds for any given $\mathcal{F}_{t^-}$, the sigma-algebra generated by $\Pi[0,t^-]$ and $\{S(s)\}_{s\le t^-}$. 
\end{prop}

\begin{proof}
	Note that the second statement follows directly from the assumptions, Proposition \ref{prop:fixed perturbed:triple int}, and the formula \ref{eq:speed:2ndorder main}. The first inequality follows similarly: we first claim that
	\begin{equation}
	\PP\left( \pi_2(s;S) \le \alpha^{-1-\frac{\epsilon}{10}}, \ \forall s\in [t^-, \hat{t}\wedge\tau] \right) \ge 1-\exp\left(-\alpha^{-\frac{\epsilon}{3000}} \right).
	\end{equation}
	This follows from the analogous argument as Lemma \ref{lem:Qint:taubase} used to establish bounds on $\tau_{\textnormal{f}2}^{\textnormal{rev}}$, in our case relying on the assumption $\sup_{t^-\le s\le \hat{t}\wedge\tau} S_1(s) \le \alpha^{1-\frac{\epsilon}{400}}$. This implies	
	\begin{equation}\label{eq:fixed perturbed:prop:med}
	\PP \left( \frac{4\alpha+4\alpha^2}{(1+2\alpha)^2} S_1(s) \le \frac{1}{2}\alpha^{\frac{3}{2}-\epsilon}\sigma_1(s;S) \le \frac{1}{2}\alpha^{1-\epsilon}\sigma_1\sigma_2(s;S), \ \forall s\in[t^-,\hat{t}\wedge\tau] \right) \ge 1-\exp\left(-\alpha^{-\frac{\epsilon}{3000}} \right),
	\end{equation} 
	and combining this with Proposition \ref{prop:fixed perturbed:double int} and the formula \eqref{eq:speed:1storder main} gives the conclusion.
\end{proof}

We remark that the assumptions \eqref{eq:error:assumptions} are actually satisfied by the speed with high probability with appropriate parameters $t^-, \hat{t}$ and the stopping time $\tau.$ The formal definitions and their proofs will be discussed in Sections \ref{sec:reg:intro}--\ref{sec:reg:next step}.

\section{The age-dependent critical branching process}\label{sec:criticalbranching}

In Section \ref{sec:inductionbase}, we saw that the speed drops down to an arbitrarily low size and can be sandwiched by two fixed rate processes. 
To continue our investigation, we need to understand the evolution of a speed process which  begins with points given by a fixed rate  such as  \eqref{eq:def:sandwichfixed}.  
In this section, we study an age-dependent critical branching process  which is closely related with the actual speed process. In the first (or second) order approximation of the speed, it will turn out that its main term stays close to a certain critical branching process, and hence understanding the latter process will be crucial in our analysis.

For two parameters $t_0^{-}<t_0$, we define the age-dependent critical branching process  starting from a configuration of points $\Pi_{0}[t_0^{-},t_0]$ in the interval $[t_0^{-},t_0]$. From now on, we use the following notation to emphasize the prescribed initial points: For $t\ge t_0$, we define
\begin{equation}\label{eq:def:Rb}
R(t) = \mathcal{R}_b(t;\Pi_0[t_0^{-},t_0],\alpha)
:=
\intop_{t_0^{-}}^{t_0} K_\alpha(t-x) d\Pi_0(x)+
\intop_{t_0}^{t} K_\alpha(t-x) d\Pi_R(x),
\end{equation}
and set $R(t)=\alpha$ for $t<t_0.$
In words, the process $R(t)$ can be described as follows. The points $\Pi_0[t_0^{-},t_0]$ are given initially, acting as roots of the branching process. Each particle at $x\in [t_0^{-},t_0]$ independently gives births to its children on $[t_0,\infty)$, at position $y\geq t_0$ at rate $K_\alpha(y-x)dy$. Each child of a root then gives birth to another generation with the same rule, and the branching continues.  

Since $\int_0^\infty K_\alpha(x)dx=1$, all particles except the roots ($\Pi_0[t_0^-,t_0]$) perform critical branching process (i.e., average number of offspring is 1). Moreover, the distance between a child and a parent is determined by the density $K_\alpha$, and hence it is an age-dependent branching process.

  Throughout the section (and the rest of the paper), $\theta>0$ denotes a large absolute constant, say, $\theta=10000$ and $C_\circ=50$ is a fixed constant that is significantly smaller than $\theta$. For given positive number $\alpha$ and $t_0$, we denote $$\beta:= \log(1/\alpha).$$ Further, let $t_0^-$ be an arbitrary fixed number satisfying 
  $$	t_0-2\alpha^{-2}\beta^\theta < t_0^- <t_0-\frac{1}{2}\alpha^{-2}\beta^\theta,$$
  and set $ \acute{t}_0:= t_0 + 3\alpha^{-2} \beta^{\theta}$.

In what follows, we illustrate a way of interpreting $R(t)$ as a perturbation of the fixed rate process. Consider the process $R_0(t)$ defined as 
\begin{equation}\label{eq:def:R0:ib}
R_0(t):= \begin{cases}
\alpha & \ \textnormal{ for } t_0^{-}\le t<t_0;\\
\intop_{t_0^{-}}^t K_\alpha(t-x) d\Pi_\alpha(x) & \ \textnormal{ for } t\ge t_0,
\end{cases}
\end{equation}

For $t\ge t_0$, observe that $R(t) = R_0(t)$ until we see the first point in $\Pi_{R_0 \triangle \alpha}$. This motivates us to define the following two processes $R_0^-(t)$ and $R_0^+(t)$ as follows: For all $t_0^{-} \le t < t_0$, we set $R_0^-(t)=R_0^+(t)=\alpha$. For $t\ge t_0$, let
\begin{equation}\label{eq:def:R:ib:aux}
\begin{split}
R_0^-(t)&:= \intop_{t_0^{-}}^t K_\alpha(t-x) d\Pi_{R_0^- - (R_0-\alpha)_-}(x);\\
R_0^+(t) &:= \intop_{t_0^{-}}^t K_\alpha(t-x) d\Pi_{R_0^+ + (R_0-\alpha)_+}(x),
\end{split}
\end{equation}
where $(s)_+ := s\vee 0$ and $(s)_-:= (-s) \vee 0$. Later in Proposition \ref{prop:branching:ib:sandwich}, we will see that $R_0^-$ (resp. $R_0^+$) lower (resp. upper) bounds $R(t)$.

To state the main result of this section,  let 
\begin{equation}\label{eq:def:R0plushat}
\hat{R}_0^+(t) := \left(R_0^+(t) + |R_0(t)-\alpha| \right) \vee \alpha.
\end{equation}
Recalling the definition of $\sigma_1(t;g)$ \eqref{eq:def:sig closest points},  define 
\begin{equation}\label{eq:def:tau:branching:ib}
\begin{split}
&\tau_{{\textnormal{B}1}} = \tau_{{\textnormal{B}1}}(\alpha,t_0):= \inf \left\{t\ge t_0: \left(R (t)-\alpha \right) \notin \left( - \alpha^{\frac{3}{2}}\beta^{5\theta}, \, \alpha \beta^{5\theta} \sigma_1(t; \hat{R}_0^+)+ \alpha^{\frac{3}{2}}\beta^{5\theta}\right) \right\};\\
&\tau_{{\textnormal{B}2}}=\tau_{{\textnormal{B}2}}(\alpha,t_0):= \inf\left\{t\ge t_0:  \left|\Pi_{\hat{R}_0^+}[(t-\alpha^{-1})\vee t_0, t] \right| \ge \beta^{C_\circ} \right\} ;
\\
&\tau_{\textnormal{B}}=\tau_{\textnormal{B}}(\alpha,t_0):=
 \tau_{{\textnormal{B}1}} \wedge \tau_{{\textnormal{B}2}}.
\end{split}
\end{equation} 
Then, we have the following result for $R(t)$.

\begin{thm}\label{thm:branching:ib}
	Under the above setting, for all sufficiently small $\alpha>0$, we have 
	\begin{equation}
	\PP \left( \tau_{\textnormal{B}}(\alpha,t_0)   >\acute{t}_0 \right) \ge 1-\exp\left(-\beta^2 \right).
	\end{equation}
\end{thm}
 The theorem will be essential later, serving as the induction base of our inductive argument in settling the regularity (Section \ref{sec:reg:intro}; Theorem \ref{thm:induction:base:main}).
Establishing the theorem consists of the four main steps as follows.
\begin{itemize}
\item Section \ref{subsec:branching:mg}: We develop further theory on representing the branching process using martingales.  Not only will this be helpful in studying $\tau_{\textnormal{B}1}$, but also serve as a fundamental tool in the later sections. 

\item Section \ref{subsec:branching:numberofpts}: We study the critical branching process that starts from a single initial particle. The analysis provides a useful tool to understand the gap between $R_0(t)$ and $R(t)$.

\item Section \ref{subsec:branching:ib}: We  investigate the processes $R_0(t)$, $R_0^-(t)$ and $R_0^+(t)$, exploiting the martingale concentration lemmas from Section \ref{subsec:fixed:mgconcen}.

	\item Section \ref{subsubsec:branching:ib:tau4}: Combining the above  analysis, we deduce the bound on $\tau_{{\textnormal{B}1}}$ and $\tau_{{\textnormal{B}2}}$. 
\end{itemize}

Before moving on, we record a simple lemma that will be used througout the rest of the paper.
\begin{lem}\label{lem:ind1:pi1int basicbd:basic}
	Let $g:[0,h]\to (0,\infty)$ be a positive function which defines a (deterministic) set of points $\Pi_g[0,h]$, and set $n=|\Pi_g[0,h]|$. Recalling the definition $\pi_1(t;g)$ \eqref{eq:def:pi closest points}, there exists an absolute constant $C>0$ such that
	\begin{equation}
	\begin{split}
	&\intop_0^h \frac{dt}{\pi_1(t;g)+1} \le (n+1)\log (h+1);\\
	&\intop_0^h \frac{dt}{\sqrt{\pi_1(t;g)+1}} \le C\sqrt{(n+1)h}.
	\end{split}
	\end{equation}
\end{lem}

\begin{proof}
		Let $\Pi_g[0,h]:= \{p_1<p_2<\ldots<p_n \}$ with $p_0=0, p_{n+1}=h.$
		For the first one, we have
		\begin{equation}
		\intop_0^h \frac{dt}{\pi_1(t;g)+1} = \sum_{i=1}^{n+1} 
		\intop_{0}^{p_i-p_{i-1}} \frac{dx}{x+1} \le (n+1) \log(h+1).
		\end{equation}
		The second one follows by
		\begin{equation}\label{eq:intbasicaux}
		\intop_0^h \frac{dt}{\sqrt{\pi_1(t;g)+1}} =
		\sum_{i=1}^{n+1}
		\intop_0^{p_i-p_{i-1}} \frac{dx}{\sqrt{x+1}} 
		\le \sum_{i=1}^{n+1} 2\sqrt{ p_i-p_{i-1}}\le 2\sqrt{(n+1)h},
		\end{equation}
		where we used Cauchy-Schwarz inequality to obtain the last inequality.
\end{proof}

\subsection{Connection to  martingales}\label{subsec:branching:mg}

Recall the definition $\mathcal{R}_b(t;\Pi_0[t_0^-,t_0],\alpha)$ \eqref{eq:def:Rb}. For a given point process $\Pi_Q$ with respect to $Q(t)>0$, we introduce a similar notation given by
\begin{equation}\label{eq:def:Qst}
\mathcal{R}_c(s,t;\Pi_Q[t_0^-,s],\alpha):= 
\intop_{t_0^{-}}^s K_{\alpha_0}(t-x) d\Pi_Q(x) + \intop_{t_0^{-}}^{s} \intop_{s}^t K^*_{\alpha_0}(t-u)K_{\alpha_0}(u-x)du d\Pi_Q(x).
\end{equation}
In words, it is a conditional expectation of  the rate at time $t$ of the critical branching process starting from the points $\Pi_Q[t_0^-,s]$, conditioned on $\Pi_Q[t_0^-,s]$. In particular,
$\mathcal{R}_c(t,t;\Pi_Q[t_0^-,t],\alpha) = \mathcal{R}_b(t;\Pi_Q[t_0^-,t],\alpha)$. 

Abbreviating the notation by $Q_1(s,t):= \mathcal{R}_c(s,t;\Pi_Q[t_0^-,s],\alpha)$, 
the main advantage of studying this process comes from the following observation, which is an analogue of \eqref{eq:Lt diff:basic}: 
\begin{equation}\label{eq:integralform:branching:pre}
\begin{split}
\frac{\partial}{\partial s} Q_1(s,t) &= \left(K_{\alpha}(t-s) +\intop_s^t K_{\alpha}^*(t-u) K_{\alpha}(u-s) du \right) \frac{d\Pi_Q(s)}{ds}  -\intop_{t_0^-}^s  K_{\alpha}(s-x) d\Pi_Q(x) \\
&=K^*_{\alpha}(t-s) \left( \frac{d\Pi_Q(s)}{ds} - Q_1(s,s) \right) , 
\end{split}
\end{equation}
where the second line comes from the fact that $K^*_{\alpha_0}=K_{\alpha_0}+K_{\alpha_0}*K^*_{\alpha_0}$. This gives 
\begin{equation}\label{eq:integralform:branching}
\begin{split}
Q_1(s_1,t)-Q_1(s_0,t) &= \intop_{s_0}^{s_1} K^*_\alpha (t-s) \big\{ d\Pi_Q(s)- Q_1(s,s)ds \big\}\\
&= \intop_{s_0}^{s_1} K^*_\alpha (t-s) \left\{ d\widetilde{\Pi}_Q(s)+ \big[Q(s)-Q_1(s,s) \big]ds \right\} ,
\end{split}
\end{equation}
where we wrote $d\widetilde{\Pi}_Q(s):= d\Pi_Q(s)-Q(s)ds$. This decomposes the identity to the ``martingale part'' and the ``drift part.''

In the definition of $\mathcal{R}_c$, the case $t=\infty$ will play important roles, and we specify this case with another notation:
\begin{equation}\label{eq:def:L}
\begin{split}
 \mathcal{L}(t; \Pi_Q[t_0^{ -}, t],\alpha) :=
\intop_{t_0^{ -}}^{t} \intop_{t}^{\infty}  K^*_\alpha \cdot K_\alpha(x-s) dx d\Pi_Q (s)
.
\end{split}
\end{equation}
where $K^*_\alpha = \lim_{t\to\infty} K^*_\alpha(t)= \frac{2\alpha^2}{1+2\alpha}$ (Lemma \ref{lem:estimat for K tilde:intro}). Note that the definition is an  analogue of \eqref{eq:def:Lt:basic form}.

\subsection{The critical branching from a single particle}\label{subsec:branching:numberofpts}

In this subsection, we study the age-dependent critical branching process that starts from a single initial particle.  The results in this subsection will be useful tools in the next subsection.

We let $\{r(t)\}_{t\ge 0}$ denote the rate of the critical branching process starting from a single point at the origin ($t=0$). That is, we define $r(t)$ by
\begin{equation}\label{eq:def:rt:branching singlept}
r(t) = r(t,\alpha) := K_\alpha(t)+ \intop_{0+}^t K_\alpha(t-x) d\Pi_r(x).
\end{equation}
Also, we set $\hat{h}:= 2\alpha^{-2}\beta^\theta$. Recall the definitions of $\pi_i(t;g)$  \eqref{eq:def:pi closest points} and $ \sigma_i(t;g)$    \eqref{eq:def:sig closest points}, and
consider the stopping times
\begin{equation}\label{eq:def:tauSB}
\begin{split}
\tau_{\textnormal{SB}1} &:= \inf\left\{t\ge 0: |\Pi_r[0,t]| \ge \beta^{\frac{7}{2}\theta} \right\};\\
\tau_{\textnormal{SB}2} &:=
\inf \left\{t\ge 0: r(t) \ge \alpha \beta^{4\theta} \sigma_1(t;r) + \alpha^2  \right\};\\
\tau_{\textnormal{SB}} &:= \tau_{\textnormal{SB}1} \wedge \tau_{\textnormal{SB}2}.
\end{split}
\end{equation}

Then, our main estimate can be described as follows.

\begin{prop}\label{prop:branching:sb:main}
	Under the above setting, we have
	\begin{equation}
	\PP \left(\tau_{\textnormal{SB}} \le \hat{h} \right) \le \exp\left(-\beta^4 \right).
	\end{equation}
\end{prop}

Furthermore, we will deduce a bound on the probability that $\Pi_r$ contains more than certain number of points in a short interval. For instance, for any $t\in [k\alpha^{-\frac{3}{2}}, (k+1)\alpha^{-\frac{3}{2}} ]$ with $k\in \mathbb{N}$, we can write
\begin{equation}\label{eq:branching:sb:ge 1}
\PP \left(|\Pi_r[t-\alpha^{-\frac{3}{2}}, t]| \ge 1 \right) \le \mathbb{E} |\Pi_r[t-\alpha^{-\frac{3}{2}},t]| 
=
\intop_{t-\alpha^{-\frac{3}{2}}}^t K^*_\alpha(x) dx \le C \left\{k^{-\frac{1}{2}}\alpha^{\frac{1}{4}} \vee \alpha^{\frac{1}{2}} \right\},
\end{equation}
for some constant $C>0$. Note that the last inequality is from the estimate on $K_\alpha^*(x)$ (Lemma \ref{lem:estimat for K tilde:intro}). Our goal is to deduce a similar inequality as follows.
\begin{prop}\label{prop:branching:sb:ge3}
	Under the above setting, for all $\alpha^{-\frac{3}{2}}\le t\le \hat{h}$ we have
	\begin{equation}
	\PP \left(\big|\Pi_r[(t-\alpha^{-\frac{3}{2}})\wedge \tau_{\textnormal{SB}},t \wedge \tau_{\textnormal{SB}}]  \big| \ge 3 \right) \le \alpha^{\frac{3}{4}} \beta^{19\theta}.
	\end{equation}
\end{prop}

The reason why we are interested in the intervals of length $\alpha^{-3/2}$ will become more clear in the next subsection. To explain it briefly, we interpret $R(t)$ as a perturbation of the rate-$\alpha$ process, and it turns out that $\alpha^{-3/2}$ is the length of an interval where we start to see particles between $R(t)$ and $\alpha$. The proof of Proposition \ref{prop:branching:sb:ge3} follows as a consequence of Proposition \ref{prop:branching:sb:main}, and we discuss it at the end of this subsection.

Now we establish Proposition \ref{prop:branching:sb:main},  beginning with the control on $\tau_{\textnormal{SB}1}$.

\begin{lem}\label{lem:numberofpts:branching}
For any sufficiently small $\alpha>0$, we have $\PP \left( \tau_{\textnormal{SB}1} \le \hat{h} \right) \le \exp\left(-\beta^{\frac{1}{200}\theta} \right)$. In other words,
	\begin{equation}
	\PP \left( |\Pi_r[0,\hat{h}]| \ge \beta^{\frac{7}{2}\theta} \right) \le \exp\left(-\beta^{\frac{1}{200}\theta} \right).
	\end{equation}
\end{lem}

To establish Lemma \ref{lem:numberofpts:branching}, we start with verifying a similar property for the critical Galton-Watson branching process with Poisson offspring.

\begin{lem}\label{lem:bound on branching}
	Let $Z_n$ be a critical branching process with offspring distribution $\text{Poisson}(1)$ and $Z_0=1$. Then
	\begin{equation}
	\mathbb P \left( \sum _{k=1}^n Z_k\ge n^3 \right)\le  e^{-cn} 
	\end{equation}
\end{lem}

\begin{proof}
	Let $\varphi _n$ be the probability generating function of $X_n$. For any $u >0$ we have 
	\begin{equation}
	\begin{split}
	\varphi _n (u)=\mathbb E u ^{Z_n}= \mathbb E \left[  \mathbb  E \left[ u ^{Z_n} |Z_{n-1}\right] \right]=\mathbb E \left[ (\mathbb E u ^{X_1} )^{X_{n-1}} \right]= \varphi _{n-1}(e^{u-1 }) = \psi^{(n)}(u),
	\end{split}
	\end{equation}
	where $\psi(u):= e^{u-1}$ and where $\psi^{(n)}(u)$ is the $n$-fold composition of $\psi $ with itself.
	Let $u :=1+ 1/5n$. We prove using induction on $k \le n$ that $\varphi _k(u ) \le 1+\frac{1}{5 n}+\frac{k}{16 n^2}$. In fact,
	\begin{equation}
	\varphi _{k+1}(u )=\exp (\varphi_{k}(u )-1)\le \varphi _k (u )+(\varphi _k (u )-1)^2 \le 1+\frac{1}{5n}+\frac{k+1}{16n^2}
	\end{equation}
	where in the last inequality we used that $\varphi _k(u )\le 1+1/4n$ which follows from the induction hypothesis when $k \le n$.
	
	Finally, using Markov's inequality we get for any $k \le n$
	\begin{equation}
	\mathbb P (Z_k \ge n^2 )= \mathbb P \left( u ^{Z_k}\ge u ^{n^2 } \right) \le \varphi _k (u  ) \cdot u  ^{-n^2}\le e^{-cn}. 
	\end{equation}
	The result now follows from a union bound.
\end{proof}

\begin{proof}[Proof of Lemma  \ref{lem:numberofpts:branching}]
	Let $X_i$ be $i.i.d$ with density $K_\alpha$. We have $\mathbb P (X_1\le \alpha ^{-2})\le 1-c$ for some $c>0$. Thus 
	\begin{equation}\label{eq:bound on sum}
	\mathbb P \left( \sum _{k=1}^{\beta^{\frac{11}{10}\theta}  } X_k \le \hat{h}\right) \le \mathbb P \left( \sum _{k=1}^{\beta^{\frac{11}{10}\theta}  } \mathds{1}_{\{ X_k \le \alpha ^{-2} \}} \ge \left(1-\frac{c}{2}\right) \beta^{\frac{11}{10}\theta} \right) \le \exp\left(-\beta ^{\frac{1}{100}\theta}\right),
	\end{equation} 
	where in the last inequality we used Azuma. Now, if there are more than $\beta^{\frac{7}{2}\theta}$ particles in the branching before time $\hat{h}$ then either there are more than $\beta^{4\theta}$ particles up to the $\beta^{\frac{11}{10}\theta}$ generation or there are less particles than that up to the $\beta^{\frac{11}{10}\theta}$ generation but one of the particles in this generation came before time $\hat{h}$. The first event happens with very small probability by Lemma \ref{lem:bound on branching} and the second one happens with small probability by union bound and  \eqref{eq:bound on sum}.
\end{proof}

We conlude the proof of Proposition \ref{prop:branching:sb:main}.

\begin{proof}[Proof of Proposition \ref{prop:branching:sb:main}]
	For all $t\le \tau_{\textnormal{SB}1}$, observe that
	\begin{equation}
	\begin{split}
	r(t) &=K_\alpha(t)+ \intop_{0+}^t K_\alpha(t-x) d\Pi_r(x) \\&\le K_\alpha(\pi_1(t;r)\wedge \alpha^{-2}\beta^\theta) \cdot \big| \Pi_r[0,t] \big| \le 
	\alpha\beta^{4\theta} \sigma_1(t;r) + \alpha^2,
	\end{split}
	\end{equation} 
	where the last inequality follows from the estimate on $K_\alpha(x)$ (Lemma \ref{lem:estimate on K:intro}). Thus, we obtain the desired result by combining with Lemma \ref{lem:numberofpts:branching}
\end{proof}

We conclude this subsection by verifying Proposition \ref{prop:branching:sb:ge3}.

\begin{proof}[Proof of Proposition \ref{prop:branching:sb:ge3}]
	For any $\alpha^{-3/2}\le t\le \hat{h}$, denote $s_1  := t-\alpha^{-\frac{3}{2}}\wedge \tau_{\textnormal{SB}}$ and $s_2:= t\wedge \tau_{\textnormal{SB}}$. First, based on the definition of $\tau_{\textnormal{SB}2}$,  applying Lemma \ref{lem:ind1:pi1int basicbd:basic} with $n=\beta^{\frac{7}{2}\theta}$ gives that
	\begin{equation}
	\intop_{s_1}^{s_2} r(x)dx \le \alpha \beta^{4\theta} \cdot \alpha^{-\frac{3}{4}} \sqrt{n}\le \alpha^{\frac{1}{4}}\beta^{6\theta}.
	\end{equation} 
	Thus, we can write
	\begin{equation}
	\begin{split}
	\PP \left( \left| \Pi_r[s_1,s_2] \right| \ge 3\right) 
	\le \intop_{s_1\le x<y<z\le s_2} r(x)r(y)r(z)dzdydx
	\le \alpha^{\frac{3}{4}}\beta^{18\theta},
	\end{split}
	\end{equation}
	noting that any addition of three extra points in $\Pi_r[s_1,s_2]$ does not change the bound given in the definition of $\tau_{\textnormal{SB}2}$. Thus, we conclude the proof by combining the above with Proposition \ref{prop:branching:sb:main}.
\end{proof}

\subsection{Understanding the fixed rate process and its perturbations}\label{subsec:branching:ib}
In this  subsection, we study the processes $R_0(t), R_0^-(t)$, and $R_0^+(t)$ introduced in \eqref{eq:def:R0:ib}, \eqref{eq:def:R:ib:aux}. We begin with observing that $R_0^-$ (resp. $R_0^+$) lower (resp. upper) bounds both $R_0$ and $R$.

\begin{prop}\label{prop:branching:ib:sandwich}
	Under the above setting, $R_0^-(t)\le R(t), R_0(t) \le R_0^+(t)$ for all $t\ge t_0$.
\end{prop}

\begin{proof}
	Observe first that at $t = t_0$, $R^+(t_0)=R^-(t_0)=R(t_0)=R_0(t_0)$, and for $t\in[t_0^{-},t_0)$ they are all identical to $\alpha$. Recalling the definition of $R(t)$, which can be written as
	\begin{equation}
	R(t)= \intop_{t_0^{-}}^{t_0} K_{\alpha}(t-x)d\Pi_\alpha(x) + \intop_{t_0}^{t} K_{\alpha}(t-x)d\Pi_{R}(x),
	\end{equation}
	we observe that $R_0^-(t)\le R(t)\le R_0^+(t)$: They maintain the same value until time $t_0$, and then after that $R_0^+(t)$  picks up more  points than $R(t)$, since it absorbs additional points from $(R_0-\alpha)_+$ in addition to all the points taken by $R (t)$. The same  reasoning for opposite direction works for $R_0^-(t)$. 
	
	To see that the conclusion holds for $R_0$,  we observe that  $R_0^+(t) \ge R_0(t)$ for $t\ge t_0$ inductively: Initially, we have $R_0^+(t_0) = R_0(t_0)$. Furthermore, as long as we have  $ R_0^+(s) \ge R_0(s)$ for all $s < t$, it holds that
	\begin{equation}
	R_0^+(t) + (R_0(t)-\alpha)_+
	\ge
	R_0(t) + (\alpha - R_0(t) ) \ge \alpha, 
	\end{equation}
	meaning that the new points picked up by $R_0(t)$ will also be included in $R_0^+(t)$. Thus, it will continue to hold that $R_0^+(t) \ge R_0(t)$. The other inequality, $R_0^- \le R_0$, can be obtained analogously.
\end{proof}

We move on to the study of $R_0(t)$. 
Define the notation
\begin{equation}\label{eq:def:Pi triangle}
\Pi_{R_0 \triangle \alpha} [ x_0,x_1]:=
\Pi_{R_0}[x_0,x_1] \triangle \Pi_{\alpha}[x_0,x_1] ,
\end{equation}
that is, the collection of points lying between $R_0$ and $\alpha$.
Then, define
\begin{equation}\label{eq:def:tauf}
\begin{split}
\tau_{\textnormal{f}1}&:= \inf\left\{t\ge t_0: (R_0(t)-\alpha) \notin \left(-\alpha^{\frac{3}{2}}\beta^{C_\circ},\,  \alpha\beta^{C_\circ} \sigma_1(t;\alpha) + \alpha^{\frac{3}{2}}\beta^{C_\circ} \right) \right\};\\
\tau_{\textnormal{f}2}&:= \inf \left\{t\ge t_0^{-}: |\Pi_\alpha[(t-\alpha^{-1})\vee t_0^{-}, \, t  ]| \ge \beta^{6} \right\};\\
\tau_{\textnormal{f}3}&:= \inf \left\{ t \ge t_0: \left| \Pi_{R_0 \triangle \alpha}[(t-\alpha^{-\frac{3}{2}})\vee t_0,\, t] \right| \ge \beta^{2C_\circ} \right\};\\
\tau_{\textnormal{f}4}&:=\inf \left\{t\ge t_0: R_0(t)\ge\alpha \beta^{C_\circ} \right\};\\
\tau_{\textnormal{f}}&:=
\tau_{\textnormal{f}1}\wedge \tau_{\textnormal{f}2}\wedge \tau_{\textnormal{f}3}\wedge \tau_{\textnormal{f}4}.
\end{split}
\end{equation}
The goal of this subsection is establishing the following lemma.
\begin{lem}\label{lem:branching:ib:R0 vs alpha}
	We have $$\PP(\tau_{\textnormal{f}}  \le \acute{t}_0) \le \exp(-\beta^5).$$ 
\end{lem}

\begin{proof}
	By an elementary estimate on Poisson processes, we have
	 \begin{equation}\label{eq:branching:ib:R0:med}
	\PP(\tau_{\textnormal{f}2} \le \acute{t}_0 ) \le \exp \left(-\beta^6 \right).
	\end{equation}
	To obtain appropriate control on $\tau_{\textnormal{f}1}$, we apply Corollary
	\ref{lem:concentrationofint:continuity} to $R_0(t)$. Recalling the estimate Lemma \ref{lem:estimate on K:intro} on $K_\alpha$, we may set $M = \alpha^3\beta^2$. Writing $d\widetilde{\Pi}_\alpha(x) := d\Pi_\alpha(x) - \alpha dx$, this gives that
	\begin{equation}\label{eq:branching:ib:R0:med1}
\begin{split}
	\PP \left(  \intop_{t_0^{-}}^{t\wedge \tau_{\textnormal{f}2}} K_\alpha(t-x) d\widetilde{\Pi}_\alpha (x) \le \alpha \beta^{C_\circ}\sigma_1(t;\alpha) + \frac{1}{2}\alpha^{\frac{3}{2}}\beta^{C_\circ}, \ \forall t\in[t_0,\acute{t}_0] \right) \ge 1- \exp\left(-\beta^6\right);\\
		\PP \left(  \intop_{t_0^{-}}^{t\wedge \tau_{\textnormal{f}2}} K_\alpha(t-x) d\widetilde{\Pi}_\alpha (x) \ge -\frac{1}{2}\alpha^{\frac{3}{2}}\beta^{C_\circ}, \ \forall t\in[t_0,\acute{t}_0] \right) \ge 1- \exp\left(-\beta^6\right).
\end{split}
	\end{equation}
	Moreover, we have for all $t\in [t_0,\acute{t}_0]$ that
	\begin{equation}\label{eq:branching:ib:R0:med0}
\begin{split}
	R_0(t)-\alpha &= \intop_{t_0^{-}}^t K_\alpha(t-x) d\widetilde{\Pi}_\alpha(x) 
	+
	\intop_{t_0^{-}}^{t} K_\alpha(t-x) \alpha dx - \alpha\\
&= \intop_{t_0^{-}}^t K_\alpha(t-x) d\widetilde{\Pi}_\alpha(x) +O\left(\alpha^{100}\right),
\end{split}
	\end{equation}
	where the last identity follows again from the estimate of $K_\alpha(s)$ (Lemma \ref{lem:estimate on K:intro}) applied to $s\ge t_0-t_0^-$. Thus, combining this with \eqref{eq:branching:ib:R0:med} and \eqref{eq:branching:ib:R0:med1}, we obtain that
	\begin{equation}\label{eq:branching:ib:R0:med2}
	\PP \left(\tau_{\textnormal{f}1} \le \acute{t}_0 \right) \le 2\exp\left(-\beta^6\right).
	\end{equation}
	
	Moving  on to $\tau_{\textnormal{f}3}$, write $\tilde{\tau}_{\textnormal{f}}:= \tau_{\textnormal{f}1} \wedge \tau_{\textnormal{f}2}$, and observe that from applying Lemma \ref{lem:ind1:pi1int basicbd:basic} with $n=\alpha^{-\frac{1}{2}} \beta^{C_\circ}$, we get
	\begin{equation}\label{eq:reg:branching:tau2:split0}
	\intop_{(t-\alpha^{-\frac{3}{2}})\wedge \tilde{\tau}_{\textnormal{f}}}^{t\wedge \tilde{\tau}_{\textnormal{f}}} |R_0(t)-\alpha| \le 2\beta_0^{\frac{3}{2}C_\circ +1}.
	\end{equation}
	Thus, applying Corollary \ref{cor:concentration:numberofpts each interval} to $\Pi_{R_0 \triangle \alpha}$ gives that
	\begin{equation}\label{eq:branching:ib:R0:med3}
	\PP\left({\tau}_{\textnormal{f}3} \le \acute{t}_0 \wedge \tilde{\tau}_{\textnormal{f}} \right) \le \exp\left(-\beta^6 \right).
	\end{equation}
	
	Lastly, we note that $\tau_{\textnormal{f}4} \ge \tau_{\textnormal{f}1}$ deterministically from their definitions.  
	Thus, along with \eqref{eq:branching:ib:R0:med},  \eqref{eq:branching:ib:R0:med2}, and \eqref{eq:branching:ib:R0:med3}, we obtain the desired conclusion.
\end{proof}

\subsection{Proximity to $\alpha$ for small $t$}\label{subsubsec:branching:ib:tau4}

In this subsection, we conclude the proof of Theorem \ref{thm:branching:ib}, combining the results from the previous subsections. 
We begin with investigating $\tau_{{\textnormal{B}1}}$. Thanks to Proposition \ref{prop:branching:ib:sandwich}, we can write
\begin{equation}\label{eq:branching:ib:decomp}
|R(t)-\alpha| \le |R_0(t)-\alpha| +\left\{R_0^+(t)-R_0^-(t) \right\}.
\end{equation}
Since the bound on $|R_0(t)-\alpha|$ has already been given by $\tau_{\textnormal{f}1} $ in Lemma \ref{lem:branching:ib:R0 vs alpha}, we focus on the term $\left\{R_0^+(t)-R_0^-(t) \right\}$.

 By the definition of $R_0^+$ and $R_0^-$ \eqref{eq:def:R:ib:aux}, the process $\delta R_0(t):= R_0^+(t)- R_0^-(t)$ satisfies
\begin{equation}\label{eq:def:delR}
\delta R_0(t) = \intop_{t_0}^t K_\alpha(t-x) d\Pi_{(R_0^++(R_0-\alpha)_+) \triangle (R_0^--(R_0-\alpha)_-)}(x),
\end{equation}
for all $t\ge t_0$, starting with $\delta R_0(t_0) = 0$.  To lighten up our notation, we symbolically define
\begin{equation}\label{eq:def:branching:ib:delRhat}
\hat{\Delta} {R}_0 := (R_0^++(R_0-\alpha)_+) \triangle (R_0^--(R_0-\alpha)_-).
\end{equation}

\begin{remark}\label{rmk:branching:ib:delR}
	In \eqref{eq:def:delR}, we can interpret $\delta R_0(t)$ in the following way: $\delta R_0$ starts as an empty process at time $t_0$. As time goes on, new particles are added to $\delta R_0$ at rate $|R_0-\alpha|$. We call these particles \textbf{the external additions}. Then, each particle from  external addition  performs the critical branching process. 
\end{remark}

To study the process $\delta R_0(t)$, we introduce the following stopping time.
\begin{equation}
\begin{split}
\tau_{\delta \textnormal{B}1}= \tau_{\delta \textnormal{B}1}(\alpha, t_0) := \inf \left\{ t\ge t_0: \left|\Pi_{\hat{\Delta} {R}_0}[(t-\alpha^{-\frac{3}{2}})\vee t_0, t] \right| \ge \beta^{4\theta} \right\}.
\end{split}
\end{equation}
Our first goal is to show the following lemma on $\tau_{\delta \textnormal{B}1}$.

\begin{lem}\label{lem:branching:ib:taudelta}
	Under the above setting, we have
	\begin{equation}
	\PP \left(\tau_{\delta \textnormal{B}1} \le \acute{t}_0 \right) \le \exp\left(-\beta^3 \right).
	\end{equation}
\end{lem}

\begin{proof}
	We interpret $\delta R(t)$ as mentioned in Remark \ref{rmk:branching:ib:delR}, recalling the notion of  \textit{external additions}. For each $j\in \mathbb{N}, \, 0\le j \le 2\alpha^{-\frac{1}{2}}\beta^\theta$, we write
	$ \{ p_{j, 1}, \ldots, p_{j, n_j} \}$ to be the {external additions} at $\delta R_0(t)$ during time $[j\alpha^{-\frac{3}{2}}, (j+1)\alpha^{-\frac{3}{2}}]$. 
	For each point $p_{j,l}$ and the critical branching process $r_{j,l}(t)$ that starts from a single point at $p_{j,l}$, we define the stopping time $\tau_{\textnormal{SB}}(j,l)$ analogously as \eqref{eq:def:tauSB}:
	\begin{equation}
	\begin{split}
	\tau_{\textnormal{SB}1}(j,l)&:= \inf \left\{ t\ge p_{j,l}: \big| \Pi'_{r_{j,l}}[0,t] \big| \ge \beta^{\frac{7}{2}\theta} \right\};\\
		\tau_{\textnormal{SB}2}(j,l)&:=
		\inf \left\{t\ge p_{j,l}: r_{j,l} \ge \alpha \beta^{4\theta}\sigma_1(t;r_{j,l}) +\alpha^2 \right\};\\
		\tau_{\textnormal{SB}}(j,l)&:=
			\tau_{\textnormal{SB}1}(j,l)\wedge 	\tau_{\textnormal{SB}2}(j,l);\\
			\tau_{\textnormal{SB}} &:= 
			\tau_{\textnormal{f}} \wedge 
			\bigwedge_{j=0}^{2\alpha^{-\frac{1}{2}}\beta^\theta} \bigwedge_{l=1}^{n_j} 	\tau_{\textnormal{SB}}(j,l).
	\end{split}
	\end{equation}
	(Recall the definition of $\tau_{\textnormal{f}}$ in \eqref{eq:def:tauf}) Moreover, for $k, j,l \in \mathbb{N}$ with $0\le k , j \le 2\alpha^{-\frac{1}{2}}\beta^\theta$ and $1\le l\le n_j$, we define
	\begin{equation}
	\begin{split}
	N_{j,l}^{(k)} :=& \textnormal{ the number of points in } \left|\Pi_{\hat{\Delta} R_0}[k\alpha^{-\frac{3}{2}}\wedge \tau_{\textnormal{SB}},\, (k+1) \alpha^{-\frac{3}{2}}\wedge \tau_{\textnormal{SB}}  ] \right| \\ &\textnormal{ that are descendents of } p_{j,l}.
	\end{split}
	\end{equation}
	Note that for each $k$, $\{N^{(k)}_{j,l}\}_{j,l}$ forms a collection of independent random variables. Furthermore, due to the definition of $\tau_{\textnormal{f}3}$, we can consider $l\le  \beta^{2C_\circ}$ only. 
	
	For each $k,j,l$ as above, define $\bar{N}^{(k)}_{j,l}$ to be the random variable as follows:
	\begin{itemize}
	\item $\bar{N}^{(k)}_{j,l} \equiv 0 $ for any $j>k$.
	
	\item Otherwise, we define
	\begin{equation}
	\bar{N}^{(k)}_{j,l} = \begin{cases}
	2, &\textnormal{ with probability } C\left\{ (k-j+1)^{-\frac{1}{2}}\alpha^{\frac{1}{4}} \vee \alpha^{\frac{1}{2}} \right\};\\
	\beta^{\frac{7}{2}\theta},& \textnormal{ with probability } \alpha^{\frac{3}{4}} \beta^{19\theta};\\
	0,  & \textnormal{ otherwise.}
	\end{cases}
	\end{equation}
	
	\item For each $k$, $\{\bar{N}_{j,l}^{(k)} \}_{j,l}$ is a collection of independent random variables.
	\end{itemize}
	Then, \eqref{eq:branching:sb:ge 1} and Proposition \ref{prop:branching:sb:ge3} tell us that $\{\bar{N}^{(k)}_{j,l} \}_{k,j,l}$ stochastically dominates $\{N^{(k)}_{j,l} \}_{k,j,l}$. 
	
	Now, we  estimate
	\begin{equation}
	\left|\Pi_{\hat{\Delta} R_0}[k\alpha^{-\frac{3}{2}}\wedge \tau_{\textnormal{SB}},\, (k+1) \alpha^{-\frac{3}{2}}\wedge \tau_{\textnormal{SB}}  ] \right| \preceq \sum_{j\le k} \sum_{l=1}^{n_j} \bar{N}^{(k)}_{j,l}.
	\end{equation}
	First, we note that each $\bar{N}_{j,l}^{(k)}$ is bounded by $\beta^{\frac{7}{2}\theta}$ and also observe that
	\begin{equation}
\begin{split}
\sum_{j,l} \mathbb{E}\left[\bar{N}^{(k)}_{j,l} \right]&\le 
\sum_{j=1}^{2\alpha^{-\frac{1}{2}}\beta^\theta} \sum_{l=1}^{\beta^{2C_\circ}} C \left\{
j^{-\frac{1}{2}}\alpha^{\frac{1}{4}} \vee \alpha^{\frac{1}{2}} \right\} \le \beta^{2\theta}
\\
	\sum_{j,l} \mathbb{E}\left[\left(\bar{N}^{(k)}_{j,l} \right)^2 \right] &\le \sum_{j=1}^{2\alpha^{-\frac{1}{2}}\beta^\theta} \sum_{l=1}^{\beta^{2C_\circ}} C' \left\{
	j^{-\frac{1}{2}}\alpha^{\frac{1}{4}} \vee \alpha^{\frac{1}{2}} \right\} \le \beta^{2\theta}.
\end{split}
	\end{equation}
	Thus, Bernstein's inequality (Lemma \ref{lem:branching:bernstein} below) tells us that for each $k$,
	\begin{equation}
	\begin{split}
	\PP \left( 	\left|\Pi_{\hat{\Delta} R_0}[k\alpha^{-\frac{3}{2}}\wedge \tau_{\textnormal{SB}},\, (k+1) \alpha^{-\frac{3}{2}}\wedge \tau_{\textnormal{SB}}  ] \right| \ge \beta^{4\theta} \right) \le \exp\left(-\beta^{\theta/3} \right).
	\end{split}
	\end{equation}
	Thus, we can conclude the proof by taking a union bound over $k$, and by noting that
	\begin{equation}
	\PP \left(\tau_{\textnormal{SB}} > \acute{t}_0 \right) \ge 1-\exp\left(-\beta^3 \right),
	\end{equation}
	which comes from Proposition \ref{prop:branching:sb:main} and Lemma \ref{lem:branching:ib:R0 vs alpha} followed by a union bound.
\end{proof}

The Bernstein's inequality used in the above proof can be stated as below.
\begin{lem}[Theorem 3.6,~\cite{chung2006concentration}]\label{lem:branching:bernstein}
    Let $X_1, \ldots, X_n$ be independent random variables satisfying $|X_i|\le M$ almost surely for all $i$. Then, we have
    \begin{equation}
        \mathbb{P}\left( \sum_{i=1}^n (X_i-\mathbb{E}X_i) \ge x \right) \le \exp\left( -\frac{x^2}{2\sum_{i=1}^n \mathbb{E}[X_i^2] + \frac{2}{3} Mx } \right).
    \end{equation}
\end{lem}

The following corollary translates Lemma \ref{lem:branching:ib:taudelta} into the form that we can apply to our analysis.
\begin{cor}\label{cor:branching:ib:taudelta}
Recall the definition of $\hat{\Delta} {R}_0(t)$ \eqref{eq:def:branching:ib:delRhat}, and
 define the stopping times
	\begin{equation}
\begin{split}
	\tau_{\delta \textnormal{B}2}=\tau_{\delta \textnormal{B}2}&:=
\inf \left\{t\ge t_0 : \delta R_0(t) \le \alpha\beta^{4\theta+1}\sigma_1(t;\hat{\Delta} {R}_0) + \alpha^{\frac{3}{2}}\beta^{4\theta+1} \right\};\\
	\tau_{\delta \textnormal{B}3}=\tau_{\delta \textnormal{B}3}&:=
	\inf \left\{t\ge t_0 : \delta R_0(t) \le 2\alpha\beta^{4\theta+1}\right\}.
\end{split}
	\end{equation}
	Then, we have
	\begin{equation}
		\PP \left( \tau_{\delta\textnormal{B}3} >\acute{t}_0 \right)\ge \PP \left( \tau_{\delta\textnormal{B}2} >\acute{t}_0 \right) \ge 1-2\exp\left(-\beta^3 \right).
	\end{equation}
\end{cor}

\begin{proof}
	 The first inequality is immediate from definition. To establish the second inequality, recall the definition of $\tau_{\textnormal{f}3}$ from \eqref{eq:def:tauf}. For $t \le \tau_{\textnormal{f}3}\wedge  \tau_{\delta \textnormal{B}1}$, note that
	\begin{equation}
\begin{split}
	&\delta R_0(t) = \intop_{t_0}^t K_{\alpha_0}({t-x}) d\Pi_{\hat{\Delta} {R}_0}(x) \\
	&\le
	 \big|\Pi_{\hat{\Delta} {R}_0}[t-\alpha^{-\frac{3}{2}}, t] \big| \cdot C\alpha \sigma_1(t;\hat{\Delta} {R}_0) 
	+
	\sum_{k\ge 1} \big|\Pi_{\hat{\Delta} {R}_0}[t-(k+1)\alpha^{-\frac{3}{2}}, t-k\alpha^{-\frac{3}{2}}] \big| \cdot
	\frac{C\alpha e^{-ck\alpha^{\frac{1}{2}}}}{\sqrt{k\alpha^{-\frac{3}{2}}}}\\
	&\le \alpha\beta^{4\theta+1}\sigma_1(t;\hat{\Delta} {R}_0)+ \alpha^{\frac{3}{2}}\beta^{4\theta+1}.
\end{split}
	\end{equation}
	Thus, the proof follows  from Lemmas \ref{lem:branching:ib:R0 vs alpha} and \ref{lem:branching:ib:taudelta}.
\end{proof}

We now are ready to  establish Theorem \ref{thm:branching:ib}.

\begin{proof}[Proof of Theorem \ref{thm:branching:ib}]
Recall the decomposition \eqref{eq:branching:ib:decomp}, which can be rewritten as 
\begin{equation}\label{eq:branching:ib:decomp:coupled}
R_0^-(t)-\alpha \le R(t)-\alpha \le (R_0(t)-\alpha )+ \delta R_0(t).
\end{equation}
From the upper bound, the results on $\tau_{\textnormal{f}3} $ (Lemma \ref{lem:branching:ib:R0 vs alpha}) and $\tau_{\delta \textnormal{B}2}$ give the control 
$$R(t)-\alpha \le \alpha\beta^{5\theta}\sigma_1(t;\hat{R}_0^+) + \alpha^{\frac{3}{2}}\beta^{5\theta} . $$

For the lower bound, we study $R_0^-(t)$ instead. Let $t_0^\flat$ be another parameter such that $t_0^\flat := \frac{1}{2}(t_0^-+t_0). $ In particular, it satisfies
\begin{equation}
    t_0-t_0^\flat = t_0^\flat - t_0^- \ge \frac{1}{4}\alpha^{-2}\beta^\theta.
\end{equation}
Noting that $R_0^-(t)=R_0(t)=\alpha$ for $t<t_0$, applying \eqref{eq:integralform:branching} to $R_0^-$ gives the following:
\begin{equation}\label{eq:branching:R0minus intform}
\begin{split}
R_0^-(t) - \mathcal{R}_c(t_0^\flat,t;\Pi_\alpha[t_0^-,t_0^\flat],\alpha )
=\intop_{t_0^\flat}^t K^*_\alpha (t-s) \big\{ \widetilde{\Pi}_{R_0^--(R_0-\alpha)_-}(s) - (R_0(s)-\alpha)ds \big\},
\end{split}
\end{equation}

The first integral can be estimated using Corollary \ref{lem:concentrationofint:continuity}: Since $R_0^-(t)-(R_0(t)-\alpha)_-\le R_0(t)$ for $t\ge t_0$, the definition of $\tau_{\textnormal{f}4}$ and Lemma \ref{lem:estimat for K tilde:intro} imply that
\begin{equation}
\begin{split}
\intop_{t_0^\flat}^{t_0\wedge \tau_{\textnormal{f}4}} 
(K^*_\alpha(t-s) )^2 \alpha ds + \intop_{t_0}^{t\wedge \tau_{\textnormal{f}4}} 
(K^*_\alpha(t-s) )^2 R_0(s) ds \\
\le \intop_{t_0^\flat}^{t} C\left(\frac{\alpha^2}{t-s+1} \vee \alpha^4 \right) \alpha\beta^{C_\circ} ds \le \alpha^3\beta^{\theta+2C_\circ},
\end{split}
\end{equation}
for all $t\le \acute{t}_0$. Thus, Corollary \ref{lem:concentrationofint:continuity} and Lemma \ref{lem:branching:ib:R0 vs alpha} tell us that
\begin{equation}\label{eq:branching:Rlbd:med1}
\PP \left( \intop_{t_0^\flat}^t K^*_\alpha (t-s)  \widetilde{\Pi}_{R_0^--(R_0-\alpha)_-}(s) \ge -\alpha^{\frac{3}{2}}\beta^\theta, \ \forall t\in [t_0,\acute{t}_0] \right) \ge 1- \exp\left(-\beta^5 \right).
\end{equation}
On the other hand, for $t\le \tau_{\textnormal{f}1}$ we can write (note that when $s\le t_0$,  $R_0(s)-\alpha=0$)
\begin{equation}\label{eq:branching:Rlbd:med2}
\left|\intop_{t_0^\flat}^t K^*_\alpha (t-s)  (R_0(s)-\alpha)ds\right| 
\le \intop_{t_0}^t C\left(\frac{\alpha}{\sqrt{t-s+1}} \vee \alpha^2 \right) \left(\alpha\beta^{C_\circ} \sigma_1(s;\alpha) + \alpha^{\frac{3}{2}}\beta^{C_\circ} \right)ds.
\end{equation}
Then, the RHS can be bounded using Lemmas \ref{lem:ind1:pi1int basicbd:basic}, and \ref{lem:ind1:pi1int basicbd} below, by setting the parameters in Lemma \ref{lem:ind1:pi1int basicbd} as $\Delta_0=\Delta_1=\alpha^{-1}, K=\alpha^{-1}\beta^\theta$, and $ N=\beta^6.$ Whenever there is an empty interval of length $\Delta_1$, we can add a point to justify the choice of $\Delta_1$ which will only increase the value of the integral. This implies that the RHS is smaller than $\alpha^{\frac{3}{2}}\beta^{2\theta}.$ 

Lastly, we need to understand difference between the term $\mathcal{R}_c(t_0^\flat,t;\Pi_\alpha[t_0^-,t_0^\flat],\alpha)$ and $\alpha$. Recalling the definition of $\mathcal{L} $ \eqref{eq:def:L}, we first observe that 
\begin{equation}\label{eq:branching:Rlbd:med3}
\begin{split}
&\big|\mathcal{R}_c(t_0^\flat,t;\Pi_\alpha[t_0^-, t_0^\flat],\alpha) - \mathcal{L}(t_0^\flat;\Pi_\alpha[t_0^-,t_0^\flat],\alpha) \big|
\\&\le
\intop_{t_0^-}^{t_0^\flat} K_\alpha(t-x) d\Pi_\alpha (x) +
\intop_{t_0^-}^{t_0^\flat} \intop_t^\infty K^*_\alpha \cdot K_\alpha(u-x) du d\Pi_\alpha(x) \\
&\quad +\intop_{t_0^-}^{t_0^\flat} \intop_{t_0^\flat}^t \big|K^*_\alpha - K^*_\alpha(t-u)\big| \cdot K_\alpha(u-x) du d\Pi_\alpha(x).
\end{split}
\end{equation}
For any $t\in [t_0, \tau_{\textnormal{f}2}]$, we see that the RHS is at most $\alpha^{100}$ since $t-t_0^\flat \ge \alpha^{-2} \beta^{C_\circ}$: The first two integrals are small due to the decay of $K_\alpha$ (Lemma \ref{lem:estimate on K:intro}), and in the last integral, either $|K^*_\alpha - K^*_\alpha(t-u)|$ or $K_\alpha (u-x)$ is small depending on the size of $u$ (see Lemma \ref{lem:estimat for K tilde:intro}). Moreover, letting $F(x) = \intop_x^\infty K_\alpha(y)dy$, we also have
\begin{equation}\label{eq:reg:ib:med1}
\begin{split}
\mathcal{L}(t_0^\flat;\Pi_\alpha[t_0^-,t_0^\flat],\alpha)& = K_\alpha^* \intop_{t_0^-}^{t_0^\flat} F(t_0^\flat-x) d\Pi_\alpha(x);\\
&=K_\alpha^* \intop_{t_0^-}^{t_0^\flat} F(t_0^\flat -x) d\widetilde{\Pi}_\alpha(x) + K_\alpha^* \intop_{t_0^-}^{t_0^\flat} F(t_0^\flat -x)\alpha dx  .
\end{split}
\end{equation}
The first term in the second line can be bounded by Lemma \ref{lem:concentration of integral} using $|F(x)| \le1$, and the last term is dealt with integration by parts which is
\begin{equation}\label{eq:branching:Rlbd:med4}
\intop_{0}^{t_0^\flat-t_0^-} F(x)dx = \intop_{0}^{\infty} F(x)dx + O(\alpha^{50})  = \intop_0^\infty xK_\alpha(x)dx + O(\alpha^{50}) = (K_\alpha^*)^{-1} + O(\alpha^{50}).
\end{equation}
Thus, we get 
\begin{equation}\label{eq:reg:ib:med2}
\big|\mathcal{L}(t_0^\flat;\Pi_\alpha[t_0^-,t_0^\flat],\alpha) - \alpha\big| \le \alpha^{\frac{3}{2}}\beta^{C_\circ},
\end{equation}
with probability at least $1-e^{-\beta^5}$.

Therefore, combining  \eqref{eq:branching:R0minus intform}, \eqref{eq:branching:Rlbd:med1},  \eqref{eq:branching:Rlbd:med2}, \eqref{eq:branching:Rlbd:med3} and \eqref{eq:branching:Rlbd:med4}, we obtain 
\begin{equation}\label{eq:branching:ib:tauB4}
\PP \left( \tau_{\textnormal{B}1} >\acute{t}_0 \right) \ge 1-3\exp\left(-\beta^3\right).
\end{equation}

The remaining task is to control $\tau_{\textnormal{B}2}. $ 
First, for $t\le\tau_{\textnormal{f}}\wedge \tau_{\delta \textnormal{B}3}$, observe that
\begin{equation}
\begin{split}
\hat{R}_0^+(t)& := R_0^+(t) +|R_0(t)-\alpha| \le R_0(t) + \delta R_0(t) + |R_0(t)-\alpha|\\
&\le \alpha + 2|R_0(t)-\alpha| + \delta R_0(t)\le  \alpha\beta^{C_\circ+1} + 2\alpha \beta^{4\theta+1} \le 3\alpha \beta^{4\theta+1},
\end{split}
\end{equation}
Thus, we have
\begin{equation}\label{eq:branching:ib:tauB6}
\begin{split}
&\tau_{ \textnormal{B}{3}} := \inf 
\left\{ 
t\ge t_0: \big|\Pi_{\hat{R}_0^+}[(t-\alpha^{-1}) \vee t_0, t] \big| \ge \beta^{5\theta}
\right\};\\
&\PP \left( \tau_{ \textnormal{B}{3}} >\acute{t}_0 \right) \ge 1-3\exp\left(-\beta^3 \right),
\end{split}
\end{equation}
from Corollary \ref{cor:concentration:numberofpts each interval}, 
based on Lemma \ref{lem:branching:ib:R0 vs alpha} and Corollary \ref{cor:branching:ib:taudelta}. We also note that $\tau_{\textnormal{B}3}$ is a weaker version of $\tau_{\textnormal{B}2}$.

The decomposition \eqref{eq:branching:ib:decomp:coupled} works similarly for $\hat{R}_0^+(t)$ as follows.
\begin{equation}
\begin{split}
|\hat{R}_0^+(t)-\alpha| \le  \big| R_0^+(t) + |R_0(t)-\alpha| -\alpha \big| \le 2|R_0(t)-\alpha| + \delta R_0(t).
\end{split}
\end{equation}
Let ${\tau}_{\delta \textnormal{B}} := {\tau}_{\delta \textnormal{B}1}\wedge {\tau}_{\delta \textnormal{B}2} \wedge {\tau}_{ \textnormal{f}}\wedge \tau_{ \textnormal{B}{3}}$. By following similar computation as  \eqref{eq:reg:branching:tau2:split0}  and what follows afterwards,
\begin{equation}
\begin{split}
\intop_{(t-\alpha^{-1})\wedge{\tau}_{\delta \textnormal{B}}}^{t\wedge {\tau}_{\delta \textnormal{B}}}& \left| \hat{R}_0^+(s)-\alpha \right|ds
\le
\intop_{(t-\alpha^{-1})\wedge{\tau}_{\delta \textnormal{B}}}^{t\wedge {\tau}_{\delta \textnormal{B}}} 2\left| {R}_0(s)-\alpha \right|ds
+
\intop_{(t-\alpha^{-1})\wedge{\tau}_{\delta \textnormal{B}}}^{t\wedge {\tau}_{\delta \textnormal{B}}} \delta R_0(s) ds \\
&\le
\intop_{(t-\alpha^{-1})\wedge{\tau}_{\delta \textnormal{B}}}^{t\wedge {\tau}_{\delta \textnormal{B}}} \left\{2\alpha \beta^{C_\circ +\frac{1}{2}}\sigma_1(s;\hat{R}_0^+) + 
\alpha \beta^{4\theta+1}\sigma_1(s;\hat{R}_0^+) + 2\alpha^{\frac{3}{2}} \beta^{4\theta+1}
\right\} ds \le 1,
\end{split}
\end{equation}
where the second inequality follows from the definitions of $\tau_{\textnormal{f}1}$ and $\tau_{\delta \textnormal{B}2}$, and the last one comes from Lemma \ref{lem:ind1:pi1int basicbd:basic} along with the bounds from $\tau_{ \textnormal{B}{3}}$. Then, applying Corollary \ref{cor:concentration:numberofpts each interval} implies that
\begin{equation}\label{eq:branching:ib:tauB5}
\PP \left( \tau_{\textnormal{B}2} >\acute{t}_0 \right) \ge 1-7\exp\left(-\beta^3\right).
\end{equation}

Thus, we conclude the proof from \eqref{eq:branching:ib:tauB4} and \eqref{eq:branching:ib:tauB5}.
\end{proof}

To conclude the section, we record a technical but simple lemma used above, and it will be useful for the rest of the paper.

\begin{lem}\label{lem:ind1:pi1int basicbd}
	Let $g:[0,h]\to (0,\infty)$ be a positive function which defines a (deterministic) set of points $\Pi_g[0,h]$. 
	Furthermore, assume additionally that $K, N_0, \Delta_0, \Delta_1 >0$ satisfy 
	\begin{itemize}
		\item 	$\sup\{|\Pi_g[t-\Delta_0 , t] | : t\in[\Delta_0, h] \} \le N_0$;
		
		\item $\Delta_1\ge \Delta_0$ and $\inf \{|\Pi_g[t-\Delta_1,t]| : t\in [\Delta_1,h]  \} \ge 1$;
		
		\item $h = K\Delta_0$.
	\end{itemize}
	Then, there exists an absolute constant $C>0$ such that the following hold true:
	\begin{equation}
	\begin{split}
	&\intop_0^h \frac{dt}{\sqrt{(h-t+1)(\pi_1(t;g)+1) } } \le C \left( N_0 + \sqrt{\frac{K\Delta_1N_0}{\Delta_0}} \right) ;\\
	&\intop_0^h 
	\frac{dt}{\sqrt{h-t+1}(\pi_1(t;g)+1)} \le 
	CN_0\left(1 + \frac{\sqrt{K}  \log \Delta_1 }{\sqrt{\Delta_0}} \right);\\
	&\intop_0^h \frac{dt}{(h-t+1)\sqrt{\pi_1(t;g)+1} }  \le C \left( \frac{N_0 \log \Delta_1}{\sqrt{\pi_1(h;g)+1}} + \frac{N_0\sqrt{\Delta_1}}{{\pi_1(h;g)+1}} + \log K {\frac{\sqrt{\Delta_1N_0}}{\Delta_0}} \right)
	;\\
	&\intop_0^h \frac{dt}{(h-t+1)(\pi_1(t;g)+1) }  \le CN_0 \log \Delta_1 \left( \frac{1}{\pi_1(h;g)+1} + \frac{\log K}{\Delta_0}  \right).
	\end{split} 
	\end{equation}
\end{lem}

\begin{proof}
	Let $\Pi_g[0,h]:= \{p_1<p_2<\ldots<p_n \}$ with $p_0=0, p_{n+1}=h.$
	To establish the first inequality, we let $\Gamma_k := \Pi_g[(K-k-1)\Delta_0, (K-k)\Delta_0 ]$. In the first integral, note that the integrals from $p_i$ to $p_{i+1}$ is bounded by an absolute constant, for  $p_i \in \Gamma_0$. Hence,
	\begin{equation}
	\begin{split}
	\intop_0^h \frac{dt}{\sqrt{(t-x)(\pi_1(x;g)+1) } } \le CN_0 + \sum_{k=1}^{K-1} \sum_{p_i \in \Gamma_{k+1}} \intop_{p_i}^{p_i+1} \frac{dx}{\sqrt{k\Delta_0 (x-p_i)} }\\
	\le CN_0 + \sum_{k=1}^{K-1} \frac{\sqrt{\Delta_1N_0}}{\sqrt{k\Delta_0}}\le CN_0+ C \sqrt{\frac{K\Delta_1 N_0 }{\Delta_0}},
	\end{split}
	\end{equation}
	where the second inequality came from Lemma \ref{lem:ind1:pi1int basicbd:basic}. The other inequalities can be obtained similarly, and we omit the details. For the last two inequalities, we just  keep in mind that the contribution from the regime $[h-2\Delta_0, h]$ are dealt separately.
\end{proof}

\section{Regularity of the speed} \label{sec:reg:intro}

In this section, we formulate the notion of regularity, and build up an inductive analysis to study the regularity property of the speed.  To motivate the works done in Sections \ref{sec:reg:intro}--\ref{sec:reg:next step}, we briefly review our future goal to explain why establishing the regularity  is essential.

Recall the definition of $L$ given in \eqref{eq:def:Lt:basic form} (see also \eqref{eq:def:L}). The main goal of this paper  is to establish the scaling limit of $L$ driven by the speed process. To this end, we need to study the mean and the variance of its increment at each interval, which can be conceptually described by
\begin{equation}\label{eq:Lincrement:conceptual}
\mathcal{L}(t_1;\Pi_S[t_1^-,t_1],\alpha_1) - \mathcal{L}(t_0;\Pi_S[t_0^-,t_0], \alpha_0),
\end{equation}
conditioned on the information up to time $t_0$. Here, $t_0^-<t_0<t_1^-<t_1$, and $\alpha_1, \alpha_0$ are the \textit{frame of reference} in each interval which are essentially the  constants that approximates the speed. That is, on the interval $[t_0,t_1]$, the speed $S(t)$ behaves roughly like $\alpha_0$, with a smaller order  fluctuation which will indeed turn out to be essentially $\alpha_0^{3/2} \textnormal{polylog}(\alpha_0)$. 

The primary difficulty of computing the mean and variance of \eqref{eq:Lincrement:conceptual} stems from the complicated nature of the process $S(t)$. As we have seen in \eqref{eq:Lt:increment} (and \eqref{eq:integralform:branching}), the increment \eqref{eq:Lincrement:conceptual} can be described as an integral involving $d\Pi_S$, and if $S(t)$ does not behave nicely enough we may not be able to estimate such an integral accurately. This motivates us to demonstrate regularity properties of $S(t)$, which is going to tell us that $S(t)$ stays close to $\alpha_0$ during the interval $[t_0,t_1]$ in an appropriate sense,  so that we can calculate the first and second moment of \eqref{eq:Lincrement:conceptual} conditioned on $\mathcal{F}_{t_0}$.

With this goal in mind, our argument is based on an inductive study of the speed. We will begin with defining the regularity property, which is a fairly complicated mixture of all the properties that $S(t)$ and $\Pi_S$ should satisfy. Then, we will show that
\begin{center}
	If $S(t)$ is regular at time $t_0$ with respect to $\alpha_0$, \\
	then $S(t)$ is regular at time $t_1$ with respect to $\alpha_1$, with high probability.
\end{center}

Before moving on to the definition of regularity, we briefly explain how we define the parameters used in the section. Throughout the rest of the paper, we fix $\epsilon=\frac{1}{10000}$ to be a small constant, and $C_\circ=50$, $\theta=10000$ be large constants (where $\theta$  needs to be larger depending on $C_\circ$) as in Sections \ref{sec:fixedrate} and \ref{sec:criticalbranching}. For given $\alpha_0>0$ and $t_0>0$, we set $\beta_0:= \log(1/\alpha_0)$, and the time parameters $t_0^{-}, t_0^+, \acute{t}_0,  \textnormal{ and } \hat{t}_0$ are defined as  follows:
\begin{itemize}
	\item $t_0^-,t_0^+$ are arbitrary numbers satisfying 
	\begin{equation}\label{eq:def:t0t1}
	t_0-2\alpha_0^{-2}\beta_0^\theta < t_0^- <t_0-\frac{1}{2}\alpha_0^{-2} \beta_0^{\theta}, \quad 
	t_0 + \frac{1}{2}\alpha_0^{-2}\beta_0^{\theta} < t_0^+ < t_0 + 2\alpha_0^{-2}\beta_0^{\theta}
	.
	\end{equation}
	
	\item $\acute{t}_0, \hat{t}_0$ are fixed numbers set to be
	\begin{equation}\label{eq:def:t0t11}
	\acute{t}_0:= t_0 + 3\alpha_0^{-2} \beta_0^{\theta}, \quad 
	\hat{t}_0 := t_0 +2\alpha_0^{-2}\beta_0^{10\theta}.
	\end{equation} 
\end{itemize} 
Note that the definitions of $t_0^-$ and $\acute{t}_0$ are consistent with those from Section \ref{sec:criticalbranching}. We also stress that $$(\acute{t}_0-t_0^+)\wedge (t_0^+-t_0) \wedge (t_0-t_0^-) \ge \frac{1}{2}\alpha_0^{-2}\beta_0^{\theta}.$$

For the speed $S(t)$,  its \textit{first-order approximation}  $S_1(t)$ \eqref{eq:def:S1:basic form}  is set to be
\begin{equation}\label{eq:def:S1:reg}
S_1(t)= S_1(t;t_0^{-},\alpha_0):=
\intop_{t_0^{-}}^t K_{\alpha_0} (t-s) d\Pi_S(s).
\end{equation} 
 We also recall the definitions of $\pi_i(t;S)$ and $ \sigma_i(t;S)$ from \eqref{eq:def:pi closest points} and \eqref{eq:def:sig closest points}.

The rest of the section is organized as follows. We begin with  giving a formal definition of regularity  in Section \ref{subsec:regoverview:reg}, which will be studied exhaustively in Sections \ref{sec:reg:conti of reg} and \ref{sec:reg:next step}.   In Section \ref{subsec:regoverview:overview}, we state the main theorem and give a more involved overview on the induction argument. We also discuss some important consequences of regularity in Section \ref{subsec:reg:conseq}.
Finally, in Section \ref{subsec:regoverview:history}, we introduce a preliminary analysis on the speed that makes it possible to ignore the history of the far past, enabling us to study $S'(t)$ \eqref{eq:def:Sprime:basic form} instead of $S(t)$.

\subsection{Formal definition of regularity}\label{subsec:regoverview:reg}

Let $\mathcal{F}_{t_0}$ be the $\sigma$-algebra generated by $(S(s))_{s\le t_0}$ and $\Pi_S[0,t_0]$. Since the process $(S(s))_{s\le t_0}$ can completely be recovered from the given configuration of points $\Pi_S[0,t_0]$, we can also identify $\mathcal{F}_{t_0}$ with $\Pi_S[0,t_0]$. Our idea is to define $\Pi_S[0,t_0]$ to be regular if
\begin{equation}
\PP \left(\left. S \textnormal{ behaves }nicely \textnormal{ until time } \acute{t}_0\ \right|\ \Pi_S[0,{t_0}] \right) \approx 1.
\end{equation}
To introduce the \textit{nice} traits that $S$ should satisfy, we will investigate the process from the following perspectives.

\begin{enumerate}
	
	\item Control on the size of the speed;
	
	\item Control on the size of the aggregate;
	
	\item Refined control of the first order approximation;
	
	\item Regularity of the history until time $t_0$.
	
\end{enumerate}

In the following subsections, we detail the formal definitions of regularity for each category, and explain their purpose. Each criterion will be described in the language of stopping times, as we saw in Section \ref{sec:criticalbranching}. In the rest of the section and throughout Sections \ref{sec:reg:conti of reg} and \ref{sec:reg:next step}, $\kappa>0$ denotes a constant which is either $\frac{1}{2}$ or $2$. This constant is the parameter used to distinguish the ``stronger'' ($\kappa=\frac{1}{2}$) regularity from the ``weaker'' ($\kappa=2$) one: We show in Section \ref{sec:reg:conti of reg} that the initial ``weaker''  regularity leads to the ``stronger'' regularity at later times.

\subsubsection{Control on the size of the speed} \label{subsubsec:reg:Scontrol}

Maintaining an appropriate control on $S(t)$ is the primary property we want from the regularity. We begin with giving the formal definitions of the stopping times of interest. For simplicity, we write $\sigma_1\sigma_2(t):= \sigma_1(t;S)\sigma_2(t;S)$. Also, recall the definition of $S_1(t)$ from \eqref{eq:def:S1:reg} 
\begin{equation}
\begin{split}
\tau_{1} (\alpha_0,t_0,\kappa )&:=
\inf\{t\geq t_0: S(t) \geq \kappa \alpha_0 \beta_0^{C_\circ} \};
\\
\tau_{2}(\alpha_0,t_0,\kappa)&:= \inf\{ t\geq t_0: S_1(t) \ge \kappa \alpha_0 \beta^{C_\circ}_0 \};
\\
\tau_{3}(\alpha_0,t_0) &:=
\inf\left\{t\ge t_0: S(t)-S_1(t)\notin \left(-\alpha_0^{\frac{3}{2}-\epsilon}\sigma_1(t)  ,\alpha_0^{1-\epsilon} \sigma_1\sigma_2(t) \right) \right \};
\\
\tau_{4}(\alpha_0,t_0,\kappa)&:=
\inf\left\{t\ge t_0: \intop_{t_0}^{t} (S(s)-\alpha_0)^2ds \ge \kappa \alpha_0 \beta_0^{25\theta} \right\};
\\
\tau_{5}(\alpha_0,t_0,\kappa)&:=
\inf \left\{t\ge t_0: \intop_{t_0}^{t} (S(s)-\alpha_0)^2 S(s)ds \ge \kappa \alpha_0^2 \beta_0^{25\theta}  \right\}.
\end{split}
\end{equation}

We briefly explain the purpose of each stopping time as follows:
\begin{itemize}
	\item $\tau_{1}$  not only gives a fundamental understanding on $S$, but also useful in controlling other stopping times such as $\tau_{4}$ and $\tau_{5}$. In particular, it is heavily used in bounding the quadratic variations of various martingales, such as the second criterion in \eqref{eq:concentration:conditions:basic}.

	\item $\tau_{3}$ provides a crucial estimate in bounding $|S(t)-\alpha_0|\leq |S(t)-S_1(t)|+|S_1(t)-\alpha_0|$, which is important in studying various different quantities including $\tau_{4}$ and $\tau_{5}$. Moreover, the lower bound on $S(t)-S_1(t)$ will lead us to obtaining the lower bound on $S(t)$: See Proposition~\ref{prop:reg:lbd of speed}
	
	\item $\tau_{2}$, $\tau_{4}$ and $\tau_{5}$  verify the main assumptions of Proposition~\ref{prop:fixed perturbed:error}, enabling us to utilize the first- and second-order approximations.  
\end{itemize}

\subsubsection{Size of the aggregate}\label{subsubsec:reg:aggsize}

Note that the number of points in $\Pi_S[0,t]$ describes the size of the aggregate, i.e., $X_t = |\Pi_S[0,t]|$.  The stopping times that deals with the size of the aggregate have similar role as $\tau_{\textnormal{B}2}$ in Section \ref{sec:criticalbranching}, but we need several of them for different purposes.
\begin{equation}
\begin{split}
\tau_{6}(\alpha_0,t_0,\kappa)
&:= 
\inf\left\{t\ge t_0: |\Pi_S[t_0^{-},\, t]| \le \frac{1}{100\kappa} \alpha_0 (t-t_0^{-}) \right\};
\\
\tau_{7}(\alpha_0,t_0,\kappa)
&:=
\inf\left\{ t\ge t_0:
|\Pi_S[(t-\alpha_0^{-1})\vee t_0,\, t]| \ge\kappa \beta_0^4
\right\};
\\
\tau_{8}(\alpha_0,t_0,\kappa)
&:=
\inf \left\{ t\ge t_0+\kappa \alpha_0^{-1}\beta_0^{C_\circ}:
|\Pi_S[t-\kappa\alpha_0^{-1}\beta_0^{C_\circ},\, t]| =0 \right\}.
\end{split}
\end{equation}

\begin{itemize}
	\item $\tau_{6}$ provides useful estimates needed in Section \ref{subsec:regoverview:history}: If there are reasonably many particles in the interval $[t_0^{-},t]$, then it turns out we can  eliminate the effect of information before $t_0^{-}$. 
	
	\item $\tau_{7}$ plays a similar role as $\tau_{\textnormal{B}2}$ from Section \ref{sec:criticalbranching}, such as helping to control $\tau_{2}$ and $\tau_{3}$.
	
	\item $\tau_{8}$ makes it easier to control quantities involving $\sigma_1$ and $\sigma_2$, since the all neighboring points before $\tau_{8}$ cannot be too far from each other. 
\end{itemize}

\subsubsection{Refined control of the first order approximation} \label{subsubsec:reg:refined1storder}

We introduce a stopping time which is a strengthened versions of $\tau_1$ and $\tau_{3}$. Instead of  $\tau_1$, we describe a sharper control on the average of $|S(t)-\alpha_0|$.  On the other hand, in $\tau_{3}$, the bound on $|S(t)-S_1(t)|$ has the term $\alpha_0^{-\epsilon}$ which is not strong enough in the later analysis (see Theorem \ref{thm:increment:formal}). Thus, we seek for a better bound, which is of $\beta_0^C$ rather than $\alpha_0^{-\epsilon}$. We define the stopping times $\tau_1^\sharp$ and $\tau_{3}^\sharp$ to be
\begin{equation}
\begin{split}
&\tau_1^\sharp = \tau_1^\sharp(\alpha_0,t_0):= 
\inf \left\{ t\ge t_0: \intop_{t_0}^t |S(s)-\alpha_0| ds \ge \alpha_0^{-\frac{1}{2}-\epsilon} \right\};\\
&\tau_{3}^\sharp=\tau_{3}^\sharp(\alpha_0,t_0) := \inf \left\{ t\ge t_0: \, \intop_{t_0}^t |S(s)-S_1(s)| ds \ge \beta_0^{4\theta}  \right\}.
\end{split}
\end{equation}
It is clear that $\tau_1$ does not imply $\tau_1^\sharp$, and also is not difficult to see that $\tau_3^\sharp $ is stronger than $\tau_3$; since integrating $|S(s)-S_1(s)|$ based on the bound in $\tau_3$  (along with $\tau_7, \tau_8$, and Lemma \ref{lem:ind1:pi1int basicbd:basic}) results in $\alpha_0^{-\epsilon}$, rather than $\beta_0^{4\theta}$-bound given in $\tau_3^\sharp$.

We remark that in $\tau_3^\sharp$, having a smaller  exponent $4\theta$ of $\beta_0$  than $10\theta$ (which is the exponent in $\hat{t}_0-t_0$) plays a very important role in this analysis. See Section \ref{subsubsec:regoverview:overview:nextstep} for more discussion; details will be presented in Section \ref{sec:reg:next step}.

\subsubsection{Regularity of the history}\label{subsubsec:reg:history}

We introduce several additional events which are measurable with respect to $\mathcal{F}_{t_0}$. Recalling the definition of $\mathcal{L}$ \eqref{eq:def:L}, let
\begin{equation}\label{eq:def:alpha0prime}
\alpha_0'=\alpha_0'(\alpha_0,t_0,t_0^-):= \mathcal{L}(t_0;\Pi_S[t_0^{-},t_0],\alpha_0),
\end{equation}
where $\mathcal{L}$ is defined in \eqref{eq:def:L}. We also set $S(s) = \infty$ for $s<0$. Define
\begin{equation}\label{eq:def of the events A1 A2 A3}
\begin{split}
\mathcal{A}_1(\alpha_0,t_0)
&:=
\left\{ |\Pi_S[s, t_0^{-}]| \ge \sqrt{t_0^{-}-s+C_\circ} \log^2 (t_0^{-} -s +C_\circ) -\alpha_0^{-1} \beta_0^\theta, \ \ \forall 0\le s\le t_0^{-} \right\};
\\
\mathcal{A}_2(\alpha_0,t_0)&:= \left\{ \frac{1}{4} \le \frac{ |\Pi_S[t_0^{-},\, t_0]|}{\alpha_0^{-1}\beta_0^{\theta}} \le 4 \right\};
\\
\mathcal{A}_3(\alpha_0,t_0)&:=
\left\{|\alpha_0-\alpha_0'| \le \alpha_0^{\frac{3}{2}} \beta_0^{\theta} \right\}.
\end{split}
\end{equation} 

\begin{itemize}
	\item $\mathcal{A}_1$ is the key event used in Section \ref{subsec:regoverview:history}. By setting $|\Pi_S[s, t_0^{-}]|$ to be of strictly bigger order than $\sqrt{t_0^{-}-s}$ (which must actually be linear in $(t_0^{-}-s)$), the contribution from the history before $t_0^{-}$ becomes negligible from the representation \eqref{eq:speed:basic expression}, since the  probability of the simple random walk to hit the aggregate before time $t_0^{-}$ is small enough.
	
		\item $\mathcal{A}_2$ provides a basic estimate on the size of the aggregate in the past, and is used in investigating $\tau_{6}$.
		
	\item $\mathcal{A}_3$, is introduced to ensure that the choice of $\alpha$ is appropriate. Later,   our choice of $\alpha_1$ will satisfy   $|\alpha_1-\alpha_0|=\alpha_0^{3/2} \beta_0^{O(1)}$, which is the same order as the bound in $\mathcal{A}_3$.

\end{itemize}

\subsubsection{Definition of regularity}\label{subsubsec:regoverview:regdef}

We give the formal definition of regularity combining the notions introduced in Sections \ref{subsubsec:reg:Scontrol}--\ref{subsubsec:reg:history}. We let
\begin{equation}\label{eq:def:tau:induction}
\begin{split}
&\tau(\alpha_0,t_0,\kappa ) := \hat{t}_0 \wedge \min\{\tau_i(\alpha_0,t_0,\kappa) : 1\le i \le 8 \},\\
&\tau^\sharp(\alpha_0,t_0,\kappa ):= \tau(\alpha_0,t_0,\kappa )\wedge \tau_1^\sharp(\alpha_0,t_0)\wedge  \tau_3^\sharp(\alpha_0,t_0),
\end{split}
\end{equation}
where we define $\tau_3(\alpha_0,t_0,\kappa):= \tau_3(\alpha_0,t_0)$ which does not have dependence on $\kappa$. Note that $\tau$ is increasing in $\kappa$. Moreover, we set
\begin{equation}\label{eq:def:A:induction}
\begin{split}
\mathcal{A}(\alpha_0,t_0) := \mathcal{A}_1(\alpha_0,t_0) \cap \mathcal{A}_2(\alpha_0,t_0) \cap \mathcal{A}_3(\alpha_0,t_0)  .
\end{split}
\end{equation}
Note that these events are $\mathcal{F}_{t_0}$-measurable. We now introduce the two notions of regularity.

\begin{remark}
In the following definitions and throughout the paper, the notation
\begin{equation}
    \PP \left( \ \ \cdot \ \ | \, \Pi_S(-\infty,t_0] \right)
\end{equation}
denotes the probability with respect to the law of the aggregate starting from time $t_0$ under the initial points given by $\Pi_S(-\infty,t_0]$, which was defined in Definition \ref{def:aggregate with initial condition}. Moreover, note that if $S(s)=\infty$ on $s< t^\flat$ for some $t^\flat$, then $\Pi_S(-\infty,t^\flat)$ is not well defined (for instance, $\overline{S}(t)$ from  \eqref{eq:def:sandwich speed}). Even in this a case when there exists such an $t^\flat\le t_0$, we still use the notation $\Pi_S(-\infty,t_0]$ for convenience, which will actually refer to
\begin{equation}
    \Pi_S(-\infty,t_0] := \textnormal{ the points in }[t^\flat,t_0] \textnormal{ are given by } \Pi_S[t^\flat,t_0], \textnormal{ and }  S(s)=\infty  \textnormal{ for all } s\le t^\flat.
\end{equation}
\end{remark}

\begin{definition}[Regularity]\label{def:reg} For $\alpha_0, t_0, t_0^-,   r>0$, we denote $[t_0] := (t_0, t_0^-)$ and set $\acute{t}_0$ as \eqref{eq:def:t0t11}. We say  $\Pi_0(-\infty,t_0]$  is $(\alpha_0,r; [t_0])$-\textit{regular} if it satisfies the following two conditions:
	\begin{itemize}
		\item $\PP \left(\left.\tau (\alpha_0,t_0,\kappa=2) \le \acute{t}_0 \, \right| \, \Pi_S(-\infty,t_0] = \Pi_0(-\infty,t_0] \right) \le r;$
		
		\item $\Pi_0(-\infty,t_0] \in \mathcal{A}(\alpha_0,t_0)$.
	\end{itemize}
	We also write $\Pi_0(-\infty,t_0]\in \mathfrak{R}(\alpha_0,r;[t_0])$ to denote that $\Pi_0(-\infty,t_0]$ is $(\alpha_0,r;[t_0])$-regular. Note that the dependence on $t_0^-$ comes from the definitions of $S_1$ and $\mathcal{A}$.
\end{definition}

\begin{definition}[Sharp regularity]\label{def:reg:sharp} For $\alpha_0, t_0,t_0^-,t_0^+, r>0$ with $[t_0]:=(t_0,t_0^-,t_0^+)$, and set $\acute{t}_0$ as \eqref{eq:def:t0t11}. We say  $\Pi_0(-\infty,t_0]$  is $(\alpha_0,r;[t_0])$-\textit{sharp-regular} if it satisfies the following two conditions:
	\begin{itemize}
		\item $\PP \left(\left.\tau^\sharp (\alpha_0,t_0,\kappa=2) \le \acute{t}_0  \, \right| \, \Pi_S(-\infty,t_0] = \Pi_0(-\infty,t_0] \right) \le r;$

		\item $\Pi_0(-\infty,t_0] \in \mathcal{A}(\alpha_0,t_0)$.
	\end{itemize}
	We also write $\Pi_0(-\infty,t_0]\in \mathfrak{R}^\sharp(\alpha_0,r;[t_0])$ to denote that $\Pi_0(-\infty,t_0]$ is $(\alpha_0,r;[t_0])$-sharp-regular.
\end{definition}

 The reason for separating two concepts of regularity will be clear when we state the main result in the next subsection (see also Remark \ref{rmk:regvsshreg}). There, we also briefly explain the main ideas of proof.

\subsection{Overview of the argument}\label{subsec:regoverview:overview}
 Let $t_1$ be an arbitrary number that satisfies
\begin{equation}\label{eq:t1 regime}
t_0 + \alpha_0^{-2} \beta_0^{9\theta}  \le  t_1 \le t_0+\alpha_0^{-2} \beta_0^{10\theta} .
\end{equation}
We note that the upper bound of $t_1$ is substantially smaller than $\hat{t}_0$, namely, $t_1\le  t_0 + \alpha_0^{-2}\beta_0^{10\theta} = \hat{t}_0-\alpha_0^{-2}\beta_0^{10\theta}.
$ Then, the goal of Sections \ref{sec:reg:intro}--\ref{sec:reg:next step} is to establish the following theorems.

\begin{thm}\label{thm:induction:main}
	Let $\alpha_0, t_0, r>0$, let $t_0^-$ be any number satisfying \eqref{eq:def:t0t1}, set $\acute{t}_0$ as \eqref{eq:def:t0t11}, and let $t_1$ be an arbitrary fixed number selected within \eqref{eq:t1 regime}. Set $t_1^- := t_1- \alpha_0^{-2} \beta_0^\theta$, and  define
	\begin{equation}\label{eq:def:alpha1}
	\alpha_1:= \mathcal{L}(t_1^-; \Pi_S[t_0^-,t_1^-],\alpha_0 ), \quad \beta_1:= \log(1/\alpha_1).
	\end{equation} 
	Then, the following hold true for all sufficiently small	$\alpha_0$ and for any $\Pi_S(-\infty,t_0]\in \mathfrak{R}(\alpha_0,r;[t_0])$:
	\begin{enumerate}

		\item $\PP \left(\left.\Pi_S(-\infty,t_1] \notin \mathfrak{R}^\sharp(\alpha_1, e^{-\beta_1^{3/2}}; [t_1] ) \, \right| \, \Pi_S(-\infty,t_0] \right) \le r+ e^{-\beta_0^{3/2}}$;
		
		\item $\PP (| \alpha_1 - \alpha_0 |\ge 2\alpha_0^{3/2} \beta_0^{6\theta} \ | \ \Pi_S(-\infty,t_0]  ) \le r+ e^{-\beta_0^{3/2}}$.
	\end{enumerate}
\end{thm}
Note that the second statement directly implies
$$\PP ( t_1^-  \textnormal{ satisfies \eqref{eq:def:t0t1} in terms of } t_1, \alpha_1 \,| \, \Pi_S(-\infty,t_0] ) \ge 1-r-e^{-\beta_0^{3/2}},$$
which means that w.h.p., we can reapply the theorem for $\Pi_S(-\infty,t_1]$ and $[t_1]$ to obtain $\Pi_S(-\infty,t_2] \in \mathfrak{R}^\sharp(\alpha_2, e^{-\beta_2^{3/2}}; [t_2])$, and so on.

\begin{remark}
	 In the interval starting from $t_1$, the first-order approximation $S_1$ is defined as $S_1(t)=S_1(t;t_1^-,\alpha_1)$. Thus, by  setting  $\alpha_1$ to be $\mathcal{F}_{t_1^-}$-measurable as above, the points in $\Pi_S[t_1^-,t_1]$ before time $t_1$ do not affect the frame of reference $\alpha_1$ of $S_1(t)$.
\end{remark}

In addition to the above theorem,
 we need the regularity of the fixed rate process obtained in Section \ref{sec:Sandwich}, which serves as the induction base of our argument. Recall the definitions of $\overline{Y}_{t_0, \alpha_0}$ and  $\underline{Y}_{t_0, \alpha_0}$, and let $ \overline{S}(t)= \overline{S}_{t_0,\alpha_0}(t)$ and $\underline{S}(t)=\underline{S}_{t_0,\alpha_0}(t)$ be the corresponding speed processes  \eqref{eq:def:sandwich speed}, respectively. For $t_0^-$ and $\acute{t}_0$ as above, we define
\begin{equation}
\begin{split}
\overline{S}_1(t) = \overline{S}_1(t;t_0^-,\alpha_0) := \intop_{t_0^-}^{t_0} K_{\alpha_0}(t-x) d\Pi_{\alpha_0}(x) + \intop_{t_0}^t K_{\alpha_0}(t-x) d\Pi_{\overline{S}}(x)
\end{split} 
\end{equation}
 Thus, using these processes, we can define the stopping times  $\{\overline{\tau}_i(\alpha_0,t_0,\kappa) \}_{1\le i\le 8}$ (Sections \ref{subsubsec:reg:Scontrol}--\ref{subsubsec:reg:aggsize}), and the event $\overline{\mathcal{A}}(\alpha_0,t_0)$ \eqref{eq:def:A:induction}, analogously as before. Then, they define the regularity (Definition \ref{def:reg}) of $\overline{S}(t)$. Everything can be done in the same way for $\underline{S}(t)$. The main theorem on the regularity of $\overline{S}(t)$ and $\underline{S}(t)$ is stated as follows.

\begin{thm}\label{thm:induction:base:main}
	Let $\alpha_0, t_0>0$, set $\acute{t}_0$ as \eqref{eq:def:t0t11}, and set $t_0^- := t_0-\alpha_0^{-2}\beta^\theta$.   Then, under the above setting, we have
	\begin{equation}
	\PP\left(\Pi_{\overline{S}}(-\infty,t_0], \, \Pi_{\underline{S}}(-\infty,t_0] \in \mathfrak{R}\left(\alpha_0, e^{-\beta_0^{3/2}}; [t_0]\right) \right) \ge 1-e^{-\beta_0^{3/2}}.
	\end{equation}
\end{thm}

\begin{remark}\label{rmk:regvsshreg}
	After some large constant time since the critical aggregate starts growing, we require the speed to be very small and the process to be sharp-regular. However, due to Theorem \ref{thm:induction:main}, a regular (does not have to be sharp-regular) interval  beginning will result in sharp-regular intervals afterwards w.h.p., and hence Theorem \ref{thm:induction:base:main} will be enough for our purpose.
\end{remark}

The rest of the section and Sections \ref{sec:reg:conti of reg}--\ref{sec:reg:next step} are mostly devoted to the  proof of Theorem \ref{thm:induction:main}, which consists of three major steps  as follows.
\begin{enumerate}
	\item Approximating $S(t)$ by $S'(t)$ \eqref{eq:def:Sprime:basic form} to validate the first and second order approximations;
	
	\item Continuation of regularity (i.e., the stopping times) until time $\hat{t}_0$, which will be larger than $\acute{t}_1:=t_1+2\alpha_1^{-2}\beta_1^\theta$;
	
	\item Regularity at time $t_1$, with new the frame of reference $\alpha_1$.
\end{enumerate}

In the remaining  subsection, we briefly explain the main purposes and ideas of (1), (2) and (3).

\subsubsection{The first and second order approximation of the speed} \label{subsubsec:regoverview:overview:history}

To study $S(t)$, we  attempt to use the first- and second-order approximations given in Section \ref{subsec:speed:speedagg}. However, the formulas  \eqref{eq:speed:1storder main}, \eqref{eq:speed:2ndorder main} give expansions on $S'(t)$ defined in \eqref{eq:def:Sprime:basic form}, not $S(t)$. Thus, we need to justify that $S(t)$ is sufficiently close to $S'(t)$.

This analysis is conducted in the next subsection by studying the original formula   \eqref{eq:speed:basic expression} of $S(t)$. We show that under the regularity condition, the hitting probability of the random walk affected by conditioning on $\mathcal{F}_{t_0^{-}}$ is negligible, since $t_0^{-}$ is far away from $t_0$.

\subsubsection{Continuation of regularity}\label{subsubsec:regoverview:overview:contiofreg}

The next step is to show that if $\tau(\alpha_0,t_0,\kappa=2)\ge \acute{t}_0$, then it is unlikely to be smaller than $\hat{t}_0$. The purpose of such an analysis is fairly obvious: In order to establish regularity at time ${t_1}$ in terms of $\alpha_1$, we would like to argue that (1) the process at time $t_1$ maintains the desired regularity properties in terms of $\alpha_0$, and (2) changing $\alpha_0$ to $\alpha_1$ does not have a significant effect. For (2), see the discussion in the following subsection.

 To establish (1), we essentially argue that $\{\tau_i = \tau \}$ happens with small enough probability for all~$i$. In the following, we discuss several major ideas needed to process the argument.

\begin{itemize}
	\item 
	We study $\tau_1$ via the original formula \eqref{eq:speed:basic expression}, by upper bounding $Y_t$ by an appropriate choice of $Y_t'$ that achieves the desired estimate. The construction of $Y_t'$ is done based on the control  of $X_t$ which is given by $\tau_{7}$.
	
	\item  To obtain estimates on $\tau_{3}$, we  appeal to the results we obtained in Section \ref{subsec:fixed:error} where $\tau_{1}, \tau_{2}, \tau_{4}$ and $\tau_{5}$ provide the appropriate assumptions we need for utilizing Proposition \ref{prop:fixed perturbed:error}.
	
	\item To control the rest of the stopping times, we first express
	\begin{equation}
	|S(t)-\alpha_0|
	\le 
	|S(t)-S_1(t)|+|S_1(t)-\alpha_0|.
	\end{equation}
	The first  term in the RHS are controlled by $\tau_{3}$, and the remaining task is to control $|S_1(t)-\alpha_0|$. To this end, we observe that $S_1(t)=\mathcal{R}_c(t,t;\Pi_S[t_0^-,t],\alpha_0)$ (see \eqref{eq:def:Qst} for its definition) and interpret it via \eqref{eq:integralform:branching}.
	After we understand the gap $|S(t)-\alpha|$, the results will follow mostly by applying the lemmas from Section \ref{subsec:fixed:mgconcen}.
\end{itemize}

\subsubsection{Regularity in the next time step}\label{subsubsec:regoverview:overview:nextstep}

The remaining step is to establish (2)  in the first paragraph of Section \ref{subsubsec:regoverview:overview:contiofreg}. We note several major issues in achieving this goal:

\begin{itemize}
	\item Our argument is mostly based on analyzing what happens when we change $\alpha_0$ to $\alpha_1$, and hence the control on $|\alpha_1-\alpha_0|$ is essential. This is done by expressing the difference in the integral form which is similar to \eqref{eq:integralform:branching} and studying it based on martingale techniques such as Lemma \ref{lem:concentration of integral} and the analytic properties of $K_\alpha(s)$ from Section \ref{sec:fourier and renewal}.

	\item The most crucial task is to understand $\tau_{3}^\sharp(\alpha_1,t_1)$. We need to avoid having $\alpha^{-\epsilon}$ term, which is present in $\tau_{3}$. To this end, we need to study the gap between $S(t)$ and its second-order approximation $S_2(t)$ \eqref{eq:def:S2:basic form}, rather than $|S(t)-S_1(t)|$. Then, however,  it turns out that the double integral term of $J_{t-s,t-u}$ in the formula of $S_2(t)$ cannot be studied in the same way as Section \ref{sec:fixedrate},  since the perturbation argument in Section \ref{subsec:fixed:perturbed} is too costly to guarantee a sharp control like $\tau_{3}^\sharp$.
	To overcome this issue, we study this double integral rather directly, appealing to the proximity of $S(t)$  to $\alpha_0$ obtained from  the previous time step. 
	
	One interesting issue we  emphasize is that $\tau^\sharp_3 (\alpha_1,t_1)$ is typically much smaller than $t_2=t_1+\alpha_1^{-2} \beta^{10\theta}_1$: As  $t$ gets closer to $t_2$, the variance accumulated from the double integral becomes so large that we need a larger exponent on $\beta_1$ for a correct bound.
\end{itemize}

\subsection{Consequences of regularity}\label{subsec:reg:conseq}

In this subsection, we highlight some properties that a regular aggregate must have. These will all be needed in Section \ref{sec:double int},  where we study the  formal version of Theorem \ref{thm:speed:increment:informal}.

We begin with the control on the number of particles in $\Pi_S[t_0,t]$ when the process is regular. Its proof will be discussed in Section \ref{subsec:ind1:growth}.

\begin{prop}\label{prop:ind1:growth}
	Let $\alpha_0,t_0,r>0$, set $\acute{t}_0, \hat{t}_0$ as \eqref{eq:def:t0t11}, and let $t_0^-$ be any number satisfying \eqref{eq:def:t0t1}. For any  $\Pi_S(-\infty,t_0]\in \mathfrak{R}(\alpha_0,r;[t_0])$, we have
	\begin{equation}
	\begin{split}
	\PP \left( \left. \Big||\Pi_S[t_0, t]| - \alpha_0(t-t_0) \Big| \le \alpha_0^{-\frac{1}{2}-10\epsilon},\ \forall t\in[\acute{t}_0,\hat{t}_0] \, \right| \, \Pi_S(-\infty,t_0] \right) \ge 1-r-2\exp\left(-\beta_0^{1.9} \right).
	\end{split}
	\end{equation}
\end{prop}

Next, the following describes an important property of how the frame of reference changes.

\begin{prop}\label{prop:reg:newrates}
	Let $\alpha_0,t_0,r>0$, and let $\acute{t}_0, t_0^-, t_1, t_1^-$ be as Theorem \ref{thm:induction:main}. 	
	Furthermore, we define $\alpha_1':=\mathcal{L}(t_1;\Pi_{S}[t_1^{-},t_1], \alpha_1)$ analogously as \eqref{eq:def:alpha0prime}, and set
	\begin{equation}
	\tilde{\alpha}_1:= \mathcal{L}(t_1;\Pi_S[t_0^-,t_1],\alpha_0).
	\end{equation}
	Then, the following holds true for all sufficiently small $\alpha_0$ and for any $\Pi_S(-\infty,t_0]\in \mathfrak{R}(\alpha_0,r;[t_0])$:
	\begin{equation}
	\PP \left(\left. |\alpha_1'-\tilde{\alpha}_1| \ge 2\alpha_0^{2} \beta_0^{8\theta+1} \, \right| \, \Pi_S(-\infty,t_0] \right) \le r+\exp\left(-\beta_0^2 \right).
	\end{equation} 
\end{prop}

In fact, we will later derive scaling limit of $\alpha'$ (see Sections \ref{sec:double int} and \ref{sec:scalinglimit}). To this end, we study the increment $\alpha_1'-\alpha_0'$ (defined in \eqref{eq:def:alpha0prime}) in Theorem \ref{thm:increment:formal}. Expressing that $$\alpha_1'-\alpha_0'=(\alpha_1'-\tilde{\alpha}_1) + (\tilde{\alpha}_1-\alpha_0'),$$
the estimate in the proposition gives an essentially sharp estimate on $(\alpha_1'-\tilde{\alpha}_1) $. Also, note that it is strictly stronger than the error bound in $\mathcal{A}_3$ (Section \ref{subsubsec:reg:history}). The proof of Proposition \ref{prop:reg:newrates} will be discussed in Section \ref{subsec:ind2:ratechange}.

Furthermore, we record estimates on the integrals of the errors such as $(S(t)-\alpha_0)$ in the following lemma. 
\begin{lem}\label{lem:reg:errorint}
	Let $\alpha_0,t_0,r>0$, and let $\acute{t}_0, t_0^-, t_1, t_1^-$ be as Theorem \ref{thm:induction:main}. Furthermore, let $S_1(t)=S_1(t;t_0^-,\alpha_0)$ and $S_2(t)=S_2(t;t_0^-,\alpha_0)$ be the first- and second-order approximations of the speed given by \eqref{eq:def:S1:basic form}, \eqref{eq:def:S2:basic form}. For any $\Pi_S(-\infty,t_0]\in \mathfrak{R}(\alpha_0,r;[t_0])$, we have
		\begin{eqnarray}
	 \textnormal{(1)}\qquad \qquad  &	\PP \left(\left. \intop_{\acute{t}_0}^{t_1} \big|S(t)-\alpha_0\big|dt  \ge \alpha_0^{-\frac{1}{2}-\epsilon}  \, \right| \, \Pi_S(-\infty,t_0] \right) \le r+ \exp\left(-\beta_0^{1.9} \right); &\qquad \qquad  \\
	  \textnormal{(2)} \qquad \qquad  &	\PP \left(\left. \intop_{\acute{t}_0}^{t_1} \big|S_1(t)-\alpha_0\big|dt  \ge \alpha_0^{-\frac{1}{2}-\epsilon}  \, \right| \, \Pi_S(-\infty,t_0] \right) \le r+ \exp\left(-\beta_0^{1.9} \right); & \qquad  \qquad \\
	  \textnormal{(3)} \qquad \qquad  &	\PP \left(\left.\intop_{{t}_0}^{t_1} \big|S(t)-S_1(t)\big|dt  \ge \alpha_0^{-\epsilon}  \, \right| \, \Pi_S(-\infty,t_0] \right) \le r+ \exp\left(-\beta_0^{1.9} \right); & \qquad  \qquad \\
	  \textnormal{(4)} \qquad \qquad  &	\PP \left(\left. \intop_{\acute{t}_0}^{t_1} \big|S(t)-S_2(t)\big|dt  \ge \alpha_0^{\frac{1}{2}-\epsilon}  \, \right| \, \Pi_S(-\infty,t_0] \right) \le r+ \exp\left(-\beta_0^{1.9} \right). & \qquad  \qquad 
		\end{eqnarray}

\end{lem}
Note that all the regimes of the integrals are $[\acute{t}_0,t_1]$, except for (3) which is $[t_0,t_1]$. One can essentially prove the same for the integral over $[\acute{t}_0,t_1]$, but the above is enough for our purpose as well and is simpler to establish. The proof is discussed later:
\begin{itemize}
	\item (1) and (2) are proven in  Corollary \ref{cor:ind1:Sint S1int}
	
	\item (3) is established in Lemma \ref{lem:ind1:intofSminS1}.
	
	\item (4) is verified in Section \ref{subsec:ind1:1storder error}.
\end{itemize}

{ }Lastly, we state  a lower bound on $S(t)$.  Although the bounds given in $\tau_{1}$ implies that $S(t)$ can get pretty high compared to $\alpha_0$, it turns out that $S(t)$ cannot stay too low from $\alpha_0$.

\begin{prop}\label{prop:reg:lbd of speed}
	Let $\alpha_0,t_0,r>0$, and let $\acute{t}_0, t_0^-,t_0^+, \hat{t}_0$ be as \eqref{eq:def:t0t1} and \eqref{eq:def:t0t11}. 
	Define the stopping time $\tau_{\textnormal{lo}}$ by
	\begin{equation}
	\tau_{\textnormal{lo}}=\tau_{\textnormal{lo}}(\alpha_0,t_0,t_0^+):= \inf \left\{ t\ge t_0^+: \ S(t) \le \alpha_0 - \alpha_0^{\frac{3}{2}-\epsilon} \right\}.
	\end{equation}
	Then, for any $\Pi_S(-\infty,t_0] \in \mathfrak{R}(\alpha_0,r;[t_0])$, we have $$\PP(\tau_{\textnormal{lo}}\le \hat{t}_0 \, |\, \Pi_S(-\infty,t_0]  )\le r+ \exp\left(-\beta_0^2 \right) . $$
\end{prop}
This proposition will be useful later in Section \ref{sec:double int}, and its proof is presented in Section \ref{subsec:ind1:1storder error}.

\subsection{Neglecting the far away history}\label{subsec:regoverview:history}

In this section, we resolve the issue addressed in Section \ref{subsubsec:regoverview:overview:history}. Recall the formula \eqref{eq:speed:basic expression}:
\begin{equation}\label{eq:def:S:in induction}
S(t) = \frac{1}{2}\PP \left( W(s) \le Y_t(s), \ \forall s\ge 0 \ | \ Y_t \right).
\end{equation} 
Our goal is to show that its dependence on   $\mathcal{F}_{t_0^{-}} = \{Y_t(s)\}_{s\ge t-t_0^{-}}$ is negligible. To this end, we introduce an approximate version of speed and study its first and second order approximations. Let $\Pi'$ be the rate-1 Poisson point process on $\mathbb{R}^2$ independent from $\Pi$, and define $Y_t'(s)$ as
\begin{equation}
Y_t'(s):= \begin{cases}
Y_t(s), &\textnormal{ for } s\le t-t_0^{-};\\
Y_t(t_0^{-}) + |\Pi_{\alpha_0}'[t-s, t_{0}^{-}]|, &\textnormal{ for } s> t-t_0^{-}. 
\end{cases}
\end{equation}
Then, we define the \textit{quenched} approximated speed $S_q'(t)$ as
\begin{equation}\label{eq:def:Sprime}
\begin{split}
S_q'(t) = S_q'(t;t_0^{-},\alpha_0) := \frac{1}{2} \PP\left(W(s)\le Y_t'(s),\ \forall s\ge 0 \ | \ Y_t' \right) .
\end{split}
\end{equation}
We recall the previous definition of $S'(t)$ in \eqref{eq:def:Sprime:basic form}, and observe that
\begin{equation}
\mathbb{E}_{\Pi_{\alpha_0}'}[S_q'(t;t_0^{-},\alpha_0)] = S'(t;t_0^{-},\alpha_0).
\end{equation}
Recalling the definitions of $\tau_{6}$ and $\mathcal{A}$ from Sections \ref{subsubsec:reg:aggsize} and \ref{subsubsec:reg:history}, respectively, the main purpose of this subsection is to establish the following property. 

\begin{prop}\label{prop:SvsSprime}
	Suppose that we have $\Pi_S[t_0^-,t_0]\in \mathcal{A}(\alpha_0,t_0)$. Then, for all sufficiently small $\alpha_0>0$, we have $$|S(t)-S'(t;t_0^{-},\alpha_0)| \le \alpha_0^{100},$$ for all $t\in [t_0, \tau_{6}(\alpha_0,t_0,\kappa=2)]$. Furthermore, for any $t\in[\acute{t}_0,\tau_{6}\wedge\tau_{8}(\alpha_0,t_0,\kappa=2)]$, we have
	\begin{equation}
	|S(t)-S'(t;t_0,\alpha_0)| \le \alpha_0^{100}.
	\end{equation}	
\end{prop}

It turns out that the same proof as Proposition \ref{prop:SvsSprime} implies the following corollary, which will be used to study the induction base.

\begin{cor}\label{cor:SvsSprime}
	 Under the setting of Theorem \ref{thm:induction:base:main} and the discussions above, define $\overline{S}'(t;t_0^-,\alpha_0)$ analogously as above in terms of $\overline{S}(t)$, and suppose that $\Pi_{\overline{S}}(-\infty,t_0] \in \overline{\mathcal{A}}(\alpha_0,t_0)$. Then, for all sufficiently small $\alpha_0>0$, we have
	\begin{equation}
	|\overline{S}(t) - \overline{S}'(t;t_0^-,\alpha_0)| \le \alpha_0^{100},
	\end{equation}
	for all $t\in [t_0,\overline{\tau}_6(\alpha_0,t_0,2)].$ The same thing holds for $\overline{S}(t)$ and $\overline{S}'(t;t_0^-,\alpha_0)$.
\end{cor}

\begin{proof}[Proof of Proposition \ref{prop:SvsSprime}]
	We prove the first statement of the proposition, and then the second one follows analogously. Let $t\in [t_0,\tau_{6}(\alpha_0,t_0,2)]$, and for simplicity we write $S'(t)=S'(t;t_0^{-},\alpha_0)$. Moreover, we write $\PP(\ \cdot \ )$ to denote the probability over both the random walk $W$ and the independent point process $\Pi_{\alpha_0}'$ together. The main idea is to compare the formulas \eqref{eq:def:S:in induction} and \eqref{eq:def:Sprime} by coupling the random walk $W$ together. This gives that
	\begin{equation}\label{eq:reg:history:SminusSprime}
	\begin{split}
	|S(t)- S'(t)|
	&\leq
	\PP_{\Pi_{\alpha_0}'}\left(\left.W(s)\le Y_t(s) \ \forall s\le t-t_0^{-}, \textnormal{ and }  \exists s'\ge t-t_0^{-} : W(s')\ge Y_t(s') \ \right| \ Y_t \right)
	\\
	&\ +\PP_{\Pi_{\alpha_0}'}\left(\left.W(s)\le Y_t(s) \ \forall s\le t-t_0^{-}, \textnormal{ and }  \exists s'\ge t-t_0^{-} : W(s')\ge Y_t'(s') \ \right| \ Y_t\right)\\
	&=:P_1 + P_2.
	\end{split}
	\end{equation}

	We begin with bounding $P_2$. Since $t\leq \tau_{6}(\alpha_0,t_0,2)$, \begin{equation}\label{eq:reg:SvsSprime:base}
	Y_t(t_0^{-}) \ge \frac{\alpha_0}{200}(t-t_0^{-}).
	\end{equation} Thus, we can write
	\begin{equation}\label{eq:reg:history:P2}
	\begin{split}
	P_2 &\le \PP \left(W(t-t_0^{-}) \ge \frac{\alpha_0}{400} (t-t_0^{-})  \right)\\
	&+
	\PP \left( \exists s\ge t_0^{-}: \ W(s) \ge |\Pi'_{\alpha_0}[t-s, t-t_0^{-} ]| + \frac{\alpha_0}{400}(t-t_0^{-}) \right).
	\end{split} 
	\end{equation}
	Note that since $t-t_0^{-} \ge \alpha_0^{-2} \beta_0^\theta$, 
	\begin{equation}
	\frac{\alpha_0}{400}(t-t_0^{-}) \ge
	\sqrt{t-t_0^{-}} \log^{\theta/3}(t-t_0^{-}).
	\end{equation}
	Plugging this into the first term of \eqref{eq:reg:history:P2}, we can see that
	\begin{equation}\label{eq:reg:history:P2:1}
	\PP \left(W(t-t_0^{-}) \ge \frac{\alpha_0}{400} (t-t_0^{-})  \right)\le \alpha_0^{101}.
	\end{equation}
	On the other hand, to study the second term of \eqref{eq:reg:history:P2}, we define
	\begin{equation}
	U'(s) = |\Pi_{\alpha_0}'[0,s] |-W(s) + \frac{\alpha_0}{400}(t-t_0^{-}),
	\end{equation}
	and define the stopping time $T':= \inf\{s>0: U'(s)\le0\}$. Then, we can write
	\begin{equation}
	\begin{split}
		\PP \left( \exists s\ge t_0^{-}: \ W(s) \ge |\Pi'_{\alpha_0}[t-s, t-t_0^{-} ]| + \frac{\alpha_0}{400}(t-t_0^{-}) \right)
		=
		\PP \left(T'<\infty \right)\\
		 = \lim_{s\to \infty} \mathbb{E} \left[ \left(\frac{1}{1+2\alpha_0} \right)^{U'(s) \wedge T' } \right].
	\end{split}
	\end{equation}
	As we saw in \eqref{eq:speed:Ut mg conv}, $(1+2\alpha_0)^{-U'(s)}$ is a martingale, and hence the Optimal Stopping Theorem tells us that
	\begin{equation}\label{eq:reg:history:P2:2}
	\PP(T'<\infty) = \mathbb{E}\left[\left(\frac{1}{1+2\alpha_0} \right)^{U'(0)} \right] \le \left(\frac{1}{1+2\alpha_0} \right)^{\frac{\alpha_0^{-1}\beta_0^\theta}{400}} \le \exp\left( -\frac{\beta_0^\theta}{400}\right) \le \alpha_0^{101}.
	\end{equation}
	Combining \eqref{eq:reg:history:P2:1} and \eqref{eq:reg:history:P2:2} tells us that $P_2 \le 2\alpha_0^{101}$.
	
	To control $P_1$, we first note that $\mathcal{A}_1(\alpha_0,t_0)$ and $t\le \tau_{6}(\alpha_0,t_0,2)$ implies the following for all $u\le t_0^{-}$:
	\begin{equation}
	\begin{split}
	X_{t } - X_u
	&=
	X_t - X_{t_0^{-} }+ X_{t_0^{-} } - X_u\\
	&\ge 
	\frac{\alpha_0}{200}(t-t_0^{-})+
	\sqrt{t_0^{-}-u+C_\circ} \log^2 (t_0^{-}-u +C_\circ) -\alpha_0^{-1}\beta_0^{\theta/2}\\
	&\ge
	\frac{\alpha_0}{300}(t-t_0^{-})+\sqrt{t_0^{-}-u+C_\circ} \log^2 (t_0^{-}-u +C_\circ),
	\end{split}
	\end{equation}
	where the last line is from $t-t_0^{-}\ge \alpha_0^{-2} \beta_0^\theta$. We can also see that
	\begin{equation}
	\frac{\alpha_0}{300}(t-t_0^{-}) \ge \sqrt{t-t_0^{-}+C_\circ} \log^2 (t-t_0^{-}+C_\circ), 
	\end{equation}
	which gives
	\begin{equation}
	\begin{split}
	X_{t } - X_u
	&\ge
	\sqrt{t-t_0^{-}+C_\circ} \log^2 (t-t_0^{-}+C_\circ)
	+
	\sqrt{t_0^{-}-u+C_\circ} \log^2 (t_0^{-}-u +C_\circ)\\
	&\ge
	\sqrt{t-u+C_\circ} \log^2 (t-u+C_\circ),
	\end{split}
	\end{equation}
	where the last inequality follows from the fact that the function $\sqrt{x+20}\log^2(x+20)$ is increasing and concave. From this estimate, we can see that
	\begin{equation}\label{eq:reg:history:P1}
	P_1 \le 
	\PP\left(\exists s'\ge t-t_0^{-}: W(s') \ge \sqrt{s'}\log^2(s') \right)
	\le e^{-\beta_0^3
	} \le \alpha_0^{101},
	\end{equation}
	where the second inequality is a standard estimate on a simple random walk. We conclude the proof by combining \eqref{eq:reg:history:SminusSprime}, \eqref{eq:reg:history:P2} and \eqref{eq:reg:history:P1}. 
	
	To obtain the second statement of the proposition, we apply the same argument with 
	\begin{equation}
	Y_t(t_0) \ge \frac{\alpha_0}{2\beta_0^{C_\circ}} (t-t_0),
	\end{equation}
	instead of \eqref{eq:reg:SvsSprime:base}. This estimate is obtained from the definition of $\tau_8$.
\end{proof}

As a concluding remark of the section, we record a simple lemma on $\overline{S}(t)$ and $\underline{S}(t)$ as a step towards establishing Theorem \ref{thm:induction:base:main}. 
\begin{lem}\label{lem:reg:base:A}
	Under the setting of Theorem \ref{thm:induction:base:main} and the discussions above, we have
	\begin{equation}
	\PP \left(\Pi_{\overline{S}}(-\infty,t_0] \in \overline{\mathcal{A}}(\alpha_0,t_0) \right) \ge 1- \exp \left(-\beta_0^{4} \right).
	\end{equation}
	The same holds for $\underline{S}(t)$.
\end{lem}

\begin{proof}
	The estimate on the event $\{\Pi_{\overline{S}}(-\infty,t_0] \in \overline{\mathcal{A}}_1(\alpha_0,t_0) \cap \overline{\mathcal{A}}_2(\alpha_0,t_0) \}$ follows straight-forwardly from the properties of a fixed rate Poisson process, and we omit the details. The  bound on the event $\overline{\mathcal{A}}_3(\alpha_0,t_0)$ comes from the same argument as \eqref{eq:reg:ib:med1} and \eqref{eq:reg:ib:med2}, replacing $t_0^\flat$ by $t_0$.
\end{proof}

\section{The inductive analysis on the speed: Part 1} \label{sec:reg:conti of reg}

The purpose of this section is to formulate the argument discussed in Section \ref{subsubsec:regoverview:overview:contiofreg}. Before we state the main theorem, we briefly explain a technical device required to understand the procedure of the induction argument. In the definition of $\tau$ from \eqref{eq:def:tau:induction}, recall that there was a parameter $\kappa \in \{\frac{1}{2},\, 2 \}$ that provides a constant multiplicative room in controlling the quantities inside $\tau_i$'s. One of the main observations in the induction argument is that even if we have $\{\tau(\alpha_0, t_0, \kappa) > \acute{t}_0\}$ with a weaker (i.e., larger) $\kappa$, the control on the regularity will bootstrap as the time passes by and at a larger time we still have $\{\tau(\alpha,t_0',\kappa') \ge \hat{t}_0 \}$ but with a stronger (i.e., smaller) $\kappa'$, with $t_0' > t_0$. Note that we can only have  this bootstrapped estimate with $t_0'>t_0$, not with $t_0$, since the initial assumption is made under a weaker $\kappa$. By doing this bootstrapping, we acquire enough room to cover the error that comes from changing $\alpha_0$ to $\alpha_1$ in the next time step. 

Having this idea in mind, we introduce another notation before stating the main objective. We define ${\tau}^+$ to be
\begin{equation}\label{eq:def:ind1:tauacute:basic}
{\tau}^+(\kappa ) := \tau(\alpha_0, t_0^+, \kappa),
\end{equation} 
except that we use the same $S_1(t)=S_1(t;t_0^-,\alpha_0)$ when defining $\tau(\alpha_0, t_0^+, \kappa)$. That is, for instance, we define
\begin{equation}\label{eq:def:ind1:tauacute}
\begin{split}
{\tau}^+_2(\kappa) &:= \inf \{t\ge t_0^+: S_1(t;t_0^-,\alpha_0) \ge \kappa \alpha_0 \beta_0^{C_\circ} \};\\ 
{\tau}^+_4(\kappa)&:=
\inf\left\{t\ge {t}_0^+: \intop_{t_0^+}^{t} (S(s)-\alpha_0)^2ds \ge \kappa \alpha_0 \beta_0^{25\theta} \right\};\\
{\tau}^+_6(\kappa)
&:= 
\inf\left\{t\ge {t}_0^+: |\Pi_S[t_0^{-},\, t]| \le \frac{1}{100\kappa} \alpha_0 (t-t_0^{-}) \right\},
\end{split}
\end{equation}
and corresponding analogue for all other ${\tau}^+_i(\kappa), \ i=1,3,5,7,8$, which are all defined in a straight-forward way unlike the  three mentioned above. Then, set ${\tau}^+(\kappa)$ to be the minimum of ${\tau}^+_i$'s. The main result for this section can be stated as follows.

\begin{thm}\label{thm:reg:conti:main}
	Recall the definitions of $\tau$ and $\mathcal{A}$ from \eqref{eq:def:tau:induction} and \eqref{eq:def:A:induction}, respectively. Write $\tau(\kappa)=\tau(\alpha_0,t_0,\kappa)$,   $\mathcal{A}=\mathcal{A}(\alpha_0,t_0)$, and set ${\tau}^+(\kappa)$ as above. For all sufficiently small $\alpha_0>0$ and any $t_0>0$, we have
	\begin{equation}\label{eq:ind1:main}
	\PP \left({\tau}^+(1/2 ) < \hat{t}_0 \ | \ \mathcal{A} \right)
	\le
	\PP \left(\tau(2 ) \le  \acute{t}_0 \ | \ \mathcal{A} \right)
	+
	\exp\left(-\beta_0^{1.9} \right).
	\end{equation}
\end{thm} 

\begin{remark}
	The equations involving the notation $\PP (\ \cdot \ | \, \mathcal{A} ) $ are understood as follows throughout Sections \ref{sec:reg:conti of reg} and \ref{sec:reg:next step}: The equation holds for any $\Pi_S(-\infty,t_0]$ such that $\Pi_S(-\infty,t_0]\in \mathcal{A}$.
\end{remark}

As done in Section \ref{sec:criticalbranching}, our idea is to show that each stopping time ${\tau}^+_i(1/2)$ is likely to be larger than $\tau(2)$ if $\tau(2)\ge \acute{t}_0$. In the following sections, we conduct this task through three major steps as follows.
\begin{enumerate}
	\item In Section \ref{subsec:ind1:1storder error}, we study $|S(t)-S_1(t)|$ to control ${\tau}^+_3$, done by utilizing the result we saw in Section \ref{subsec:fixed:error}.
	
	\item In Section \ref{subsec:ind1:S1vsalpha}, we estimate the gap $|S_1(t)-\alpha_0|$ based on ideas from Section \ref{subsec:branching:mg}.

	\item In Section \ref{subsec:ind1:Svsalpha}, we control all stopping times except ${\tau}^+_3$, with the estimate 
	\begin{equation}\label{eq:ind1:Svsalpha:decomp:basic}
	|S(t)-\alpha_0|
	\le 
	|S(t)-S_1(t)|+|S_1(t)-\alpha_0|
	\end{equation}
	obtained from Sections \ref{subsec:ind1:1storder error}, \ref{subsec:ind1:S1vsalpha}.
\end{enumerate}  
 In each of these subsections, we add brief explanations on how the result is generalized to the case of induction base, Theorem \ref{thm:induction:base:main}. The only major difference compared to the proof of Theorem \ref{thm:reg:conti:main} is the availability of the assumption $\{\tau(2)>\acute{t}_0 \}$. Although we can only be interested in $t \ge t_0^+$ when establishing Theorem \ref{thm:reg:conti:main}, we need to consider all $t\ge t_0$ for Theorem \ref{thm:induction:base:main}. This is the place where Theorem \ref{thm:branching:ib} comes in to play.

As a consequence of the arguments used in Theorem \ref{thm:reg:conti:main}, we establish
\begin{itemize}
	\item [(4)]  Theorem \ref{thm:induction:base:main} in Section \ref{subsec:ind1:reg of fixed};
	
	\item [(5)] Proposition \ref{prop:ind1:growth} in Section \ref{subsec:ind1:growth}.
\end{itemize}
For convenience, we restate Theorem \ref{thm:induction:base:main} to shape it into a more convenient  form to work with. For its proof, we only discuss the case of $\overline{S}(t)$, since the corresponding result for $\underline{S}(t)$ follows analogously.

\begin{prop}\label{prop:ind1:base:main}
	Assume the setting of Theorem \ref{thm:induction:base:main} and the discussions above, recall the definition of $\tau_{\textnormal{B}}=\tau_{\textnormal{B}}(\alpha_0,t_0)$ from \eqref{eq:def:tau:branching:ib} we define
	\begin{equation}
\begin{split}
	\overline{\tau}=\overline{\tau}(\alpha_0,t_0):=\tau_{\textnormal{B}} \wedge \min\left\{ \overline{\tau}_i(\alpha_0,t_0,\kappa=2) \,:\, 1\le i \le 8 \right\}.
\end{split}
	\end{equation}
	(Note that $\overline{\tau}_3(\alpha_0,t_0)$ does not depend on $\kappa$ although we wrote as above for convenience.) 
	Then, 
	\begin{equation}
	\PP \left( \overline{\tau}(\alpha_0,t_0) >\acute{t}_0, \ \Pi_{\overline{S}}(-\infty,t_0]\in \overline{\mathcal{A}}(\alpha_0,t_0) \right) \ge 1-\exp\left(-\beta_0^{1.9}\right).
	\end{equation}
	The same result holds for $\Pi_{\underline{S}}$ and $\underline{\tau}(\alpha_0,t_0).$
\end{prop}

Before moving on, 
 we establish the following lemma which is not only a useful tool in the later section, but also leads to the proof of (3) of Lemma \ref{lem:reg:errorint}.

\begin{lem}\label{lem:ind1:intofSminS1}
	Under the setting of Theorem \ref{thm:reg:conti:main}, we have  that 
	\begin{equation}
	\begin{split}
	\intop_{ t_0}^{\tau(2)} |S(t)-S_1(t)|dt   \le  \alpha_0^{-2\epsilon} .
	\end{split}
	\end{equation}
\end{lem}

\begin{proof}
	From the definition of $\tau_{3}$ (Section \ref{subsubsec:reg:Scontrol}) and Lemma \ref{lem:ind1:pi1int basicbd:basic}, we have
	\begin{equation}
	\intop_{ t_0}^{\tau(2)} |S(t)-S_1(t)|dt   \le \intop_{t_0}^{\tau(2)} \frac{\alpha_0^{1-\epsilon} dt }{\pi_1(t;S)+1}  \le  \alpha_0^{1-\epsilon}\beta_0^2 |\Pi_S[t_0,\tau(2)]| \le  \alpha_0^{-2\epsilon} ,
	\end{equation}
	where we obtained the last inequality from the definition of $\tau_{7}$ (Section \ref{subsubsec:reg:aggsize}). 
\end{proof}

\begin{proof}[Proof of Lemma \ref{lem:reg:errorint}-(3)]
	Recall in Theorem \ref{thm:reg:conti:main} that ${\tau}^+(1/2) \le \tau(2)$ if $\tau(2)>\acute{t}_0$. Thus, combining Theorem \ref{thm:reg:conti:main} and Lemma \ref{lem:ind1:intofSminS1} directly implies (3) of Lemma \ref{lem:reg:errorint}.
\end{proof}

\subsection{The error from the first order approximation}\label{subsec:ind1:1storder error}
As the first step towards establishing Theorem \ref{thm:reg:conti:main}, we begin with studying $\tau_{3}$, the error of the first-order approximation.

\begin{lem}\label{lem:ind1:tau6}
	Under the setting of Theorem \ref{thm:reg:conti:main}, we have
	\begin{equation}
	\PP \left({\tau}^+_6 = \tau(2)\wedge \hat{t}_0,\ \tau(2)> \acute{t}_0 \ | \ \mathcal{A}  \right) \le \exp\left(-\beta_0^2 \right).
	\end{equation}
\end{lem}
\noindent (Note that ${\tau}^+_6 \ge \tau(2)$ by definition.)

\begin{proof}
	Let $t\in[\acute{t}_0, \hat{t}_0]$, and we begin with expressing that
	\begin{equation}\label{eq:ind1:SminS1:decomp}
\begin{split}
	|S(t)-S_1(t;t_0^-,\alpha_0)|\le& |S(t)-S'(t;t_0,\alpha_0)|+ |S'(t;t_0,\alpha_0)-S_1(t;t_0,\alpha_0)| \\
	&+ |S_1(t;t_0^-,\alpha_0)-S_1(t;t_0,\alpha)|,
\end{split}
	\end{equation}
	and note that the first term in the RHS is bounded by $\alpha_0^{100}$ for any $\Pi_S(-\infty,t_0]\in \mathcal{A}$, which is from Proposition \ref{prop:SvsSprime}. Moreover, for any $t\in [\acute{t}_0, \tau(2)\wedge \hat{t}_0],$ Lemma \ref{lem:estimate on K:intro} tells us that
	\begin{equation}\label{eq:ind1:S1timechange}
	|S_1(t;t_0^-,\alpha_0)-S_1(t;t_0,\alpha_0)| = \intop_{t_0^-}^{t_0}  K_{\alpha_0}(t-x)d\Pi_S(x) \le K_{\alpha_0}(\acute{t}_0-t_0) \cdot |\Pi_S[t_0^-,t_0]|\le \alpha_0^{100}.
	\end{equation}

	The rest of the proof follows as a direct consequence of Proposition \ref{prop:fixed perturbed:error}. In fact, by setting $\alpha$, $t^-$,  $\hat{t}$ and $\tau$ in Proposition \ref{prop:fixed perturbed:error} to be 
	\begin{equation}
	\alpha = \alpha_0, \ \ t^- = t_0,   \ \ \hat{h} = \hat{t}_0, \ \ \tau = \tau(2),
	\end{equation}
	we see that the assumptions in the proposition are all satisfied.  Thus, we obtain that 
	\begin{equation}\label{eq:ind1:tau6:1}
\begin{split}
	\PP \left( |S(t)-S_1(t;t_0^-,\alpha_0)| \le 4\alpha_0^{1-\epsilon}\sigma_1\sigma_2(t;S), \ \forall t\in[\acute{t}_0, \hat{t}_0\wedge \tau(2)]  \, | \, \mathcal{A} \right) \\
	\ge 1-4\exp\left(-\alpha_0^{-\frac{\epsilon}{3000}}\right) .
\end{split}
	\end{equation}
	The lower bound on $S(t)-S_1(t)$ is obtained analogously, from the expression 
		\begin{equation}
		\begin{split}
		S(t)-S_1(t;t_0^-,\alpha_0)\ge&  (S'(t;t_0,\alpha_0)-S_1(t;t_0,\alpha_0)) -|S(t)-S'(t;t_0,\alpha_0)| \\
		&- |S_1(t;t_0^-,\alpha_0)-S_1(t;t_0,\alpha)|,
		\end{split}
		\end{equation}
	 and Proposition \ref{prop:fixed perturbed:error}. 
\end{proof}

Note that the corresponding analogue of \eqref{eq:ind1:tau6:1} holds the same for the second-order approximation $S_2(t)=S_2(t;t_0^-,\alpha_0)$ \eqref{eq:def:S2:basic form}, and we record this result in the following corollary.

\begin{cor}\label{cor:ind1:S2error}
	Under the setting of Theorem \ref{thm:reg:conti:main}, there exists a constant $c_\epsilon>0$ such that
	\begin{equation}
	\begin{split}
	\PP \left(|S(t)-S_2(t)|\le 4\alpha_0^{1-\epsilon} \sigma_1\sigma_2\sigma_3(t;S) ,\ \forall t\in[\acute{t}_0, \hat{t}_0\wedge\tau(2)] \,| \,\mathcal{A}  \right)\\
	\ge 1-4\exp\left(-\alpha_0^{-c_\epsilon} \right)	.
	\end{split}
	\end{equation}
\end{cor}

\begin{proof}
	The only difference in the proof compared to the previous lemma is the way we control $|S_2(t;t_0,\alpha)-S_2(t;t_0^-,\alpha)|$. The difference is bounded by
	\begin{equation}
	\frac{1}{(1+2\alpha)^2} \intop_{t_0^-}^{t_0} K_{\alpha_0}(t-x) d\Pi_S(x) + \frac{\alpha}{1+2\alpha} \intop_{t_0^-}^t \intop_{t_0^-}^{t_0\wedge s} J_{t-s,t-u}^{(\alpha_0)} d\widehat{\Pi}_S(u)d\widehat{\Pi}_S(s),
	\end{equation}
	which is also smaller than $\alpha_0^{100}$ for all $t\in [\acute{t}_0,\tau(2)\wedge\hat{t}_0]$ due to the decay properties of $K$ and $J$ (Lemma \ref{lem:bound on deterministic J}).
\end{proof}

Along with Lemma \ref{lem:fixed perturbed:error int:perturbed}, this leads to the proof of (4) of Lemma \ref{lem:reg:errorint}:

\begin{proof}[Proof of Lemma \ref{lem:reg:errorint}-(4)]
	The result is obtained by integrating the bound on $|S(t)-S_2(t)|$ in Corollary \ref{cor:ind1:S2error} using Lemma \ref{lem:fixed perturbed:error int:perturbed}, and relying on the estimate on $\tau(2) $ from Theorem \ref{thm:reg:conti:main}.
\end{proof}

We conclude this subsection by presenting the analogue of Lemma \ref{lem:ind1:tau6} for the induction base.

\begin{cor}\label{cor:ind1:tau6}
	Under the setting of Proposition \ref{prop:ind1:base:main}, we have
	\begin{equation}
	\PP\left( \overline{\tau}_3 = \overline{\tau}\wedge \hat{t}_0, \ \overline{\tau}>t_0 \right) \le \exp\left(-\beta_0^3 \right).
	\end{equation}
\end{cor}

\begin{proof}
	Relying on the same decomposition \eqref{eq:ind1:SminS1:decomp}, we can obtain the same estimate on $\overline{S}(t)-\overline{S}_1(t;t_0^-,\alpha_0)$ from Lemma \ref{lem:reg:base:A} and Corollary \ref{cor:SvsSprime} along with Lemma \ref{lem:ind1:tau6}. 
\end{proof}

\subsection{Connection to the  critical branching process}\label{subsec:ind1:S1vsalpha}

In this subsection, we estimate $|S_1(t)-\alpha_0|$. By understanding the size of this quantity, we will eventually be able to control $|S(t)-\alpha_0|$ in the next subsection. Define the stopping time $\tau_{\textnormal{b}}$ as
\begin{equation}\label{eq:def:taux}
\tau_{\textnormal{b}} = \tau_{\textnormal{b}}(\alpha_0,t_0,t_0^+) := \inf\left\{t\ge t_0^+:( S_1(t)-\alpha_0) \notin \left(-2\alpha_0^{\frac{3}{2}}\beta_0^{6\theta} ,\alpha_0\beta_0^{C_\circ}\sigma_1(t;S) + 2\alpha_0^{\frac{3}{2}}\beta_0^{6\theta} \right) \right\}.
\end{equation}

Our goal is to establish the following Lemma:

\begin{lem}\label{lem:ind1:taux}
	Under the setting of Theorem \ref{thm:reg:conti:main}, we have
	\begin{equation}
	\PP \left( \tau_{\textnormal{b}} <\hat{t}_0,\ \tau(2) >\acute{t}_0  \ | \ \mathcal{A} \right) \le \exp\left(-\beta_0^5 \right).
	\end{equation}
\end{lem}

To prove this lemma, we rely on the methods in Section \ref{subsec:branching:mg}. Recall the definition \eqref{eq:def:Qst}, and for $t\ge s \ge t_0$, define
\begin{equation}
\begin{split}
R(s,t)=\mathcal{R}_c(s,t;\Pi_{S}[t_0^{-},s],\alpha_0).
\end{split}
\end{equation}
We note that  $R(t,t)=S_1(t)$ and for $t\ge t_0^+$ that
\begin{equation}
\begin{split}
|R(t_0,t)-\alpha_0'|& \le \intop_{t_0^-}^{t_0} K_{\alpha_0}(t-x) d\Pi_S(x) + \intop_{t_0^-}^{t_0} \intop_t^{\infty} K^*_{\alpha_0}\cdot K_{\alpha_0}(u-x) du d\Pi_S(x)\\
&\quad + \intop_{t_0^-}^{t_0} \intop_{t_0}^t |K^*_{\alpha_0}-K^*_{\alpha_0}(t-u) | K_{\alpha_0}(u-x) du d\Pi_S(x)\\
&\le  \alpha_0^{100},
\end{split}
\end{equation}
for all $t\in[\acute{t}_0, \tau(2)\wedge\hat{t}_0]$, similarly as \eqref{eq:branching:Rlbd:med3}. 
Then, \eqref{eq:integralform:branching} gives
\begin{equation}\label{eq:integralform:branching:ind1}
\begin{split}
S_1(t) - \alpha_0' + O(\alpha_0^{100}) &=  R(t,t)-R(t_0,t) \\
&= \intop_{t_0}^t K^*_{\alpha_0}(t-x) \left\{ d\widetilde{ \Pi}_S(x) + (S(x)-S_1(x))dx\right\}  ,
\end{split}
\end{equation}
where we wrote $d\widetilde{ \Pi}_S(x):= d\Pi_S(x) - S(x)dx$. This decomposes $S_1(t)-\alpha_0'$ into a martingale part and a drift part.

\begin{proof}[Proof of Lemma \ref{lem:ind1:taux}]
	We begin with estimating the martingale part of $S_1(t)-\alpha_0'$ in \eqref{eq:integralform:branching:ind1}. We apply Corollary \ref{lem:concentrationofint:continuity} in the following setting:
	\begin{itemize}
		\item Set $\tau = \tau(2)$, $f_t(x) = K^*_{\alpha_0}(t-x)$, $D=1$ (such a choice of $D$ is justified by Corollary \ref{cor:bound on K'}).
		
		\item From the definition of $\tau_{1}$ and the bound in Lemma  \ref{lem:estimat for K tilde:intro}, we can set
		\begin{equation}\label{eq:ind1:S1minusalpha QV}
		M:=\alpha_0^3\beta_0^{12\theta} \ge \intop_{t_0}^{t\wedge \tau(2)} (K^*_{\alpha_0}(t-x))^2 S(x)dx.
		\end{equation}
		
		\item $\Delta = \alpha_0^{-1}$, $A = C\alpha_0^{\frac{3}{2}}$, $\eta = \alpha_0\beta_0^{C_\circ}$, $N= \beta_0^5$ and $\delta = \alpha_0^{10}$ (see definitions of $\tau_{1}$ and $\tau_{7}$).
	\end{itemize}
	Then Corollary \ref{lem:concentrationofint:continuity} gives
	\begin{equation}
\begin{split}
	\PP \left( \left.\intop_{t_0}^{t\wedge \tau(2)} K_{\alpha_0}^*(t-x) d\widetilde{\Pi}_S(x ) \le \alpha_0\beta_0^{5}\sigma_1(t;S) +\alpha_0^{\frac{3}{2}}\beta_0^{6\theta},  \, \forall t\in[t_0, \hat{t}_0 ]  \right| \, \mathcal{A}  \right) \ge 1- e^{-\beta_0^{C_\circ}}; \\
	\PP \left( \left.\intop_{t_0}^{t\wedge \tau(2)} K_{\alpha_0}^*(t-x) d\widetilde{\Pi}_S(x ) \ge -\alpha_0^{\frac{3}{2}}\beta_0^{6\theta},  \, \forall t\in[t_0, \hat{t}_0 ]  \right| \, \mathcal{A}  \right) \ge 1- e^{-\beta_0^{C_\circ}}.
\end{split}
	\end{equation}
	
	To control the drift term of \eqref{eq:integralform:branching:ind1}, we use the bound given by $\tau_{3}$. Express that
	\begin{equation}
\begin{split}
	\intop_{t_0}^{t\wedge \tau(2)} K^*_{\alpha_0}(t-x) |S(x)-S_1(x)| dx \le \intop_{t_0}^{t\wedge \tau(2)} C\left(\frac{\alpha_0}{\sqrt{t-x}} \vee \alpha_0^2 \right) \frac{\alpha_0^{1-\epsilon}dx }{\pi_1(x;S)+1}
	\le \alpha_0^{2-2\epsilon},
\end{split}
	\end{equation}
	where the last inequality is from Lemma \ref{lem:ind1:pi1int basicbd}, with parameters $\Delta_0 = \alpha_0^{-1}, \Delta_1= \alpha_0^{-1}\beta_0^{C_\circ}, K=\alpha_0^{-1}\beta_0^{11\theta}$ and $N_0 = \beta_0^{5}$. Thus, the conclusion follows by combining the above two bounds, along with the condition $|\alpha_0'-\alpha_0|\le \alpha_0^{\frac{3}{2}}\beta_0^{6\theta}$ from $\mathcal{A}_3(\alpha_0,t_0)$.
\end{proof}

We record  a direct consequence of our analysis, for a future usage in Section \ref{subsec:increment:double int}.
\begin{cor}\label{cor:ind1:taux:fixedrate}
	Under the setting of Theorem \ref{thm:reg:conti:main}, define the stopping time $\tilde{\tau}_{b}$ as
	\begin{equation}
	\tilde{\tau}_{\textnormal{b}}:= \inf\left\{ t\ge t_0^+: \left|\intop_{t_0^+}^t K^*_{\alpha_0}(t-x) d\widetilde{\Pi}_{\alpha_0}(x) \right| \ge \alpha_0\beta_0^{C_\circ} \sigma_1(t; \alpha_0) +2\alpha_0^{\frac{3}{2}}\beta_0^{6\theta}  \right\}.
	\end{equation}
	Then, we have $\PP(\tilde{\tau}_{\textnormal{b}}<\hat{t}_0,\ \tau(2)>\acute{t}_0 \,| \, \mathcal{A} ) \le \exp(-\beta_0^5) .$
\end{cor}

Note that the exponent $6\theta$ of the error term $\alpha_0^{\frac{3}{2}}\beta_0^{6\theta}$ is related with the length of the interval $[\acute{t}_0, \hat{t}_0]$, as seen in \eqref{eq:ind1:S1minusalpha QV}. If we are interested in a shorter interval, we can obtain the following stronger bound, which will be essential in the analysis of $\tau_{3}^\sharp$ in Section \ref{subsec:ind2:bootstrappedJ}. Its proof is omitted since it is identical to that of Lemma \ref{lem:ind1:taux}.

\begin{cor}\label{cor:ind1:tauxprime}
	Under the setting of Theorem \ref{thm:reg:conti:main}, define the stopping time
	\begin{equation}\label{eq:def:tauxprime}
	\tau_{\textnormal{b}}' = \tau_{\textnormal{b}}'(\alpha_0,t_0,t_0^+) := \inf\left\{t\ge t_0^+:| S_1(t)-\alpha_0'| \le \alpha_0\beta_0^{C_\circ}\sigma_1(t;S) + \alpha_0^{\frac{3}{2}}\beta_0^{\theta}  \right\},
	\end{equation}
    with $\alpha_0' = \mathcal{L}(t_0;\Pi_S[t_0^-,t_0], \alpha_0)$ as before. 	Then, we have
    \begin{equation}
    \PP \left( \tau_{\textnormal{b}}' \le \acute{t}_0, \ \tau(2)> \acute{t}_0 \ | \ \mathcal{A} \right) \le \exp\left(-\beta_0^5 \right).
    \end{equation}
\end{cor} 
We stress that  the dependence on $t_0^-$ in the definition \eqref{eq:def:tauxprime} comes from $\alpha_0'$ and $S_1(t)=S_1(t;t_0^-,\alpha_0)$. We include $t_0^-$ in the notation (unlike the case of $\tau_{\textnormal{b}}$) to prevent confusion when it is used later in Section \ref{subsec:ind2:bootstrappedJ}.

We conclude this subsection by deriving an analogue for the process $\overline{S}(t)$ (and $\underline{S}(t)$). To be more concrete, we set $\overline{S}_1(t):= \intop_{t_0^-}^t K_{\alpha_0}(t-x) d\Pi_{\overline{S}}(x)$, and define the processes ${R}_0^+(t)$ and $\hat{R}_0^+(t)$ as in \eqref{eq:def:R:ib:aux} and \eqref{eq:def:R0plushat}, replacing $\alpha$ with $\alpha_0$. Define $\overline{S}^+(t):= \overline{S}(t)\vee \hat{R}_0^+(t)$, and let
\begin{equation}\label{eq:def:taux:init}
\overline{\tau}_{{\textnormal{b}}} = \overline{\tau}_{{\textnormal{b}}}(\alpha_0,t_0) := \inf\left\{t\ge t_0:( \overline{S}_1(t)-\alpha_0) \notin \left(-2\alpha_0^{\frac{3}{2}}\beta_0^{6\theta} ,\alpha_0\beta_0^{5\theta}\sigma_1(t;\overline{S}^+) + 2\alpha_0^{\frac{3}{2}}\beta_0^{6\theta} \right) \right\}.
\end{equation}

The major difference in the definition is that the stopping time reads the process starting from $t_0$, not from $t_0^+$, where we entail Theorem \ref{thm:branching:ib}. 
Then, the following corollary holds true:
\begin{cor}
	Under the setting of Proposition \ref{prop:ind1:base:main}, we have
	\begin{equation}
	\PP \left( \overline{\tau}_{\textnormal{b}} = \overline{\tau} \wedge \hat{t}_0, \ \overline{\tau}>t_0 \right) \le \exp\left(-\beta_0^4 \right).
	\end{equation}
\end{cor}

\begin{proof}
	The proof is analogous as Lemma \ref{lem:ind1:taux}, using the decomposition \eqref{eq:integralform:branching:ind1} corresponding to $\overline{S}_1(t)$. This gives
	\begin{equation}
	\overline{S}_1(t)-\overline{R}(t_0,t) \in \left(-\alpha_0^{\frac{3}{2}}\beta_0^{6\theta} , \, \alpha_0\beta_0^5 \sigma_1(t;\overline{S}) + \alpha_0^{\frac{3}{2}}\beta_0^{6\theta} \right)
	\end{equation}
	for all $t\in [t_0, \overline{\tau}(2)\wedge \hat{t}_0]$, which holds with high probability.
	
	The only difference is that we are not guaranteed with the bound on $|\overline{R}(t_0,t)-\alpha_0|$ as before, since we need to cover $t\ge t_0$, not just $t\ge t_0^+$.
		To study $|\overline{R}(t_0,t)-\alpha_0|$,  consider the process
	 ${R}(t) = \mathcal{R}_b(t;\Pi_{\alpha_0}[t_0^-,t_0],\alpha_0)$ \eqref{eq:def:Rb}. Then, $\tau_{\textnormal{B}1}$ from Theorem \ref{thm:branching:ib} gives the conclusion, since $\overline{R}(t_0,t)$ is an averaged version of $R(t)$. 
\end{proof}

\subsection{Proximity of the speed from the base rate}\label{subsec:ind1:Svsalpha}

The goal of this subsection is to control all  stopping times introduced in Sections \ref{subsubsec:reg:Scontrol} and \ref{subsubsec:reg:aggsize}, except $\tau_{3}$. Furthermore, we also establish Proposition \ref{prop:reg:lbd of speed}. To this end, we exploit the bound on $|S(t)-\alpha_0|$ based on \eqref{eq:ind1:Svsalpha:decomp:basic}.

Let $F(\alpha_0,t)$ be the following random function which is essentially the sum of error bounds from  $\tau_{3}$ (Section \ref{subsubsec:reg:Scontrol}), and $\tau_{\textnormal{b}}$ \eqref{eq:def:taux}:
\begin{equation}
\begin{split}
F(\alpha_0,t):= 4\alpha_0^{1-\epsilon}\sigma_1\sigma_2(t;S) + \alpha_0\beta_0^{C_\circ+1} \sigma_1(t;S) + 3\alpha_0^{3/2} \beta_0^{6\theta}.
\end{split}
\end{equation}

Define
\begin{equation}\label{eq:def:tau12}
{\tau}_{9}^+:= \inf\{t \ge t_0^+: |S(t)-\alpha_0| \ge F(\alpha_0,t) \}.
\end{equation}

\begin{cor}\label{cor:ind1:tau12}
	Under the setting of Theorem \ref{thm:reg:conti:main}, we have
	\begin{equation}
	\PP \left({\tau}^+_{9} \le \tau(2),\ \tau(2)> \acute{t}_0 \ | \ \mathcal{A}  \right) \le 3\exp\left(-\beta_0^2 \right).
	\end{equation}
\end{cor}

\begin{proof}
	This is an immediate consequence of Lemmas~\ref{lem:ind1:tau6} and  \ref{lem:ind1:taux}.
\end{proof}

Recall the definition of ${\tau}^+_i(1/2)$ \eqref{eq:def:ind1:tauacute}. In the following subsections, we show that each stopping time ${\tau}^+_i$ is unlikely to be smaller than or equal to \begin{equation}\label{eq:def:tau2tilde}
\tilde{ \tau}(2):= \tau(2)\wedge {\tau}^+_{9}.
\end{equation}\begin{itemize}
	\item In Section \ref{subsubsec:ind1:squareint}, we control ${\tau}^+_4(1/2), {\tau}^+_5(1/2)$, the stopping times related with the square integral of $|S(t)-\alpha_0|$.
	
	\item In Section \ref{subsubsec:ind1:aggsize}, we study ${\tau}^+_6(1/2),{\tau}^+_{7}(1/2),{\tau}^+_{8}(1/2),$ which describe the size of the aggregate.
	
	\item In Section \ref{subsubsec:ind1:Sbd}, we work with $\tau_{1}^+(1/2), \tau_{2}^+(1/2)$.
	
	\item In Section \ref{subsubsec:ind1:base}, we show how the analysis from Sections \ref{subsubsec:ind1:squareint}--\ref{subsubsec:ind1:Sbd} can be done in the setting of Proposition \ref{prop:ind1:base:main}.
	
	\item In Section \ref{subsubsec:ind1:Sminusalpha conseq}, we establish several consequences of our analysis which will be useful later in Section \ref{sec:double int}.
\end{itemize}

\subsubsection{The square integrals}\label{subsubsec:ind1:squareint}
The main goal of this subsection is to deduce the following lemma.

\begin{lem}\label{lem:ind1:tau78}
	Under the setting of Theorem \ref{thm:reg:conti:main}, let $\tilde{\tau}(2)$ be as \eqref{eq:def:tau2tilde}.  We have
	\begin{equation}
	\PP \left({\tau}^+_{4}(1/2) \wedge {\tau}^+_5(1/2) \le \tilde{\tau}(2),\ \tau(2)> t_0^+ \ | \ \mathcal{A}  \right) \le \exp\left(-\beta_0^3 \right).
	\end{equation}
\end{lem}

\begin{proof}
	The proof is based on estimating the integrals in the definition of $\tau_{4}^+,\tau_{5}^+$ directly from the bound we have from ${\tau}^+_{9}$. We first observe that
	\begin{equation}\label{eq:ind1:squareint:splitinto4}
	\begin{split}
	\frac{1}{9}\intop_{t_0^+}^{ \tilde{ {\tau}}(2)} (S(t)-\alpha_0)^2 dt 
	\le
	\intop_{ t_0^+}^{ \tilde{ {\tau}}(2)}
	\left\{\frac{\alpha_0^{2-2\epsilon}}{(\pi_1(t;S)+1)(\pi_2(t;S)+1)} + \frac{\alpha_0^2\beta_0^{2C_\circ+2}}{\pi_1(t;S)+1}+ \alpha_0^3\beta_0^{12\theta}
	\right\} \ dt.
	\end{split}
	\end{equation}
	We control the integral of each term in the RHS separately. 
	
	The integral over the constant $\alpha_0^3\beta_0^{12\theta}$ turns out to give the leading order and satisfies 
	\begin{equation}\label{eq:ind1:tau78:term4}
	\intop_{t_0^+}^{\hat{t}_0} \alpha_0^3 \beta_0^{12\theta} dt \le \alpha_0 \beta_0^{23\theta}.
	\end{equation}
	
	To control the first term in the RHS of  \eqref{eq:ind1:squareint:splitinto4}, we write
	\begin{equation}
	\Pi_S [t_0^+, \tilde{ {\tau}}(2)] = \{p_1 < p_2 < \ldots < p_N \}.
	\end{equation} 
	Setting $p_0=p_{-1}=t_0^+, p_{N+1}= \tilde{ {\tau}}(2)$, we express
	\begin{equation}
	\begin{split}
	\intop_{{t}_0^+}^{ \tilde{ {\tau}}(2)}
	\frac{\alpha_0^{2-2\epsilon}dt}{(\pi_1(t;S)+1)(\pi_2(t;S)+1)}&
	\le 2\intop_{{t}_0^+}^{ \tilde{ {\tau}}(2)}
	\frac{\alpha_0^{2-2\epsilon}dt}{(\pi_1(t;S)+1)(\pi_2(t;S)+2)}
	\\ 
	&\le 2\sum_{i=0}^{N} \int_{0}^{p_{i+1}-p_{i}} \frac{dx}{(x+1)(x+2+p_{i}-p_{i-1})}.
	\end{split}
	\end{equation}
	We can bound the summand at $i=0$ by $1$, and the rest can be written as
	\begin{equation}
	\begin{split}
	\sum_{i=1}^{N}& \int_{0}^{p_{i+1}-p_{i}} \frac{dx}{(x+1)(x+2+p_{i}-p_{i-1})}\\
	&=
	\sum_{i=1}^{N} \int_{0}^{p_{i+1}-p_{i}} 
	\frac{dx}{p_{i}-p_{i-1}+1} \left(\frac{1}{x+1}- \frac{1}{x+2+p_i-p_{i-1}} \right)\le  \sum_{i=1}^N \frac{\beta_0^2}{p_i-p_{i-1}+1} ,
	\end{split}
	\end{equation}
	where we used $\log(\hat{t}_0-{t}_0^+)\le \beta_0^2$ to obtain the last inequality. The following statement gives a control on the quantity in the RHS, and its proof is given after finishing the proof of Lemma \ref{lem:ind1:tau78}.
	
	\begin{claim}\label{claim:ind1:pointsum:S}
		Under the above setting,
		\begin{equation}
		\PP \left( \left.
		\sum_{i=1}^N \frac{1}{p_i-p_{i-1}+1} \ge \alpha_0^{-\frac{1}{2}-\epsilon} \ \right|
		\ \mathcal{A} \right) \le \exp\left(-\beta_0^{5} \right).
		\end{equation}
	\end{claim}
	Using the claim, we obtain
	\begin{equation}\label{eq:ind1:tau78:term1}
	\PP\left( \left.	\intop_{ {t}_0^+}^{ \tilde{ {\tau}}(2)}
	\frac{\alpha_0^{2-2\epsilon}dt}{(\pi_1(t;S)+1)(\pi_2(t;S)+1)}
	\ge \alpha_0^{\frac{3}{2}-4\epsilon} \ \right| \ \mathcal{A} \right) \le \exp\left(-\beta_0^5 \right).
	\end{equation}
	
	We move on to the second term in the RHS of \eqref{eq:ind1:squareint:splitinto4}. Using Lemma \ref{lem:ind1:pi1int basicbd:basic},
	\begin{equation}\label{eq:ind1:tau78:term3}
	\intop_{ {t}_0^+}^{ \tilde{ {\tau}}(2)}\frac{\alpha_0^2\beta_0^{2C_\circ+2}dt}{\pi_1(t;S)+1} \le \alpha_0^2\beta_0^{2C_\circ+4} \left|\Pi_{ S} [ {t}_0^+, \tilde{ \tau}(2)] \right| \le
	\alpha_0 \beta_0^{11\theta},
	\end{equation}
	Note that we bounded $\log(\hat{t}_0-{t}_0^+)\le \beta_0^2.$ 
	
	Thus, we obtain the conclusion for ${\tau}^+_4(1/2)$ by combining \eqref{eq:ind1:squareint:splitinto4}, \eqref{eq:ind1:tau78:term4}, \eqref{eq:ind1:tau78:term1}, and \eqref{eq:ind1:tau78:term3}, which gives that
	\begin{equation}
	\PP \left(\left. \intop_{ {t}_0^+}^{ \tilde{ {\tau}}(2)} (S(t)-\alpha_0)^2 dt \ge \alpha_0 \beta_0^{24\theta} \ \right| \ \mathcal{A} \right) \le \exp\left(-\beta_0^3 \right).
	\end{equation}
	
	Note that the result for ${\tau}^+_5(1/2)$ follows as well, by writing
	\begin{equation}
	\intop_{{t}_0^+}^{ \tilde{ {\tau}}(2)} (S(t)-\alpha_0)^2 S(t)dt \le
	2\alpha_0 \beta_0^{C_\circ} \intop_{{t}_0^+}^{ \tilde{ {\tau}}(2)} (S(t)-\alpha_0)^2 dt ,
	\end{equation}
	from the definition of $\tau_{1}(\kappa=2)$ (Section \ref{subsubsec:reg:Scontrol}). 
\end{proof}

\begin{proof}[Proof of Claim \ref{claim:ind1:pointsum:S}]
	Let $\hat{ \alpha}_0:= 2\alpha_0 \beta_0^{C_\circ }$, and let 
	\begin{equation}
	\Pi_{ \hat{ \alpha}_0}[{t}_0^+,\hat{t}_0] = \{p_1'<p_2'<\ldots <p_{N'}' \}.
	\end{equation}
	Due to the definition of $\tau_{1}(\kappa = 2)$ (Section \ref{subsubsec:reg:Scontrol}), it suffices to establish the main inequality in terms of
	\begin{equation}
	\sum_{i=1}^N \frac{1}{p_i'-p_{i-1}'+1},
	\end{equation}
	where we set $p_0'={t}_0^+$ as before.
	
	Our idea is to count the number of neighboring pairs of points with distance less than $\alpha_0^{-1/2}$. To this end, we make the following simple observation:
	\begin{itemize}
		\item If $p'_{i}-p'_{i-1} \le \alpha_0^{-1/2}$, then there exists $k\in \mathbb{Z}$ such that $p'_i, p'_{i-1}\in {t}_0^+ + \alpha_0^{-1/2}\big[k,k+2\big]$.
	\end{itemize}
	Note that for each $k\in \mathbb{Z}$,
	\begin{equation}
	\PP\left(\left|\Pi_{ \hat{ \alpha} _0} \left[{t}_0^++k\alpha_0^{-\frac{1}{2}},\, {t}_0^++(k+2)\alpha_0^{-\frac{1}{2}}\right]\right| \ge 2 \right) \le C\alpha_0 \beta_0^{2C_\circ},
	\end{equation}
	where $C>0$ is an absolute constant. Further, for all even $k$, the above events are independent (and same for odd $k$). Thus, we apply a Chernoff bound for even $k$ and odd $k$ separately, and then use a union bound over the two to deduce that
	\begin{equation}\label{eq:ind1:tau78:claim:numberofk}
	\begin{split}
	\PP\bigg( \sharp \bigg\{ k\in \mathbb{Z}, \, k\le \alpha_0^{-3/2}\beta_0^{10\theta} :  \Big|\Pi_{ \hat{ \alpha} _0} \big[{t}_0^++k\alpha_0^{-\frac{1}{2}},\, {t}_0^++(k+2)\alpha_0^{-\frac{1}{2}}\big]\Big| \ge 2 \bigg\}\,\ge  \alpha_0^{-\frac{1}{2}}\beta_0^{11\theta} \bigg)\\
	\le \exp\left(-\beta_0^\theta \right).
	\end{split}
	\end{equation}
	Furthermore, in each interval $[{t}_0^++k'\alpha_0^{-1}, {t}_0^++(k'+1)\alpha_0^{-1}]$, there can be more than $2\beta_0^{C_\circ}$ particles with probability at most $\exp(-\beta_0^{C_\circ} )$. Thus, we take a union bound over $0\le k' \le \alpha_0^{-1}\beta_0^{10\theta}$ and obtain that
	\begin{equation}
	\PP \left(\left|\Pi_{ \hat{ \alpha}_0} [t-\alpha_0^{-1},t] \right| \le 4\beta_0^{C_\circ}, \ \forall t\in [{t}_0^+,\hat{t}_0] \right) \ge 1- \exp\left(-\beta_0^{C_\circ}/2 \right).
	\end{equation}
	This implies that at each interval ${t}_0^++ [k\alpha_0^{-1/2}, (k+2)\alpha_0^{-1/2}]$ with at least two points, there cannot be more than $4\beta_0^{C_\circ}$ points. Also, we have that there are at most $\alpha_0^{-1} \beta_0^{11\theta}$ points in the entire interval.
	
	Combining the above information, we have
	\begin{equation}
	\sum_{i=1}^{N'}\frac{1}{p_i'-p_{i-1}'+1} \le \alpha_0^{-\frac{1}{2}}\beta_0^{11\theta} \cdot 4\beta_0^{C_\circ} + \alpha_0^{\frac{1}{2}} \cdot \alpha_0^{-1}\beta_0^{11\theta} \le \alpha_0^{-\frac{1}{2}-\epsilon},
	\end{equation}
	with probability at least $1- \exp(-\beta_0^5)$.
\end{proof}

\subsubsection{The size of the aggregate}\label{subsubsec:ind1:aggsize}

The goal of this subsection is to study ${\tau}^+_6(1/2), {\tau}^+_{7}(1/2), {\tau}^+_{8}(1/2)$ (Section \ref{subsubsec:reg:aggsize}).

\begin{lem}\label{lem:ind1:tau9 10 11}
	Under the setting of Theorem \ref{thm:reg:conti:main}, let $\tilde{\tau}(2)$ be as \eqref{eq:def:tau2tilde}, and set $\acute{ \tau}^+(\kappa)=\min\{{\tau}^+_i(\kappa):i=6,7,8 \}$. Then, we have
	\begin{equation}
	\PP \left(\acute{\tau}^+(1/2) \le \tilde{\tau}(2),\ \tau(2)> \acute{t}_0 \ | \ \mathcal{A}  \right) \le \exp\left(-\beta_0^2 \right).
	\end{equation}
\end{lem}

A key quantity in establishing the lemma is the following. For each $\Delta>\alpha_0^{-1}$, we define
\begin{equation}\label{eq:def:ind1:tau13}
{\tau}^+_{10}(\Delta):= \inf \left\{
t\ge t_0^+  : \left|\int_{(t- \Delta)\vee t_0^+}^t |S(s)-\alpha_0|ds \right| \ge \alpha_0^{\frac{3}{2}-7\epsilon} \Delta  
\right\}.
\end{equation}
In fact, if we have a good control on ${\tau}^+_{10}$, then we can estimate the number of points in an interval of size $\Delta$ using Corollary \ref{cor:concentration:numberofpts each interval}. Based on Corollary \ref{cor:ind1:tau12}, we begin with  showing the following result.

\begin{lem}\label{lem:ind1:tau13}
	Under the setting of Theorem \ref{thm:reg:conti:main}, let $\tilde{\tau}(2)$ be as \eqref{eq:def:tau2tilde}. Then, we have for all $\Delta \geq \alpha_0^{-1}$ that 
	\begin{equation}
	\PP \left({\tau}^+_{10}(\Delta) \le \tilde{\tau}(2),\ \tau(2)> \acute{t}_0 \ | \ \mathcal{A}  \right) \le \exp\left(-\beta_0^3 \right).
	\end{equation}
\end{lem}

\begin{proof}
	Based on the definition of $\tilde{\tau}(2)$ and $\tau_{9}$ (\eqref{eq:def:tau12}, \eqref{eq:def:tau2tilde}), we have for $t\in[{t}_0^++\Delta,\tilde{ {\tau}}(2)]$ that
	\begin{equation}\label{eq:ind1:Sminualpha:splitinto4}
	\begin{split}
	\intop_{ t-\Delta}^{t} |S(t)-\alpha_0| dt 
	\le
	\intop_{t-\Delta}^{ t}
	\left\{\frac{3\alpha_0^{1-\epsilon}}{\pi_1(t;S)+1}+ \frac{\alpha_0\beta_0^{5}}{\sqrt{\pi_1(t;S)+1}}+ 3\alpha_0^{3/2}\beta_0^{6\theta}
	\right\} \ dt.
	\end{split}
	\end{equation}
	Using Lemma \ref{lem:ind1:pi1int basicbd:basic}, the RHS is upper bounded by
	\begin{equation}\label{eq:ind1:Sminualpha:splitbd}
	\begin{split}
	\alpha_0^{1-\epsilon} \beta_0^{2} \left|\Pi_{ S}[t-\Delta,t] \right| +
	2\alpha_0 \beta_0^5 \sqrt{\Delta\left|\Pi_{ S}[t-\Delta,t] \right| } + 3\Delta \alpha_0^{3/2} \beta_0^{6\theta}.
	\end{split}
	\end{equation}
	Among these terms, we control $\left|\Pi_{ S}[t-\Delta,t] \right|$ from Corollary \ref{cor:concentration:numberofpts each interval}. The definition of $\tau_{1}$  gives that
	\begin{equation}
	\intop_{ (t-\Delta)\wedge \tilde{ \tau}(2)}^{ t\wedge \tilde{ \tau}(2)} S(x)dx \le 2\Delta \alpha_0 \beta_0^{C_\circ},
	\end{equation}
	and the RHS is at least $\beta_0^{C_\circ}\ge 1$ since $\Delta\ge \alpha_0^{-1}$. Therefore, Corollary \ref{cor:concentration:numberofpts each interval} tells us that
	\begin{equation}\label{eq:ind1:tau13:Sptsbd}
	\PP \left( \left|\Pi_{ S}[(t-\Delta)\wedge \tilde{ \tau}(2),\,t\wedge\tilde{ \tau}(2)] \right| \le 5\Delta \alpha_0 \beta_0^{C_\circ}, \ \forall t\in[{t}_0^++\Delta, \hat{t}_0] \right) \ge 1- \exp\left(-\beta_0^{C_\circ/3} \right).
	\end{equation}
	When the above event holds, \eqref{eq:ind1:Sminualpha:splitbd} is upper bounded by
	\begin{equation}
	\alpha_0^{2-2\epsilon} \Delta + \alpha_0^{1-\epsilon}\sqrt{\Delta} + \alpha_0^{\frac{3}{2}-\epsilon} \Delta \le \alpha_0^{\frac{3}{2}-\epsilon}\Delta, 
	\end{equation}
	where the inequality followed from $\alpha_0^{1/2} \Delta \ge \sqrt{\Delta}$. Thus, combining the error probabilities in the three events, we deduce conclusion. 
\end{proof}

We can conduct an analogous but simpler investigation on $|S(t)-\alpha_0|$ and $|S_1(t)-\alpha_0|$, leading to (1) and (2) of Lemma \ref{lem:reg:errorint}. We restate them in the following corollary without spelling out the details of the proof due to similarity.

\begin{cor}\label{cor:ind1:Sint S1int}
	Under the setting of Theorem \ref{thm:reg:conti:main}, we have that
	\begin{equation}
	\begin{split}
	\PP\left(\left.\intop_{{t}_0^+}^t \big| S(s)-\alpha_0 \big| ds \le \alpha_0^{-\frac{1}{2}-\epsilon}, \ \forall t\in [{t}_0^+,\tau(2)] \, \right| \,\mathcal{A}  \right) \le \exp\left(-\beta_0^{C_\circ/3} \right);\\
		\PP\left(\left.\intop_{{t}_0^+}^t \big| S_1(s)-\alpha_0\big| ds  \le \alpha_0^{-\frac{1}{2}-\epsilon}, \ \forall t\in [{t}_0^+,\tau(2)] \, \right| \,\mathcal{A}  \right) \le \exp\left(-\beta_0^{C_\circ/3} \right).
	\end{split}
	\end{equation}
\end{cor}

We conclude the subsection by establishing Lemma \ref{lem:ind1:tau9 10 11}.

\begin{proof}[Proof of Lemma \ref{lem:ind1:tau9 10 11}]
	Set $\tilde{ \tau}'(\Delta):= \tilde{ \tau}(2)\wedge {\tau}^+_{10}(\Delta) = { \tau}(2)\wedge {\tau}^+_{9}\wedge {\tau}^+_{10}(\Delta)$.
	Having Lemma \ref{lem:ind1:tau13} on hand, Lemma \ref{lem:ind1:tau9 10 11} follows from applying Corollary \ref{cor:concentration:numberofpts each interval} for different values of $\Delta$. Indeed, we can choose $\Delta=\alpha_0^{-1}$ for ${\tau}^+_{7}(1/2)$ and $\Delta=\alpha_0^{-1}\beta_0^{C_\circ}$ for ${\tau}^+_{8}(1/2)$ to see that
	\begin{equation}\label{eq:ind1:tau1011bd}
	\begin{split}
	&\PP \left( {\tau}^+_{7}(1/2) \le \tilde{ \tau}'(\alpha_0^{-1}), \ \tau(2)> \acute{t}_0 \ | \ \mathcal{A}  \right) \le \exp\left(-\beta_0^3 \right);\\
	&\PP \left( {\tau}^+_{8}(1/2) \le \tilde{ \tau}'(\alpha_0^{-1}\beta_0^{C_\circ}), \ \tau(2)>\acute{t}_0 \ | \ \mathcal{A}  \right) \le \exp\left(-\beta_0^{C_\circ/3} \right).
	\end{split}
	\end{equation}
	
	The remaining task is  to study ${\tau}^+_6(1/2)$. On the event $\mathcal{A}_2 \supset \mathcal{A}$ (Section \ref{subsubsec:reg:history}), we have that
	\begin{equation}\label{eq:ind1:tau9primevstau9}
	\acute{\tau}^+_9:= \inf \left\{t\ge {t}_0^+: X_t-X_{t_0} \le \frac{\alpha_0}{50}(t-t_0)  \right\} \le {\tau}^+_9(1/2).
	\end{equation}
	Letting $\Delta_0=\alpha_0^{-3/2}$, we have for all $t\in[{t}_0^+,\hat{t}_0]$, $t\le \tilde{\tau}'(\Delta_0)$ that
	\begin{equation}
	\intop_{t-\Delta_0 }^{t} S(t)dt \in \left[  \frac{\alpha_0\Delta_0}{2}, 2\alpha_0\Delta_0\right]. 
	\end{equation}
	Choose $M=2\Delta_0$ to apply Corollary \ref{cor:concentration of integral}, and union bound over all $t\in \mathcal{T}_{\Delta_0}$, where  $\mathcal{T}_{ \Delta_0}:= \{t\in({t}_0^+,\hat{t}_0]: t={t}_0^++k\Delta_0,\, k\in\mathbb{Z}$. This gives that
	\begin{equation}
	\PP \left(\left|\Pi_{ S}[t-\Delta_0,\, t] \right| \ge \frac{\alpha_0\Delta_0}{3},\,\forall t\in \mathcal{T}_{\Delta_0} , \, t\le \tilde{\tau}'(\Delta_0) \right) \ge 1- \exp\left(-\alpha_0^{-\epsilon} \right).
	\end{equation}
	We can also see that if this event holds, then $\acute{\tau}^+_9\ge \tilde{ {\tau}}'.$ Therefore, combining this with \eqref{eq:ind1:tau9primevstau9},
	\begin{equation}\label{eq:ind1:tau9bd}
	\PP \left( \tau_{9}^+(1/2) \le \tilde{ \tau}'(\alpha_0^{-3/2}), \ \tau(2)> \acute{t}_0 \ | \ \mathcal{A}  \right) \le\exp\left(-\alpha_0^{-\epsilon} \right)
	\end{equation}
	
	Thus, we obtain conclusion from Lemma \ref{lem:ind1:tau13}, equations \eqref{eq:ind1:tau1011bd} and \eqref{eq:ind1:tau9bd}.
\end{proof}

\subsubsection{The magnitude of the speed}\label{subsubsec:ind1:Sbd}

We conclude this section by establishing the estimates on $\tau_{1}$ and $\tau_{2}$, and their corresponding analogue for $\overline{S}(t)$. Further, we conclude the proof Theorem \ref{thm:reg:conti:main} by combining all results obtained throughtout this section. We begin with controlling $\tau_{2}$. 

\begin{lem}\label{lem:ind1:tau5}
	Under the setting of Theorem \ref{thm:reg:conti:main},  write $\tau_{2}^+(\kappa)=\tau_{2}(\alpha_0,t_0^+,\kappa).$ Conditioned on $\mathcal{A}$ and $\{\tau(2)> \acute{t}_0 \}$, we have $\tau_{2}^+(1/2) > \tau(2)$ almost surely.
\end{lem}

\begin{proof}
	Recall the definition of $S_1(t)$ given by \eqref{eq:def:S1:basic form}: 
	\begin{equation}
	S_1(t) = \intop_{ t_0^{-}}^t K_{\alpha_0} (t-x) d\Pi_S(x).
	\end{equation}
	Based on the definitions of $\tau_{7}$ (Section \ref{subsubsec:reg:aggsize}) and $\mathcal{A}_2$ (Section \ref{subsubsec:reg:history}), we express that
	\begin{equation}
	\begin{split}
	\intop_{t_0^{-}}^{t} K_{\alpha_0}(t-x)d\Pi_S(x) &\leq \sum_{k:\,0\le k \alpha^{-1} \le t-t_0^{-}} \big|\Pi_S[t-(k+1)\alpha^{-1}, t-k\alpha_0^{-1} ]\big| K_{\alpha_0}(k\alpha_0^{-1})\\ 
	&\le
	\sum_{k:\, 0\le k \alpha_0^{-1} \le t'-t_0}  \frac{C\alpha_0\beta_0^{C_\circ} }{\sqrt{k\alpha_0^{-1}+1} } e^{-c\alpha_0 k}  + \alpha_0^{-1} \beta_0^{\theta+10} K_{\alpha_0}(t- t_0)
	\leq  \frac{1}{2}\alpha_0\beta_0^{C_\circ+1},
	\end{split}
	\end{equation} 
	where we used Lemma \ref{lem:estimate on K:intro} to obtain the second inequality.
\end{proof}

\begin{cor}\label{cor:ind1:tau5:base}
	Under the setting of Proposition \ref{prop:ind1:base:main}, we have
	\begin{equation}
	\PP \left( \overline{\tau}_2= \overline{\tau}\wedge\hat{t}_0, \ \overline{\tau}>t_0 \right) \le \exp\left(-\beta_0^{3} \right).
	\end{equation}
\end{cor}

\begin{proof}
	The proof is identical to the previous lemma, except that we are not conditioning on ${\mathcal{A}}_2$. Instead, we use Lemma \ref{lem:reg:base:A} to conclude the proof.
\end{proof}

\begin{lem}\label{lem:ind1:tau4}
	Under the setting of Theorem \ref{thm:reg:conti:main},  write $\tau_{1}^+(\kappa)=\tau_{1}(\alpha_0,t_0^+,\kappa).$ Conditioned on $\mathcal{A}$ and $\{\tau(2) >\acute{t}_0 \}$, we have $\tau_{1}(1/2) > \tau(2)$ almost surely.
\end{lem}

We establish Lemma \ref{lem:ind1:tau4} by studying the original formula \eqref{eq:speed:basic expression} of $S(t)$. To this end, we first introduce the following lemma.

\begin{lem}\label{lem:ind1:tau4:aux}
	Let $\alpha >0$ sufficiently small and  $0<x \le \alpha ^{-1}$. Let $W(s)$ be a continuous time simple random walk. We have 
	\begin{equation}
	\mathbb P (\forall s>0 , \ W(s) < x +\alpha s) \le 3 \alpha x
	\end{equation}
\end{lem}

\begin{proof}
	Let $\xi _\alpha $ be the unique solution in $(0,1)$ to the equation 
	\begin{equation}\label{eq:equation for theta}
	\left(\frac{\xi +\xi ^{-1}}{2}-1\right)+\alpha \log \xi =0.
	\end{equation}
	We claim that $\xi _\alpha= 1-2 \alpha +O(\alpha ^2 )$. Indeed, if $\xi =1 -\delta $ then
	\begin{equation}
	\xi ^{-1}=1+\delta +\delta ^2 +O(\delta ^3 ), \quad \log \xi =-\delta +O(\delta ^2 ). 
	\end{equation}
	Substituting these estimates in \eqref{eq:equation for theta} we get the equation $\frac{\delta ^2 }{2}+O(\delta ^3 )-\alpha \delta +O(\alpha \delta ^2 )=0$ and therefore $\delta =2 \alpha +O(\alpha ^2 )$. We have that 
	\begin{equation}
	\mathbb E \xi ^{\alpha s-W(s)}=\xi ^{\alpha s} \sum _{k=1}^{\infty } \frac{e^{-s}s^k}{k!} \left( \frac{\xi +\xi ^{-1}}{2}\right)^k=\exp \left(\alpha s \log \xi +\left(\frac{\xi +\xi ^{-1}}{2}-1\right)s\right).
	\end{equation}
	Thus, by the definition of $\xi _{\alpha }$ the process $M_s:=\xi _\alpha ^{\alpha s-W(s)}$ is a martingale. Define the stopping time
	\begin{equation}
	T_0 := \inf \{s>0 : \ W(s)\ge x +\alpha s \}.
	\end{equation}
	Almost surely we have $M_{t \wedge T_0 } \to \mathds 1 _{T_0 <\infty }\xi ^{\alpha T_0 -W(T_0) } $. On $\{ T_0 <\infty \}$ we have  $-x \le \alpha T_0 -W(T_0) \le -x+1 $ and therfore, by the bounded convergence theorem we have 
	\begin{equation}
	1=\mathbb E M_0=\mathbb E  \left[ \mathds 1 \{ T_0 <\infty \} \xi ^{\alpha T_0 -W(T_0) } \right] \le \theta ^{-x} \mathbb P (T_0 <\infty ).
	\end{equation}
	Thus $\mathbb P (T_0 <\infty ) \ge \xi ^{x}\ge (1-3\alpha)^{x} \ge 1-3 \alpha x$. This finishes the proof of the lemma.
\end{proof}

\begin{proof}[Proof of Lemma \ref{lem:ind1:tau4}]
	Throughout the proof, we condition on $\mathcal{A}$ and $\{\tau(2) > \acute{t}_0 \}$, and let $t\in[{t}_0^+, \tau(2)]$. Because of $\tau_{7} $ and $\mathcal{A}_2$, we have for all $s\in [t_0^{-},\, t]$ that
	\begin{equation}\label{eq:ind1:tau4:linearbd on Y}
	|\Pi_S[t-s,t]| \le 2\beta_0^4 + 2\alpha_0\beta_0^4 s =: Y'(s).
	\end{equation}
	From Lemma \ref{lem:ind1:tau4:aux}, we have
	\begin{equation}
	S(Y') \le 12\alpha_0\beta_0^8, 
	\end{equation}
	where $S(Y')$ is the speed process \eqref{eq:def:speedproc} in terms of $Y'$. Let us define $Y''_t$ as
	\begin{equation}
	Y''_t(s) := \begin{cases}
	|\Pi_S[t-s,t]| & \quad \textnormal{ for } s\le t-t_0^{-};\\
	|\Pi_S[t_0^{-},\,t]| + \alpha_0\beta_0^4(s-t+t_0^{-}) & \quad \textnormal{ for } s \ge t-t_0^{-}.
	\end{cases}
	\end{equation}
	Then, an analogous argument as Proposition \ref{prop:SvsSprime} implies that under the event $\mathcal{A}$,
	\begin{equation}
	|S(t) - S(Y_t'')| \le \alpha_0^{100},
	\end{equation}
	for any $t\ge t_0$. 
	Thus, since $S(Y') \ge S(Y_t'')$, we obtain $S(t) \le 12\alpha_0\beta_0^8+\alpha_0^{100} \le \frac{1}{2} \alpha_0\beta_0^{C_\circ}$ which holds deterministically. This tells us that $\tau_{1}^+(1/2) > \tau(2)$ almost surely.
\end{proof}

It turns out that the proof of  Lemma \ref{lem:ind1:tau4} applies similarly to the case of $\overline{S}(t)$, and we record the result below.

\begin{cor}\label{cor:ind1:tau4:base}
	Under the setting of Proposition \ref{prop:ind1:base:main}, we have
	\begin{equation}
	\PP \left(\overline{\tau}_1=\overline{\tau}\wedge \hat{t}_0  \right) \le \exp\left(-\beta_0^{3}\right).
	\end{equation}
\end{cor}

\begin{proof}
	The proof of Lemma \ref{lem:ind1:tau4} applies analogously  except for the following issues:
	\begin{itemize}
		\item The statement we have does not condition on $\overline{\mathcal{A}}$, in contrast to Lemma \ref{lem:ind1:tau4};
		
		\item Proof of Lemma \ref{lem:ind1:tau4} deals only with $t\ge {t}_0^+$, whereas we need to cover all $t\ge t_0$ in our case.
	\end{itemize}
	The first item is resolved using Lemma \ref{lem:reg:base:A}. For the second one, observe that the reason we wanted $t\ge {t}_0^+$ in the previous proof was to deduce \eqref{eq:ind1:tau4:linearbd on Y}: since we did not have control on $|\Pi_S[t-\alpha_0^{-1},t]|$ for $t\le t_0$, we started looking at ${t}_0^+$ to ensure such a linear bound on $Y_t(s)$.
	
	 In our case, since $\overline{S}(t)=\alpha_0$ on $t\in[t_0^-,t_0)$, we do not have such an issue. In fact, we can obtain \eqref{eq:ind1:tau4:linearbd on Y} for $\overline{S}(t)$, $t\ge t_0$  from \eqref{eq:def:tauf} and Lemma \ref{lem:branching:ib:R0 vs alpha} along with $\overline{\tau}_9$. One advantage of this is that we also have
	 \begin{equation}\label{eq:ind1:tau4:base:t0}
	 \PP \left(\overline{\tau}_1 = t_0 \right) \le \exp\left(-\beta_0^5 \right),
	 \end{equation}
	 which means that we do not need to include the event $\{\overline{\tau} >t_0 \} $ as before. Thus, having this discussion in mind, following the proof of Lemma \ref{lem:ind1:tau4} concludes the proof.
\end{proof}

We conclude this subsection by giving the proof of Theorem \ref{thm:reg:conti:main}.

\begin{proof}[Proof of Theorem \ref{thm:reg:conti:main}]
	The proof follows by combining the lemmas discussed in this section. Namely,
	\begin{itemize}
		\item Control on ${\tau}^+_3$: Lemma \ref{lem:ind1:tau6}.
		
		\item Control on ${\tau}^+_4(1/2), {\tau}^+_5(1/2)$: Corollary \ref{cor:ind1:tau12} and Lemma \ref{lem:ind1:tau78}.
		
		\item Control on ${\tau}^+_6(1/2),{\tau}^+_{7}(1/2),{\tau}^+_{8}(1/2)$: Lemma \ref{lem:ind1:tau9 10 11} (with Corollary \ref{cor:ind1:tau12}). 
		
		\item Control on $\tau_{1}^+(1/2)$ and $\tau_{2}^+(1/2)$: Lemmas \ref{lem:ind1:tau4} and\ref{lem:ind1:tau5}.

	\end{itemize}
	Combining all the aforementioned results gives the conclusion.
\end{proof}

\subsubsection{Proximity to $\alpha$ of the induction base }\label{subsubsec:ind1:base}

In this section, we explain how to deduce the corresponding analogues of Lemmas \ref{lem:ind1:tau78} and \ref{lem:ind1:tau9 10 11} for $\overline{S}(t)$. The only difference is that the definition of $\overline{\tau}_{\textnormal{b}}$ \eqref{eq:def:taux:init} compared to $\tau_{\textnormal{b}}$ \eqref{eq:def:taux}, and hence the corresponding change needs to be made in $\tau_9^+$ \eqref{eq:def:tau12}.

Recall the definition of $\overline{S}^+$ \eqref{eq:def:taux:init}. Let us define
\begin{equation}
\overline{F} (\alpha_0,t):= 4\alpha_0^{1-\epsilon}\sigma_1\sigma_2(t;S)  + \alpha_0\beta_0^{5\theta} \sigma_1(t; \overline{S}^+) + 3\alpha_0^{3/2} \beta_0^{6\theta}.
\end{equation}
 Then, the analogue of ${\tau}^+_{9}$  is given by
\begin{equation}
\overline{\tau}_{9}:= \inf \left\{t\ge t_0: |\overline{S}(t)-\alpha_0| \ge \overline{F}(\alpha_0,t) \right\}.
\end{equation}
Furthermore, we set $\overline{\tau}_{10}(\Delta)$ to be the analogue of \eqref{eq:def:ind1:tau13}, which is
\begin{equation}
\overline{\tau}_{10}(\Delta):= \inf\left\{t\ge t_0: \left| \intop_{(t-\Delta)\vee t_0}^{t} \left\{\overline{S}(s)-\alpha_0 \right\}ds \right| \ge \alpha_0^{\frac{3}{2}-7\epsilon} \Delta \right\}.
\end{equation}
Note the important difference from the previous definitions: they read off the process starting  from $t_0$, not ${t}_0^+$. Then, we have the following as corresponding counterparts of Corollary \ref{cor:ind1:tau12} and Lemma \ref{lem:ind1:tau13}:
\begin{equation}
\begin{split}
&\PP \left(\overline{\tau}_{9} \le \overline{\tau}\wedge \hat{t}_0, \ \overline{\tau}>t_0 \right) \le \exp\left( -\beta_0^3\right);\\
&\PP \left( \overline{\tau}_{10}(\Delta) \le \overline{\tau}\wedge \overline{\tau}_{9}\wedge \hat{t}_0, \ \overline{\tau}>t_0 \right) \le 5\exp \left(-\beta_0^3 \right),
\end{split}
\end{equation}
where the second one holds for all $\Delta \ge \alpha_0^{-1}$. In fact, the first one is immediate from the definitions of  $\overline{\tau}_3$ and $\overline{\tau}_{\textnormal{b}}$, and the second one follows from the same proof as Lemma \ref{lem:ind1:tau13} relying on $\overline{\tau}_7$ and $\tau_{\textnormal{B}2}$ (Theorem \ref{thm:branching:ib})  when estimating $|\Pi_{\overline{S}^+}[t-\Delta,t]|$. Thus, the following corollary can be obtained analogously as Lemmas \ref{lem:ind1:tau78} and \ref{lem:ind1:tau9 10 11}, whose details are omitted thanks to similarity.

\begin{cor}\label{cor:ind1:tau7891011:base}
	Under the setting of Proposition \ref{prop:ind1:base:main}, let $\overline{\tau}' = \min \{ \overline{\tau}_i: 4\le i \le 8 \}$. Then, we have
	\begin{equation}
	\PP \left(\overline{\tau}' \le \overline{\tau}\wedge \hat{t}_0, \ \overline{\tau}>t_0 \right) \le 2\exp\left(-\beta_0^2 \right).
	\end{equation}
\end{cor}

\subsubsection{The martingale generated by the gap $S-\alpha_0$}\label{subsubsec:ind1:Sminusalpha conseq}

In this subsection, we conduct further analysis based on ${\tau}^+_{10}(\Delta)$ \eqref{eq:def:ind1:tau13} which will be useful later in Section \ref{sec:double int}. Define the \textit{gap} process
\begin{equation}\label{eq:def:Strianglealpha}
\Pi_{S\triangle\alpha_0}[a,b] := \Pi_S[a,b] \triangle \Pi_{\alpha_0}[a,b],
\end{equation} 
and write $d\Pi_{S-\alpha_0}(x):= d\Pi_S(x)-d\Pi_{\alpha_0}(x).$  We begin with stating a direct consequence of  Theorem \ref{thm:reg:conti:main} and  Lemma \ref{lem:ind1:tau13}

\begin{cor}\label{cor:ind1:gap num of pts}
	Let $\alpha_0,t_0,r>0$, set $t_0^-, t_0^+, \hat{t}_0$ as before, and let $\Delta_0:= \alpha_0^{-\frac{3}{2}+7\epsilon}$. Define 
	\begin{equation}
	\tau_{\textnormal{gap}}^{(1)}:= \inf \left\{t\ge t_0^++\Delta_0: \, |\Pi_{S\triangle \alpha_0}|[(t-\Delta_0), t]| \ge\beta_0^5 \right\}.
	\end{equation}
	If $\Pi_S(-\infty,t_0]\in \mathfrak{R}(\alpha_0,r;[t_0])$, then we have
	\begin{equation}
	\PP \left(\left. \tau_{\textnormal{gap}}^{(1)} <\hat{t}_0 \, \right| \, \mathcal{F}_{t_0} \right) \le 2\exp\left(-\beta_0^2 \right)+r.
	\end{equation}
\end{cor}

\begin{proof}
	Proof follows directly from applying Corollary \ref{cor:concentration:numberofpts each interval} based on Lemma \ref{lem:ind1:tau13}, and then combining with Theorem \ref{thm:reg:conti:main}.
\end{proof}

The next object of interest is the martingale defined as
\begin{equation}
G(t):= \intop_{t_0^+}^t K^*_{\alpha_0}(t-x) \big\{d{\Pi}_{S-\alpha_0} (x) -(S(x)-\alpha_0 )dx \big\}, 
\end{equation}
which satisfies the following property.
\begin{cor}\label{cor:ind1:gap mg}
Define the stopping time
\begin{equation}
\tau_{\textnormal{gap}}^{(2)}:= \left\{
t\ge t_0^+: |G(t)| \ge {\alpha_0\beta_0^{C_\circ}}\sigma_1(t;S\triangle \alpha_0) + \alpha_0^{\frac{7}{4}-\epsilon}
 \right\}.
\end{equation}
	Then, under the setting of Corollary \ref{cor:ind1:gap num of pts}, we have
	\begin{equation}
	\PP \left(\left.\tau_{\textnormal{gap}}^{(2)} <\hat{t}_0 \, \right| \, \mathcal{F}_{t_0} \right) \le \exp\left(-\beta_0^2 \right) +r
	\end{equation}
\end{cor}

\begin{proof}
	Recall the definition of $\tau(2)$ (Theorem \ref{thm:reg:conti:main}), $F(\alpha_0,t)$ and $\tau_9^+$ \eqref{eq:def:tau12} and write
	\begin{equation}
	\begin{split}
	&\intop_{t_0^+}^{t\wedge \tau(2)\wedge\tau_9^+ } (K^*_{\alpha_0}(t-x))^2 |S(x)-\alpha_0|dx \\
	&\le \intop_{t_0^+}^t  C\left(\frac{\alpha_0^2}{t-x+1} \vee \alpha_0^4 \right) \left( \alpha_0^{1-\epsilon}\sigma_1(t;S)^2 + \alpha_0\beta_0^{C_\circ+1}\sigma_1(t;S)+\alpha_0^{\frac{3}{2}}\beta_0^{6\theta} \right) dx.
	\end{split}
	\end{equation}
	Each terms in the integral can be estimated using Lemmas \ref{lem:ind1:pi1int basicbd:basic} and  \ref{lem:ind1:pi1int basicbd}, with parameters $\Delta_0 = \alpha_0^{-1}, \Delta_1 = \alpha_0^{-1} \beta_0^{C_\circ}, K=\alpha_0^{-1-\epsilon}$ and $N_0 = \beta_0^5.$ This gives that for $t\le \hat{t}_0$,
	\begin{equation}
	\intop_{t_0^+}^t (K^*_{\alpha_0}(t-x))^2 |S(x)-\alpha_0|dx \le \alpha_0^{\frac{7}{2}-\epsilon},
	\end{equation}
	and hence we can apply Corollary \ref{lem:concentrationofint:continuity} with parameters 
	\begin{equation}
	A^2=M= \alpha_0^{\frac{7}{2}-\epsilon}, \ \Delta = \alpha_0^{-\frac{3}{2}-7\epsilon}, \ D=1, \ \eta = \alpha_0\beta_0^{C_\circ},\ N= \beta_0^5,\ \delta = \alpha_0^{10}.
	\end{equation}
	Combining the result with Theorem \ref{thm:reg:conti:main} concludes the proof.
\end{proof}

The same proof applies to establishing the next corollary, whose details are omitted due to similarity.

\begin{cor}\label{cor:ind1:gap mg 2}
	Let $\underline{\alpha}_0:= \alpha_0 - \alpha_0^{\frac{3}{2}-\epsilon},$ $d\widetilde{\Pi}_{\triangle \alpha_0} (x) := d\Pi_{\alpha_0}(x)-d\Pi_{\underline{\alpha}_0}(x) - (\alpha_0-\underline{\alpha}_0)dx$,  and define the stopping time
	\begin{equation}
	\tau_{\textnormal{gap}}^{(3)} := \left\{t\ge t_0^+: \left|\intop_{t_0^+}^t K^*_{\alpha_0}(t-x) d\widetilde{\Pi}_{\triangle \alpha_0} (x) \right| \ge \alpha_0\beta_0^{C_\circ}\sigma_1(t; \triangle\alpha_0) + \alpha_0^{\frac{7}{4}-\epsilon}  \right\}.
	\end{equation} Under the setting of Corollary \ref{cor:ind1:gap num of pts}, we have
	\begin{equation}
	\PP \left(\left. \tau_{\textnormal{gap}}^{(3)} <\hat{t}_0 \,\right|\,\mathcal{F}_{t_0} \right) \le \exp\left(-\beta_0^2 \right)+r.
	\end{equation}
\end{cor}

\subsection{The induction base: regularity of the fixed rate process}\label{subsec:ind1:reg of fixed} 

We conclude the proof of Proposition \ref{prop:ind1:base:main} and Theorem \ref{thm:induction:base:main}. Since we collected most of the ingredients from the previous subsections, we need the last piece of argument that tells us $\overline{\tau}$ is not likely to be trivial, i.e., $\overline{\tau}>t_0$.

\begin{lem}\label{lem:ind1:tau t0:base}
Under the setting of Proposition \ref{prop:ind1:base:main}, we have
\begin{equation}
\PP \left(\overline{\tau}>t_0 \right) \ \ge 1-3\exp\left(-\beta_0^2 \right).
\end{equation}	
\end{lem}

\begin{proof}
From their definitions, it is clear that $\overline{\tau}_i >t_0$ for all $4\le i \le 8$. Furthermore, we saw from \eqref{eq:ind1:tau4:base:t0} in Corollary \ref{cor:ind1:tau4:base} that $\overline{\tau}_1>t_0$ w.h.p.. Thus, what remain to investigate are $ \overline{\tau}_2$ and $\overline{\tau}_3$.
	
	For $\overline{\tau}_2$, recall the definition of $R_0(t)$ from \eqref{eq:def:R0:ib} and note that $R_0(t_0) = \overline{S}(t_0).$ Thus, Lemma \ref{lem:branching:ib:R0 vs alpha} implies the desired estimate on $\overline{\tau}_2$.
	
	To understand $\overline{\tau}_3$, we write
	\begin{equation}
\left| \big[	\overline{S}(t_0) - \overline{S}_1(t_0)\big] -\big[ \overline{S}'(t_0;t_0^-,\alpha_0) - \overline{S}_1(t_0)\big] \right| \le \big|\overline{S}(t_0) - \overline{S}'(t_0;t_0^-,\alpha_0)\big| .
	\end{equation}
	The  RHS is $O(\alpha^{50})$  w.h.p. due to Corollary \ref{cor:SvsSprime} and Lemma \ref{lem:reg:base:A}, and the second term of the LHS can be controlled from Proposition \ref{prop:fixed perturbed:error}. Since we are only interested in estimating the above at point $t_0$ (not for entire $t$), it is not difficult to see that the third assumption of \eqref{eq:error:assumptions} can be weakened into
	\begin{equation}
	\sup_{t_0^- \le t <t_0} \overline{S}(t) \le \alpha_0^{1-\frac{\epsilon}{400}},\quad \overline{S}_1(t_0) \le \alpha_0^{1-\frac{\epsilon}{400}}.
	\end{equation}
	We also have the first two assumptions since $\overline{S}(t)=\alpha_0$ for $t\in[t_0^-,t_0)$. Combining all the above discussions concludes the proof. 
\end{proof}

\begin{proof}[Proof of Proposition \ref{prop:ind1:base:main}]
	We can obtain Proposition \ref{prop:ind1:base:main} by linking all the ingredients we observed so far: Building upon Lemmas \ref{lem:reg:base:A} and \ref{lem:ind1:tau t0:base}, the estimates on $\{\overline{\tau}_i\}_{1\le i \le 3} $ (Corollaries \ref{cor:ind1:tau6} \ref{cor:ind1:tau5:base}, \ref{cor:ind1:tau4:base}), and $\{\overline{\tau}_i \}_{4\le i \le 8}$ (Corollary \ref{cor:ind1:tau7891011:base}) deduce the desired result. The proof of the corresponding result for $\underline{S}(t)$ follows analogously.
\end{proof}

We conclude this subsection by verifying Theorem \ref{thm:induction:base:main} based on Proposition \ref{prop:ind1:base:main}.

\begin{proof}[Proof of Theorem \ref{thm:induction:base:main}]
	From Markov's inequality, Proposition \ref{prop:ind1:base:main} implies that
	\begin{equation}
	\PP\left[ \PP \left( \left.\overline{\tau}\le \acute{t}_0 \, \right| \,  \Pi_{\overline{S}}(-\infty,t_0] \right) \ge e^{-\beta_0^{3/2}} \right] \le  e^{-\beta_0^{1.8}},
	\end{equation}
	and it also tells us that
	\begin{equation}
	\PP \left(\Pi_{\overline{S}}(-\infty,t_0] \in \overline{\mathcal{A}}(\alpha_0,t_0) \right) \ge 1- \exp\left(-\beta_0^{1.9} \right).
	\end{equation}	
	This concludes Theorem \ref{thm:induction:base:main} for $\Pi_{\overline{S}}(-\infty,t_0]$, and the result for  $\Pi_{\underline{S}}(-\infty,t_0]$ follows analogously.
\end{proof}

\subsection{The growth of a regular aggregate}\label{subsec:ind1:growth}

This section, we verify Proposition \ref{prop:ind1:growth}, which follows as a consequence of Theorem \ref{thm:reg:conti:main} and  Lemma \ref{lem:ind1:tau13}.

\begin{proof}[Proof of Proposition \ref{prop:ind1:growth}]
	Recall the definition of $\tau(\kappa)$ and ${\tau}^+(\kappa)$ from Theorem \ref{thm:reg:conti:main}, and ${\tau}^+_{10}(\Delta)$ from \eqref{eq:def:ind1:tau13}. In the proof, we let $$\tilde{\tau}:= \tau(2)\wedge {\tau}^+(1/2) \wedge {\tau}^+_{10}(\alpha_0^{-1}).$$ 
	Then, Theorem \ref{thm:reg:conti:main} and Lemma \ref{lem:ind1:tau13} tell us that
	\begin{equation}\label{eq:ind1:growth:1}
	\PP \left( \left. \tilde{\tau} = \hat{t}_0 \, \right| \, \Pi_S(-\infty,t_0] \in \mathfrak{R}(\alpha_0,r;[t_0]) \right) \ge 1-r-\exp\left(-\beta_0^{1.9} \right).
	\end{equation} 
	(Note that $\tau(2)\ge {\tau}^+(1/2)$ if $\tau(2) > \acute{t}_0$, and also that by definition $\tilde{\tau}^+ \le \hat{t}_0$.) 
	
	Now, we apply Lemma \ref{lem:concentration of integral} under the following setting: $$f\equiv 1,  \ g(t) = S(t),\ \tau = \tilde{\tau}, \ M= \alpha_0^{-1}\beta_0^{12\theta},\ \lambda = M^{-\frac{1}{2}}, \ \textnormal{ and } a = \beta_0^{\theta}.$$ 
	The definition of $\tau_{1}$ justifies the assumption \eqref{eq:concentration:conditions:basic}. Hence, from the lemma we obtain that
	\begin{equation}\label{eq:ind1:growth:2}
	\PP \left( \left. \sup_{\acute{t}_0\le t \le \tilde{\tau}^+} \left| |\Pi_S[t_0,t]| - \intop_{t_0}^t S(s)ds  \right| \ge \alpha^{-\frac{1}{2}}\beta_0^{7\theta} \, \right| \, \Pi_S(-\infty,t_0] \right)
	\le \exp\left(-\frac{1}{2}\beta_0^\theta \right),
	\end{equation}
	and this holds for any given $\Pi_S(-\infty,t_0]$.
	
	Lastly, we observe that the definition of ${\tau}^+_{10}(\alpha_0^{-1})$  gives that
	\begin{equation}\label{eq:ind1:growth:3}
	\sup_{\acute{t}_0 \le t \le {\tau}^+_{10}\wedge \hat{t}_0} \left|\intop_{t_0}^t S(s) ds - \alpha_0(t-t_0) \right| \le \alpha_0^{-\frac{1}{2}-8\epsilon}.
	\end{equation}
	Thus, we obtain the conclusion by combining \eqref{eq:ind1:growth:1}, \eqref{eq:ind1:growth:2}, and \eqref{eq:ind1:growth:3}.
\end{proof}

\section{The inductive analysis on the speed: Part 2} \label{sec:reg:next step}

In this section, we present the proof of Theorem \ref{thm:induction:main}, completing the inductive argument on regularity. Recall the notations $t_0^-, t_0^+,\acute{t}_0$ and $\hat{t}_0$ (\eqref{eq:def:t0t1} and \eqref{eq:def:t0t11}), and the definitions of $\tau^\sharp$,  $\mathcal{A}$, and ${\tau}^+(\kappa)$ given in Section \ref{subsec:regoverview:reg} (\eqref{eq:def:tau:induction},   \eqref{eq:def:A:induction}) and in \eqref{eq:def:ind1:tauacute:basic}. To emphasize the time step and the frame of reference where ${\tau}^+(\kappa)$ is defined, we write ${\tau}^+(\alpha_0, t_0,\kappa) = {\tau}^+(\kappa)$ (Hence, ${\tau}^+(\alpha_0, t_0,\kappa)= \tau(\alpha_0,t_0^+,\kappa)$). We also set $t_1$ to be any number that satisfies \eqref{eq:t1 regime}, and let 
\begin{equation}
t_1^- := t_1 - \alpha_0^{-2}\beta_0^\theta, \quad t_1^+:= t_1+\alpha_0^{-2} \beta_0^{\theta}, \quad t_1^{\sharp}:= t_1+4\alpha_0^{-2} \beta_0^\theta.
\end{equation}
Further, let $\alpha_1$ be given as \eqref{eq:def:alpha1} and let $\beta_1:= \log(1/\alpha_1)$.

\begin{thm}\label{thm:ind2:main}
	Let $\alpha_0,t_0>0$, let $t_0^-, t_0^+$ be as above, and let 
	$t_1$ be any number satisfying \eqref{eq:t1 regime}. Under the above notations, let $\tau':= \tau(\alpha_0,t_0,2)\wedge {\tau}^+(\alpha_0, t_0,1/2)$.
	Then, for all sufficiently small $\alpha_0$, the following hold true:
	\begin{enumerate}
		\item $\PP (\tau^\sharp(\alpha_1,t_1,2) \le t_1^\sharp  \ | \ \mathcal{A}(\alpha_0,t_0) )
		\le
		\PP (\tau' < \hat{t}_0 \ | \ \mathcal{A}(\alpha_0,t_0) )
		+ 4\exp\left(-\beta_0^{2} \right);$
		
		\item $\PP (\mathcal{A}(\alpha_1,t_1)^c \ | \ \mathcal{A}(\alpha_0,t_0) )
		\le
		\PP (\tau' <  \hat{t}_0  \ | \ \mathcal{A}(\alpha_0,t_0) )
		+ \exp\left(-\beta_0^{2} \right).$
	\end{enumerate}
	
\end{thm}

Our program is to follow the description given in Section \ref{subsubsec:regoverview:overview:nextstep}. 

\begin{itemize}
	\item In Section \ref{subsec:ind2:ratechange}, we study $|\alpha_1-\alpha_0|$, the change of the frame of reference. Moreover, we establish that $\mathcal{A}(\alpha_1,t_1)$ holds w.h.p. at time $t_1$, proving Theorem \ref{thm:ind2:main}-(2). As a byproduct of our analysis, we obtain control on $\tau_1^\sharp$ and establish Proposition \ref{prop:reg:newrates} as well.

	\item In Section \ref{subsec:ind2:remaining reg}, we study all the stopping times except $\tau_1^\sharp$ and $\tau_{3}^\sharp$. 
	
	\item In Section \ref{subsec:ind2:bootstrappedJ}, we establish  control on  $\tau_{3}^\sharp $.
	
	\item In Section \ref{subsec:ind2:fin}, we combine all the analysis done in Sections \ref{subsec:ind2:ratechange}--\ref{subsec:ind2:bootstrappedJ} and complete the proofs of Theorems \ref{thm:ind2:main} and  \ref{thm:induction:main}.

\end{itemize}

Before moving on to the proofs, we clarify the definitions of $\alpha_0',\alpha_1,\alpha_1'$: given $\alpha_0,t_0$, we have
\begin{equation}\label{eq:def:alpha:ind2}
\begin{split}
\alpha_0' &= \mathcal{L}(t_0; \Pi_S[t_0^{-},t_0],\alpha_0)=\intop_{t_0^{-}}^{t_0} \intop_{ t_0}^{\infty} K^*_{\alpha_0} \cdot K_{\alpha_0}(u-x) du d\Pi_S(x)  ;\\
\alpha_1&=\mathcal{L}(t_1^-;\Pi_S[t_0^{-},t_1^-],\alpha_0) 
=
\intop_{t_0^{-}}^{t_1} \intop_{ t_1}^{\infty} K^*_{\alpha_0} \cdot K_{\alpha_0}(u-x) du d\Pi_S(x)  ;\\
\alpha_1'&=\mathcal{L}(t_1;\Pi_S[t_1^{-},t_1],\alpha_1) 
=
\intop_{t_1^{-}}^{t_1} \intop_{ t_1}^{\infty} K^*_{\alpha_1} \cdot K_{\alpha_1}(u-x) du d\Pi_S(x) ,
\end{split}
\end{equation}
where $K^*_\alpha = \frac{2\alpha^2}{1+2\alpha}$.

\subsection{Change of  rates for critical branching}\label{subsec:ind2:ratechange}

Our goal is to establish that $\mathcal{A}(\alpha_1,t_1)$ \eqref{eq:def:A:induction} happens with high probability, conditioned on the regularity in the previous step. As a consequence of our analysis, we prove Proposition \ref{prop:reg:newrates}.

 We begin with studying $\mathcal{A}_3(\alpha_1,t_1)$. The other events $\mathcal{A}_1$ and $\mathcal{A}_2$ are investigated in Sections \ref{subsubsec:ind2:A1}, \ref{subsubsec:ind2:A3} respectively.

\begin{lem}\label{lem:ind2:ratechange}
	Under the setting of Theorem \ref{thm:ind2:main}, we have
	\begin{equation}
	\PP\left(\mathcal{A}_3(\alpha_1,t_1)^c,\ \tau'\ge \hat{t}_0 \ | \ \mathcal{A}(\alpha_0,t_0) \right) \le 10\exp\left(-\beta_0^5 \right).
	\end{equation}
\end{lem}

Since the base rates in the definition of $\alpha_1$ and $\alpha_1'$ are $\alpha_0$ and $\alpha_1$, respectively, we may expect that understanding the size of $|\alpha_1-\alpha_0|$ is important, which is also what we want in Theorem \ref{thm:induction:main}-(2). We define the following $\mathcal{F}_{t_1}$-measurable event
\begin{equation}\label{eq:def:A4}
\mathcal{A}_4=\mathcal{A}_4(\alpha_0,t_0):= \left\{|\alpha_1-\alpha_0|\le 2\alpha_0^{3/2}\beta_0^{6\theta} \right\},
\end{equation} 
and begin with showing that $\mathcal{A}_4$ happens with high probability.

\begin{lem}\label{lem:ind2:alpha1vsalpha0}
	Under the setting of Theorem \ref{thm:ind2:main}, we have
	\begin{equation}
	\PP\left(\mathcal{A}_4^c ,\ \tau' \ge \hat{t}_0 \ | \ \mathcal{A}(\alpha_0,t_0) \right) \le 2\exp\left(-\beta_0^5 \right).
	\end{equation}
\end{lem}

\begin{proof}
	In the proof, set  $\mathcal{A}=\mathcal{A}(\alpha_0,t_0)$ for convenience.
	
Since the event $\mathcal{A}_3(\alpha_0,t_0)\supset \mathcal{A}$ is given, we can focus on deriving $|\alpha_1-\alpha_0'|\le \alpha_0^{3/2}\beta_0^{6\theta}$. Define $S_1(t)= S_1(t;t_0^-,\alpha_0)$ as before. From  equation \eqref{eq:integralform:branching}, we can write
	\begin{equation}\label{eq:ind2:a1vsa0:split}
	\begin{split}
	\alpha_1-\alpha_0' 
	&= \mathcal{L}( t_1^-; \Pi_S[t_0^{-},t_1^-],\alpha_0)
	-\mathcal{L}(t_0; \Pi_S[t_0^{-},t_0],\alpha_0)\\
	&=
\intop_{t_0}^{t_1^-} K^*_{\alpha_0} d\widetilde{ \Pi}_S(x) + \intop_{t_0}^{t_1^-} K^*_{\alpha_0} (S(x)-S_1(x))dx ,
\end{split}
	\end{equation}
	where we wrote $d\widetilde{\Pi}_S(x)= d\Pi_S(x)-S(x)dx$.
	
	To study the first term, we note that
	\begin{equation}\label{eq:ind2:rate:a1vsa0:qv}
	\intop_{t_0}^{t_1^-\wedge \tau'}( K^*_{\alpha_0})^2 S(x)dx \le \alpha_0^3 \beta_0^{10\theta+2C_\circ},
	\end{equation}
	due to the definition of $\tau_{1}$ (Section \ref{subsubsec:reg:Scontrol}). Since $K^*_{\alpha_0} \le C\alpha_0^2 \le \alpha_0^{3/2}$, we apply Corollary \ref{cor:concentration of integral} to deduce that
	\begin{equation}\label{eq:ind2:rate:a1vsa0:mg}
	\PP\left(\left.\left|	\intop_{t_0}^{t_1^-\wedge \tau'} K^*_{\alpha_0} d\widetilde{\Pi}_S(x)\right| \ge \alpha_0^{3/2}\beta_0^{5\theta+2C_\circ} \ \right| \ \mathcal{A} \right) \le \exp\left(-\beta_0^{C_\circ/2} \right).
	\end{equation}
	
	The second term of \eqref{eq:ind2:a1vsa0:split} follows directly from Lemma \ref{lem:ind1:intofSminS1} (note that $\tau'\le \tau(2)$ for $\tau(2)$ in Lemma \ref{lem:ind1:intofSminS1}):
	\begin{equation}\label{eq:ind2:rate:a1vsa0:drift}
	\PP \left( \left. \left|\intop_{ t_0}^{t_1^-\wedge \tau'} K^*_{\alpha_0}(S(x)-S_1(x))dt  \right| \ge  \alpha_0^{2-3\epsilon} \ \right|  \ \mathcal{A}  \right) \le \exp\left(-\beta_0^{C_\circ/3} \right).
	\end{equation}
	
	\noindent Combining \eqref{eq:ind2:a1vsa0:split}, \eqref{eq:ind2:rate:a1vsa0:mg}, \eqref{eq:ind2:rate:a1vsa0:drift} and the condition from $\mathcal{A}_3(\alpha_0,t_0)$, we obtain conclusion.
\end{proof}

Before moving on, we establish the following lemma, which gives a desired control on the stopping time $\tau_1^\sharp$ (Section \ref{subsubsec:reg:refined1storder}.

\begin{lem}\label{lem:ind2:tau1sharp}
    Under the setting of Theorem \ref{thm:ind2:main}, we have
    \begin{equation}
        \mathbb{P} \left( \tau_1^\sharp(\alpha_1,t_1) \le t_1^\sharp,\ \tau'\ge \hat{t}_0 \,| \, \mathcal{A}(\alpha_0,t_0) \right) \le \exp\left(-\beta_0^4 \right).
    \end{equation}
\end{lem}

\begin{proof}
    This comes as a direct consequence of Lemmas \ref{lem:ind1:tau13} and \ref{lem:ind2:alpha1vsalpha0}: For  $t' =  t_1^\sharp\wedge \tau_{10}^+(\alpha_0^{-1})\wedge \tau'$ and on the event $\mathcal{A}_4$, note that
    \begin{equation}
        \intop_{t_1}^{t'} |S(s)-\alpha_1| ds \le \intop_{t_1}^{t'} |S(s)-\alpha_0|ds + \intop_{t_1}^{t'} |\alpha_0-\alpha_1|ds \le \alpha_0^{-\frac{1}{2}-\epsilon}.
    \end{equation}
\end{proof}

To establish Lemma \ref{lem:ind2:ratechange}, we set
\begin{equation}\label{eq:def:t1minmin}
t_1^{--}:= t_1^- - \alpha_0^{-2}\beta_0^\theta = t_1-2\alpha_0^{-2}\beta_0^\theta,
\end{equation} 
and introduce auxiliary parameters $\alpha_1^{(1)}, \alpha_1^{(2)}$ and $\alpha_1^{(3)}$ defined as
\begin{equation}\label{eq:def:alphaprimeprime}
\begin{split}
\alpha_1^{(1)}& := \mathcal{L}(t_1;\Pi_S[t_1^{-},t_1],\alpha_0) = \intop_{ t_1^{-}}^{t_1} \intop_{t_1}^\infty K^*_{\alpha_0 } \cdot K_{\alpha_0}(u-x) du d\Pi_S(x);\\
\alpha_1^{(2)}& := \mathcal{L}(t_1;\Pi_S[t_1^{--},t_1],\alpha_0) = \intop_{ t_1^{--}}^{t_1} \intop_{t_1}^\infty K^*_{\alpha_0 } \cdot K_{\alpha_0}(u-x) du d\Pi_S(x);\\
\alpha_1^{(3)} &:= \mathcal{L}(t_1^-;\Pi_S[t_1^{--},t_1^-],\alpha_0) = \intop_{ t_1^{--}}^{t_1^-} \intop_{t_1^-}^\infty K^*_{\alpha_0 } \cdot K_{\alpha_0}(u-x) du d\Pi_S(x).
\end{split}
\end{equation}
From these notations, we write
\begin{equation}\label{eq:ind2:rate:a1vsa1prime:decomp}
|\alpha_1-\alpha_1'| \le |\alpha_1-\alpha_1^{(3)}|+|\alpha_1^{(2)}-\alpha_1^{(3)}| +|\alpha_1^{(1)}-\alpha^{(2)}|+|\alpha_1' -\alpha_1^{(1)}|.
\end{equation}

 Recalling that $\alpha_1:= \mathcal{L}(t_1^-; \Pi_S[t_0^-,t_1^-], \alpha_0)$, we can see that $|\alpha_1^{(2)}-\alpha_1^{(3)}|$ can be investigated by the same method as Lemma \ref{lem:ind2:alpha1vsalpha0}, leading us to the following corollary.

\begin{cor}\label{cor:ind2:alpha1vsalphaaux3}
	Under the setting of Theorem \ref{thm:ind2:main} and the above notations, we have
	\begin{equation}
	\PP \left(\left. \left|\alpha_1^{(2)} -\alpha_1^{(3)}  \right| \ge \frac{1}{2}\alpha_0^{\frac{3}{2}}\beta_0^{\theta}, \ \tau' \ge \hat{t}_0 \ \right|\, \mathcal{A}(\alpha_0,t_0) \right) \le 2\exp\left(-\beta_0^5\right).
	\end{equation}
\end{cor}

\begin{proof}
	From the proof of Lemma \ref{lem:ind2:alpha1vsalpha0}, our task is to control 
	\begin{equation}
	\intop_{t_1^{-}}^{t_1} K^*_{\alpha_0} \big\{d\widetilde{\Pi}_S(x) + (S(x)-S_1(x;t_1^{--},\alpha_0))dx \big\}.
	\end{equation}
	The rest of the proof is analogous to Lemma \ref{lem:ind2:alpha1vsalpha0} except the following two things:
	\begin{itemize}
		\item Since the regime of integration $[t_1^-,t_1]$ is much shorter than $[t_0^-,t_1^-]$ from the previous lemma, the RHS \eqref{eq:ind2:rate:a1vsa0:qv} can be improved to $\alpha_0^3\beta_0^{\theta+2C_\circ}$, resulting in a corresponding enhancement in \eqref{eq:ind2:rate:a1vsa0:mg}.
		
		\item For any $t_1^-\le t \le t_1\wedge \tau'$, we have
		\begin{equation}
		|S_1(t;t_0^-,\alpha_0)-S_1(t;t_1^{--},\alpha_0)| = \intop_{t_0^{-}}^{t_1^-} K_{\alpha_0}(t-x) d\Pi_S(x) \le \alpha_0^{100},
		\end{equation}
		due to the decay property of $K_{\alpha_0}(s)$. Thus, we can obtain \eqref{eq:ind2:rate:a1vsa0:drift} for our case as well.
	\end{itemize}
\end{proof}

 Next, we observe that the first and the third terms in the RHS are negligible, which comes as a simple consequence of regularity.

\begin{lem}\label{lem:ind2:alphavsalphaPP}
	Under the setting of Theorem \ref{thm:ind2:main} and \eqref{eq:def:alphaprimeprime},  we have
	\begin{equation}
	\begin{split}
	\PP\left(|\alpha_1-\alpha_1^{(3)}|\ge \alpha_0^{100} ,\ \tau' \ge \hat{t}_0 \ | \ \mathcal{A}(\alpha_0,t_0) \right) \le 2\exp\left(-\beta_0^5 \right);\\
		\PP\left(|\alpha_1^{(1)}-\alpha_1^{(2)}|\ge \alpha_0^{100} ,\ \tau' \ge \hat{t}_0 \ | \ \mathcal{A}(\alpha_0,t_0) \right) \le 2\exp\left(-\beta_0^5 \right).
	\end{split}
	\end{equation}
	
\end{lem}

\begin{proof}
	We start with investigating the first inequality. We observe that
	\begin{equation}\label{eq:ind2:rate:a1vsa1PP}
	\alpha_1-\alpha_1^{(3)}
	= \intop_{ t_0^{-}}^{t_1^{--}} \intop_{t_1^{-}}^\infty K^*_{\alpha_0 } \cdot K_{\alpha_0}(u-x) du d\Pi_S(x)
	\le Ce^{-c\beta_0^{\theta}}  \left|\Pi_S [t_0^{-},t_1^{--}] \right|,
	\end{equation}
	where the last inequality came from the estimate Lemma \ref{lem:estimate on K:intro}, noting that $t_1^--t_1^{--} = \alpha_0^{-2}\beta_0^\theta$.

	Moreover, from the definition of ${\tau}^+_{10}(\alpha_0, t_0,1/2)$ and $\mathcal{A}_2(\alpha_0,t_0)$  (Sections \ref{subsubsec:reg:aggsize}, \ref{subsubsec:reg:history}), we have $|\Pi_S[t_0^{-},\tau']|\le \alpha_0^{-1}\beta_0^{11\theta}$. Plugging this estimate to \eqref{eq:ind2:rate:a1vsa1PP},  we obtain the first estimate of Lemma \ref{lem:ind2:alphavsalphaPP}.
	
	For the second one, observe that
	\begin{equation}
	\begin{split}
	\alpha_1^{(2)}-\alpha_1^{(1)} = \intop_{t_1^{--}}^{t_1^-} \intop_{t_1}^\infty K^*_{\alpha_0}\cdot K_{\alpha_0}(u-x) du d\Pi_S(x)
	\le
	Ce^{-c\beta_0^\theta} |\Pi_S[t_1^{--},t_1^-]|,
	\end{split}
	\end{equation}
analogously as \eqref{eq:ind2:rate:a1vsa1PP}. Thus, the same reasoning as above gives the conclusion.
\end{proof}

	In \eqref{eq:ind2:rate:a1vsa1prime:decomp}, what remains to understand is $|\alpha_1'-\alpha_1^{(1)}|$.  The two terms differ only  in  the parameter $\alpha_1$ and $\alpha_0$ in the integrand, and hence we can expect to study the difference $|\alpha_1'-\alpha_1^{(1)}|$  based on analytic properties of $K^*_\alpha$ and $K_\alpha$. 
	Building upon such an intuition, we prove the  following result.

\begin{lem}\label{lem:ind2:alphaPvsalphaPP}
	Under the setting of Theorem \ref{thm:ind2:main} and \eqref{eq:def:alphaprimeprime}, we have
	\begin{equation}
	\PP\left(|\alpha_1'-\alpha_1^{(1)} |\ge \alpha_0^2 \beta_0^{8\theta+1},\ \tau'\ge \hat{t}_0 \ | \ \mathcal{A}(\alpha_0,t_0) \right) \le \exp\left(-\beta_0^5 \right).
	\end{equation}
\end{lem}

Proof of Lemma \ref{lem:ind2:alphaPvsalphaPP} is given in the following subsection. Before delving into the proof, we establish Lemma \ref{lem:ind2:ratechange}.

\begin{proof}[Proof of Lemma \ref{lem:ind2:ratechange}]
	Lemma \ref{lem:ind2:ratechange} follows directly from a union bound over the four probabilities in Corollary \ref{cor:ind2:alpha1vsalphaaux3}, Lemmas  \ref{lem:ind2:alphavsalphaPP} and \ref{lem:ind2:alphaPvsalphaPP}. Note that by the four estimates along with Lemma \ref{lem:ind2:alpha1vsalpha0}, we get
	$$|\alpha_1-\alpha_1'|\le \frac{2}{3}\alpha_0^{\frac{3}{2}}\beta_0^\theta  \le \alpha_1^{\frac{3}{2}}\beta_1^\theta, $$
	with high probability, which is much stronger than $\mathcal{A}_3(\alpha_1,t_1)$ in terms of the exponent of $\beta_1$.
\end{proof}

Before moving on, we also conclude the proof of Proposition \ref{prop:reg:newrates}.

\begin{proof}[Proof of Proposition \ref{prop:reg:newrates}]
	Recall the definition of $\tilde{\alpha_1}$, and we write
	$$|\alpha_1'-\tilde{\alpha}_1|\le |\alpha_1'-\alpha_1^{(1)}| + |\tilde{\alpha}_1- \alpha_1^{(1)}|, $$
	where $\alpha_1^{(1)}$ is given as \eqref{eq:def:alphaprimeprime}.
	Note that we can write
	\begin{equation}
	\tilde{\alpha}_1-\alpha_1^{(1)} = \intop_{t_0^-}^{t_1^-} \intop_{t_1}^\infty K^*_{\alpha_0} \cdot K_{\alpha_0}(u-x)du d\Pi_S(x),
	\end{equation}
	which can be bounded in the same way as Lemma \ref{lem:ind2:alphavsalphaPP}. Thus, combining Theorem \ref{thm:reg:conti:main}, Lemmas \ref{lem:ind2:alphavsalphaPP} and \ref{lem:ind2:alphaPvsalphaPP} concludes the proof.
\end{proof}

\subsubsection{Proof of Lemma \ref{lem:ind2:alphaPvsalphaPP}: stability under change of rate}

Let $Y_\alpha (s)$ be a rate $\alpha $ Poisson process and let $W(s)$ be a continuous time random walk. Define the stopping time  $T _\alpha : =\inf \{ t>0: \ W(s)>Y_\alpha (s) \}$ and the function $h_\alpha(t):=\mathbb P (t \le T _\alpha <\infty )$. Recall that $K_\alpha (t)=\alpha (1+2 \alpha ) h_\alpha(t)$ (Definition \ref{def:Kalpha:main}). In the following claim we bound the derivative of $h_\alpha (t)$ with respect to $\alpha $.

\begin{claim}\label{claim:derivative of h}
    We have that 
    \begin{equation}
       \left|  \frac{d}{d \alpha } h_\alpha (t) \right| \le Ce^{-c \alpha ^2 t}
    \end{equation}
\end{claim}

\begin{proof} 
    Throughout the proof we write $h_\alpha '(x)$ for the derivative of $h$ with respect to $x$ and write explicitly $\frac{d}{d \alpha }h_\alpha (x)$ when we differentiate with respect to $\alpha $. Let $\delta \le \alpha ^{10}$ and let $Y_\alpha ,Y_{\delta  }$ be independent Poisson processes with rates $\alpha $ and $\delta $ respectively. Let $Y_{\alpha +\delta }:=Y_\alpha +Y _{\delta }$ and note that  $T _\alpha \le T _{\alpha +\delta } $. Recall that $\mathbb P (T _\alpha <\infty ) =h_\alpha (0)=1/(1+2 \alpha )$. Integrating over the values $x$  that $T _\alpha $ can take we obtain
	\begin{equation}
	\begin{split}
	&\left| h_{\alpha +\delta }(t)-h_{\alpha }(t) \right| \le \mathbb P (T _\alpha <t,\ t \le T _{\alpha +\delta } <\infty ) + \mathbb P (t \le T _\alpha <\infty , \ T _{\alpha +\delta } =\infty )    \\
	& \le - \intop _0^t h_\alpha '(x) \mathbb P (Y_{\delta }(x)=1)h_{\alpha +\delta }(t-x)   - \frac{2(\alpha +\delta ) }{1+2(\alpha+\delta ) }\intop _t^{\infty }h_\alpha '(x) \mathbb P (Y_{\delta }(x)=1)  
	- \intop _0^{\infty } h_\alpha '(x) \mathbb P (Y_\delta (x)\ge 2 )  \\
	&\le C \intop _0^t x^{-\frac{3}{2}}e^{-c \alpha ^2 x} \cdot \delta x \cdot  (t-x)^{-\frac{1}{2}} e^{-c \alpha ^2(t-x)}dx+ C\alpha  \intop _t^{\infty } x^{-\frac{3}{2}} e^{-c \alpha ^2 x} \cdot \delta x dx  +C\intop _0^\infty  x^{-\frac{3}{2}} e^{-c\alpha ^2 x} \delta ^2 x^2  \\
	&\le C\delta e^{-c \alpha ^2 t } \intop _0^t  x^{-\frac{1}{2}} (t-x)^{-\frac{1}{2}} + C \delta \alpha e^{-c \alpha ^2 t } \intop _0^{\infty } x^{-\frac{1}{2}} e^{-c \alpha ^2 x} + C \alpha ^{-3 } \delta ^2 \le C \delta e^{-c \alpha ^2 t } +C_\alpha \delta ^2 ,
	\end{split}
	\end{equation} 
	where in the third inequality we used Corollary~\ref{cor:bound on K'} in order to bound $h_\alpha '$. The claim follows from the last bound.
\end{proof}

We claim  that we can approximate the speed by a fixed value in the time interval $[t_1^-,t_1]$.

\begin{claim}\label{claim:speed is almost fixed}
    Under the setting of Theorem \ref{thm:induction:main}, there exists a random variable $\alpha _2 \in \mathcal F _{t_1^-}$ such that 
    \begin{equation}
       \PP \left(\left. \intop _{t_1^{-}} ^{t_1} |S(t)-\alpha _2 | dt\ge \alpha_0 ^{-\frac{1}{2}} \beta_0^{2 \theta }, \, \tau' \ge \hat{t}_0 \, \right| \, \mathcal{A}(\alpha_0,t_0) \right) \le \exp\left(-\beta_0^5 \right).
    \end{equation}
\end{claim}

\begin{proof}
Recalling the definition of $\alpha_1^{(3)}$ from \eqref{eq:def:alphaprimeprime},	we show that the choice $\alpha_2=\alpha_1^{(3)}$ is enough for our purpose. To this end, we split the integrand into 
$$ |S(t)-\alpha_1| \le |S(t)-S_1(t;t_1^{--},\alpha_0)| + |S_1(t;t_1^{--},\alpha_0) - \alpha_1^{(3)}|, $$
where $t_1^{--}$ is given as \eqref{eq:def:t1minmin}.

Note that if $t_1^-\le t\le t_1\wedge  \tau'$, then $|S_1(t;t_0^-,\alpha_0)-S_1(t;t_1^{--},\alpha_0)| \le \alpha_0^{100}  $, by the same argument as \eqref{eq:ind1:S1timechange}. Thus, the definition of $\tau_3$ (Section \ref{subsubsec:reg:Scontrol}) gives 
\begin{equation}\label{eq:ind2:SminS1timechange}
|S(t)-S_1(t;t_1^{--},\alpha_0)| \le 2\alpha_0^{1-\epsilon} \sigma_1\sigma_2(t;S).
\end{equation}

On the other hand, the same expression as \eqref{eq:integralform:branching:ind1} gives
\begin{equation}
S_1(t;t_1^{--},\alpha_0) - \alpha_1^{(3)} = \intop_{t_1^-}^{t} K_{\alpha_0}^*(t-x) \big\{ d\Pi_S(x)-S_1(x;t_1^{--},\alpha_0) dx \big\}.
\end{equation}
Using the bound \eqref{eq:ind2:SminS1timechange}, we can control this term analogously as the proof of Lemma \ref{lem:ind1:taux}, which gives that
\begin{equation}\label{eq:ind2:S1minusalphapp}
\begin{split}
\PP \left( \left. \bigcap_{t\in[t_1^-,t_1]} \left\{ \big|S_1(t;t_1^{--},\alpha_0)-\alpha_1^{(3)} \big| \le \alpha_0\beta_0^{C_\circ}\sigma_1(t;S) + \alpha_0^{\frac{3}{2}}\beta_0^{\theta} \right\}, \ \tau'\ge\hat{t}_0 \ \right| \ \mathcal{A}(\alpha_0,t_0) \right)\\
\ge 1-\exp\left(-\beta_0^5 \right).
\end{split}
\end{equation}
Note that the exponent $\theta$ in the term $\alpha_0^{\frac{3}{2}}\beta_0^\theta$ is smaller than that of Lemma \ref{lem:ind1:taux} since the length of the interval we are looking at is of scale $t_1-t_1^-=\alpha_0^{-2}\beta_0^\theta$, not $\alpha_0^{-2}\beta_0^{10\theta}$.

When we have both \eqref{eq:ind2:SminS1timechange} and \eqref{eq:ind2:S1minusalphapp}, we can control the integral of $|S(t)-\alpha_1^{(3)}|$ by Lemma \ref{lem:ind1:pi1int basicbd:basic}. This tells us that
\begin{equation}
\intop_{t_1^-}^{t_1} |S(t)-\alpha_1^{(3)}|dt \le \intop_{t_1^-}^{t_1} \left(\frac{\alpha_0^{1-\epsilon}}{\pi_1(t;S)+1} + \frac{\alpha_0\beta_0^{C_\circ}}{\sqrt{\pi_1(t;S)+1}} + \alpha_0^{\frac{3}{2}}\beta_0^{\theta} \right)dt \le 2\alpha_0^{-\frac{1}{2}}\beta_0^{2\theta},
\end{equation}
concluding the proof.
\end{proof}

We are now ready to prove Lemma~\ref{lem:ind2:alphaPvsalphaPP}

\begin{proof}[Proof of Lemma~\ref{lem:ind2:alphaPvsalphaPP}]
	Let $\alpha _2$ be the random variable from Claim~\ref{claim:speed is almost fixed}. Recall that 
	\begin{equation}
	\alpha _1^{(1)} =\frac{2\alpha_0 ^2 }{1+2 \alpha_0 }\intop _{t_1^-}^{t_1} \intop _{t_1}^\infty K_{\alpha _0 }(y-x) dy \, d \Pi _S (x) =\frac{2\alpha_0 ^2 }{1+2 \alpha _0 }\intop _{t_1^-}^{t_1} \intop _{t_1-x}^\infty K_{\alpha _0 }(z) dz \, d \Pi _S (x)
	\end{equation}
	Next, we break the measure $d \Pi _S$ into $(d \Pi _S(x)- \alpha _2 dx)+\alpha _2 dx$ and start by estimating the integral with respect to the Lebesgue measure. We have
	\begin{equation}
	\begin{split}
	\frac{2\alpha_0 ^2 \alpha _2 }{1+2 \alpha_0 } \intop _{t_1^-}^{t_1}\intop _{t_1-x}^\infty K_{\alpha _0}(z) dz \, dx &=\frac{2\alpha_0 ^2 \alpha _2 }{1+2 \alpha_0  } \intop _{0}^{t_1-t_1^-} \intop _{t_1-z}^{t_1} K_{\alpha _0}(z)dx\, dz+\frac{2\alpha_0 ^2 \alpha _2 }{1+2 \alpha_0  }\intop _{t_1-t_1^-}^{\infty } \intop _{t_1^-}^{t_1} K_{\alpha _0}(z) dx \, dz\\
	=\frac{2\alpha_0 ^2 \alpha _2 }{1+2 \alpha_0  } &\intop _{0}^{t_1-t_1^-} z K_{\alpha _0} (z) dz+\frac{(t_1-t_1^-)2 \alpha _0 ^2 \alpha _2 }{1 +2 \alpha _0 } \intop _{t_1-t_1^-}^{\infty } K_{\alpha _0} (z)dz= \alpha _2 +O(\alpha _0^{10}),
	\end{split}
	\end{equation}
	where in the last equality we used Corollary~\ref{cor:moments of K} and Lemma~\ref{lem:estimate on K:intro} and the fact that $t_1-t_1^- \ge \alpha _0^{-2 }\beta_0 ^2 $. Thus, 
	\begin{equation}
	\alpha _1 '=O(\alpha_0 ^{10})+\alpha _2+\frac{2\alpha _0 ^2 }{1+2 \alpha_0 }\intop _{t_1^-}^{t_1} \intop _{t_1-x}^\infty K_{\alpha }(z) dz \, (d \Pi _S (x)-\alpha _2dx).
	\end{equation}
	Using the same arguments for $\alpha _1'$ we get 
	\begin{equation}\label{eq:integral of F}
	\alpha _1'- \alpha _1^{(1)}=O(\alpha _0 ^{10})+\intop _{t_1^{-}}^{t_1} F(x)  (d \Pi _S (x)-\alpha _2dx),
	\end{equation}
	where
	\begin{equation}
	F(x):=\intop _{t_1-x}^\infty  \frac{2\alpha_1 ^2 }{1+2 \alpha_1 }K_{\alpha _1}(z) - \frac{2\alpha _0 ^2 }{1+2 \alpha _0} K_{\alpha _0 }(z) dz=2\intop _{t_1-x}^\infty  \alpha_1 ^3 h_{\alpha _1}(z) - \alpha _0^3 h_{\alpha _0 }(z) dz.
	\end{equation}
	We turn to bound the function $F$. We have  
	\begin{equation}
	    \left| \frac{d}{d\alpha }(\alpha ^3 h_\alpha (z) ) \right| \le 3 \alpha ^2 h_\alpha (z)+\alpha ^3 \left| \frac{d}{d\alpha} h_\alpha (z) \right|  \le C \alpha ^2 z^{-\frac{1}{2}}e^{-c \alpha ^2 z} +C \alpha ^3e^{-c \alpha ^2 z}\le C \alpha ^2 z^{-\frac{1}{2}}e^{-c \alpha ^2 z},
	\end{equation}
	where in the second inequality we used Lemma~\ref{lem:estimate on K:intro} and Claim~\ref{claim:derivative of h} and in the third inequality we changed the values of $C$ and $c$ slightly. Thus,
	\begin{equation}\label{eq:bound on F in stability}
	|F(x)|\le C |\alpha _1 -\alpha _0| \alpha _0^2  \intop _0^{\infty }  z^{-\frac{1}{2}}e^{-c\alpha _0 ^2 z}dz = C \alpha _0 |\alpha _1 -\alpha _0| \intop _0^\infty y^{-\frac{1}{2}} e^{-cy} dy\le  C \alpha _0 |\alpha _1 -\alpha _0|.
	\end{equation}
	
	Next, we use this bound in order to bound the integral on the right hand side of \eqref{eq:integral of F}. To this end, define
	\begin{equation}
	    I_1:=\intop _{t_1^{-}}^{t_1} F(x)  (d \Pi _S (x)-S(x) dx),\quad I_2:= \intop _{t_1^{-}}^{t_1} F(x)  (S (x)-\alpha _2) dx.
	\end{equation}
	We start by bounding $I_1$ using Corollary~\ref{cor:concentration of integral}. We compute the quadratic variation
	\begin{equation}
	    M:=\intop _{t_1^-} ^{t_1}  F(x)^2 S(x) dx \le \alpha _0^2 |\alpha _1-\alpha _0|^2 \intop _{t_1^-} ^{t_1} S(x) dx \le C \alpha _0^4  \beta  _0 ^{13 \theta +C_{\circ}} .
	\end{equation}
	where the last inequality holds on the event $\mathcal A _4 \cap \{ \tau'\ge \hat{t}_0 \}$ using the definition of $\tau_1$. Thus, by Lemma~\ref{lem:ind2:alpha1vsalpha0} and Corollary~\ref{cor:concentration of integral} we have 
	\begin{equation}\label{eq:bound on I _1 in the proof of stability}
	\PP\left( |I_1|\ge  \alpha _0^{2}\beta_0^{6.5\theta+C_\circ },\ \tau'\ge \hat{t}_0 \ | \ \mathcal{A}(\alpha_0,t_0) \right) \le \exp\left(-\beta_0^5 \right).
	\end{equation}
	We turn to bound $I_2$. On the event in Claim~\ref{claim:speed is almost fixed} we have 
	\begin{equation}\label{eq: bound on I_2 in the proof of stability}
	    |I_2|\le C \alpha _0|\alpha _1-\alpha _0 | \intop _{t_1^-} ^{t_1} |S(x)-\alpha _2 | dx \le  C \alpha _0^{2} \beta _0^{8 \theta }  .
	\end{equation}
	Thus, using Claim~\ref{claim:speed is almost fixed} and substituting the bounds on $I_1$ and $I_2$ into \eqref{eq:integral of F} finishes the proof of the lemma. 
\end{proof}

We also record a generalized version of Lemma \ref{lem:ind2:alphaPvsalphaPP} which comes from a simple union bound.

\begin{cor}\label{cor:ind2:Ldiff}
	Under the setting of Theorem \ref{thm:ind2:main}, let $L(\alpha)$ denote
	\begin{equation}
	L(\alpha):= \mathcal{L}(t_1^-; \Pi_S[t_1^{--},t_1^-],\alpha).
	\end{equation}
	Then, we have
	\begin{equation}
	\PP \left(\left. \sup_{\alpha : \, |\alpha-\alpha_0| \le \alpha_0^{\frac{3}{2}}\beta_0^{6\theta} } \left|L(\alpha)-L(\alpha_0)  \right| \ge 2\alpha_0^2 \beta_0^{7\theta}, \ \tau' \ge \hat{t}_0 \, \right| \, \mathcal{A}(\alpha_0,t_0)  \right) \le \exp\left( -\beta_0^4\right).
	\end{equation}
\end{cor}

\begin{proof}
	Let $\delta = \alpha_0^{10}$, and define $\mathcal{D}_{\alpha_0}:= \left\{\alpha: \, \left|\alpha-\alpha_0 \right| \le \alpha_0^{\frac{3}{2}}\beta_0^{6\theta}, \, \alpha -\alpha_0 = k \delta \textnormal{ for some } k\in \mathbb{Z}  \right\}$. Note that Lemma \ref{lem:ind2:alphaPvsalphaPP} controls the difference $|L(\alpha_1)-L(\alpha_0)|$, and we can apply a union bound over the choice of $\alpha_1$, extending the result into 

		\begin{equation}
		\PP \left(\left. \sup_{\alpha \in \mathcal{D}_{\alpha_0} } \left|L(\alpha)-L(\alpha_0)  \right| \ge \alpha_0^2 \beta_0^{7\theta}, \ \tau' \ge \hat{t}_0 \, \right| \, \mathcal{A}(\alpha_0,t_0)  \right) \le \exp\left( -\beta_0^4\right).
		\end{equation}
	Then, note that if $|\alpha- \hat{\alpha}| \le \delta$ for some $\hat{\alpha } \in \mathcal{D}_{\alpha_0}$ and $t_1\le \tau'$,  then $|L(\alpha)-L(\hat{\alpha})| \le \alpha_0^5$, by crudely bounding \eqref{eq:integral of F}, concluding the proof.
\end{proof}

\subsubsection{Estimating the event $\mathcal{A}_1$}\label{subsubsec:ind2:A1}

Recall the definition of $\mathcal{A}_1(\alpha_1,t_1)$ in Section \ref{subsubsec:reg:history}. Our goal in this subsection is to establish the following lemma.

\begin{lem}\label{lem:ind2:A1}
	Under the setting of Theorem \ref{thm:ind2:main}, we have
	\begin{equation}
	\PP\left(\mathcal{A}_1(\alpha_1,t_1)^c,\ \tau'\ge \hat{t}_0\ | \ \mathcal{A}(\alpha_0,t_0) \right) \le
	\exp \left(-\beta_0^3 \right)
	\end{equation}
\end{lem}

To establish this lemma, we study the size of $X_{t_1^{-}}-X_s$ in two separate regime: for $s>t_0^{-}$ and $s \le t_0^{-}$. To analyze the first case, we consider the following event $\mathcal{A}_{1,1}$:
\begin{equation}
\mathcal{A}_{1,1}=\mathcal{A}_{1,1}(\alpha_0,t_0) :=
\left\{\forall \Delta \in [\alpha_0^{-2}\beta_0^{\theta/2}, t_1^{-}-t_0^{-}], \ X_{t_1^{-}}-X_{t_1^{-}-\Delta} \ge \frac{\alpha_0\Delta}{3} \right\}.
\end{equation}

\begin{lem}\label{lem:ind2:A11}
	Under the setting of Theorem \ref{thm:ind2:main}, we have
	\begin{equation}
	\PP \left(\mathcal{A}_{1,1}^c,\ \tau'\ge \hat{t}_0 \ | \ \mathcal{A}(\alpha_0,t_0) \right) \le \exp\left(-\beta_0^4\right). 
	\end{equation}
\end{lem}

\begin{proof}
	This comes as a consequence of Lemma \ref{lem:ind1:tau13} and Corollary \ref{cor:concentration of integral}, applied to
	\begin{equation}
	|\Pi_S[s,t_1^-]|-\intop_{s}^{t_1^-} S(x)dx.
	\end{equation}
	The argument goes analogously as the proof of Lemma \ref{lem:ind1:tau9 10 11}, and the details are omitted due to similarity.
\end{proof}

\begin{proof}[Proof of Lemma \ref{lem:ind2:A1}]
	Suppose that  $s\in [t_1^{-}-\alpha_0^{-2}\beta_0^{\theta/2},t_1^{-} ] $. Then, given the event $\mathcal{A}_4$ from $\eqref{eq:def:A4}$, the condition for $\mathcal{A}_1(\alpha_1,t_1)$ is automatically satisfied, since
	\begin{equation}\label{eq:ind2:A1:pf1}
	\alpha_1^{-1} \beta_1^{\theta} \ge \frac{1}{2}\alpha_0^{-1}\beta_0^{\theta} \ge
	\sqrt{\alpha_0^{-2}\beta_0^{\theta/2}} \log^2(\alpha_0^{-2}\beta_0^{\theta/2})\ge \sqrt{t_1^{-}-s+C_\circ}\log^2 (t_1^{-}-s+C_\circ).
	\end{equation}
	
	Now consider the case $s\in[t_0^{-},t_1^{-}-\alpha_0^{-2}\beta_0^{\theta/2}]$. Conditioned on the event $\mathcal{A}_{1,1}$ from Lemma \ref{lem:ind2:A11}, we have
	\begin{equation}\label{eq:ind2:A1:pf2}
	X_{t_1^{-}}-X_s \ge \frac{1}{3}{\alpha_0(t_1^{-}-s)} \ge \sqrt{t_1^{-}-s+C_\circ} \log^2 (t_1^{-}-s+C_\circ),
	\end{equation}
	which satisfies the criterion for $\mathcal{A}_1(\alpha_1,t_1)$.
	
	Finally, for the case $s\le t_0^{-}$ conditioned on $\mathcal{A}(\alpha_0,t_0)$, we observe that
	\begin{equation}\label{eq:ind2:A1:pf3}
	\begin{split}
	X_{t_1^{-}}-X_s &= X_{t_1^{-}}-X_{t_0^{-}} + X_{t_0^{-}} X_s\\
	&\ge 
	\frac{\alpha_0}{3}(t_1^{-}-t_0^{-}) + \sqrt{t_0^{-}-s+C_\circ}\log^2 (t_0^{-}-s+C_\circ)-\alpha_0^{-1}\beta_0^{\theta/2}\\
	&\ge 
	\frac{\alpha_0}{4}(t_1^{-}-t_0^{-})+ \sqrt{t_0^{-}-s+C_\circ}\log^2 (t_0^{-}-s+C_\circ)\\
	&\ge 
	\sqrt{t_1^{-}-t_0^{-}+C_\circ}\log^2 (t_1^{-}-t_0^{-}+C_\circ)
	+
	\sqrt{t_0^{-}-s+C_\circ}\log^2 (t_0^{-}-s+C_\circ)\\
	&\ge \sqrt{t_1^{-}-s+C_\circ}\log^2 (t_1^{-}-s+C_\circ),
	\end{split}
	\end{equation}
	where the last step follows from the fact that the function $f(x)=\sqrt{x+50}\log^2(x+50)$ is positive, increasing and concave.
	
	Thus, from \eqref{eq:ind2:A1:pf1}, \eqref{eq:ind2:A1:pf2} and \eqref{eq:ind2:A1:pf3}, we see that $\mathcal{A}_1(\alpha_1,t_1)$ holds if $\mathcal{A}(\alpha_0,t_0)$, $\mathcal{A}_{1,1}$ and $\mathcal{A}_4$ are given. Thus, due to Lemmas \ref{lem:ind2:alpha1vsalpha0} and \ref{lem:ind2:A11}, we obtain the conclusion.
\end{proof}

\subsubsection{Estimating the event $\mathcal{A}_2$}\label{subsubsec:ind2:A3}

Recall the definition of $\mathcal{A}_2(\alpha_1,t_1)$ in Section \ref{subsubsec:reg:history}. In  this  subsection, we show that   $\mathcal{A}_2(\alpha_1,t_1)$ occurs with high probability.

\begin{lem}\label{lem:ind2:A3}
	Under the setting of Theorem \ref{thm:ind2:main}, we have
	\begin{equation}
	\PP\left(\mathcal{A}_2(\alpha_1,t_1)^c,\ \tau'\ge \hat{t}_0\ | \ \mathcal{A}(\alpha_0,t_0) \right) \le \exp\left(-\beta_0^3 \right)
	\end{equation}
\end{lem}

\begin{proof}
This is a direct consequence of Lemma \ref{lem:ind1:tau13} and Corollary \ref{cor:concentration of integral}, following the same argument discussed in Lemma \ref{lem:ind1:tau9 10 11} (see also Lemma \ref{lem:ind2:A11}). We omit its details due to similarity to the previous proofs.
\end{proof}

Wrapping up the discussion in Section \ref{subsec:ind2:ratechange}, we finish the proof of Theorem \ref{thm:ind2:main}-(2).

\begin{proof}[Proof of Theorem \ref{thm:ind2:main}-(2)]
	Combining Lemmas \ref{lem:ind2:ratechange}, \ref{lem:ind2:A1} and \ref{lem:ind2:A3}, we obtain that 	
	\begin{equation}
	\PP\left(\mathcal{A}(\alpha_1,t_1)^c, \ \tau'\ge\hat{t}_0\ | \ \mathcal{A}(\alpha_0,t_0) \right) \le \exp\left(-\beta_0^2\right),
	\end{equation}
	which directly implies Theorem \ref{thm:ind2:main}.
\end{proof}

\subsection{Regularity of speed in the next step}\label{subsec:ind2:remaining reg}

Recall the definition of $\tau(\alpha_1,t_1,2)$, ${\tau}^+(\alpha_0,t_0,1/2)$, $\mathcal{A}(\alpha_1,t_1)$ and $\mathcal{A}(\alpha_0,t_0)$ given in \eqref{eq:def:tau:induction} and \eqref{eq:def:A:induction}. Our goal in this subsection is to deduce the desired control on \begin{equation}
\tilde{ \tau} := \tau(\alpha_1,t_1,2)
\end{equation} 
(excluding $\tau_{3}^\sharp(\alpha_1,t_1)$), which is given by the following proposition.

\begin{prop}\label{prop:ind2:prereg}
	Under the setting of Theorem \ref{thm:ind2:main}, the following holds true:
	$$\PP(\tilde{ {\tau}}\le \tau' ,\ \tau'\ge \hat{t}_0 \ | \ \mathcal{A}(\alpha_0,t_0) ) \le \exp(-\beta_0^{2}) .$$
\end{prop}

Recall the event $\mathcal{A}_4$ in \eqref{eq:def:A4}. On the event $\{\tau' \ge \hat{t}_0 \}\cap \mathcal{A}_4$, observe that the following hold true with probability 1:
\begin{itemize}
	\item  $\tau_{1}(\alpha_1,t_1,2)\ge {\tau}^+_1(\alpha_0,t_0,1/2) \ge \tau'$, since $\frac{1}{2}\alpha_0 \beta_0^{C_\circ} \le  2\alpha_1 \beta_1^{C_\circ}$.
	
	\item  $\tau_{7}(\alpha_1,t_1,2)\ge {\tau}^+_{7}(\alpha_0,t_0,1/2)\ge \tau'$, since for any $t\le {\tau}^+_{7}(\alpha_0,t_0,1/2)$, \begin{equation}
	X_t - X_{t-\alpha_1^{-1}} \le X_t - X_{t-2\alpha_0^{-1}} \le 2\cdot \frac{1}{2} \beta_0^4 \le 2\beta_1^4.
	\end{equation}
	
	\item  $\tau_{8}(\alpha_1,t_1,2) \ge {\tau}^+_{8}(\alpha_0,t_0,1/2) \ge \tau'$, since for any $t\le {\tau}^+_{8}(\alpha_0,t_0,1/2)$, we have
	\begin{equation}
	X_t - X_{t-2\alpha_1^{-1} \beta_1^{C_\circ}} 
	\ge X_t - X_{t-\frac{1}{2}\alpha_0^{-1 } \beta_0^{C_\circ}} >0.
	\end{equation}
\end{itemize}

Moreover, conditioned on $\{\tau' \ge \hat{t}_0 \}\cap \mathcal{A}_4$, with probability 1 we have 
\begin{equation}
\tau_{4}(\alpha_1,t_1,2) \wedge \tau_{5} (\alpha_1,t_1,2) \ge {\tau}^+_4(\alpha_0,t_0,1/2) \wedge {\tau}^+_5 (\alpha_0,t_0,1/2)\ge \tau',
\end{equation}
since for any $ t\le {\tau}^+_1(\alpha_0,t_0,1/2) {\tau}^+_4(\alpha_0,t_0,1/2) \wedge {\tau}^+_5 (\alpha_0,t_0,1/2) $, 
\begin{equation}
\begin{split}
\intop_{t_1}^t (S(t)-\alpha_1)^2 dt
&\le \intop_{t_1}^t 2(S(t)-\alpha_0)^2 dt + 2(\hat{t}_0-t_1) (\alpha_1-\alpha_0)^2
\\
&\le 2\cdot \frac{1}{2}\alpha_0 \beta_0^{25\theta} + 2\cdot 2\alpha_0^{-2}\beta_0^{10\theta}\cdot 4\alpha_0^3 \beta_0^{12\theta}
\le 2\alpha_1 \beta_1^{25\theta},
\end{split}
\end{equation}
where the last inequality followed from the definitions of ${\tau}^+_4(\alpha_0,{t_0},1/2)$ and $\mathcal{A}_4$.
The integral in the definition of $\tau_{5}(\alpha_1,t_1,2)$ follows similarly, using the definition of ${\tau}^+_1(\alpha_0,t_0,1/2)$ together. Wrapping up the above discussion, $\tilde{ {\tau}}^{(1)} := \tau_{1}\wedge\tau_{4}\wedge \tau_{5} \wedge \tau_{7}\wedge \tau_{8}(\alpha_1,t_1,2) $, and obtain that
\begin{equation}\label{eq:ind2:tau4781911}
\PP\left(
\tilde{ \tau}^{(1)} \le \tau', \ \tau'\ge \hat{t}_0 \ | \ \mathcal{A}(\alpha_0,t_0)
\right) \le \exp\left(-\beta_0^4 \right).
\end{equation}

Thus, the remaining task is to control $\tau_{2}, \tau_{3}, $ and $\tau_{6}$. We conduct this separately in Lemmas \ref{lem:ind2:tau5}, \ref{lem:ind2:tau6} and \ref{lem:ind2:tau9} below. We begin with studying $\tau_{2}.$

\begin{lem}\label{lem:ind2:tau5}
	Under the setting of Theorem \ref{thm:ind2:main}, we have
	\begin{equation}
	\PP \left(\tau_{2}(\alpha_1,t_1,2) \le \tau',\ \tau'\ge \hat{t}_0\ | \ \mathcal{A}(\alpha_0,t_0) \right) \le \exp\left( -\beta_0^4\right).
	\end{equation}
\end{lem}

\begin{proof}
	In the proof, we set 
	\begin{equation}
	\begin{split}
	S_1(t)=S_1(t,t_1^{-},\alpha_1) := \mathcal{R}_b (t;\Pi_S[t_1^{-},t],\alpha_1 ).
	\end{split}
	\end{equation}
	We can bound this quantity very similarly as in Lemma \ref{lem:ind1:tau5} as follows. We use the definition of ${\tau}^+_{7}(\alpha_0,t_0,1/2)$ to see that for all $t\in[t_1,\tau']$,
	\begin{equation}
	\begin{split}
	S_1(t)&= \intop_{t_1^{-}}^{t } K_{\alpha_1}(t-x) d\Pi_S(x)\\
	&\le 
	\sum_{k: 0\le k\alpha_0^{-1} \le 2\alpha^{-2}\beta^{10\theta} } 
	\frac{\beta_0^4}{2} \cdot \frac{C\alpha_1}{\sqrt{k\alpha_0^{-1}+1}} e^{-c\alpha_1^2 (k\alpha_0^{-1})} \le \alpha_1 \beta_1^{C_\circ}, 
	\end{split}
	\end{equation} 
	where the last inequality holds if the event $\mathcal{A}_4$ is given. Thus, thanks to Lemma \ref{lem:ind2:alpha1vsalpha0}, we obtain the conclusion.
\end{proof}

\begin{lem}\label{lem:ind2:tau6}
	Under the setting of Theorem \ref{thm:ind2:main}, we have
	\begin{equation}
	\PP \left(\tau_{3}(\alpha_1,t_1,2) \le \tau',\ \tau'\ge \hat{t}_0\ | \ \mathcal{A}(\alpha_0,t_0) \right) \le 2\exp\left( -\beta_0^2\right).
	\end{equation}
\end{lem}

\begin{proof}
	The proof is analogous as Lemma \ref{lem:ind1:tau6}, using the definitions of $\tau_{4}(\alpha_1,t_1,2)$ and $\tau_{5}(\alpha_1,t_1,2),$ and the estimate \eqref{eq:ind2:tau4781911}.
\end{proof}

\begin{lem}\label{lem:ind2:tau9}
	Under the setting of Theorem \ref{thm:ind2:main}, we have
	\begin{equation}
	\PP \left(\tau_{6}(\alpha_1,t_1,2) \le \tau',\ \tau'\ge \hat{t}_0\ | \ \mathcal{A}(\alpha_0,t_0) \right) \le \exp\left( -\beta_0^3\right).
	\end{equation}
\end{lem}

\begin{proof}
This comes as a consequence of Lemma \ref{lem:ind1:tau13} and Corollary \ref{cor:concentration of integral}, applied to 
\begin{equation}
|\Pi_S[t_1^-,t]| - \intop_{t_1^-}^t S(x)dx.
\end{equation}
We also rely on $\mathcal{A}_4$ to switch $\alpha_0$ to $\alpha_1$. The rest of the proof is analogous to that of Lemma \ref{lem:ind1:tau9 10 11}, and we omit the details due to similarity (see also Lemmas \ref{lem:ind2:A11} and \ref{lem:ind2:A3}).
\end{proof}

We conclude Section \ref{subsec:ind2:remaining reg} by proving Proposition \ref{prop:ind2:prereg}.

\begin{proof}[Proof of Proposition \ref{prop:ind2:prereg}]
	It follows from combining \eqref{eq:ind2:tau4781911}, Lemmas \ref{lem:ind2:tau5}, \ref{lem:ind2:tau6} and \ref{lem:ind2:tau9}.
\end{proof}

\subsection{Refined control on the first order approximation}\label{subsec:ind2:bootstrappedJ}

In this subsection, we provide a refined analysis of the first order approximation, deducing an appropriate control on $\tau_{3}^\sharp$ (Section \ref{subsubsec:reg:refined1storder}). The main result we aim to establish is the following.

\begin{lem}\label{lem:ind2:tau6sharpreal}
	Under the setting of Theorem \ref{thm:ind2:main} and Proposition  \ref{prop:ind2:prereg}, we have
	\begin{equation}
	\PP \left(\tau_{3}^\sharp(\alpha_1,t_1 ) \le t_1^\sharp, \ \tilde{\tau}\wedge \tau' \ge \hat{t}_0 \ | \ \mathcal{A}(\alpha_0,t_0) \right) \le 10\exp\left(-\beta_0^2 \right).
	\end{equation}
\end{lem}

Having this on hand, we can deduce Theorem  \ref{thm:ind2:main} by combining the results obtained so far.
\begin{proof}[Proof of Theorem \ref{thm:ind2:main}]
    Theorem~\ref{thm:ind2:main}-(1) follows directly from combining Proposition \ref{prop:ind2:prereg}, Lemmas \ref{lem:ind2:tau1sharp} and \ref{lem:ind2:tau6sharp}. The second part of the theorem was already established in Section \ref{subsec:ind2:ratechange}.
\end{proof}

$\tau_{3}^\sharp$ studies the integral of $|S(t)-S_1(t)|$, and our approach is to first  estimate $|S(t)-S_1(t)|$ itself.  To this end, we define the stopping time $\tau_{{3,1}}^\sharp$ as
\begin{equation}
\tau_{{3,1}}^\sharp(\alpha_0,t_0)
:=
\inf \left\{t\ge t_0: |S(t)-S_1(t)| \ge \alpha_0 \beta_0^{3\theta} \sigma_1(t;S)^2 + \alpha_0^{1-\epsilon} \sigma_1\sigma_2\sigma_3(t;S) \right\}.
\end{equation}

\begin{lem}\label{lem:ind2:tau6sharp}
	Under the setting of Theorem \ref{thm:ind2:main} and Proposition  \ref{prop:ind2:prereg}, we have
	\begin{equation}
	\PP \left(\tau_{{3,1}}^\sharp(\alpha_1,t_1 ) \le t_1^\sharp, \ \tilde{\tau}\wedge \tau' \ge \hat{t}_0 \ | \ \mathcal{A}(\alpha_0,t_0) \right) \le 10\exp\left(-\beta_0^2 \right).
	\end{equation}
\end{lem}

Assuming Lemma \ref{lem:ind2:tau6sharp}, we conclude the proof of Lemma \ref{lem:ind2:tau6sharpreal}.

\begin{proof}[Proof of Lemma \ref{lem:ind2:tau6sharpreal}]
	We set
	$\tau_{\textnormal{int}}$ to be
	\begin{equation}
	\tau_{\textnormal{int}} := \inf\left\{ t\ge t_1: \intop_{t_1}^t \sigma_1\sigma_2\sigma_3(t;S) dt \ge \alpha^{-\frac{1}{2}-\epsilon} \right\}.
	\end{equation}
	We also recall the definition of  $\tau'$ in Theorem \ref{thm:ind2:main}, and set
	\begin{equation}
	\tilde{\tau}_3^\sharp := \tau_{3,1}^\sharp\wedge \tau' \wedge \tau_{\textnormal{int}} .
	\end{equation}
	Then, Lemmas \ref{lem:fixed perturbed:error int:perturbed},  \ref{lem:ind2:tau6sharp} and Proposition \ref{prop:ind2:prereg} tell us that 
	\begin{equation}
	\PP \left( \tilde{\tau}_{3}^\sharp \le t_1^\sharp , \ \tau' \ge \hat{t}_0 \, | \,\mathcal{A}(\alpha_0,t_0) \right) \le \exp \left(-\beta_0^2 \right).
	\end{equation}
	The proof of the desired statement follows by observing that
	\begin{equation}
	\begin{split}
\intop_{t_1}^{\tilde{\tau}_3^\sharp} |S(t)-S_1(t)|dt \le 	\intop_{t_1}^{\tilde{\tau}_3^\sharp} \left\{\frac{\alpha_0\beta_0^{3\theta}}{\pi_1(t;S)+1} + \alpha_0^{1-\epsilon} \sigma_1\sigma_2\sigma_3(t;S) \right\} dt <\beta_1^{4\theta},
	\end{split}
	\end{equation}
	where we used Lemma \ref{lem:ind1:pi1int basicbd:basic} and $\tau_7$ to deduce the last inequality, relying on $\mathcal{A}_4$ \eqref{eq:def:A4} to show that $\beta_0$ and $\beta_1$ are close.
\end{proof}

The rest of this subsection is devoted to the proof of Lemma \ref{lem:ind2:tau6sharp}. Unlike when we studied ${\tau}^+_3$ in Lemmas \ref{lem:ind1:tau6} and \ref{lem:ind2:tau6}, we can no longer work with the first-order approximation of the speed. Thus, the second-order approximation  \eqref{eq:def:S2:basic form} is required. However, for technical reasons (which will be clarified in Remark \ref{rmk:ind2:tau6sh:frame shift}), we will read the process from time $t_1^{--}$ given by
\begin{equation}
t_1^{--}:= t_1 -2\alpha_0^{-2}\beta_0^\theta = t_1^- - \alpha_0^{-2}\beta_0^\theta,\quad t_1^\flat := t_1 - \frac{1}{3}\alpha_0^{-2}\beta_0^\theta,
\end{equation}
with respect to the same frame of reference $\alpha_1$.
Recall that $\alpha_1$ is measureable with respect to $\mathcal{F}_{t_1^-}.$ In what follows, we first introduce the setting we use in this subsection, which the framework in terms of $\alpha_1$ and $t_1^-$.

To begin with, the second-order approximation that we work with in this subsection is defined as 
\begin{equation}\label{eq:def:S2:ind2}
S_2(t)=S_2(t;t_1^{--},\alpha_1) := \frac{2\alpha_1^2}{(1+2\alpha_1)^2} + \frac{S_1(t)}{(1+2\alpha_1)^2} + \frac{\alpha_1}{(1+2\alpha_1)} \intop_{t_1^-}^t \intop_{t_1^-}^{s} J_{t-s, t-u}^{(\alpha_1)} d\widehat{\Pi}_{S}(u) d\widehat{\Pi}_S(s),
\end{equation}
where 
$S_1(t)$ in this subsection is given by
\begin{equation}\label{eq:def:S1:ind2}
\begin{split}
S_1(t)= S_1(t,t_1^{--},\alpha_1)
=
\intop_{ t_1^{--}}^t K_{\alpha_1}(t-x) d\Pi_S(x),
\end{split}\end{equation}
and $J_{u,s}^{(\alpha_1)}$ is as \eqref{eq:def:J:basic form} and \eqref{eq:def:J:expected val}. We also remark that  $d\widehat{\Pi}_S(x)= d\Pi_S(x)-\alpha_1 dx$.

Recall that Lemma \ref{lem:ind2:alpha1vsalpha0} implies that the event
$\mathcal{A}_4:= \{|\alpha_1-\alpha_0|\le 2\alpha_0^{3/2}\beta_0^{6\theta}\}$ 
happens with high probability if $\Pi_S(-\infty,t_0]$ was regular, and it is $\mathcal{F}_{t_1^-}$-measurable. In such a case, $(t_1^-)^-:= t_1^{--}$, $(t_1^-)^+:= t_1^\flat$ satisfies \eqref{eq:def:t0t1} in terms of $t_1^-$ and $\alpha_1$, and thus we can redefine $\tau(\alpha_1, t_1^-,\kappa)$ \eqref{eq:def:tau:induction} and $\mathcal{A}(\alpha_1, t_1^-)$ \eqref{eq:def:A:induction} in terms of $(t_1^-)^- , t_1^-$ and $ S_1(t)$. Furthermore, the stopping time $\tau_{\textnormal{b}}'(\alpha_1,t_1^-,{t_1^\flat})$ from \eqref{eq:def:tauxprime} can be defined in the same way, and it will play an important role in this subsection. Since the choice of $t_1$ is flexible in Theorem \ref{thm:ind2:main}, the results from  Theorem \ref{thm:ind2:main}-(2), Proposition \ref{prop:ind2:prereg} and Corollary \ref{cor:ind1:tauxprime} tell us the following.

\begin{cor}\label{cor:ind2:prereg:past}
	Under the setting of Theorem \ref{thm:ind2:main}, let $\alpha_1 , $ $t_1^-, (t_1^-)^-=t_1^{--}, (t_1^-)^+ = t_1^\flat$, $\tau(\alpha_1,t_1^-,\kappa)$, $\tau_{\textnormal{b}}'(\alpha_1,t_1^-,t_1^\flat)$ and $\mathcal{A}(\alpha_1,t_1^-) $ be as above. Then, we have
	\begin{itemize}
		\item[(1)] $\PP \left( \tau(\alpha_1,t_1^-,2) <\hat{t}_0 \ | \ \mathcal{A}(\alpha_0,t_0) \right) \le \PP\left(\tau' <\hat{t}_0 \ | \ \mathcal{A}(\alpha_0,t_0) \right) + 4 \exp (-\beta_0^2) $;
		
		\item[(2)] $\PP \left( \tau_{\textnormal{b}}'(\alpha_1,t_1^-,t_1^\flat) \le t_1^\sharp \ | \ \mathcal{A}(\alpha_0,t_0) \right) \le \PP\left(\tau' <\hat{t}_0 \ | \ \mathcal{A}(\alpha_0,t_0) \right) + 4 \exp (-\beta_0^2) $;
		
		\item[(3)] $\PP \left(\mathcal{A}(\alpha_1,t_1^-)^c \, | \, \mathcal{A}(\alpha_0,t_0) \right) \le \PP \left(\tau'<\hat{t}_0 \, | \, \mathcal{A}(\alpha_0,t_0) \right) + \exp (-\beta_0^2)$.
	\end{itemize}
\end{cor}

Before moving on, we clarify the role of $\tau_{\textnormal{b}}'(\alpha_1,t_1^-, t_1^\flat )$. It tells us the bound on $|S_1(t)- \alpha_1''|$, where we denoted
$$\alpha_1'' := \mathcal{L}(t_1^-; \Pi_S[t_1^{--},t_1^-], \alpha_1 ). $$
On the event $\mathcal{A}_4$ where $|\alpha_1-\alpha_0|\le \alpha_0^{\frac{3}{2}}\beta_0^{6\theta}$, we have with high probability that
\begin{equation}
\begin{split}
|\alpha_1'' - \alpha_1| \le& |\mathcal{L}(t_1^-;\Pi_S[t_1^{--},t_1^-],\alpha_1) -\mathcal{L}(t_1^-;\Pi_S[t_1^{--},t_1^-],\alpha_0)| 
\\ 
&+ |\mathcal{L}(t_1^-;\Pi_S[t_1^{--},t_1^-],\alpha_0) - \mathcal{L}(t_1^-;\Pi_S[t_0^{-},t_1^-],\alpha_0)|\\
\le & \alpha_0^2 \beta_0^{7\theta},
\end{split}
\end{equation}
where we bounded two terms using Lemma \ref{lem:ind2:alphavsalphaPP} and Corollary \ref{cor:ind2:Ldiff}. Note that Corollary \ref{cor:ind2:Ldiff} is necessary (instead of Lemma \ref{lem:ind2:alphaPvsalphaPP}) since $\alpha_1$ is not $\mathcal{F}_{t_1^{--}}$-measurable. Thus, combining with the estimates on $\tau_{\textnormal{b}}'(\alpha_1,t_1^-,t_1^{\pm})$, we have control on the stopping time $\tau_{\textnormal{b}}''$ given by
\begin{equation}\label{eq:def:tauxpprime}
\tau_{\textnormal{b}}'':=\inf \left\{ t
\ge t_1^{\flat} :\, |S_1(t)-\alpha_1| \le \alpha_0 \beta_0^{C_\circ} \sigma_1(t;S) + 2\alpha_0^{\frac{3}{2}}\beta_0^\theta  \right\},
\end{equation}
in such a way that
\begin{equation}\label{eq:ind2:tauxpprime}
 \PP ( \tau_{\textnormal{b}}'' \le t_1^\sharp \, | \, \mathcal{A}(\alpha_0,t_0) ) \le \PP(\tau' <\hat{t}_0 \,|\, \mathcal{A}(\alpha_0,t_0)) + \exp (-\beta_0^2) . 
\end{equation}

Moving on to estimating $|S(t)-S_1(t;t_1^-,\alpha_1)|$ in $\tau_{{3,1}}^\sharp(\alpha_1,t_1)$, we write
\begin{equation}\label{eq:ind2:SminS1decomp}
\begin{split}
|S(t)-S_1(t;t_1^-,\alpha_1)|
\le& \,
|S(t)-S'(t;t_1^{--},\alpha_1)| +|S'(t;t_1^{--},\alpha_1)-S_2(t)|\\
& + |S_1(t)-S_2(t)| + |S_1(t;t_1^-,\alpha_1)-S_1(t)| 
\end{split}
\end{equation}
If $\Pi_S(-\infty,t_1^-]\in\mathcal{A}(\alpha_1,t_1^-)$ (which happens with high probability due to (3) of Corollary \ref{cor:ind2:prereg:past}), Proposition \ref{prop:SvsSprime} tells us that
\begin{equation}\label{eq:ind2:SvsSprime}
|S(t)- S'(t;t_1^{--},\alpha_1)|\le \alpha_1^{100}.
\end{equation}
Also, if $\mathcal{A}_4$ is given, then for any $t_1^\flat \le t\le \tau'$ we have 
\begin{equation}\label{eq:ind2:S1vsS1minus}
|S_1(t;t_1^-,\alpha_1)-S_1(t)| =\intop_{t_1^{--}}^{t_1^-} K_{\alpha_1}(t-x) d\Pi_S(x) \le e^{-c\beta_0^\theta} |\Pi_S[t_1^{--},t_1^-]| \le \alpha_1^{100},
\end{equation}
due to the decay property of $K_{\alpha}(s)$.
Moreover, note that the assumptions of Proposition \ref{prop:fixed perturbed:error} are satisfied with $t^- = t_1^{--}$, $\hat{t} = \hat{t}_0$, $\alpha=\alpha_1$, and $\tau = \tau(\alpha_1,t_1^-,2)$. Hence, it tells us that
\begin{equation}\label{eq:ind2:SprimevsS2}
\begin{split}
\PP \left(\left. |S'(t;t_1^{--},\alpha_1)-S_2(t)| \le 2\alpha_1^{1-\frac{\epsilon}{2}}\sigma_1\sigma_2\sigma_3(t;S), \ \forall t\in[t_1^-,\, \hat{t}_0\wedge \tau(\alpha_1,t_1^-,2)] \,\right| \, \mathcal{A}(\alpha_0,t_0) \right)\\
\ge
1-\exp\left(- \beta_0^4 \right).
\end{split}
\end{equation}
Note that  we again used $\mathcal{A}_4$ (Lemma \ref{lem:ind2:alpha1vsalpha0}) to claim that $\alpha_1$ is close enough to $\alpha_0$, enabling us to exchange $\alpha_1$ with $\alpha_0$ (in RHS). 

Therefore, among the terms in \eqref{eq:ind2:SminS1decomp}, what remains is estimating $|S_2(t)-S_1(t)|$.  It can written as
\begin{equation}\label{eq:ind2:S2vsS1}
S_2(t)-S_1(t)
=
\frac{2\alpha_1^2}{(1+2\alpha_1)^2} - \frac{4\alpha_1+4\alpha_1^2}{(1+2\alpha_1)^2} S_1(t)+ \frac{\alpha_1}{(1+2\alpha_1)} \intop_{t_1^-}^t \intop_{t_1^-}^{s} J_{t-s, t-u}^{(\alpha_1)} d\widehat{\Pi}_{S}(u) d\widehat{\Pi}_S(s).
\end{equation}
Here, the first two terms $\frac{2\alpha_1^2}{(1+2\alpha_1)^2}$ and $\frac{4\alpha_1+4\alpha_1^2}{(1+2\alpha_1)^2}S_1(t)$ are going to absorbed by $\alpha_1\beta_1^{3\theta} \sigma_1(t;S)^2$ in $\tau_{3,1}^\sharp$, due to the definitions of $\tau_{2}(\alpha_1,t_1^-,2)$, $\tau_{8}(\alpha_1,t_1^-,2)$, and $\mathcal{A}_4$ \eqref{eq:def:A4}. Thus, our main duty now is to investigate the double integral, the third term in RHS.

To this end, we recall that $t_1^\flat=t_1-\frac{1}{3}\alpha_0^{-2}\beta_0^\theta$, and we attempt to switch $t_1^-$ in the double integral to $t_1^\flat$, which will give us a significant technical advantage (see Remark \ref{rmk:ind2:tau6sh:frame shift}). We claim that if we have $ |\alpha_1-\alpha_0| \le 2\alpha_0^{\frac{3}{2}}\beta_0^{6\theta}$, then for any $t_1\le t\le \tau'$  
\begin{equation}\label{eq:ind2:tau6sh:switch base}
\begin{split}
\left|\intop_{t_1^-}^t \intop_{t_1^-}^{s} J_{t-s, t-u}^{(\alpha_1)} d\widehat{\Pi}_{S}(u) d\widehat{\Pi}_S(s) - \intop_{t_1^\flat}^t \intop_{t_1^\flat}^{s} J_{t-s, t-u}^{(\alpha_1)} d\widehat{\Pi}_{S}(u) d\widehat{\Pi}_S(s)\right|\\
= \left|\intop_{t_1^-}^t \intop_{t_1^-}^{s\wedge t_1^\flat} J_{t-s, t-u}^{(\alpha_1)} d\widehat{\Pi}_{S}(u) d\widehat{\Pi}_S(s) \right| \le
\alpha_1^{50}.
\end{split}
\end{equation}
This is becuase for any $u,s,t$ such that $t_1^-\le u \le t_1^\flat$, and $t\ge t_1$, we have $t - u \ge \frac{1}{3}\alpha_0^{-2}\beta_0^{\theta}$, and hence the estimate on $J$ (Lemma \ref{lem:bound on deterministic J}) implies
\begin{equation}
\left| J_{t-s,t-u}^{(\alpha_1)} \right| \le \alpha_1^{100}.
\end{equation}
Combined with the definitions of $\tau_{1}, \tau_{7}$, we obtain \eqref{eq:ind2:tau6sh:switch base}.

Thus,   it suffices to consider the following stopping time:
\begin{equation}
\begin{split}
\tau_{{3,2}}^\sharp &= \tau_{{3,2}}^\sharp(\alpha_1,t_1;\alpha_0)\\
&:=
\inf \left\{t\ge t_1: \left| \intop_{t_1^\flat}^t \intop_{t_1^\flat}^{s} J_{t-s, t-u}^{(\alpha_1)} d\widehat{\Pi}_{S}(u) d\widehat{\Pi}_S(s)\right| \ge \beta_0^{2\theta+4C_\circ} \sigma_1(t;S) \right\}.
\end{split}
\end{equation}

\begin{remark}\label{rmk:ind2:tau6sh:frame shift}
	In the definition of $\tau_{{3,2}}^\sharp$, note that it reads the process starting from $t_1^\flat$, not $t_1^-$. In controlling the double integral in $\tau_{{3,2}}^\sharp$, we will rely on the bound on $|S_1(t)-\alpha_1|$, which works for $t\ge t_1^\flat$, due to the definition of $\tau_{\textnormal{b}}''$. This is the main reason why we shift the interval of interest to $[t_1^{\flat}, t]$ from $[t_1^-,t]$ and consider an estimate such as \eqref{eq:ind2:tau6sh:switch base}.
\end{remark}

\begin{proof}[Proof of Lemma \ref{lem:ind2:tau6sharp}]
	Summing up the discussions from \eqref{eq:ind2:SminS1decomp} to \eqref{eq:ind2:tau6sh:switch base}, the proof follows directly from Lemma \ref{lem:ind2:tau61sharp} below.
\end{proof}

\begin{lem}\label{lem:ind2:tau61sharp}
	Under the setting of Theorem \ref{thm:ind2:main} and Proposition  \ref{prop:ind2:prereg}, we have
	\begin{equation}
	\PP\left(\tau_{{3,2}}^\sharp \le t_1^\sharp, \ \tilde{\tau}\wedge \tau' \ge \hat{t}_0\ | \ \mathcal{A}(\alpha_0,t_0)  \right) \le 3 \exp\left(-\beta_0^2 \right).
	\end{equation}
\end{lem}

Proof of Lemma \ref{lem:ind2:tau61sharp} turns out to be more involved than the methods used in Section \ref{sec:fixedrate}, since we now do not allow the error to be of size $\alpha_1^{-\epsilon}$. We rather analyze the double integral directly, relying on the fact that the length scale of $t_1^\sharp-t_1^\flat $, which is roughly $\alpha_0^{-2}\beta_0^\theta$, is much smaller than $\alpha_0^{-2}\beta_0^{10\theta}$.

To begin with, we first study the inner integral. Define the stopping time $\tau_{{3,3}}^\sharp$ as
\begin{equation}\label{eq:def:tau63sharp}
\tau_{{3,3}}^\sharp:= \inf\left\{ t\ge t_1: \,\exists s\in[t_1^\flat,t] \textnormal{ s.t. } \left|\intop_{t_1^\flat}^s J^{(\alpha_1)}_{t-s,t-u} d\widehat{ \Pi}_S(u) \right| \ge \frac{\alpha_1^{\frac{1}{2}}\beta_1^{\theta+3 }}{\sqrt{t-s+1}} + \frac{\beta_1^{C_\circ +1} }{t-s+1} \right\}.
\end{equation}

\begin{lem}\label{lem:ind2:tau63sharp}
	Under the setting of Theorem \ref{thm:ind2:main} and Proposition \ref{prop:ind2:prereg}, we have
	\begin{equation}
	\PP\left(\tau_{{3,3}}^\sharp \le t_1^\sharp, \ \tilde{\tau}\wedge \tau'\ge \hat{t}_0\ | \ \mathcal{A}(\alpha_0,t_0)  \right) \le 3\exp\left(-\beta_0^2 \right).
	\end{equation}
\end{lem}

Recall the definitions $S_1(t)$ and $R(t)$.
In order to study the integrals over $d\widehat{\Pi}_S(x)$ which appear in Lemmas \ref{lem:ind2:tau61sharp} and \ref{lem:ind2:tau63sharp}, we decompose it as follows:
\begin{equation}\label{eq:ind2:tausharp:tilPiSdecomposition}
\begin{split}
d\widehat{\Pi}_S (x) = ~d \widetilde{\Pi}_S(x) + \big[S(x)-S_1(x) \big] + \big[S_1(x)-\alpha_1 \big],
\end{split}
\end{equation}
where we remind that $d\widetilde{\Pi}_S(x)=d\Pi_S(x)-S(x)dx$.
We will rely on this formula when establishing both Lemmas \ref{lem:ind2:tau61sharp} and \ref{lem:ind2:tau63sharp}.

In the remaining of Section \ref{subsec:ind2:bootstrappedJ}, we first prove Lemma \ref{lem:ind2:tau63sharp} in Section \ref{subsubsec:ind2:tau63sharp}, and then establish Lemma \ref{lem:ind2:tau61sharp} in Section \ref{subsubsec:ind2:tau61sharp}. 

\subsubsection{Proof of Lemma \ref{lem:ind2:tau63sharp}}\label{subsubsec:ind2:tau63sharp}
We decompose the integral given in the definition of $\tau^\sharp_{3,3}$ into four parts using the decomposition \eqref{eq:ind2:tausharp:tilPiSdecomposition}, and then we study the following lemmas in order to show Lemma \ref{lem:ind2:tau63sharp}. Recall the definition of the event $\mathcal{A}_4$ given in \eqref{eq:def:A4}, and set
\begin{equation}\label{eq:def:ind2:Aprime}
\mathcal{A}':= \mathcal{A}(\alpha_0,t_0)\cap \mathcal{A}_4.
\end{equation}
Moreover, recall the definitions of  $\tau_{\textnormal{b}}''$ \eqref{eq:def:tauxpprime}  and set
\begin{equation}\label{eq:def:tauhat:ind2}
\hat{\tau}:= \tilde{\tau} \wedge \tau_{\textnormal{b}}'' .
\end{equation}
We note that Lemmas \ref{lem:ind2:tau63sharp:1} and \ref{lem:ind2:tau63sharp:2} below hold the same with the weaker stopping time $\tilde{\tau}$ instead of $\hat{\tau}$. However, we state them in terms of $\hat{\tau}$ to be consistent with Lemma \ref{lem:ind2:tau63sharp:4}.

\begin{lem}\label{lem:ind2:tau63sharp:1}
	Assume the setting of Theorem \ref{thm:ind2:main}. For each  $t\in [t_1,t_1^\sharp]$  and $s\in[t_1^\flat,t]$, define the event \begin{equation}
	\mathcal{B}_1(s,t):= \left\{ \left| \intop_{t_1^\flat \wedge\hat{\tau}}^{s\wedge \hat{\tau}} J^{(\alpha_1)}_{t-s,t-u} d\widetilde{\Pi}_S(u) \right|
	\le \left(\frac{\beta_0^{C_\circ+1}}{t-s+1} +\frac{\alpha_0^{\frac{1}{2}} \beta_0^{2C_\circ} }{\sqrt{t-s+1} }
	\right) e^{-c\alpha_0^2(t-s)} \right\}.
	\end{equation}
	Then, we have
	\begin{equation}
	\PP\left(\left. \bigcap_{t\in [t_1, t_1^\sharp] } \bigcap_{s\in [t_1^\flat,t] }\mathcal{B}_1(s,t)  \ \right| \ \mathcal{A}' \right) \ge 1-\exp\left(-\beta_0^4 \right).
	\end{equation}
\end{lem}

\begin{lem}\label{lem:ind2:tau63sharp:2}
	Under the setting of Theorem \ref{thm:ind2:main} suppose that we condition on the event $\mathcal{A}_4$. Then, we have
	\begin{equation}
	\left| \intop_{t_1^{\flat}\wedge\hat{\tau}}^{u \wedge \hat{\tau}}  J^{(\alpha_1)}_{t-s,t-u} (S(u)-S_1(u)) du\right| \le \frac{\alpha_0^{1-\epsilon}\beta_0^{C_\circ} }{\sqrt{t-s+1}}e^{-c\alpha_0^2(t-s)},
	\end{equation}
	for all $t\in[t_1,t_1^\sharp]$ and $s\in [t_1^\flat,t]$.
	
\end{lem}

\begin{lem}\label{lem:ind2:tau63sharp:4}
	Under the setting of Theorem \ref{thm:ind2:main}, suppose that we condition on the event $\mathcal{A}_4$. For all $t\in [t_1,t_1^+]$  and $s\in[t_1^\flat,t]$, we have
	\begin{equation}
	\left| \intop_{t_1^\flat\wedge\hat{\tau}}^{s\wedge \hat{\tau}} J^{(\alpha_1)}_{t-s,t-u} (S_1(u)-\alpha_1)du \right|
	\le \frac{\alpha_0^{\frac{1}{2}} \beta_0^{\theta+2}} {\sqrt{t-s+1} }e^{-c\alpha_1^2(t-s)}.
	\end{equation}
\end{lem}

Before we delve into the details, we address a simple fact that is useful to keep in mind for the rest of the subsection. Note that \eqref{eq:ind2:tauxpprime} and Theorem \ref{thm:ind2:main}-(2) shown in Section \ref{subsec:ind2:ratechange} tells us that
\begin{equation}\label{eq:ind2:tau63sh:aux1}
\PP \left( \left. \hat{\tau} \le t_1^\sharp, \ \tilde{\tau}\wedge\tau'\ge \hat{t}_0 \ \right| \ \mathcal{A}(\alpha_0,t_0) \right) \le 2\exp\left(-\beta_0^2 \right).
\end{equation}
Furthermore, we see from Lemma \ref{lem:ind2:alpha1vsalpha0} and Theorem \ref{thm:ind2:main}-(2) that
\begin{equation}\label{eq:ind2:tau63sh:aux2}
\begin{split}
\PP\left( \left. (\mathcal{A}_4)^c \cup \mathcal{A}(\alpha_1,t_1^-)^c,\ \tilde{\tau} \wedge \tau' \ge \hat{t}_0\ \right| \ \mathcal{A}(\alpha_0,t_0) \right) \le 2 \exp\left(-\beta_0^5 \right).
\end{split}
\end{equation}
This tells us that along with the event $\{\tilde{\tau}\wedge\tau'\ge \hat{t}_0 \}$, conditioning on $\mathcal{A}'$ or on $\mathcal{A}(\alpha_1,t_1^-)$ instead of $\mathcal{A}(\alpha_0,t_0)$ produces only a small additional error probability. Moreover, we stress that $\mathcal{A}_4$ and $\mathcal{A}(\alpha_1,t_1^-)$ are $\mathcal{F}_{t_1^-}$-measurable.

\begin{proof}[Proof of Lemma \ref{lem:ind2:tau63sharp}]
	
	Having \eqref{eq:ind2:tau63sh:aux1} and  \eqref{eq:ind2:tau63sh:aux2} combining the results of Lemmas \ref{lem:ind2:tau63sharp:1}--\ref{lem:ind2:tau63sharp:4},  we have with probability at least $1- 3\exp(-\beta_0^4)$ conditioned on $\mathcal{A}(\alpha_0,t_0)$ that
	\begin{equation}
	\left| \intop_{t_1^\flat \wedge\hat{\tau} }^{u\wedge\hat{\tau}} J^{(\alpha_1)}_{t-s,t-u} d\widehat{\Pi}_S(u)\right| \le \left(\frac{ \alpha_0^{\frac{1}{2}}\beta_0^{\theta+3} }{\sqrt{t-s+1}} + \frac{\beta_0^{C_\circ +1}}{t-s+1}\right) e^{-c\alpha_0^2(t-s)},
	\end{equation}
	and hence we obtain the conclusion.
\end{proof}

In the remaining of this subsection, we prove Lemmas \ref{lem:ind2:tau63sharp:1}--\ref{lem:ind2:tau63sharp:4}. In the proof, we repeatedly use Lemma \ref{lem:bound on deterministic J} which is
\begin{equation}\label{eq:ind2:Jbd}
J^{(\alpha_1)}_{x,y} \le \frac{Ce^{-2c y\alpha_0^{2} }}{\sqrt{(x+1)(y+1)}} \le \frac{Ce^{-c (x+y)\alpha_0^{2} }}{\sqrt{(x+1)(y+1)}},
\end{equation}
where the first inequality holds upon the event $\mathcal{A}_4.$

\begin{proof}[Proof of Lemma \ref{lem:ind2:tau63sharp:1}]
	Define the set 
	\begin{equation}\label{eq:ind2:tau63sh:1:discretizedset}
	\mathcal{T}:= \left\{ \alpha_0^{10}k \in [t_1^\flat,t_1^\sharp] : k\in \mathbb{Z}  \right\}.
	\end{equation}
	For any $t,u\in \mathcal{T}$ with $u\le t$, we first observe from \eqref{eq:ind2:Jbd} that on $\mathcal{A}_4$,
	\begin{equation}\label{eq:ind2:tau61sh1:small s}
	\begin{split}
	\left|\intop_{(s-\alpha_0^{-1})\wedge \hat{\tau}}^{s\wedge \hat{\tau}} J^{(\alpha_1)}_{t-s,t-u} d\widetilde{\Pi}_S(u)\right|
	&\le
	\frac{\beta_0^{5}e^{-c\alpha_0^2(t-s)}}{t-s+1 } + \frac{Ce^{-c\alpha_0^2(t-s)}}{\sqrt{t-s+1}} \intop_{s-\alpha_0^{-1}}^s \frac{\alpha_0 \beta_0^{C_\circ}du}{\sqrt{t-u+1}} \\&\le \frac{2C\beta_0^{C_\circ}e^{-c\alpha_0^2(t-s)}}{t-s+1}.
	\end{split}
	\end{equation}
	
To take care of the integral from $t_1^\flat$ to $(s-\alpha_0^{-1})$,   note that we have
	\begin{equation}
	\begin{split}
	\intop_{t_1^\flat\wedge \hat{\tau}}^{s\wedge \hat{\tau}} \left\{J^{(\alpha_1)}_{t-s,t-u}\right\}^2 S(u)du 
	\le \frac{Ce^{-2c\alpha_1^2(t-s)}}{t-s+1} \intop_{t_1^{\flat}}^{s} \frac{\alpha_0 \beta_0^{C_\circ}du}{t-u+1}& \le
	\frac{ \alpha_0 \beta_0^{C_\circ +2} }{t-s+1}e^{-2c\alpha_0^2(t-s)},\\
	\sup \left\{\left|J^{(\alpha_1)}_{t-s,t-u} \right|: s\in[t_1^\flat,s-\alpha_0^{-1}]  \right\} &\le \frac{C\alpha_0^{\frac{1}{2}}e^{-c\alpha_0^2(t-s)}}{\sqrt{t-s+1}}.
	\end{split}
	\end{equation}
	Thus, we apply Lemma \ref{lem:concentration of integral} and use \eqref{eq:ind2:tau61sh1:small s} to obtain that
	\begin{equation}
	\PP \left(\left.\left| \intop_{t_1^\flat\wedge \hat{\tau} }^{s\wedge \hat{\tau}} 
	J^{(\alpha_1)}_{t-s,t-u} d\widetilde{\Pi}_S(u)
	\right|	\ge \left(\frac{2C\beta_0^{C_\circ}}{t-s+1} + \frac{\alpha_0^{\frac{1}{2}}\beta_0^{C_\circ} }{\sqrt{t-s+1}} \right)e^{-c\alpha_0^2(t-s)} \ \right| \ \mathcal{A}' \right) \le \exp\left(-\beta_0^5 \right).
	\end{equation}
	Note that we can condition on $\mathcal{A}'$ when applying Lemma \ref{lem:concentration of integral} since it is $\mathcal{F}_{t_1^-}$-measurable. Moreover, the event $\mathcal{A}_4$ is used to ensure that $\alpha_1$ and $\alpha_0$ are close enough. 
	
	Now define 
	\begin{equation}
	\mathcal{B}'_1(s,t):= \left\{ \left| \intop_{t_1^\flat\wedge \hat{\tau} }^{s\wedge \hat{\tau}} 
	J^{(\alpha_1)}_{t-s,t-u} d\widetilde{\Pi}_S(u)
	\right|	\le \left(\frac{2C\beta_0^{C_\circ}}{t-s+1} + \frac{\alpha_0^{\frac{1}{2}}\beta_0^{C_\circ} }{\sqrt{t-s+1}}\right)e^{-c\alpha_0^2(t-s)}  \right\}.
	\end{equation}
	Note that its bound is slightly stronger than $\mathcal{B}_1(s,t)$, to leave some room to take a union bound and then conduct a continuity argument.
	Then, we use a union bound to deduce that
	\begin{equation}
	\PP \left(\left.\bigcap_{t\in \mathcal{T} } \bigcap_{s\in \mathcal{T}:\, s\le t} \mathcal{B}_1'(s,t)\ \right| \ \mathcal{A}' \right) \ge 1- \exp\left(-\beta_0^4 \right).
	\end{equation}
	The remaining duty is to improve this bound to work on general $t,s$ which are not necessarily in $\mathcal{T}$. This is done by the exact same argument presented in Lemmas \ref{lem:Qint 1st} and \ref{lem:concentrationofint:continuity}, and the details of the proof are left to the reader.
\end{proof}

\begin{proof}[Proof of Lemma \ref{lem:ind2:tau63sharp:2}]
	
	In the proof, we assume that we have $\mathcal{A}_4$.
	Using \eqref{eq:ind2:Jbd} and the definition of $\tau_{3}(\alpha_1,t_1^-)$ (Section \ref{subsubsec:reg:Scontrol}), we express that
	\begin{equation}\label{eq:ind2:tau63sh:2:main}
	\begin{split}
	\intop_{t_1^\flat\wedge \hat{\tau}}^{s\wedge\hat{\tau}} \left|J^{(\alpha_1)}_{t-s,t-u} \right| \left|S(u)-S_1(u) \right|du
	\le
	\intop_{t_1^\flat\wedge \hat{\tau}}^{s\wedge\hat{\tau}} \frac{Ce^{-c\alpha_0^2(t-s) }e^{-c\alpha_0^2(t-u)}}{\sqrt{(t-u+1)(t-s+1)}} \cdot\frac{\alpha_0^{1-\epsilon}}{\pi_1(u;S)+1}ds\\
	= \frac{C\alpha_0^{1-\epsilon}e^{-c\alpha_0^2(t-s)}}{\sqrt{t-s+1}} 	\intop_{t_1^\flat\wedge \hat{\tau}}^{s\wedge\hat{\tau}} \frac{Ce^{-c\alpha_0^2(t-u)}}{\sqrt{t-u+1}} \cdot\frac{ds}{\pi_1(u;S)+1}.
	\end{split}
	\end{equation}
	To simplify the RHS, we recall Lemma \ref{lem:ind1:pi1int basicbd}, with parameters $\Delta_0=\alpha_0^{-1}, \Delta_1= \alpha_0^{-1}\beta_0^{C_\circ}$ and $N_0 = \beta_0^5$. This gives
	\begin{equation}
	\intop_{t_1^\flat\wedge \hat{\tau}}^{s\wedge\hat{\tau}} \frac{Ce^{-c\alpha_0^2(t-u)}}{\sqrt{t-u+1}} \cdot\frac{ds}{\pi_1(u;S)+1} \le \beta_0^{C_\circ}.
	\end{equation}
	Note that although we need to set $K\ge \alpha_0^{-1}\beta_0^{10\theta}$ in Lemma \ref{lem:ind1:pi1int basicbd}, the contribution from $k\ge \alpha_0^{-1}\beta_0^2$ is negligible due to the extra exponential decay.  This concludes the proof of Lemma \ref{lem:ind2:tau63sharp:2}.
\end{proof}

\begin{proof}[Proof of Lemma \ref{lem:ind2:tau63sharp:4}]
	Suppose that we are on the event $\mathcal{A}_4.$
	Using the bound \eqref{eq:ind2:Jbd} and the definition of $\tau_{\textnormal{b}}''$ from \eqref{eq:def:tauxpprime}, we have
	\begin{equation}
	\intop_{t_1^\flat\wedge\hat{\tau}}^{s\wedge\hat{\tau}} \left| J^{(\alpha_1)}_{t-s,t-u} \right| |S_1(u)-\alpha_1| du
	\le \frac{Ce^{-c\alpha_0^2(t-s)}}{\sqrt{t-s+1}}
	\intop_{t_1^\flat\wedge\hat{\tau}}^{s\wedge\hat{\tau}} 
	\frac{ e^{-c\alpha_0^2(t-u)}}{\sqrt{t-u+1}} \left(\frac{\alpha_0\beta_0^{C_\circ}}{\sqrt{\pi_1(u;S)+1}} + \alpha_0^{\frac{3}{2}}\beta_0^\theta \right)ds.
	\end{equation}
	We first note that
	\begin{equation}\label{eq:ind2:tau63sh:4:int1}
	\intop_{t_1^\flat}^s \frac{\alpha_0^{\frac{3}{2}}\beta_0^\theta e^{-c\alpha_0^2(t-u)}}{\sqrt{t-u+1}} ds \le  \alpha_0^{\frac{1}{2}}\beta_0^{\theta+1}.
	\end{equation}
	Furthermore, from Lemma \ref{lem:ind1:pi1int basicbd} with $\Delta_0 = \alpha_0^{-1}, \Delta_1= \alpha_0^{-1}\beta_0^{C_\circ}$, and $N_0 =\beta_0^5 $,
	\begin{equation}\label{eq:ind2:tau63sh:4:int2}
	\begin{split}
	\intop_{t_1^\flat\wedge\hat{\tau}}^{s\wedge\hat{\tau}} \frac{\alpha_0\beta_0^{C_\circ}e^{-c\alpha_0^2(t-u)}du}{\sqrt{ (t-u+1)(\pi_1(u;S)+1) }}
	\le 
	\alpha_0^{\frac{1}{2}} \beta_0^{2C_\circ}
	.
	\end{split}
	\end{equation}
	Note that  the contribution from $k\ge \alpha_0^{-1}\beta_0^2$ is negligible similarly as in Lemma \ref{lem:ind2:tau63sharp:2}. Combining the two estimates, we obtain conclusion.
\end{proof}

\subsubsection{Proof of Lemma \ref{lem:ind2:tau61sharp}}\label{subsubsec:ind2:tau61sharp}
To finish the proof of Lemma \ref{lem:ind2:tau61sharp}, we define 
\begin{equation}
Z(s,t):= \intop_{t_1^\flat}^s J^{(\alpha_1)}_{t-s,t-u} d\widehat{\Pi}_S(u),
\end{equation}
which we have studied in the previous subsection. Moreover, for $\hat{\tau}$ and $\tau^\sharp_{3,3}$ given in \eqref{eq:def:tauhat:ind2} and \eqref{eq:def:tau63sharp}, respectively, we set
\begin{equation}
\hat{\tau}' := \hat{\tau} \wedge \tau^\sharp_{3,3}.
\end{equation}
Splitting $d\widehat{\Pi}_S(s)$ into four parts as described in \eqref{eq:ind2:tausharp:tilPiSdecomposition} and recalling the definition of $\mathcal{A}'$ from \eqref{eq:def:ind2:Aprime}, we prove the following lemmas.

\begin{lem}\label{lem:ind2:tau61sharp:1}
	Assume the setting of Theorem \ref{thm:ind2:main}. For each $t\in [t_1^\flat,t_1^\sharp]$, define the event 
	\begin{equation}
	\mathcal{C}_1(t):= \left\{ \left| \intop_{t_1^\flat}^{t} Z(s,t) d\widetilde{\Pi}_S(s) \right|
	\le \frac{ \beta_0^{2\theta+3C_\circ} }{ \pi_1(t;S)+1 } 
	\right\}.
	\end{equation}
	Then, we have
	\begin{equation}
	\PP\left(\left. \bigcap_{t\in [t_1^\flat, t_1^\sharp] } \mathcal{C}_1(t)  \ \right| \ \mathcal{A}' \right) \ge 1-\exp\left(-\beta_0^4 \right).
	\end{equation}
\end{lem}

\begin{lem}\label{lem:ind2:tau61sharp:2}
	Under the setting of Theorem \ref{thm:ind2:main}, suppose that we condition on $\mathcal{A}'$. Then, we have
	\begin{equation} \left| \intop_{t_1^\flat\wedge\hat{\tau}'}^{t\wedge \hat{\tau}'} Z(s,t) (S(s)-S_1(s))ds \right|
	\le \frac{ \alpha_0^{1-2\epsilon} }{ \pi_1(t;S)+1 } + \alpha_0^{\frac{3}{2}-2\epsilon} 
	,
	\end{equation}
	for all $t\in[t_1^\flat,t_1^\sharp]$
\end{lem}

\begin{lem}\label{lem:ind2:tau61sharp:4}
	Under the setting of Theorem \ref{thm:ind2:main}, suppose that we condition on $\mathcal{A}'$. Then, we have 
	\begin{equation}
	\left| \intop_{t_1^\flat\wedge\hat{\tau}'}^{t\wedge \hat{\tau}'} Z(s,t) (S_1(s)-\alpha_1)ds \right|
	\le \alpha_0\beta_0^{2\theta+5} 
	,
	\end{equation}
	for all $t\in [t_1^\flat, t_1^\sharp].$
\end{lem}

Proofs of Lemmas \ref{lem:ind2:tau61sharp:1}--\ref{lem:ind2:tau61sharp:4} are similar to those of Lemmas \ref{lem:ind2:tau63sharp:1}--\ref{lem:ind2:tau63sharp:4}. Before delving into those, we first conclude the proof of Lemma \ref{lem:ind2:tau61sharp}.

\begin{proof}[Proof of Lemma \ref{lem:ind2:tau61sharp}]
	As in the proof of Lemma \ref{lem:ind2:tau63sharp}, \eqref{eq:ind2:tau63sh:aux1} and \eqref{eq:ind2:tau63sh:aux2} tell us that along with the event $\{\tilde{\tau}\wedge\tau'\ge \hat{t}_0 \}$, conditioning on $\mathcal{A}'$ or on $\mathcal{A}(\alpha_1,t_1^-)$ instead of $\mathcal{A}(\alpha_0,t_0)$ produces only a small additional error probability. 
	Thus, combining Lemmas \ref{lem:ind2:tau61sharp:1}--\ref{lem:ind2:tau61sharp:4}, we have with probability at least $1-3e^{-\beta_0^4}$  conditioned on $\mathcal{A}(\alpha_0,t_0)$ that
	\begin{equation}
	\begin{split}
	\left| \intop_{t_1^\flat\wedge \hat{\tau}'}^{t\wedge \hat{\tau}'}  Z(s,t) d\widehat{\Pi}_S(s) \right| \le \frac{2\beta_0^{2\theta +3C_\circ}}{\pi_1(t;S)+1}+2\alpha_0\beta_0^{2\theta+5}
	\le
	\frac{\beta_1^{2\theta+4C_\circ}}{\pi_1(t;S)+1},
	\end{split}
	\end{equation}
	where  the last inequality is from $  \pi_1(t;S) \le \alpha_0^{-1}\beta_0^{C_\circ}$ (see the definition of $\tau_{8}$ in Section \ref{subsubsec:reg:aggsize}). Then, we obtain the conclusion using Theorem \ref{thm:ind2:main}-(2) (proven in Section \ref{subsec:ind2:ratechange}) and Lemma \ref{lem:ind2:tau63sharp}.
\end{proof}

In the remaining of this subsection, we prove Lemmas \ref{lem:ind2:tau61sharp:1}--\ref{lem:ind2:tau61sharp:4}.

\begin{proof}[Proof of Lemma \ref{lem:ind2:tau61sharp:1}]
	The proof goes analogously as Lemma \ref{lem:ind2:tau63sharp:1}, by splitting the integral into two parts, from $t-\alpha_0^{-1}$ to $t$ and $t_1^\flat$ to $t-\alpha_0^{-1}$, which is to apply Lemma \ref{lem:concen of int:conti:forwardtime}.	From the definition of $\tau_{{3,3}}^\sharp$, we see that for all $t\le \hat{\tau}'$,
	\begin{equation}\label{eq:ind2:tau61sh:1:int}
	\begin{split}
	\left| 	\intop_{t-\alpha_0^{-1} }^{t} Z(s,t) d\widetilde{\Pi}_S(s)\right|
	&\le
	\frac{\beta_0^{2C_\circ}}{\pi_1(t;S)+1}+ \frac{\alpha_0^{\frac{1}{2}} \beta_0^{\theta+C_\circ+3}}{\sqrt{\pi_1(t ;S)+1}} +\alpha_0\beta_0^{\theta+C_\circ+4}
	\\&\le
	\frac{\beta_0^{\theta+2C_\circ}}{\pi_1(t ;S)+1},
	\end{split}
	\end{equation}
	where the second inequality follows from $\pi_1(t\wedge \hat{\tau}' ;S)\le 2\alpha_0^{-1}\beta_0^{C_\circ}$. 
	
	The other integral can be estimated using Lemma \ref{lem:concen of int:conti:forwardtime}, by observing that
	\begin{equation}
	\begin{split}
	&	\intop_{t_1^\flat\wedge\hat{\tau}'}^{(t-\alpha_0^{-1})\wedge \hat{\tau}} Z(s,t)^2 S(s)ds \le \intop_{t_1^\flat}^{t} \left(\frac{2\beta_0^{2C_\circ+2}}{(t-s+1)^2} +\frac{2\alpha_0\beta_0^{2\theta+6} }{t-s+1} \right) 2\alpha_0 \beta_0^{C_\circ} ds \le \alpha_0^2 \beta_0^{2\theta+2C_\circ},\\
	&\qquad \qquad \qquad 	\sup \left\{ |Z(s,t)|: s\in[t_1^{\flat}\wedge \hat{\tau}', (t-\alpha_0^{-1})\wedge \hat{\tau}'] \right\} \le 2\alpha_0 \beta_0^{\theta +3}.
	\end{split}
	\end{equation}
	Thus, applying the lemma gives that
	\begin{equation}
	\PP \left( \left.\left| \intop_{t_1^\flat\wedge\hat{\tau}'}^{(t-\alpha_0^{-1})\wedge \hat{\tau}'} Z(s,t)d\widehat{\Pi}_S(s) \right| \le \alpha_0\beta_0^{\theta+2C_\circ} , \ \forall t\in [t_1^\flat, t_1^\sharp ] \right| \ \mathcal{A}' \right) \ge 1- \exp\left(-\beta_0^{5} \right),
	\end{equation}
and hence combining with \eqref{eq:ind2:tau61sh:1:int} concludes the proof, since $\hat{\tau}'\ge t_1^\sharp$ with high probability (Proposition \ref{prop:ind2:prereg}, equation \eqref{eq:ind2:tauxpprime} and Lemma \ref{lem:ind2:tau63sharp}).
\end{proof}

\begin{proof}[Proof of Lemma \ref{lem:ind2:tau61sharp:2}]
 Recalling the definition of $\tau_{3}(\alpha_1,t_1^-)$ (Section \ref{subsubsec:reg:Scontrol}), we have
	\begin{equation}\label{eq:ind2:tau61sh:2:int}
	\begin{split}
	\intop_{t_1^\flat\wedge\hat{\tau}'}^{t\wedge \hat{\tau}'} |Z(s,t) (S(s)-S_1(s))| ds 
&	\le
	\intop_{t_1^\flat}^{t\wedge \hat{\tau}'} \left(\frac{\beta_0^{C_\circ+1}}{t-s+1} + \frac{\alpha_0^{\frac{1}{2}}\beta_0^{\theta+3} }{\sqrt{t-s+1}} \right) \frac{\alpha_0^{1-\epsilon}e^{-c\alpha_0^2(t-s)}}{\pi_1(s;S)+1}ds.
	\end{split}
	\end{equation}
	Then, applying Lemma \ref{lem:ind1:pi1int basicbd} with $\Delta_0 = \alpha_0^{-1}, \Delta_1 = \alpha_0^{-1}\beta_0^{C_\circ}$, and $N_0 = \beta_0^5$ bounds the RHS by 
	$$\frac{\alpha_0^{1-2\epsilon}}{\pi_1(s;S)+1} + \alpha_0^{\frac{3}{2}-2\epsilon}, $$
		recalling that  the contribution from $k\ge \alpha_0^{-1}\beta_0^2$ is negligible similarly as in Lemma \ref{lem:ind2:tau63sharp:2}.
\end{proof}

\begin{proof}[Proof of Lemma \ref{lem:ind2:tau61sharp:4}]
 Having the definition of $\tau_{\textnormal{b}}''$ \eqref{eq:def:tauxpprime} in mind, we express that
	\begin{equation}\label{eq:ind2:tau61sh:4:int}
	\begin{split}
	&	\intop_{t_1^\flat\wedge \hat{\tau}'}^{t\wedge \hat{\tau}'}
	\left|Z(s,t) (S_1(s)-\alpha_1) \right|ds\\
	&	\le
	\intop_{t_1^\flat\wedge\hat{\tau}'}^{t\wedge \hat{\tau}'} e^{-c\alpha_0^2(t-s)} \left( \frac{\beta_0^{C_\circ+1}}{t-s+1} + \frac{\alpha_0^{\frac{1}{2}}\beta_0^{\theta+3}}{\sqrt{t-s+1}} \right) \left( \frac{\alpha_0 \beta_0^{C_\circ+1}}{\sqrt{\pi_1(u;S)+1}} + 2\alpha_0^{\frac{3}{2}}\beta_0^\theta  \right)ds. 
	\end{split}
	\end{equation}
	We begin with observing that
	\begin{equation}\label{eq:ind2:tau61sh:4:int2}
	\intop_{t_1^\flat\wedge\hat{\tau}'}^{t\wedge \hat{\tau}'} e^{-c\alpha_0^2(t-s)} \left( \frac{\beta_0^{C_\circ+1}}{t-s+1} + \frac{\alpha_0^{\frac{1}{2}}\beta_0^{\theta+3}}{\sqrt{t-s+1}} \right)  2\alpha_0^{\frac{3}{2}}\beta_0^\theta ds
	\le \alpha_0^{\frac{3}{2}} \beta_0^{\theta +C_\circ+3} + \alpha_0 \beta_0^{2\theta +4}.
	\end{equation}
	On the other hand, the remaining integral can be bounded using Lemma \ref{lem:ind1:pi1int basicbd} as in the proof of Lemma \ref{lem:ind2:tau61sharp:2}. This gives
	\begin{equation}\label{eq:ind2:tau61sh:4:int3}
	\begin{split}
	\intop_{t_1^\flat\wedge\hat{\tau}'}^{t\wedge \hat{\tau}'} e^{-c\alpha_0^2(t-s)} \left( \frac{\beta_0^{C_\circ+1}}{t-s+1} + \frac{\alpha_0^{\frac{1}{2}}\beta_0^{\theta+3}}{\sqrt{t-s+1}} \right)  \frac{\alpha_0 \beta_0^{C_\circ+1}ds}{\sqrt{\pi_1(s;S)+1}}\le \alpha_0 \beta_0^{2\theta},
	\end{split}
	\end{equation}
	concluding the proof.
\end{proof}

\subsubsection{Some consequences of the integral calculation}\label{subsubsec:ind2:tausharpconseq}

Before moving on, we record several direct consequences of the integral calculations we have seen above. These results will be useful later in Section \ref{sec:double int} when we encounter similar formulas under a slightly different setting.

We begin with stating the analogue of Lemmas \ref{lem:ind2:tau63sharp:1} and \ref{lem:ind2:tau63sharp:2}, noting that the previous analysis works analogously in another interval  instead of $[t_1^\flat, t_1^\sharp]$. In this subsection, $S_1(t)$ is defined to be $S_1(t)=S_1(t;t_0^-,\alpha_0).$
\begin{cor}\label{cor:ind2:tau63sharp:1}
	Let $\alpha_0,t_0,r>0$, set $t_0^-, \acute{t}_0, \hat{t}_0$ as \eqref{eq:def:t0t1},  \eqref{eq:def:t0t11}, and let
	\begin{equation}
	t_0^\dagger := t_0 + \frac{5}{2}\alpha_0^{-2}\beta_0^{\theta} = \acute{t}_0 - \frac{1}{2}\alpha_0^{-2}\beta_0^\theta.
	\end{equation}
	 Define the stopping times
	\begin{equation}
	\begin{split}
	\tilde{\tau}_{3,3}^{(1)}&:= \inf \left\{
	t\ge \acute{t}_0: \, \exists s\in[t_0^\dagger,t] \textnormal{ s.t. } \left|
	\intop_{t_0^\dagger}^s J^{(\alpha_0)}_{t-s,t-u} d\widetilde{\Pi}_S(u)
	\right|
	\ge
	\left(\frac{\beta_0^{C_\circ+1}}{t-s+1} + \frac{\alpha_0^{\frac{1}{2}}\beta_0^{2C_\circ}}{\sqrt{t-s+1}} \right) e^{-c\alpha_0^2(t-s)}
	 \right\}
	;\\
	\tilde{\tau}_{3,3}^{(2)}&:=
	 \inf \left\{
	 t\ge \acute{t}_0: \, \exists s\in[t_0^\dagger,t] \textnormal{ s.t. } \left|
	 \intop_{t_0^\dagger}^s J^{(\alpha_0)}_{t-s,t-u} (S(u)-S_1(u))du
	 \right|
	 \ge
\frac{\alpha_0^{1-\epsilon}\beta_0^{C_\circ} }{\sqrt{t-s+1}}  e^{-c\alpha_0^2(t-s)}
	 \right\}
	.
	\end{split}
	\end{equation}
	 Suppose that $\Pi_S(-\infty,t_0] \in \mathfrak{R}(\alpha_0, r; [t_0])$. Then, we have
	 \begin{equation}
	 \PP \left(\left.\tilde{\tau}_{3,3}^{(1)}\wedge \tilde{\tau}_{3,3}^{(2)} <\hat{t}_0 \,\right|\,\mathcal{F}_{t_0}  \right) \le \exp\left(-\beta_0^3 \right) +r.
	 \end{equation}
\end{cor}
To be precise, we remark that the conclusion can be obtained from combining Theorem \ref{thm:reg:conti:main} and the argument from Lemma \ref{lem:ind2:tau63sharp}.

We also derive a similar estimate on the integral of $K^*_{\alpha_0}$ instead of $J_{t-s,t-u}^{(\alpha_0)}.$
\begin{cor}\label{cor:ind2:tau63sharp:Kstar}
	Under the setting of Corollary \ref{cor:ind2:tau63sharp:1}, define the stopping time
	\begin{equation}
	\tilde{\tau}_{3,3}^{(3)} :=
	\inf \left\{ 
	t\ge \acute{t}_0: \, \exists s\in[t_0^+,t] \textnormal{ s.t. } \left|\intop_{t_0^+}^s K^*_{\alpha_0}(s-u) (S(u)-S_1(u)) \right| \ge \alpha_0^{2-2\epsilon} 
	\right\}.
	\end{equation}
	Then, we have
	\begin{equation}
	\PP \left(\left.\tilde{\tau}_{3,3}^{(3)}  <\hat{t}_0\, \right|\, \mathcal{F}_{t_0} \right) \le \exp\left(-\beta_0^3 \right)+r.
	\end{equation}
\end{cor}
\begin{proof}
	The conclusion can be obtained from the same method as Lemma \ref{lem:ind2:tau63sharp:2}, since $K^*_{\alpha_0} (x) $ satisfies the estimate from Lemma \ref{lem:estimat for K tilde:intro} which is very similar to \ref{eq:ind2:Jbd}.
\end{proof}

The following analogue of Lemma \ref{lem:ind2:tau63sharp} can be obtained similarly.

\begin{cor}\label{cor:ind2:tau63sharp}
 Under the setting of Corollary \ref{cor:ind2:tau63sharp:1}, define the stopping time
 \begin{equation}\label{eq:def:tau63sharptil}
 \tilde{\tau}_{{3,3}}^\sharp:= \inf\left\{ t\ge \acute{t}_0: \,\exists s\in[t_0^\dagger,t] \textnormal{ s.t. } \left|\intop_{t_0^\dagger}^s J^{(\alpha_0)}_{t-s,t-u} d\widehat{ \Pi}_S(u) \right| \ge \frac{\alpha_0^{\frac{1}{2}}\beta_0^{6\theta }}{\sqrt{t-s+1}} + \frac{\beta_0^{C_\circ +1} }{t-s+1} \right\}.
 \end{equation}
 Then, we have
	\begin{equation}
	\PP\left(\left.\tilde{\tau}_{{3,3}}^\sharp < \hat{t}_0\ \right| \ \mathcal{F}_{t_0} \right) \le 3\exp\left(-\beta_0^2 \right) +r.
	\end{equation}
\end{cor}
Note that  in the term $\frac{\alpha_0^{1/2}\beta_0^{6\theta} }{\sqrt{t-s+1}}$, we have a larger exponent $6\theta$ than in $\tau_{3,3}^\sharp$, since we deal with a longer interval $[t_0, \hat{t}_0]$.

The other lemmas in the previous subsections can be extended similarly. Letting
$$Z_0(s,t):=\intop_{t_0^\dagger}^s J^{(\alpha_0)}_{t-s,t-u} d\widehat{\Pi}_S(u),$$
we have the analogue of Lemma \ref{lem:ind2:tau61sharp:2}, which can be deduced from using Corollary \ref{cor:ind2:tau63sharp} instead of Lemma \ref{lem:ind2:tau63sharp}.
\begin{cor}\label{cor:ind2:tau61sharp1}
	Under the setting of Corollary \ref{cor:ind2:tau63sharp:1}, define the stopping time
	\begin{equation}
	\begin{split}
	\tilde{\tau}_{3,1}^{(1)} := \inf \left\{
	 t\ge \acute{t}_0: \, \left| 
	 \intop_{t_0^\dagger}^t Z_0(s,t) (S(s)-S_1(s))ds
	 \right|
	 \ge \frac{\alpha_0^{1-2\epsilon}}{\pi_1(t;S)+1} + \alpha_0^{\frac{3}{2}-2\epsilon}
	 \right\}
	 .
	\end{split}
	\end{equation}
	Then, we have
	\begin{equation}
	\PP \left(\left. \tilde{\tau}_{3,1}^{(1)} <\hat{t}_0 \, \right| \, \mathcal{F}_{t_0} \right) \le \exp\left( - \beta_0^3\right)+r.
	\end{equation}
\end{cor}
Due to their similarity, proofs of Corollaries \ref{cor:ind2:tau63sharp:1}, \ref{cor:ind2:tau63sharp} and \ref{cor:ind2:tau61sharp1} are omitted.

\subsection{Proof of Theorem \ref{thm:induction:main}}\label{subsec:ind2:fin}

In this final section, we present the proof of Theorem \ref{thm:induction:main}. Throughout the proof, we assume that $\mathcal{F}_{t_0} \equiv \Pi_S(-\infty,t_0]$ is $(\alpha_0,r;[t_0])$-regular (see Definition \ref{def:reg}). 

We begin with verifying the second item of Theorem \ref{thm:induction:main}. Let $\tau := \tau(\alpha_0,t_0,2)$ \eqref{eq:def:tau:induction}, ${\tau}^+:= {\tau}^+(\alpha_0,t_0,1/2)$ \eqref{eq:def:ind1:tauacute:basic}, and $\tau' := {\tau}^+\wedge \tau$. Observe that
\begin{equation}\label{eq:ind2:fin:main1}
\begin{split}
\PP \left(\left.|\alpha_1-\alpha_0| \ge 2\alpha_0^{3/2}\beta_0^{6\theta} \ \right| \ \mathcal{F}_{t_0} \right)& \le
\PP \left(\tau \le \acute{t}_0 \ | \ \mathcal{F}_{t_0} \right) +  \PP \left(  \tau' < \hat{t}_0, \ \tau >\acute{t}_0 \ | \ \mathcal{F}_{t_0} \right)\\
&\quad +\PP\left( |\alpha_1-\alpha_0| \ge 2\alpha_0^{3/2}\beta_0^{6\theta},\  \tau' \ge \hat{t}_0\ | \ \mathcal{F}_{t_0}  \right)\\
&\le r+ \exp \left(-\beta_0^{3/2} \right),
\end{split}
\end{equation}
where the last line follows from the definition of regularity, Theorem \ref{thm:reg:conti:main} and Lemma \ref{lem:ind2:alpha1vsalpha0}.

To study the event $\{\mathcal{F}_{t_1} \textnormal{ is } (\alpha_1,e^{-\beta_1^{3/2}}; [t_1])\textnormal{-regular} \}$, we verify the two conditions of regularity separately. The second condition follows from Theorems \ref{thm:reg:conti:main} and \ref{thm:ind2:main}-(2). Namely, we write
\begin{equation}\label{eq:ind2:fin:main2}
\begin{split}
\PP \left( \mathcal{A}{(\alpha_1,t_1)}^c \ | \ \mathcal{F}_{t_0} \right)
&\le 
\PP \left( \tau' <\hat{t}_0 \ | \ \mathcal{F}_{t_0} \right) +\exp \left(-\beta_0^2 \right)
\\
&\le  
\PP \left( \tau' < \hat{t}_0,\ \tau>\acute{t}_0\ | \ \mathcal{F}_{t_0} \right) + \PP \left(\tau\le t_0^+ \ | \ \mathcal{F}_{t_0} \right)\\
&\le \exp\left(-\beta_0^2 \right)+ r+ \exp\left(-\beta_0^2\right) \le r+ \exp\left(-\beta_0^{3/2} \right). 
\end{split}
\end{equation}

To verify the first condition of regularity, recall the definition of $\mathcal{A}_4$ \eqref{eq:def:A4}, and $t_1^\sharp = t_1+4\alpha_0^{-2}\beta_0^\theta$ which is larger than $\acute{t}_1=t_1+2\alpha_1^{-2}\beta_1^\theta$ on $\mathcal{A}_4$. Observe that
\begin{equation}\label{eq:ind2:fin:1}
\begin{split}
\PP \left(\left.  \PP\left(\left.\tau^\sharp(\alpha_1,t_1,2) \le \acute{t}_1 \ \right| \ \mathcal{F}_{t_1} \right) > e^{-\beta_1^{3/2}}, \ \tau > \acute{t}_0, \ \mathcal{A}_4 \ \right| \ \mathcal{F}_{t_0} \right)\\
\le e^{2\beta_0^{3/2}} \PP \left(\tau^\sharp (\alpha_1,t_1,2) \le t_1^\sharp,\ \tau>\acute{t}_0\ | \ \mathcal{F}_{t_0} \right),
\end{split}
\end{equation}
by Markov's inequality. Note that $\mathcal{A}_4$ ensures $\beta_1^{3/2}\le 2\beta_0^{3/2}$. Furthermore, since $$t_1^\sharp= t_1 + 4\alpha_0^{-2}\beta_0^\theta \le t_0 +\alpha_0^{-2}\beta^{10\theta} +4\alpha_0^{-2}\beta_0^\theta< t_0+2\alpha_0^{-2}\beta^{10\theta}= \hat{t}_0,$$
combining Theorem \ref{thm:reg:conti:main}, Proposition \ref{prop:ind2:prereg}, and Lemma \ref{lem:ind2:tau6sharp} gives that
\begin{equation}\label{eq:ind2:fin:2}
\PP \left(\tau^\sharp (\alpha_1,t_1,2) \le t_1^\sharp,\ \tau>\acute{t}_0\ | \ \mathcal{F}_{t_0} \right) \le 3\exp\left(-\beta_0^2 \right).
\end{equation}
We also know from Theorem \ref{thm:reg:conti:main} and Lemma \ref{lem:ind2:alpha1vsalpha0} that
\begin{equation}\label{eq:ind2:fin:3}
\PP \left(\mathcal{A}_4^c, \tau>\acute{t}_0 \ | \  \mathcal{F}_{t_0} \right) \le 2\exp\left(-\beta_0^2 \right).
\end{equation}
Thus, \eqref{eq:ind2:fin:1}, \eqref{eq:ind2:fin:2} and \eqref{eq:ind2:fin:3} tell us that
\begin{equation}\label{eq:ind2:fin:main3}
\begin{split}
\PP \left(\left.  \PP\left(\left.\tau^\sharp(\alpha_1,t_1,2) \le \acute{t}_1 \ \right| \ \mathcal{F}_{t_1} \right) > e^{-\beta_1^{3/2}} \ \right| \ \mathcal{F}_{t_0} \right)&\le 3e^{2\beta_0^{3/2}-\beta_0^2} + 2e^{-\beta_0^2}+ r  \\ &\le r+ e^{-\beta_0^{3/2}}.
\end{split}
\end{equation}
Thus, we conclude the proof of Theorem \ref{thm:induction:main} from \eqref{eq:ind2:fin:main1}, \eqref{eq:ind2:fin:main2} and \eqref{eq:ind2:fin:main3}. \qed

\section{The second order contributions and the moments of the increment}\label{sec:double int}

Building upon the analysis of regularity, the goal of this section is establishing the main estimate needed in deriving the scaling limit of the speed. We aim to state and proving the formal version of Theorem \ref{thm:speed:increment:informal}, which will be done in Section \ref{subsec:increment proof}. To this end, an essential step  is to compute the expectation of the double integral term in $S_2(t)$ \eqref{eq:def:S2:basic form}, namely,
\begin{equation}
\mathcal{J}(t)=\mathcal{J}(t;\alpha,t_0^-):= \frac{\alpha}{1+2\alpha}\intop_{t_0^-}^{t} \intop_{t_0^-}^s J^{(\alpha)}_{t-s,t-u} d\widehat{\Pi}_S(u) d\widehat{\Pi}_S(v).
\end{equation}
The following theorem gives an appropriate control on the mean of $\mathcal{J}(t)$, and it is established in Section \ref{subsec:increment:double int}.

\begin{thm}\label{thm:double int:main}
Let $\alpha, t_0>0$,  set $t_0^-, t_0^+, \acute{t}_0$ and $\hat{t}_0$ as \eqref{eq:def:t0t1} and \eqref{eq:def:t0t11}, and let $r=e^{-\beta^{3/2}}$ for $\beta:=\log(1/\alpha)$. Suppose that $\Pi_S(-\infty,t_0]$ is $(\alpha,r;[t_0])$-regular. Then, for all $t\in [\acute{t}_0,\hat{t}_0]$, we have 
\begin{equation}
\mathbb{E}[\mathcal{J}(t)\, | \, \Pi_S(-\infty,t_0]] = 2\alpha^2 \left(1+ o(\alpha^{\frac{1}{5}}) \right).
\end{equation}
\end{thm}
We note that  the error bound $\alpha^{\frac{1}{5}}$ is not essential: something better than $O(\beta^{-2})$ will suffice for our purpose.
 
 \subsection{The contribution of the double integral}\label{subsec:increment:double int}
 From the analysis in the previous sections, we can deduce that $\mathcal{J}(t)$ is typically of order $\alpha^{2-\epsilon}$. Then, the main difficulty in establishing Theorem \ref{thm:double int:main} is narrowing down its size to $2\alpha^2$ plus a smaller order error. Here, we require more refined tools to study the mean of $\mathcal{J}(t)$ accurately. 
 
 Let $t_0^-, t_0^+, \acute{t}_0$ and $\hat{t}_0$ be as in Theorem \ref{thm:double int:main}, we introduce another parameter $t_0^\dagger$ defined as
 \begin{equation}
 t_0^\dagger := t_0 + \frac{5}{2}\alpha^{-2}\beta^\theta.
 \end{equation}
 An important relation we stress here is that
 \begin{equation}
 (\acute{t}_0 - t_0^\dagger) \wedge (t_0^\dagger - t_0^+) \ge \frac{1}{2}\alpha^{-2}\beta^\theta.
 \end{equation}
 Since $\Pi_S(-\infty,t_0]$ is regular and we consider $t\ge \acute{t}_0$, we can ignore the contributions to the integral from  the regime $[t_0^-,t_0^\dagger]$, analogously as we have seen in \eqref{eq:ind2:tau6sh:switch base}. Namely, with probability $1- e^{-\beta^5},$
 \begin{equation}
 \left|\mathcal{J}(t) - \intop_{t_0^\dagger}^t\intop_{t_0^\dagger}^s J_{t-s,t-u}^{(\alpha)} d\widehat{\Pi}_S(u )d\widehat{\Pi}_S(s) \right| \le \alpha_0^{50}, \quad \textnormal{for all } t\in[\acute{t}_0,\hat{t}_0].
 \end{equation}
 Hence, from now we will be interested in investigating the double integral starting from $t_0^\dagger$. The reason for our choice of such $t_0^\dagger$ is explained in Remark \ref{rmk:doubleint:dagger}. Controlling the contribution from the error event where the above does not hold will be discussed in the proof of Lemma \ref{lem:doubleint:out} below, and it will also be bounded by $O(\alpha^{50})$.
 
 Writing $\mathcal{F}_t :=\Pi_S(-\infty,t_0]$ as before, we begin with observing that
 \begin{equation}\label{eq:doubleint:mg:basic}
 \mathbb{E} [\mathcal{J}(t) \, |\, \mathcal{F}_{t_0}  ] = \frac{\alpha}{1+2\alpha} \mathbb{E} \left[\left. \intop_{t_0^{\dagger}}^t \intop_{t_0^\dagger}^s J^{(\alpha)}_{t-s,t-u} d\widehat{\Pi}_S(u) (S(s)-\alpha)ds \, \right| \, \mathcal{F}_{t_0}   \right] + O(\alpha^{50}),
 \end{equation}  
 since the outer integral with respect to $d\widetilde{\Pi}_S(s) = d\Pi_S(s) - S(s)ds$ is a martingale and thus has mean zero. Suppose that we can switch $(S(s)-\alpha)ds$ into $(S_1(s)-\alpha)ds$ with a negligible error, recalling the definition $S_1(s) = S_1(s;t_0^-,\alpha) = \mathcal{R}_c(s,s;\Pi_S[t_0^-,s],\alpha)$. Then, from \eqref{eq:integralform:branching}, we can write
 \begin{equation}\label{eq:doubleint:S1minalpha:basic}
 (S_1(s)-\alpha) = \intop_{t_0^+}^s K^*_\alpha(s-x) (d\Pi_S(x)-S_1(x)dx) +
(\mathcal{R}_c(t_0^+,s;\Pi_S[t_0^-,t_0^+],\alpha)-\alpha).
 \end{equation}
 
 \begin{remark}\label{rmk:doubleint:dagger}
 	 Note that this integral starts from $t_0^+$, not $t_0^\dagger$, in order to keep our control on  $\mathcal{R}_c(t_0^+,s;\Pi_S[t_0^-,t_0^+],\alpha)$ for all $s\in [t_0^\dagger,\hat{t}_0]$ (see \eqref{eq:doubleint:errorbasic} below). The integral \eqref{eq:doubleint:S1minalpha:basic} needs to be  from $t_0^+$, not $t_0$, since we have control on $|S_1(u)-\alpha| $ only on $t\ge t_0^+$ (Lemma \ref{lem:ind1:taux}).
 \end{remark}

 From \eqref{eq:doubleint:S1minalpha:basic}, we  again attempt to approximate $(S_1(s)-\alpha)ds$ by
 \begin{equation}
 (S_1(s)-\alpha)ds \approx \intop_{t_0^+}^s K^*_\alpha(s-x) d\widetilde{\Pi}_\alpha (x) + (\alpha''-\alpha)ds,
 \end{equation}
 with $\alpha'':= \mathcal{L}(t_0^+;\Pi_S[t_0^-,t_0^+],\alpha)$.
 If we make these two steps of approximations rigorous and repeat a similar procedure to the inner integral, then we arrive at the following proposition.
 
 \begin{prop}\label{prop:double int:main approx}
Define the integrals $I_1$ and $I_2$ by
\begin{equation}
\begin{split}
I_1(t)&:= \intop_{t_0^{\dagger}}^t \left[ \intop_{t_0^\dagger}^s J_{t-s,t-u}^{(\alpha)} d\widetilde{\Pi}_\alpha(u) \cdot \intop_{t_0^+}^s K^*_\alpha(s-u) d\widetilde{\Pi}_\alpha(u) \right] ds;
\\
I_2(t)&:= \intop_{t_0^\dagger}^t \left[ \intop_{t_0^\dagger}^s \intop_{t_0^+}^u J^{(\alpha)}_{t-s,t-u} K^*_\alpha(u-v) d\widetilde{\Pi}_\alpha(v) du \cdot  \intop_{t_0^+}^s K^*_\alpha(s-x) d\widetilde{\Pi}_\alpha(x) \right]ds.
\end{split}
\end{equation}
 	Then, under the setting of Theorem \ref{thm:double int:main}, for all $\acute{t}_0\le t\le \hat{t}_0$ we have
 	\begin{equation}
\begin{split}
 \mathbb{E}[\mathcal{J}(t)\,| \, \mathcal{F}_{t_0}] = 
 	\frac{\alpha}{1+2 \alpha} \mathbb{E}\left[I_1(t)+I_2(t)  \right]+ O\left(\alpha^{2+\frac{1}{5}} \right).
\end{split}
 	\end{equation}
 	Note that the integrals $I_1$ and $I_2$ are independent of $\mathcal{F}_{t_0}$.
 \end{prop}

 Before establishing the proposition rigorously, we first deduce Theorem \ref{thm:double int:main} from it.
 
 \begin{proof}[Proof of Theorem \ref{thm:double int:main}]
 	Observe that for any deterministic functions $f,g$ and any numbers $a_1\le a_2$, $b_1\le b_2$, we have
 	\begin{equation}
 	\mathbb{E}\left[\intop_{a_1}^{a_2} f(x)d\widetilde{\Pi}_\alpha(x) \intop_{b_1}^{b_2} g(y)d\widetilde{\Pi}_\alpha (y) \right] = \alpha \intop_{a_1\vee b_1}^{a_2\wedge b_2} f(x)g(x )  dx,
 	\end{equation}
 	since the only nontrivial correlation comes from the cases when both $x$ and $y$ are at the same point in the point process. Thus, we obtain that
 	\begin{equation}
 	\begin{split}
 	\mathbb{E}[I_1(t)] &= 
\alpha 	\intop_{t_0^\dagger}^t \intop_{t_0^\dagger}^s J_{t-s,t-u}^{(\alpha)} K_\alpha^*(s-u) duds = \alpha \intop_0^{t-t_0^\dagger} \intop_0^s J^{(\alpha)}_{u,s} K^*_\alpha(s-u) duds.
\end{split}
\end{equation}
Similarly, we can see that
\begin{equation}
\begin{split}
\mathbb{E}[I_2(t)] &= \intop_{t_0^\dagger}^t  \intop_{t_0^\dagger}^s\mathbb{E}\left[\intop_{t_0^+}^u J_{t-s,t-u}^{(\alpha)}  K^*_\alpha(u-v)  d\widetilde{\Pi}_\alpha(v) \intop_{t_0^+}^s K_\alpha^*(s-x) d\widetilde{\Pi}_\alpha(x) \right]duds \\
&=
\alpha \intop_{t_0^\dagger}^t \intop_{t_0^\dagger}^s \intop_{t_0^+}^u J_{t-s,t-u}^{(\alpha)} K^*_\alpha(u-v) K^*_\alpha(s-v)dv du ds \\
&=\alpha \intop_0^{t-t_0^\dagger} \intop_0^s \intop_0^u J_{v,u}^{(\alpha)} K_\alpha^*(s-u) K^*_\alpha (s-v) dv du ds + O(\alpha^{100}),
 	\end{split}
 	\end{equation}
 	where in the third line, the integral over $[t_0^+,t_0^\dagger]$ of $dv$ has a negligible $O(\alpha^{100})$-order contribution due to the decay property of $J$, since $t\ge \acute{t}_0$.
 	The two deterministic integrals above are computed via a precise understanding of the quantities $K^*_\alpha$ and $J^{(\alpha)}$. This is done using analytical methods based on Fourier analysis, and can be found in Lemma \ref{lem:integral computation} in the appendix.
 	Applying the results from the lemma  concludes the proof.
 \end{proof}

 The rest of the subsection is devoted to the proof of Proposition \ref{prop:double int:main approx}. The analysis is highly technical, and relies on a similar approach as that in Section \ref{subsec:ind2:bootstrappedJ} based on the decomposition 
  \begin{equation}
  d\widehat{\Pi}_S(x) = d\widetilde{\Pi}_S(x ) + (S(x)-S_1(x)) dx + (S_1(x)- \alpha) dx.
  \end{equation} 
  ($d\widetilde{\Pi}_S(x) := d\Pi_S (x)-S(x)dx$ as before.)
 
 \subsubsection{Reshaping the outer integral}
 We begin with modifying the outer integral of $\mathcal{J}(t)$. Since it suffices to study the integral in the RHS of \eqref{eq:doubleint:mg:basic},
 define
 \begin{equation}
\begin{split}
&I_{\textnormal{out}}^{(1)}(t):=\intop_{t_0^\dagger}^t \intop_{t_0^\dagger}^s J_{t-s,t-u}^{(\alpha)} d\widehat{\Pi}_S(u) (S(s)-S_1(s))ds;\\
 &\tau_{\textnormal{out}}^{(1)}: = \inf \left\{ 
 t\ge \acute{t}_0: \, \left| I_{\textnormal{out}}^{(1)}(t) \right| \ge \alpha^{\frac{3}{2}-\epsilon} + \frac{\alpha^{1-\epsilon}}{\pi_1(t;S)+1}
 \right\}.
\end{split}
 \end{equation}
 Moreover, recalling \eqref{eq:doubleint:S1minalpha:basic}, the remaining outer integral is decomposed by the following formula:
 \begin{equation}\label{eq:doubleint:S1minusalpha decomp}
\begin{split}
 S_1(s)-\alpha = & \intop_{t_0^+}^s K^*_\alpha(s-x) d\widetilde{\Pi}_\alpha(x) + \intop_{t_0^+}^s K^*_\alpha ( s-x) d\widetilde{\Pi}_{S-\alpha} (x) \\
  &+ \intop_{t_0^+}^s K^*_\alpha(s-x) (S(x)-S_1(x))dx + \big[\mathcal{R}_c(t_0^+,s;\Pi_S[t_0^-,t_0^+],\alpha) -\alpha'' \big] + \big[\alpha''-\alpha \big],
\end{split}
 \end{equation}
 where we defined $d\widetilde{\Pi}_{S-\alpha}(x) := d\Pi_S(x)-d\Pi_\alpha(x) -(S(x)-\alpha)dx$ and $\alpha'' = \mathcal{L}(t_0^+;\Pi_S[t_0^-,t_0^+],\alpha)$. We show that the contributions to the double integral coming from the terms other than the first and the last in the RHS are negligible. To this end, recall the notation $\Pi_{S\triangle \alpha}$ \eqref{eq:def:Strianglealpha} and define
 \begin{equation}
 \begin{split}
& I_{\textnormal{out}}^{(2)}(t):= \intop_{t_0^\dagger}^t \intop_{t_0^\dagger}^s J_{t-s,t-u}^{(\alpha)} d\widehat{\Pi}_S(u) \intop_{t_0^+}^s K^*_\alpha (s-x)d\widetilde{\Pi}_{S-\alpha}(x) ds ;\\
 	&\tau_{\textnormal{out}}^{(2)} := \inf \left\{ 
 	t\ge \acute{t}_0: \left| I_{\textnormal{out}}^{(2)}(t) \right| \ge \alpha^{1-\epsilon} \sigma_1(t;S\triangle \alpha) + \alpha^{\frac{1}{4}-\epsilon} \sigma_1(t;S\triangle\alpha)^2+  \alpha^{\frac{5}{4}-\epsilon}
 	 \right\};\\
 	& I_{\textnormal{out}}^{(3)}(t):=
 	\intop_{t_0^\dagger}^t \intop_{t_0^\dagger}^s J_{t-s,t-u}^{(\alpha)} d\widehat{\Pi}_S(u) \intop_{t_0^+}^s K^*_\alpha (s-x) (S(x)-S_1(x))dx ds;\\
& 	 \tau_{\textnormal{out}}^{(3)} := \inf \left\{ 
 	 t\ge \acute{t}_0: \left| I_{\textnormal{out}}^{(3)}(t)  \right| \ge \alpha^{\frac{3}{2}-3\epsilon}
 	 \right\}. 
 \end{split}
 \end{equation}
 We remark that the contribution from the fourth term is negligible due to  Lemma \ref{lem:ind2:alphavsalphaPP}, which tells us that with probability at least $1-e^{-\beta^5}$,
 \begin{equation}\label{eq:doubleint:errorbasic}
 \big|\mathcal{R}_c(t_0^+,s;\Pi_S[t_0^-,t_0^+],\alpha) - \alpha'' \big| \le \alpha^{100},
 \end{equation}
  for all $t\in [t_0^\dagger,\hat{t}_0]$.
 
\begin{lem}\label{lem:doubleint:out1}
	Under the setting of Theorem \ref{thm:double int:main}, we have
	\begin{equation}
	\PP \left(\left. \tau_{\textnormal{out}}^{(1)} \wedge \tau_{\textnormal{out}}^{(2)} \wedge \tau_{\textnormal{out}}^{(3)}  < \hat{t}_0 \ \right| \, \mathcal{F}_{t_0} \right) \le \exp \left( -\beta^2 \right) +r.
	\end{equation}
\end{lem}

From this lemma, deduce the conclusion of this subsection which can be written as follows.
\begin{lem}\label{lem:doubleint:out}
Define the integral $I_{\textnormal{out}}(t)$ as
\begin{equation}
I_{\textnormal{out}}(t):= \intop_{t_0^\dagger}^t \intop_{t_0^\dagger}^s J^{(\alpha)}_{t-s,t-u} d\widehat{\Pi}_S(u) \cdot\left\{\intop_{t_0^+}^s K^*_\alpha(s-u) d\widetilde{\Pi}_\alpha(u) + \big[\alpha''-\alpha \big] \right\} ds.
\end{equation}
Under the setting of Theorem \ref{thm:double int:main}, we have for all $t\in[\acute{t}_0,\hat{t}_0]$ that
\begin{equation}
\mathbb{E}[\mathcal{J}(t) \, | \, \mathcal{F}_{t_0}] = \frac{\alpha}{1+2\alpha} \mathbb{E} [I_{\textnormal{out}}(t) \,|\, \mathcal{F}_{t_0} ] + o (\alpha^{2+\frac{1}{5}} ).
\end{equation}
\end{lem}

\begin{proof}
Let $\tau_{\textnormal{out}} := \tau_{\textnormal{out}}^{(1)}\wedge \tau_{\textnormal{out}}^{(2)}\wedge \tau_{\textnormal{out}}^{(3)}$.
 	We first show that
	\begin{equation}\label{eq:doubleint:remainder}
	\mathbb{E}[\mathcal{J}(t) \mathds{1}\{ \tau_{\textnormal{out}} <\hat{t}_0 \} \, | \, \mathcal{F}_{t_0} ] = o(\alpha^{40}) = 	\mathbb{E}[I_{\textnormal{out}}(t) \mathds{1}\{ \tau_{\textnormal{out}} <\hat{t}_0 \} \, | \, \mathcal{F}_{t_0} ].
	\end{equation}
	To establish the left estimate, we use the fact that $S(t)\le \frac{1}{2}$ and $|J_{u,s}^{(\alpha)} |\le 2$ from Proposition \ref{prop:speed:basicdef} and \eqref{eq:coupeling form of J}. Namely, we can express that
	\begin{equation}\label{eq:doubleint:error:upd rough}
	|\mathcal{J}(t)| \le \frac{\alpha}{1+2\alpha} \left[ |\Pi_{1/2}[t_0^\dagger,t]|^2 + \intop_{t_0^\dagger}^{t}\intop_{t_0^\dagger}^s 2\cdot \frac{1}{4} duds \right].
	\end{equation} 
	Since the event $\{\tau_{\textnormal{out}} <\hat{t}_0 \}$ happens with probability less than $2e^{-\beta^{3/2}}$, we can couple this event with the tail events of the rate-$\frac{1}{2}$ Poisson process $|\Pi_{1/2}[t_0, \hat{t}_0]|$ and deduce the first equality of  \eqref{eq:doubleint:remainder}. The second estimate can be obtained analogously.

What remains to show is that the contributions from  $I_{\textnormal{out}}^{(1)}(t)$, $I_{\textnormal{out}}^{(2)}(t)$ and $I_{\textnormal{out}}^{(3)}(t)$ are small under the event $\{\tau_{\textnormal{out}} \ge \hat{t}_0 \}.$ Note that we already know that the contribution from $I_{\textnormal{out}}^{(3)}(t)$ is small, namely,
\begin{equation}\label{eq:doubleint:1}
\frac{\alpha}{1+2\alpha}\mathbb{E} \left[I_{\textnormal{out}}^{(3)}(t ) \mathds{1}\left\{\tau_{\textnormal{out}} \ge \hat{t}_0  \right\} \right] = o\left(\alpha^{2+\frac{1}{5}} \right).
\end{equation}

Moving on to the investigation of $I_{\textnormal{out}}^{(2)}(t)$, let $\delta_0 = \alpha^{-\frac{3}{2}+8\epsilon}$. From Lemma \ref{lem:ind1:tau13}, the event 
\begin{equation}
\bigcap_{t\in[t_0^\dagger+\delta_0,\hat{t}_0]} \left\{ \intop_{t-\delta_0}^t |S(s)-\alpha|ds \le \alpha^\epsilon  \right\}
\end{equation}
happens with probability at least $1- e^{-\beta^3}-r$. Conditioned on this event, the probability that $\pi_1(t;S\triangle \alpha)<\delta_0$ is smaller than $\alpha^\epsilon$ for each fixed $t$. Thus, we have 
\begin{equation}\label{eq:doubleint:exp of closest pts}
\begin{split}
&\mathbb{E}\left[\left.\sigma_1(t;S\triangle \alpha)\mathds{1}\left\{\tau_{\textnormal{out}}\ge \hat{t}_0 \right\} \, \right| \,\mathcal{F}_{t_0} \right] \le \alpha^{\frac{3}{4}-5\epsilon};\\  &\mathbb{E}\left[\left.\sigma_1(t;S\triangle \alpha)^2\mathds{1}\left\{\tau_{\textnormal{out}}\ge \hat{t}_0 \right\} \, \right|\, \mathcal{F}_{t_0} \right] \le \alpha^{\frac{3}{2}-9\epsilon}.
\end{split}
\end{equation}
Combining with the definition of $\tau_{\textnormal{out}}^{(2)}$ and Lemma \ref{lem:doubleint:out1}, this gives that
\begin{equation}\label{eq:doubleint:2}
\frac{\alpha}{1+2\alpha}\mathbb{E} \left[\left.I_{\textnormal{out}}^{(2)}(t) \mathds{1}\left\{\tau_{\textnormal{out}}\ge\hat{t}_0 \right\} \, \right| \, \mathcal{F}_{t_0} \right] = o\left(\alpha^{2+\frac{1}{5}} \right).
\end{equation}
Similarly, we have
\begin{equation}
\mathbb{E}[\sigma_1(t;S)^2\mathds{1}\{\tau_{\textnormal{out}}\ge \hat{t}_0 \} \, | \, \mathcal{F}_{t_0} ] \le \alpha^{1-\epsilon},
\end{equation}
and hence from the definition of $\tau_{\textnormal{out}}^{(1)}$, we get 
\begin{equation}\label{eq:doubleint:3}
\frac{\alpha}{1+2\alpha} \mathbb{E}\left[\left. I_{\textnormal{out}}^{(1)}(t) \mathds{1}\{\tau_{\textnormal{out}} \ge \hat{t}_0 \} \, \right| \, \mathcal{F}_{t_0} \right] = o \left(\alpha^{2+ \frac{1}{5}} \right).
\end{equation}

To conclude the proof, recall that 
$$\mathcal{J}(t) = \frac{\alpha}{1+2\alpha} \big(I_{\textnormal{out}}(t) + I_{\textnormal{out}}^{(1)}(t)+ I_{\textnormal{out}}^{(2)}(t ) + I_{\textnormal{out}}^{(3)}(t )\big),$$
up to an $O(\alpha^{50})$-error that comes from \eqref{eq:doubleint:errorbasic}, whose contribution can be controlled analogously as  \eqref{eq:doubleint:error:upd rough}, and combine \eqref{eq:doubleint:1}, \eqref{eq:doubleint:2} and \eqref{eq:doubleint:3}.
\end{proof}
 
 \begin{proof}[Proof of Lemma \ref{lem:doubleint:out1}]
 We begin with noting that $\tau_{\textnormal{out}}^{(1)} \ge \hat{t}_0$ with high probability, due to Corollary \ref{cor:ind2:tau61sharp1}. To establish the control on $\tau_{\textnormal{out}}^{(2)}$ and $\tau_{\textnormal{out}}^{(3)}$, recall the definitions of $\tau_{\textnormal{gap}}^{(1)}$, $\tau_{\textnormal{gap}}^{(2)}$ (Corollaries \ref{cor:ind1:gap num of pts}, \ref{cor:ind1:gap mg}),  $\tilde{\tau}_{3,3}^{(3)}$ and $\tilde{\tau}_{3,3}^\sharp$ (Corollaries \ref{cor:ind2:tau63sharp:Kstar}, \ref{cor:ind2:tau63sharp}). On the event $\{ \tau_{\textnormal{gap}}^{(1)}\wedge \tau_{\textnormal{gap}}^{(2)} \wedge \tilde{\tau}_{3,3}^{(3)} \wedge \tilde{\tau}_{3,3}^\sharp \ge \hat{t}_0 \}$, we can write
 \begin{equation}
 \begin{split}
I_{\textnormal{out}}^{(2)}(t)\le
 \intop_{t_0^\dagger}^t \left( \frac{\alpha^{\frac{1}{2}}\beta^{6\theta}}{\sqrt{t-s+1}} + \frac{\beta^{C_\circ +1}}{t-s+1} \right) \left( \frac{\alpha\beta^{C_\circ}}{\sqrt{\pi_1(s;S\triangle \alpha)}} + \alpha^{\frac{7}{4}-\epsilon} \right)ds .
 \end{split}
 \end{equation}
  Then, the RHS is bounded using Lemma \ref{lem:ind1:pi1int basicbd} with parameters $\Delta_0 = \alpha^{-\frac{3}{2}+7\epsilon}, N_0 = \beta_0^5, K = \alpha^{-\frac{1}{2}} $ and $\Delta_1 = \alpha^{-\frac{3}{2}}.$
  This gives
  \begin{equation}
  \begin{split}
   \intop_{t_0^\dagger}^t \left( \frac{\alpha^{\frac{1}{2}}\beta^{6\theta}}{\sqrt{t-s+1}} + \frac{\beta^{C_\circ +1}}{t-s+1} \right) \left( \frac{\alpha\beta^{C_\circ}}{\sqrt{\pi_1(s;S\triangle \alpha)}} + \alpha^{\frac{7}{4}-\epsilon} \right)ds \\
   \le \frac{\alpha^{1-\epsilon}}{\sqrt{\pi_1(t;S\triangle \alpha)+1}} + \frac{\alpha^{\frac{1}{4}-\epsilon}}{\pi_1(s;S\triangle \alpha)+1} + \alpha^{\frac{5}{4}-2\epsilon} .
  \end{split}
  \end{equation}
  Note that if there were  intervals of length $\Delta_1$ without any points, then we can just artificially add  points to justify the choice of $\Delta_1$; this operation will only increase the value of the whole integral.
 
 Continuing our work on the same event as above, we express that
 \begin{equation}
 I_{\textnormal{out}}^{(3)}(t) \le \intop_{t_0^\dagger}^t  \left( \frac{\alpha^{\frac{1}{2}}\beta^{6\theta}}{\sqrt{t-s+1}} + \frac{\beta^{C_\circ +1}}{t-s+1} \right) \alpha^{2-2\epsilon} ds \le \alpha^{\frac{3}{2}-3\epsilon},
 \end{equation}
 concluding the proof of Lemma \ref{lem:doubleint:out1}. 
 \end{proof}
 
 \subsubsection{Reshaping the inner integral}\label{subsubsec:doubleint:inner}
 
 Building upon Lemma \ref{lem:doubleint:out}, we investigate the inner integral of $I_{\textnormal{out}}(t)$ and establish Proposition \ref{prop:double int:main approx}. Define
 \begin{equation}
\begin{split}
& F(s):= \intop_{t_0^+}^s K^*_\alpha(s-u)d\widetilde{\Pi}_\alpha(u) + \big[\alpha''-\alpha \big];\\
& \tau_{\textnormal{in}}^{(1)}:= \inf \left\{ s\ge t_0^+: |F(s)| \ge \alpha \beta^{C_\circ} \sigma_1(s;\alpha) + 3\alpha^{\frac{3}{2}}\beta^{6\theta} \right\}.
\end{split}
 \end{equation}
Moreover, note that $\alpha''-\alpha = (\alpha''-\alpha' )+ (\alpha'-\alpha)$, and \eqref{eq:integralform:branching} gives 
\begin{equation}
\alpha''-\alpha' =\intop_{t_0}^{t_0^+} K^*_{\alpha} \{d\Pi_S(x) - S_1(x)dx \}.
\end{equation}
Then, the same argument as Lemma \ref{lem:ind2:alpha1vsalpha0} tells us that
\begin{equation}
\PP \left(\left. \left| \alpha''-\alpha' \right| \ge \alpha^{\frac{3}{2}}\beta^{2\theta} \,\right| \, \mathcal{F}_{t_0} \right) \le \exp \left(-\beta^5 \right).
\end{equation}
 Thus, Theorem \ref{thm:reg:conti:main}, Corollary \ref{cor:ind1:taux:fixedrate}, and the definition of $\mathcal{A}_3(\alpha,t_0)$ (Section \ref{subsubsec:reg:history}) tell us that
 \begin{equation}
 \PP\left(\left. \tau_{\textnormal{in}}^{(1)} <\hat{t}_0 \,\right|\, \mathcal{F}_{t_0} \right) \le \exp\left(-\beta^4 \right)+r.
 \end{equation}
 
 We also decompose $d\widehat{\Pi}_S(u)$ in $I_{\textnormal{out}}(t)$ as
 \begin{equation}\label{eq:doubleint:innerint decomp}
 \begin{split}
 &d\widehat{\Pi}_S(u) = d\widetilde{\Pi}_\alpha(u) + d\widetilde{\Pi}_{S-\alpha}(u) + (S(u)-S_1(u))du + (S_1(u)-\alpha) du,
 \end{split}
 \end{equation}
 and then writing $(S_1(u)-\alpha)$ as \eqref{eq:doubleint:S1minusalpha decomp}. Define the stopping times
 \begin{equation}
 \begin{split}
 &\tau_{\textnormal{in}}^{(2)}:= \inf\left\{ t \ge \acute{t}_0: \left|\intop_{t_0^\dagger}^t\intop_{t_0^\dagger}^s J^{(\alpha)}_{t-s,t-u} (S(u)-S_1(u)) F(s)duds  \right| \ge \alpha^{\frac{3}{2}-\epsilon}  \right\};\\
 & \tau_{\textnormal{in}}^{(3)}:= \inf\left\{ t \ge \acute{t}_0: \left|\intop_{t_0^\dagger}^t\intop_{t_0^\dagger}^s J^{(\alpha)}_{t-s,t-u} F(s) \intop_{t_0^+}^u K^*_\alpha(u-v)d\widetilde{\Pi}_{S-\alpha}(v) du ds \right| \ge \alpha^{\frac{5}{4}-\epsilon}  \right\};\\
&   \tau_{\textnormal{in}}^{(4)}:= \inf\left\{ t \ge \acute{t}_0: \left|\intop_{t_0^\dagger}^t\intop_{t_0^\dagger}^s J^{(\alpha)}_{t-s,t-u} F(s) \intop_{t_0^+}^u K^*_\alpha(u-v) (S(u)-S_1(v))dvduds \right| \ge \alpha^{\frac{3}{2}-\epsilon}  \right\}.
 \end{split}
 \end{equation}
 We again stress that the integrals of $dv$ are over $[t_0^+,u]$, not $[t_0^\dagger,u]$.
 Among the rest of the integrals rather than the ones stated in the above definitions, $d\widetilde{\Pi}_\alpha(u)$ and $\intop_{t_0^+}^u K^*_\alpha(u-v) d\widetilde{\Pi}_\alpha(v)du$ form the leading order $I_1(t)$ and $I_2(t)$. The other contributions coming from $d\widetilde{\Pi}_{S-\alpha}(u)$ and $(\alpha''-\alpha)du$ will be studied later. Note that in the above we consider $u \ge t_0^\dagger$, and hence the contribution from $[\mathcal{R}_c(t_0^+,u;\Pi_S[t_0^-,t_0^+],\alpha ) -\alpha''] du$ is negligible from \eqref{eq:doubleint:errorbasic} by the argument in Lemma \ref{lem:doubleint:out}.
 
 \begin{lem}\label{lem:doubleint:in1ez}
 	Under the setting of Theorem \ref{thm:double int:main}, we have
 	\begin{equation}
 	\PP \left(\left. \tau_{\textnormal{in}}^{(2)}\wedge \tau_{\textnormal{in}}^{(3)}\wedge \tau_{\textnormal{in}}^{(4)} < \hat{t}_0 \,\right|\, \mathcal{F}_{t_0} \right) \le \exp\left(-\beta_0^2 \right) +r.
 	\end{equation}
 \end{lem}
 
 \begin{proof}
 	The proof is very similar to that of Lemma \ref{lem:doubleint:out1}, and hence we concisely explain which lemmas from the previous sections are used to establish the conclusion. 
 	
 	For $\tau_{\textnormal{in}}^{(2)},$ we rely on Corollary \ref{cor:ind2:tau63sharp:1} to estimate the inner integral over $du$, and then combine with the bound from $\tau_{\textnormal{in}}^{(1)}$, using Lemma \ref{lem:ind1:pi1int basicbd} to perform integration over $ds$.
 	
 	For $\tau_{\textnormal{in}}^{(3)}$, recall the definition of $\tau_{\textnormal{gap}}^{(1)}, \tau_{\textnormal{gap}}^{(2)}$ from Corollaries  \ref{cor:ind1:gap num of pts}, \ref{cor:ind1:gap mg}, and note that on $\{\tau_{\textnormal{gap}}^{(1)}\wedge \tau_{\textnormal{gap}}^{(2)} \ge \hat{t}_0 \}$ we have
 	\begin{equation}
\begin{split}
 &\left|\intop_{t_0^\dagger}^s J^{(\alpha)}_{t-s,t-u}  \intop_{t_0^+}^u K^*_\alpha(u-v)d\widetilde{\Pi}_{S-\alpha}(v) du\right|\\
 &\le 
\frac{e^{-c\alpha^2(t-s) }}{\sqrt{t-s+1}} \intop_{t_0^+}^s \frac{1}{\sqrt{t-u+1}} \left( \frac{\alpha\beta^{C_\circ}}{ \sqrt{\pi_1(u;S\triangle \alpha)}}  + \alpha^{\frac{7}{4}-\epsilon} \right)du
\le \frac{\alpha^{\frac{3}{4}-9\epsilon} e^{-c\alpha^2(t-s) }}{\sqrt{t-s+1}} ,
\end{split}
 	\end{equation}
 	where the last inequality follows from applying Lemma \ref{lem:ind1:pi1int basicbd} as before in Lemma \ref{lem:doubleint:out1}. Then the control on $\tau_{\textnormal{in}}^{(3)}$ is obtained by performing the outer integral over $ds$, using the bound on $F(s)$ and Lemma \ref{lem:ind1:pi1int basicbd}.
 	
 	Lastly, for $\tau_{\textnormal{in}}^{(4)}$, we rely on $\tilde{\tau}_{3,3}^{(3)}$ from  Corollary \ref{cor:ind2:tau63sharp:Kstar}, where on $\{ \tilde{\tau}_{3,3}^{(3)} \ge \hat{t}_0 \}$ we have
 	\begin{equation}
\begin{split}
 	\left|\intop_{t_0^\dagger}^s J^{(\alpha)}_{t-s,t-u}  \intop_{t_0^+}^u K^*_\alpha(u-v) (S(u)-S_1(v))dvdu \right| \le \intop_{t_0^\dagger}^s \left|J^{(\alpha)}_{t-s,t-u}\right| \alpha^{2-2\epsilon} du \le \alpha^{2-3\epsilon}.
\end{split}
 	\end{equation}
 	Then, conducting the outer integral over $ds$ similarly as above gives the estimate on $\tau_{\textnormal{in}}^{(4)}$, concluding the proof.
 \end{proof}
 
 As mentioned right before Lemma \ref{lem:doubleint:in1ez}, we now study the remaining integrals, beginning with 
 \begin{equation}
\begin{split}
& I_{\textnormal{in}}^{(5)}(t):= \intop_{t_0^\dagger}^t \intop_{t_0^\dagger}^s J^{(\alpha)}_{t-s,t-u} d\widetilde{\Pi}_{S-\alpha}(u) \cdot \intop_{t_0^+}^s K^*_\alpha(s-x) d\widetilde{\Pi}_\alpha(x)ds ;\\
  &I_{\textnormal{in}}^{(6)}(t):= \intop_{t_0^\dagger}^t \intop_{t_0^\dagger}^s J^{(\alpha)}_{t-s,t-u} d\widetilde{\Pi}_{S-\alpha}(u) \cdot \big[\alpha''-\alpha \big]ds.
\end{split}
 \end{equation} 
 
 \begin{lem}\label{lem:doubleint:in2 mg}
 	Under the setting of Theorem \ref{thm:double int:main}, we have for all $t\in[\acute{t}_0,\hat{t}_0]$ that
 	\begin{equation}
 	\mathbb{E}\left[\left. I_{\textnormal{in}}^{(5)}(t)\, \right| \, \mathcal{F}_{t_0} \right] = O\left(\alpha^{\frac{5}{4}-\epsilon} \right)= 	\mathbb{E}\left[\left. I_{\textnormal{in}}^{(6)}(t)\, \right| \, \mathcal{F}_{t_0} \right] .
 	\end{equation}
 \end{lem}
 
 Although these integrals should have a negligible size compared to the leading order,  one can see that  the same argument as the previous analysis does not give the correct estimate on $I_{\textnormal{in}}^{(5)}$. To overcome this issue, we remind ourselves  that our actual interest  is to control the expected value of this quantity.
 
 \begin{proof}[Proof of Lemma \ref{lem:doubleint:in2 mg}]
 We begin with investigating $I_{\textnormal{in}}^{(5)}$. Define $\underline{\alpha}:=\alpha - \alpha^{\frac{3}{2}-\epsilon},$ and recall from Proposition \ref{prop:reg:lbd of speed} that $\tau_{\textnormal{lo}}(\alpha,t_0) \ge \hat{t}_0$ with high probability, that is, $S(t)$ is likely to stay above from $\underline{\alpha}.$ Thus, writing $\tau_{\textnormal{lo}}=\tau_{\textnormal{lo}}(\alpha,t_0)$, observe that
 \begin{equation}
 \mathbb{E}\left[\left.
 \intop_{t_0^\dagger}^{t\wedge \tau_{\textnormal{lo}}} \left\{ \intop_{t_0^\dagger}^s J^{(\alpha)}_{t-s,t-u} d\widetilde{\Pi}_{S-\alpha}(u) \cdot \intop_{t_0^+}^s K^*_\alpha(s-x) d\widetilde{\Pi}_{\underline{ \alpha}}(x) \right\}ds
 \,\right|\, \mathcal{F}_{t_0} \right] =0,
 \end{equation}
 since $d\widetilde{\Pi}_{S-\alpha}$ and $d\widetilde{\Pi}_{\underline{\alpha}}$ define martingales and depend on disjoint set of points. Thus, under extra conditioning on $\Pi_{\underline{\alpha}}[t_0,\hat{t}_0]$, we can see that  the above will always be zero.
 Therefore, relying on Proposition \ref{prop:reg:lbd of speed} and the argument from Lemma~\ref{lem:doubleint:out} gives that for all $t\in[\acute{t}_0,\hat{t}_0]$,
 \begin{equation}
  \mathbb{E}\left[\left.
  \intop_{t_0^\dagger}^{t}\left\{ \intop_{t_0^\dagger}^s J^{(\alpha)}_{t-s,t-u} d\widetilde{\Pi}_{S-\alpha}(u) \cdot \intop_{t_0^+}^s K^*_\alpha(s-x) d\widetilde{\Pi}_{\underline{ \alpha}}(x)\right\} ds
  \,\right|\, \mathcal{F}_{t_0} \right] =O(\alpha^{50}).
 \end{equation}
 
 Writing $d\widetilde{\Pi}_{\triangle \alpha} = d\widetilde{\Pi}_\alpha - d\widetilde{\Pi}_{\underline{\alpha}}$, define the stopping time
 \begin{equation}
 \tau_{\textnormal{in}}^{(5)} :=\inf \left\{
 t \ge \acute{t}_0: \left| \intop_{t_0^\dagger}^t \left\{\intop_{t_0^\dagger}^s J^{(\alpha)}_{t-s,t-u} d\widetilde{\Pi}_{S-\alpha}(u) \cdot \intop_{t_0^+}^s K^*_\alpha(s-x) d\widetilde{\Pi}_{\triangle\alpha}(x)\right\}ds \right| \ge \alpha^{\frac{5}{4}-\epsilon}
  \right\}.
 \end{equation}
 To conclude the proof, it suffices to show that
 \begin{equation}
 \PP \left(\left. \tau_{\textnormal{in}}^{(5)} <\hat{t}_0 \, \right|\, \mathcal{F}_{t_0} \right) \le \exp\left(-\beta_0^2 \right) +r.
 \end{equation}
 Since this estimate follows similarly as the proof of Lemmas \ref{lem:doubleint:out1} and \ref{lem:doubleint:in1ez},  we give a brief explanation on its proof.  Since $J^{(\alpha)}_{t-s,t-u}$ satisfies a similar bound as $K^*_\alpha$ (Lemma \ref{lem:estimat for K tilde:intro}), analogous argument as Corollary \ref{cor:ind1:gap mg} gives an estimate the first part of the integral,  written as
 \begin{equation}
\begin{split}
&\tau_{\textnormal{in}}^{(6)}:= \inf \left\{ t\ge \acute{t}_0: \, \exists s\in[t_0^\dagger,t] \textnormal{ s.t. } \left| \intop_{t_0^\dagger}^s J^{(\alpha)}_{t-s,t-u} d\widetilde{\Pi}_{S-\alpha} (u)\right| \ge \frac{\beta_0^{C_\circ} }{t-s+1} + \frac{\alpha^{\frac{3}{4}-\epsilon}}{\sqrt{t-s+1}}\right\};\\
& \PP \left(\left.\tau_{\textnormal{in}}^{(6)} < \hat{t}_0 \, \right|\, \mathcal{F}_{t_0} \right) \le \exp\left(-\beta^3 \right)+r.
\end{split}
 \end{equation}
(Note that $\sigma_1(t;S\triangle \alpha)$ in Corollary \ref{cor:ind1:gap mg} becomes $(t-s+\pi_1(s)+1)^{-1/2} \le (t-s+1)^{-1/2}$ in this case.)
 The second part of the integral is controlled by Corollary \ref{cor:ind1:gap mg 2}, and hence combining the two bounds with an application of Lemma \ref{lem:ind1:pi1int basicbd} concludes the proof, similarly as the previous lemmas. 
 
  Finally, we note that $I_{\textnormal{in}}^{(6)}(t)$ is a martingale (note that $\alpha''$ is $\mathcal{F}_{t_0^+}$-measurable while the integral starts from $t_0^\dagger$) and hence its expectation is zero.
\end{proof}
 
 Now we are left with the two final integrals, which are 
 \begin{equation}\label{eq:doubleint:in:last piece}
 \begin{split}
& I_{\textnormal{in}}^{(7)} (t):= \intop_{t_0^\dagger}^t \left\{\intop_{t_0^\dagger}^s J^{(\alpha)}_{t-s,t-u} du \cdot \intop_{t_0^+}^s K^*_\alpha(s-x)d\widetilde{\Pi}_\alpha(x)\right\} (\alpha''-\alpha) ds;\\
 & I_{\textnormal{in}}^{(8)} (t):= \intop_{t_0^\dagger}^t \intop_{t_0^\dagger}^s J^{(\alpha)}_{t-s,t-u}  (\alpha''-\alpha)^2du ds = \intop_0^{t-t_0^\dagger} \intop_0^s J_{u,s}^{(\alpha)} (\alpha''-\alpha)^2duds.
 \end{split}
 \end{equation}
 Note that $I_{\textnormal{in}}^{(7)}(t)$ is a  martingale, since $\alpha''$ is $\mathcal{F}_{t_0^+}$-measurable. Thus, we have $\mathbb{E} [I^{(7)}_{\textnormal{in}}(t) | \mathcal{F}_{t_0} ]=0$. On the other hand, $I_{\textnormal{in}}^{(8)}(t),$ which is a deterministic integral, is bounded from  \eqref{eq:I_4}: $I_{\textnormal{in}}^{(8)}(t) = O(\alpha^{2-\epsilon}) $ for any $t\in[\acute{t}_0,\hat{t}_0]$.
 
 Combining all the analysis done in this subsection, we conclude the proof of Proposition \ref{prop:double int:main approx}.
 \begin{proof}[Proof of Proposition \ref{prop:double int:main approx}]
 	In \eqref{eq:doubleint:innerint decomp} and the discussions below, we have seen that $I_{\textnormal{out}}(t)$ in Lemma \ref{lem:doubleint:out} can be decomposed into 
 	\begin{equation}
 	I_1(t), \, I_2(t), \,\, \left\{\textnormal{integrals in } \tau_{\textnormal{in}}^{(2)},\,  \tau_{\textnormal{in}}^{(3)},\,  \tau_{\textnormal{in}}^{(4)} \right\}, \ I_{\textnormal{in}}^{(5)}(t), \, I_{\textnormal{in}}^{(6)}(t), \, I_{\textnormal{in}}^{(7)}(t),  \textnormal{ and } I_{\textnormal{in}}^{(8)}(t),
 	\end{equation}
 	up to an $O(\alpha^{50})$ error coming from \eqref{eq:doubleint:errorbasic}. 
 	Then, Lemmas \ref{lem:doubleint:in1ez}, \ref{lem:doubleint:in2 mg} and \eqref{eq:doubleint:in:last piece} tell us that the conditional expectations given $\mathcal{F}_{t_0}$ of all of the above integrals except $I_1(t)$ and $I_2(t)$ are of order $O(\alpha^{\frac{5}{4}-\epsilon})$. This implies
 	\begin{equation}
 	\mathbb{E} \left[I_{\textnormal{out}}(t) \,|\,\mathcal{F}_{t_0} \right] = \mathbb{E} \left[I_1(t)+I_2(t) \,| \,\mathcal{F}_{t_0} \right] + O\left(\alpha^{\frac{5}{4}-\epsilon}\right),
 	\end{equation}
 	which concludes the proof combined with Lemma \ref{lem:doubleint:out}.
 \end{proof}

 \subsection{The analysis on the increment}\label{subsec:increment proof}
 
 We give a formal statement and proof of Theorem \ref{thm:speed:increment:informal}, thus verifying the assumptions of Theorem 4.1 are indeed accurate.
 
 \begin{thm}\label{thm:increment:formal}
 	Let $\alpha_0,t_0>0$, set $t_0^-, t_0^+,\acute{t}_0$ and $\hat{t}_0$ as \eqref{eq:def:t0t1} and \eqref{eq:def:t0t11} in terms of $\alpha_0$, and let $r=e^{-\beta_0^{3/2}}$. Also, let $t_1$ be any number satisfying \eqref{eq:t1 regime},  set $t_1^-:= t_1-\alpha_0^{-2}\beta_0^\theta$. Furthermore, define
 	\begin{equation}
 	\alpha_1 := \mathcal{L}(t_1^-; \Pi_S[t_0^-,t_1^-], \alpha_0),
 	\end{equation}
 	and $\alpha_0', \alpha_1'$ as \eqref{eq:def:alpha0prime}.
 	Then, for any $\Pi_S(-\infty,t_0]$ that is $(\alpha_0,r;[t_0])$-sharp-regular, we have
 	\begin{equation}
 	\begin{split}
 	\mathbb{E}\left[\alpha_1'-\alpha_0'\, | \, \Pi_S(-\infty,t_0] \right] &= o\left(\alpha_0^4\beta_0^{-3} \right) \cdot (t_1 -t_0);\\
 	\var \left[ \alpha_1'-\alpha_0'\, | \, \Pi_S(-\infty,t_0] \right] &= (1+o(\alpha_0^{\frac{1}{3}})) 4\alpha_0^5(t_1-t_0).
 	\end{split}
 	\end{equation}
 \end{thm}

 \begin{proof}	
 	Let $\tau =\tau(\alpha_0,t_0,2)$, $\tau^\sharp = \tau^\sharp(\alpha_0,t_0,2)$ be as \eqref{eq:def:tau:induction}, and recall the definition of  $\tilde{\alpha}_1$ from Proposition \ref{prop:reg:newrates}. We define the event
 	\begin{equation}
 	\begin{split}
 	\mathcal{B}_1:=& \Big\{ \tau \ge \hat{t}_0 \Big\} \bigcap \Big\{\tau^\sharp \ge \acute{t}_0 \Big\} \bigcap \Big\{|\alpha_1'-\tilde{\alpha}_1| \le 2\alpha_0^2\beta_0^{8\theta+1} \Big\};\\
 	\mathcal{B}_2:=&\left\{\intop_{\acute{t}_0}^{t_1} S(t)dt = \alpha_0(t_1-\acute{t}_0) + O\left(\alpha_0^{-\frac{1}{2}-\epsilon}\right) =\intop_{\acute{t}_0}^{t_1} S_1(t)dt  \right\};\\
 		\mathcal{B}_3:=&\left\{ \intop_{\acute{t}_0}^{t_1} |S(t)-S_1(t)| dt \le \alpha_0^{-2\epsilon} \right\} \bigcap \left\{ \intop_{\acute{t}_0}^{t_1} |S(t)-S_2(t)| dt \le \alpha_0^{\frac{1}{2}-2\epsilon} \right\} ;\\
 	\mathcal{B}_4:= & \left\{ \intop_{t_0}^{t_1} S(t)dt = \alpha_0(t_1-t_0) + O \left(\alpha_0^{-\frac{1}{2}-\epsilon} \right) \right\};\\	
 	\mathcal{B}_5:=&
 	\left\{ \intop_{t_0}^{t_1}   \sigma_1\sigma_2\sigma_3(t;S) dt \le \alpha_0^{-\frac{1}{2}-\epsilon} \right\} ;\\
 	\mathcal{B}:=& \mathcal{B}_1 \bigcap \mathcal{B}_2 \bigcap \mathcal{B}_3\bigcap \mathcal{B}_4 \bigcap\mathcal{B}_5,
 	\end{split}
 	\end{equation}
 	where $S_1(t)=S_1(t;t_0^-,\alpha_0)$, $S_2(t)=S_2(t;t_0^-,\alpha_0)$ are given as \eqref{eq:def:S1:basic form}, \eqref{eq:def:S2:basic form}.
 	Conditioned on $\Pi_S(-\infty,t_0]$, $\mathcal{B}$ happens with probability at least $1-2e^{-\beta_0^{3/2}}$:
 	\begin{itemize}
 		\item $\mathcal{B}_1$ is obtained from the definition of sharp-regularity, Proposition \ref{prop:reg:newrates}, and Theorem \ref{thm:reg:conti:main}.
 		
 		\item $\mathcal{B}_2, \mathcal{B}_3$ come from Lemma \ref{lem:reg:errorint}, and the  sharp-regularity  along with $\mathcal{B}_2$ implies $\mathcal{B}_4$.
 		
 		\item $ \mathcal{B}_5$ is from Lemma \ref{lem:fixed perturbed:error int:perturbed} (whose assumptions are satisfied by $\tau_{1}, \tau_{4}$, and $ \tau_{5}$).
 	\end{itemize}
 	Hence
 	the contributions to the mean and variance of increments on $\mathcal{B}^c$ are negligible (since $||S|| \le \frac{1}{2}$).
 	
 	We write $\alpha_1' - \alpha_0' = (\alpha_1' - \tilde{\alpha}_1) + (\tilde{\alpha}_1-\alpha_0')$ and similarly as \eqref{eq:Lt diff:basic} and \eqref{eq:ind2:a1vsa0:split}, we express that
 	\begin{equation}\label{eq:increment:mean:0}
 	\tilde{\alpha}_1-\alpha_0' = \intop_{t_0}^{t_1} K^*_{\alpha_0} \{d\Pi_S (t) - S_1(t)dt\}.
 	\end{equation}
From these formulas, we begin with estimating the mean of the increment. Note that from Lemma \ref{lem:ind2:alphaPvsalphaPP},
 	\begin{equation}\label{eq:increment:mean:1}
 	\mathbb{E}[|\alpha_1'-\tilde{\alpha}_1|\,;\, \mathcal{B}  \ | \, \mathcal{F}_{0}] = o(\alpha_0^4 \beta_0^{-3}) \cdot (t_1-t_0).
 	\end{equation}
 Furthermore, we have
 	\begin{equation}\label{eq:increment:mean:2}
\begin{split}
 		\mathbb{E}[\tilde{\alpha}_1-\alpha_0'  \, | \, \mathcal{F}_{0}]
 	= 	K^*_{\alpha_0} \cdot \mathbb{E}\left[\left. \intop_{t_0}^{t_1} \{S(t)-S_1(t) \}  dt \, \right| \, \mathcal{F}_{0}\right]
\end{split}
 	\end{equation}
 	We decompose this integral into two parts, from $t_0$ to $\acute{t}_0$ and from $\acute{t}_0$ to $t_1$. For the first one, observe that the definition of $\mathcal{B}$ gives
 	\begin{equation}\label{eq:increment:mean:3}
 	\begin{split}
 \left|	\mathbb{E}\left[\left.\intop_{t_0}^{\acute{t}_0} \{S(t)-S_1(t) \}dt  \,; \, \mathcal{B}\ \right| \, \mathcal{F}_{0}\right]\right|
 \le
\beta_0^{4\theta} = o(\alpha_0^2\beta_0^{-3}) (t_1-t_0).
 	\end{split}
 	\end{equation}

 	To study the integral from $\acute{t}_0$ to $t_1$ of \eqref{eq:increment:mean:2}, we recall the formula \eqref{eq:def:S2:basic form} to see that
 	\begin{equation}
 	S(t) - S_1(t)= S(t)-S_2(t) + \frac{2\alpha_0^2}{(1+2\alpha_0)^2} - \frac{4\alpha_0+4\alpha_0^2}{(1+2\alpha_0)^2} S_1(t) + \mathcal{J}(t),
 	\end{equation}
 	where $\mathcal{J}(t)$ is defined in the beginning of this section.
 	Thus, we can see that for all $t\in[\acute{t}_0, t_1]$,
 	\begin{equation}
 	\mathbb{E} \left[ \left. \frac{2\alpha_0^2}{(1+2\alpha_0)^2} + \mathcal{J}(t)  \, ; \, \mathcal{B}\, \right| \ \mathcal{F}_0 \right] = (1+ O(\alpha^{\epsilon})) 4\alpha_0^2.
 	\end{equation}
 	 Moreover,  the event $\mathcal{B}$ gives 
 	\begin{equation}
 	\mathbb{E} \left[ \left. \intop_{\acute{t}_0}^{t_1} S_1(t)dt  \, ; \, \mathcal{B}\, \right| \ \mathcal{F}_0 \right]  = \alpha_0 (t_1-\acute{t}_0) + O\left(\alpha_0^{-\frac{1}{2}-\epsilon}\right).
 	\end{equation}
 	In addition, the integral of $|S(t)-S_2(t)|$ is bounded from the definition of $\mathcal{B}$, which is smaller than $\alpha_0^{\frac{1}{2}-2\epsilon}=o(\alpha_0^4\beta_0^{-1})\cdot (t_1-t_0)$. Combining these three estimates, we observe that the terms of order $\alpha_0^2(t_1-t_0)$ cancel out, and hence deduce that
 	\begin{equation}\label{eq:increment:mean:4}
 	\mathbb{E}\left[\left.  \intop_{\acute{t}_0}^{t_1} \{S(t)-S_1(t) \} dt \,;  \ \mathcal{B}\, \right| \ \mathcal{F}_0 \right] = o(\alpha_0^2 \beta_0^{-3})\cdot (t_1-t_0).
 	\end{equation}
 	Therefore, combining \eqref{eq:increment:mean:1}, \eqref{eq:increment:mean:2}, \eqref{eq:increment:mean:3} and \eqref{eq:increment:mean:4} concludes the proof of the first statement.
 	
 	\vspace{2mm}
 	
 	To control the variance, we first write $\alpha_1'-\alpha_0' = K_{\alpha_0}^*\mathcal{M}+\mathcal{D}$, where
 	\begin{equation}
 	\mathcal{M}:=  \intop_{t_0}^{t_1} \{d\Pi_S(t)-S(t)dt\}, \quad \mathcal{D}:= (\alpha_1'-\tilde{\alpha}_1) + K^*_{\alpha_0}\cdot\intop_{t_0}^{t_1}\{S(t)-S_1(t)\} dt,
 	\end{equation}
 	and also define $\overline{\mathcal{D}}:=\mathbb{E} [ \mathcal{D} \, | \, \mathcal{F}_0].$
 	Then, we express that
 	\begin{equation}\label{eq:increment:var:1}
 	\begin{split}
 	\left| \var\left[\left. \alpha_1'-\alpha_0'  \, \right|\,\mathcal{F}_0 \right] -\var\left[\left.  K_{\alpha_0}^*\mathcal{M} \, \right|\,\mathcal{F}_0 \right] \right| 
 	\le
 	 \var[\mathcal{D}\,|\,\mathcal{F}_0]+ 2K^*_{\alpha_0} \cdot \Big\{	\mathbb{E} [ \mathcal{M}^2\,|\, \mathcal{F}_0] \cdot \var[\mathcal{D}\,|\,\mathcal{F}_0]  \Big\}^{1/2},
 	\end{split}
 	\end{equation}
 	where we used Cauchy-Schwarz to deduce the inequality. Furthermore, since $\mathcal{M}$ is a martingale, we have
 	\begin{equation}\label{eq:increment:var:2}
 	\begin{split}
 	\mathbb{E}[\mathcal{M}^2\,|\, \mathcal{F}_0] &= \mathbb{E}\left[\left.\intop_{t_0}^{t_1} S(t)dt \, \right|\,\mathcal{F}_0 \right] = o(\alpha_0^{10}) + \mathbb{E}\left[\left.\intop_{t_0}^{t_1} S(t)dt \, ;\ \mathcal{B}\, \right|\,\mathcal{F}_0 \right] \\
 	&=
 	\alpha_0(t_1-t_0) + O\left(\alpha_0^{-\frac{1}{2}-\epsilon} \right) = \left(1+ o(\alpha_0^{\frac{1}{3}}) \right) \alpha_0(t_1-t_0).
 	\end{split}
 	\end{equation}
 	Moreover, on the event $\mathcal{B}$, we have $\mathcal{D} = O(\alpha_0^{2-2\epsilon})$. Thus, we see that $\var[\mathcal{D}\,|\, \mathcal{F}_0] = O(\alpha_0^{4-4\epsilon})$, and hence
 	\begin{equation}
 	\var[\mathcal{D} |\mathcal{F}_0 ] = o(\alpha_0^{1-5\epsilon}) (K_{\alpha_0}^*)^2 \mathbb{E} [\mathcal{M}^2 | \mathcal{F}_0 ] .
 	\end{equation}
 	Therefore, we combine the above computations \eqref{eq:increment:var:1} and \eqref{eq:increment:var:2} to deduce the second statement which finishes the proof.
 \end{proof}
 
 To conclude this section, we establish a variant of Theorem \ref{thm:increment:formal} which will be useful when applying to a slightly different setting in the next section.
 
 \begin{cor}\label{cor:increment:formal:goodevent}
     Under the setting of Theorem \ref{thm:increment:formal}, let $\mathcal{E}$ be an event that satisfies $\PP(\mathcal{E} | \mathcal{F}_{t_0}) \ge 1- \alpha_0^{5}$, for any $\mathcal{F}_{t_0}=\Pi_S(-\infty,t_0]$ that is $(\alpha_0,r;[t_0])$-sharp-regular. Then, we have
     \begin{equation}
         \var\left[\left.\left(\alpha_1'-\alpha_0'\right) \mathds{1}\{\mathcal{E} \} \, \right|\, \Pi_S(-\infty,t_0] \right] = (1+ o(\alpha_0^{\frac{1}{3}}))4\alpha_0^5(t_1-t_0).
     \end{equation}
 \end{cor}
 
 \begin{proof}
     For convenience, we write $\mathcal{F}_0= \mathcal{F}_{t_0}=\Pi_S(-\infty,t_0] $. Due to the analogous decomposition as \eqref{eq:increment:var:1}, it suffices to show that
     \begin{equation}
         \var \left[ (\alpha_1'-\alpha_0')\mathds{1}\{\mathcal{E}^c \}\, | \, \mathcal{F}_0 \right] = o(\alpha_0^4),
     \end{equation}
     since we have $ \var \left[ \alpha_1'-\alpha_0' | \mathcal{F}_0 \right]= \Theta(\alpha_0^3\beta_0^{10\theta})$. The proof goes analogously as that of Lemma \ref{lem:doubleint:out}, based on the same idea as \eqref{eq:doubleint:remainder}. Observe from \eqref{eq:ind2:a1vsa0:split} that
     \begin{equation}
\begin{split}
    \alpha_1-\alpha_0' &= \mathcal{L}(t_1^-;\Pi_S[t_0^-,t_1^-],\alpha_0) = K_{\alpha_0}^* \intop_{t_0^-}^{t_1^-} d\Pi_S(x) - S_1(x)dx\\
    &= K_{\alpha_0}^*|\Pi_S[t_0^-,t_1^-]| - K_{\alpha_0}^* \intop_{t_0^-}^{t_1^-} \intop_{y}^{t_1^-} K_{\alpha_0}(x-y) dx d\Pi_S(y).
\end{split}
     \end{equation}
     Since $\intop_y^{t_1^-} K_{\alpha_0}(x-y)dx \in [0,1]$, we can see that $|\alpha_1-\alpha_0'|$ is stochastically dominated by
     \begin{equation}
         |\alpha_1-\alpha_0'| \preceq
         K_{\alpha_0}^* |\Pi_{1/2} [t_0^-,t_1^-]|,
     \end{equation}
     due to the fact that $S\le \frac{1}{2}$ (Proposition \ref{prop:speed:basicdef}). Moreover, we also have
     \begin{equation}
         \begin{split}
             \alpha_1' = K_{\alpha_1}^* \intop_{t_1}^\infty \intop_{t_1^-}^{t_1} K_{\alpha_1}(x-y) d\Pi_S(y) dx \preceq K_{\alpha_1}^* | \Pi_{1/2}[t_1^-,t_1]|.
          \end{split}
     \end{equation}
     Thus, we can control the size of $\alpha_1$ and $\alpha_1'$ on the event $\mathcal{E}^c$ by coupling with the tail events of the Poisson process, which easily gives us the desired estimate. (Note that  $\alpha_0'$ is already close to $\alpha_0$ since the speed is $(\alpha_0,r;[t_0])$-sharp-regular.)
 \end{proof}
 
 \section{Multi-scale analysis and the scaling limit}\label{sec:scalinglimit}

In this section we repeatedly use the induction step theorem in order to prove the main results of this paper. The proof goes as follows. Let $t>0$ be sufficiently large and recall that Theorem~\ref{thm:main theorem 2} says that $t^{-\frac{2}{3}}S(st)$ should behave roughly like the solution of the SDE \eqref{eq:SDE equation in main theorem}. The proof has the following parts:
\begin{enumerate}
	\item 
	We show that in time $\approx \delta  t$ the aggregate is regular and the speed is of order $\approx \delta ^{-\frac{1}{3}}t^{-\frac{1}{3}}$.
	\item 
	We show that $t^{\frac{1}{3}}S(\delta t+st)$ is close to the solution of \eqref{eq:SDE equation in main theorem} with initial condition $Z(0)=\delta ^{-\frac{1}{3}}$.
\end{enumerate}
Theorem~\ref{thm:main theorem 2} follows from part (2) by taking $\delta \to 0$ and integrating the speed to obtain the aggregate size. The proof of the first part is given in Subsection~\ref{sec:stitching} and it includes analysis of $\log t$ many scales of the aggregate speed. We show that in each scale the aggregate speed is likely to decrease and that the time it takes for this to happen is not too long. The second part is given in Subsection~\ref{sec:convergence in distribution} and it is a direct application of a result by Helland \cite{helland1981minimal}. In \cite{helland1981minimal} it is shown that, under some assumptions, if the increments of a sequence of processes have the right conditional expectation and variance then the sequence of processes converges to the solution of the corresponding SDE.

Throughout this section it is convenient to work with a slightly different definition of regular. Recall the definition of $\mathfrak{R}^\sharp (\alpha_0,r,[t_0])$ in Section~\ref{subsec:regoverview:reg}. As usual, we identify a set of points $\Pi $ with the cumulative function $Y_t(s):=|\Pi [t-s,t]|$ and denote by $\Pi _Y$ the set of points corresponding to the step function $Y$. We also write $Y\in \mathfrak{R}^\sharp (\alpha_0,r,[t_0])$ when $\Pi _{Y}\in \mathfrak{R}^\sharp (\alpha_0,r,[t_0])$. For convenience, in the following definition and throughout this section we also switch the roles of $\alpha $ and $\alpha '$ from Section~\ref{sec:reg:intro} and equation \eqref{eq:def:alpha0prime}.

\begin{definition}\label{def:def of regularity prime}
Let $\alpha >0$ sufficiently small and $t>0$. We say that a deterministic step function $Y _0 :\mathbb R _+ \to \mathbb N \cup \{\infty \} $ satisfies $Y _0 \in \mathfrak{R}'(\alpha ,t)$ if there exists $t^{-}$ with 
\begin{equation}
    t -2 \alpha ^{-2} \log ^{\theta } (1/\alpha ) \le t^{-} \le t -\frac{1}{2} \alpha _0^{-2} \log ^{\theta } (1/\alpha )
\end{equation}
and $\alpha '>0$ such that $Y_0 \in  \mathfrak{R} ^\sharp (\alpha ' ,\alpha ^6 ;[t])$ and
\begin{equation}
    \alpha = \mathcal{L}(t; \Pi _{Y_0}[t^{-},t],\alpha '):= \frac{2 (\alpha') ^2 }{1+2 \alpha' }\intop _{t^-} ^{t} \intop _{t}^{\infty} K_{\alpha '} (z-x) dz \, d \Pi _{Y_0}  (x).
\end{equation}
\end{definition}

Using this definition we state the main results of the induction step in the following theorem

\begin{thm}\label{thm:induction step for repeatd use}
Let $\alpha   >0$ sufficiently small, $t >0$ and $\tilde{t}>0$ such that
\begin{equation}
\alpha  ^{-2}\log ^{10\theta-4 } (1/\alpha ) \le \tilde{t}-t\le \alpha ^{-2}\log ^{10\theta } (1/\alpha  ).
\end{equation}
Let $Y _{t} \in \mathfrak{R}'(\alpha  ,t)$ and let $(Y_x ,S(x),X_x)$ be the aggregate with initial condition $(Y _{t},t)$. Then, there is a random variable $\tilde{\alpha }  \in \mathcal F _{\tilde{t}}$ such that the following holds:
\begin{enumerate}
\item
\begin{equation}
| \tilde{ \alpha } -\alpha  |\le \sqrt{\alpha  ^5(\tilde{t}-t )} \log ^{2 \theta } (1/\alpha  )  
\end{equation}
\item  
\begin{equation}
\mathbb P \left( Y_{\tilde{t}} \in \mathfrak{R}'(\tilde{\alpha } ,\tilde{t} ) \right)\ge 1- \alpha ^5
\end{equation}
\item
\begin{equation}
\left| \mathbb E \left(  \tilde{\alpha }  -\alpha   \right) \right| \le \frac{\alpha  ^4 }{\log ^3 (1/\alpha  )}(\tilde{t}-t ).
\end{equation}
\item
\begin{equation}
\mathbb E \left( ( \tilde{ \alpha  } -\alpha  )^2 \right) = 4 \alpha ^5 (\tilde{t}-t )\left(1+O(\alpha  ^{\frac{1}{3} } )\right).  
\end{equation}
\item
\begin{equation}
\mathbb P \left(  \left| X_{\tilde{t} }-X_{t}-\alpha  (\tilde{t}-t) \right|  \le  \alpha  ^{1.4}(\tilde{t}-t ) \right) \ge 1-\alpha  ^5
\end{equation}
\end{enumerate}
\end{thm}

\begin{proof}
    By Definition~\ref{def:def of regularity prime} there exists $\alpha  '$ such that $Y_{t}\in \mathfrak{R} ^\sharp  (\alpha  ' ,\alpha ^6 ,[t])$ which in particular means by the definition of $\mathcal A _2$ in \eqref{eq:def of the events A1 A2 A3} that $|\alpha  -\alpha  '| \le (\alpha ') ^{\frac{3}{2}} \log ^{\theta } (1/\alpha ') \le 2 \alpha  ^{\frac{3}{2}} \log ^{\theta } (1/\alpha ) $. Thus by Theorem~\ref{thm:induction:main} we have some $\tilde{\alpha }'$ such that  
    \begin{equation}
     \mathbb P \big( Y_{t_1} \in \mathfrak{R} ^\sharp  ( \tilde{\alpha } ' , ( \tilde{ \alpha }') ^7 ,[\tilde{t}]) \big) \ge 1-2\alpha ^6  \quad \text{and} \quad  \mathbb P \big( |\alpha '- \tilde{ \alpha }'| \le 2(\alpha' ) ^{\frac{3}{2}} \log ^{6 \theta }(1/\alpha ') \big) \ge 1-2\alpha  ^6 .
     \end{equation}
    Let
    \begin{equation}
        \tilde{\alpha } :=\mathcal{L}(\tilde{t}; \Pi _{Y_{\tilde{t}}}[\tilde{t}^{-},\tilde{t}], \tilde{\alpha  }').
    \end{equation}
    Using the definition of $\mathcal A _2 $ again we have $|\tilde{ \alpha } -\tilde{ \alpha  }'|\le (\tilde{ \alpha  }')^{\frac{3}{2}} \log ^\theta  (1/\tilde{\alpha }') \le 2\alpha _0^{\frac{3}{2}} \log ^{\theta } (1/\alpha )$ with probability at least $1-4 \alpha ^6 $. Thus we get
    \begin{equation}
        \mathbb P \big(  |\alpha -\tilde{ \alpha } | \le \sqrt{\alpha ^5 (\tilde{t}-t)} \log ^{2 \theta } (1/\alpha  \big) \ge  1-4 \alpha ^6.
     \end{equation}
    We also get that $\mathbb P \big( Y_{\tilde{t}} \in \mathfrak{R} ^\sharp  ( \tilde{\alpha } ' , \tilde{ \alpha }^6 ,[\tilde{t}]) \big) \ge 1-6\alpha ^6$ and therefore by Definition~\ref{def:def of regularity prime}, $\mathbb P ( Y_{\tilde{t}} \in \mathfrak{R} '  ( \tilde{ \alpha }  , \tilde{t}) ) \ge 1-6\alpha ^6$. This finishes the proof of the second part of the theorem. 
    The third and fourth part of the theorem follows from Theorem~\ref{thm:increment:formal}. The last part of the theorem follows from Proposition~\ref{prop:ind1:growth}. 
    
    This finishes the proof of all parts of the theorem except for part one which holds with probability $1-4 \alpha ^6$ instead of deterministically. In order to fix this issue we change the definition of $\tilde{\alpha} $ slightly. On the event where part (1) doesn't hold we change $\tilde{\alpha }$ to be $\alpha$. It is clear that parts (2),(3) and (5) still hold. Part (4) of the theorem still holds by  Corollary~\ref{cor:increment:formal:goodevent}.
\end{proof}

\subsection{Stitching intervals}\label{sec:stitching}

In this section we use Theorem~\ref{thm:induction step for repeatd use} repeatedly to show that at time $t$ the aggregate speed is $\approx t^{-\frac{1}{3}}$. To this end, we stitch a lot of short intervals of time to obtain the right growth in a medium interval of time. Then we stitch a lot of medium intervals of time to obtain the right growth in a long interval of time. Finally, we stitch a lot of long intervals of time to obtain the right growth in the time interval $[0,t]$. See Subsections~\ref{sec:short},\ref{sec:medium} and \ref{sec:long}.

\subsubsection{Stitching short intervals}\label{sec:short}

In this section we build inductively a process that approximates the speed of the aggregate and study the time it takes for this process to get smaller or larger by a factor of $2$. 

For the proof of the main lemma of this subsection we'll need the following martingale concentration result due to Freedman \cite{freedman1975tail}. See also Theorem 18 in \cite{chung2006concentration} for a short proof. 

\begin{claim}\label{claim:bernstein}
	Let $X_0,X_1,\dots ,X_n $ be a martingale with filtration $\mathcal F _0,\dots ,\mathcal F _n$. Suppose that almost surely  for all $1\le i \le n$ we have: 
	\begin{equation}
	\var (X_i \ | \ \mathcal F _{i-1}) \le \sigma _i^2 \quad \text{and} \quad  \left| X_i-X_{i-1} \right| \le M 
	\end{equation}
	Then, 
	\begin{equation}
	\mathbb P \left( |X_n-X_0|\ge \lambda  \right) \le 2\exp \left(-\frac{\lambda ^2 }{\sum _{i=1}^n \sigma _i^2 +\frac{1}{3}M \lambda }\right).
	\end{equation}
\end{claim}

Let $t_0>0$ and let $\delta _0>0$ sufficiently small. Let $\alpha _0>0$ such that $|\alpha _0-\delta _0| \le \delta _0^{1.1}$. Here $\delta _0$ should be understood as the scale of the speed and $\alpha _0$ as the initial speed. Let $Y _{t_0}  \in \mathfrak{R}'(\alpha _0 ,t_0)$ and let $\Pi _S$ be the aggregate with initial condition $Y _{t_0}$. Let $t_1 > 0 $ such that 
\begin{equation}
\delta _0 ^{-2}\log ^{10\theta -3}(1/\delta _0) \le t_1- t_0 \le \delta _0 ^{-2}\log ^{10\theta -1 }(1/\delta _0)
\end{equation}
and $t_i:=t_0+i(t_1-t_0)$ for any $i\ge 2$. We define the sequence $\alpha _i$ inductively. Suppose that $\alpha _1,\dots ,\alpha _i$ were already defined. Define the stopping times
\begin{equation}
\begin{split}
&\zeta _1:=\min \left\{t_i\ge t_0:\ \alpha _i < \delta _0/2 \right\}\\
&\zeta _2 :=\min \{t_i\ge t_0:\ \alpha _i > 2 \delta _0 \} \\
&\zeta _4:=\min \{t_i \ge t_0: Y_{t_i} \notin \mathfrak{R} '(\alpha _i,t_i) \}\\
& \zeta _5:=\min \{t_{i+1}\ge t_1: \left| X_{t_{i+1}}-X_{t_{i}}-\alpha _{i} (t_1-t_0) \right|  \ge  \alpha _i ^{1.4}(t_1-t_0) \},
\end{split}
\end{equation}
$\zeta _3:=\zeta _4 \wedge \zeta _5$ and $\zeta :=\zeta _1 \wedge \zeta _2 \wedge \zeta _3 $. We note that so far $\zeta $ is not well defined because $\alpha _j$ for $j>i$ was not defined but the event $\{\zeta >t_i\}$ is well defined.  On the event $\{ \zeta >t_i \}$ we have that $Y_{t_i} \in \mathfrak{R} '(\alpha _i,t_i)$ and therefore we can define $\alpha _{i+1} \in \mathcal F _{t_i}$ to be the random variable $\tilde{\alpha }$ from Theorem~\ref{thm:induction step for repeatd use} with $t_i$ and  $t_{i+1}$ as $t$ and $\tilde{t}$ respectively and with $\alpha _i$ as $\alpha $. On the event $\{\zeta \le t_i\}$ we just let $\alpha _{i+1}:=\alpha _i$. Using Theorem~\ref{thm:induction step for repeatd use} we have almost surely on the event $\{\zeta >t_{n}\}$ that $\mathbb P (\zeta _3=t_{n+1} \ | \ \mathcal F _{t_{n}})\le 2\alpha _{n}^5 \le C \delta _0^5$ and moreover
\begin{equation}
\left| \mathbb E ( \alpha _{n+1} -\alpha _n \ | \ \mathcal F _{t_{n}}) \right| \le \frac{(t_1-t_0)\alpha _{n}^4}{\log ^3 (1/\alpha _{n})} \le C\frac{(t_1-t_0)\delta _0^4}{\log ^3  (1/\delta _0)}.
\end{equation}

\begin{lem}\label{lem:stitching short intervals}
	The stopping time $\zeta $ satisfies the following properties 
	\begin{enumerate}
		\item 
		\begin{equation}
		\mathbb P (\zeta =\zeta _3 ) \le C \delta _0  ^3
		\end{equation}
		\item 
		\begin{equation}
		\mathbb P (\zeta =\zeta _1  )=\frac{2}{3}+O(\log ^{-1}(1/\delta _0)), \quad \mathbb P (\zeta =\zeta _2 )=\frac{1}{3}+O(\log ^{-1}(1/\delta _0)) 
		\end{equation}
		\item 
		For all $k \ge 1 $,
		\begin{equation}
		\mathbb P (\zeta \ge t_0 +k \delta _0^{-3} ) \le Ce^{-ck}
		\end{equation}
		\item
		For all $ \epsilon <1$,
		\begin{equation}
		\mathbb P (\zeta \le t_0+ \epsilon \delta _0^{-3} )\le C\delta _0^4 +Ce^{-c \epsilon ^{-1}}
		\end{equation}
	\end{enumerate}
\end{lem}

\begin{proof}
	We start by studying the Doob decomposition of $\alpha _i$. Consider the martingale  
	\begin{equation}
	\beta _n:=\alpha _0+\sum _{i=1}^{n} \alpha_{i }- \mathbb E [\alpha _i \ \big|  \ \mathcal F _{t_{i-1}}]= \alpha _n-\sum _{i=1}^n \mathbb E [\alpha _i \ | \ \mathcal F _{t_{i-1}}]- \alpha _{i-1 }. 
	\end{equation}	
	We show that the increments of $\beta _n$ are close to those of $\alpha _n$. We have 
	\begin{equation}\label{eq:increments are close}
	\begin{split}
	| (\beta _{n+1}-\beta _n) -(\alpha _{n+1 }-\alpha _{n }) | \le \mathds 1\{\zeta >t_n\} \left| \mathbb E [\alpha _{n+1}\ | \ \mathcal F {t_n}]-\alpha _n \right| \le C\mathds 1 \{\zeta >t_n\} \frac{\delta _0^4 (t_1-t_0)}{\log ^3 (1/\delta _0 )}.
	\end{split}
	\end{equation}
	Thus, for any $n\le  2 \delta _0^{-3} \log ^2 (1/\delta _0)/(t_1-t_0)$ we have that
	\begin{equation}\label{eq:beta is close to alpha}
	|\beta _n-\alpha _{n }|\le  C n\frac{  \delta _0^4 (t_1-t_0)}{\log ^3 (1/\delta _0)}\le C\frac{\delta _0 }{\log (1/\delta _0)}.
	\end{equation}
	and in particular $0.4 \delta _0<\beta _n  \le 2.1 \delta _0$ for such $n$. By \eqref{eq:increments are close} we also have 
	\begin{equation}\label{eq:bound on increments of beta_n}
	|\beta _{n+1}-\beta _n |\le |\alpha _{n+1 }-\alpha _{n }|+C\frac{\delta _0^4(t_1-t_0)}{\log (1/\delta _0)}\le C\sqrt{\delta _0^5(t_1-t_0)}\log ^{2 \theta } (1/\delta _0) 
	\end{equation}
	and therefore, using \eqref{eq:increments are close} once again
	\begin{equation}
	\begin{split}
	|(\beta _{n+1}-\beta _n)^2-(\alpha _{n+1 }-\alpha _{n })^2 |\le C\sqrt{\delta _0^5(t_1-t_0)}\log ^{2\theta } (1/\delta _0) | (\beta _{n+1}-\beta _n) -(\alpha _{n+1 }-\alpha _{n }) |\\
	\le C \mathds 1 \{\zeta >t_n\}\delta _0^{\frac{13}{2}} \log ^{2 \theta } (1/\delta _0) (t_1-t_0)^{\frac{3}{2}}\le C\mathds 1 \{\zeta >t_n\} \delta _0^{\frac{7}{2}} \log ^{2 \theta }(1/ \delta _0).
	\end{split}
	\end{equation}
	Thus
	\begin{equation}\label{eq:variance of increments}
	\begin{split}
	\mathbb E \left[ (\beta _{n+1}-\beta _n)^2 \ | \ \mathcal F _{t_n} \right]&=\mathbb E \left[ (\alpha_{n+1 }- \alpha _{n })^2 | \ \mathcal F _{t_n} \right] +\mathds 1 \{\zeta >t_n\} O (\delta _0^{\frac{7}{2} }\log ^{2 \theta }(1/\delta _0 )) \\
	&=\mathds 1 \{\zeta > t_n\} 4\alpha _n ^5(t_1-t_0) (1+O(\delta _0 ^{0.4} )) 
	\end{split}
	\end{equation}
	It follows that $\mathbb E \left[ (\beta_{n+1}-\beta _n )^2 \right]\ge 0.1 \delta _0^5 (t_1-t_0)\mathbb P (\zeta >t_n )$. Now, letting $n_0:=\lceil 100\delta _0^{-3} /(t_1-t_0) \rceil $ we have that
	\begin{equation}
	\begin{split}
	5 \delta _0^2  \ge \mathbb E (\beta _{n_0+1}^2 )&\ge \var (\beta _{n_0+1} )=\sum _{n=0}^{n_0} \mathbb E \left[ (\beta _{n+1}-\beta _n)^2 \right] \\ 
	&\ge 0.1 n_0 \delta _0^5 (t_1-t_0) \mathbb P (\zeta \ge t_{n_0} ) \ge 10 \delta _0^2 \cdot \mathbb P (\zeta \ge t_{n_0}). 
	\end{split}
	\end{equation}
	We get that $ \mathbb P (\zeta > t_{n_0} )\le 1/2$. Let $k\ge 1$. By repeating the above arguments on the event $\{\zeta >t_{(k-1)n_0} \}$ with the initial condition $Y _{ t_{(k-1)n_0}}$ instead of $Y _{t_0}$ and the random variable $\alpha _{(k-1)n_0}$ instead of $\alpha _0$ we get that $\mathds 1 \{\zeta >t_{(k-1)n_0}\} \cdot \mathbb P (\zeta > t_{kn_0} \ | \ \mathcal F _{t_{(k-1)n_0}})\le 1/2$ almost surely (we note that we might not have $|\alpha _{kn_0}-\delta _0|\le \delta _0^{1.1}$ but we did not use the fact that $| \alpha _0-\delta _0|\le \delta _0^{1.1}$ in the proof of $ \mathbb P (\zeta > t_{n_0} )\le 1/2$). Thus, inductively we have that
	\begin{equation}\label{eq:probability that the stopping time is larger than kn_0}
	\begin{split}
	\mathbb P (\zeta > t_{kn_0} )=\mathbb E \left[ \mathds 1 \{\zeta \ge t _{(k-1)n_0}\} \cdot  \mathbb P (\zeta > t_{kn_0} \ | \ \mathcal F _{t_{(k-1)n_0}}) \right] \le \frac{1}{2}\mathbb P (\zeta > t_{(k-1)n_0} ) \le \cdots \le  2^{-k}.
	\end{split}
	\end{equation}
	Finally, by the definition of $n_0$, for $c_1=0.001$ and $k$ sufficiently large we have that $t_{\lfloor c_1k \rfloor n_0}=t_0+ \lfloor c_1k \rfloor n_0 (t_1-t_0)\le t_0+k\delta _0^{-3} $ and therefore 
	\begin{equation}
	\mathbb P (\zeta > t_0+k\delta _0^{-3} ) \le \mathbb P (\zeta > t_{ \lfloor c_1k \rfloor n_0} ) \le 2^{-\lfloor c_1 k \rfloor } \le C e^{-ck}.
	\end{equation}
	This finishes the proof of the third part.
	
	We turn to prove the first part. For any $n\ge 1$
	\begin{equation}\label{eq:zeta 3 happens fast}
	\begin{split}
	\mathbb P \left( \zeta _3 =\zeta \le  t_n  \right) &=\mathbb P \left( \bigcup _{i=1}^n \{\zeta _3 =t_i \} \cap \{\zeta  \ge t_{i-1} \} \right) \le  \sum _{i=1}^n \mathbb P  \left( \zeta > t_{i-1} \text{ and } \zeta _3=i \right)\\
	&=\sum _{i=1}^n \mathbb E \left[ \mathds 1 \{\zeta > t_{i-1} \} \cdot \mathbb P ( \zeta _3=t_i \ | \ \mathcal F _{t_{i-1}})  \right]\le C n \delta _0^5.\\
	\end{split}
	\end{equation}
	Thus, if we let $n_1:=\lfloor \delta _0^{-1} \rfloor n_0 \le \delta _0^{-2}$ we get by \eqref{eq:probability that the stopping time is larger than kn_0} that
	\begin{equation}
	\mathbb P \left( \zeta _3 =\zeta \right)\le  \mathbb P \left( \zeta _3 =\zeta \le  t_{n_1}   \right)+\mathbb P (\zeta >t_{n_1} )\le C\delta _0^3+Ce^{-c\delta _0^{-1}} \le C\delta _0^3.
	\end{equation}
	
	We turn to prove the fourth part of the lemma. Let $\log ^{-2} (1/\delta _0) \le \epsilon <1$ and let $n_3:= \lceil \epsilon \delta _0^{-3}/(t_1-t_0)\rceil \le \delta _0 ^{-1}$. Define
	\begin{equation}
	M:=C\sqrt{\delta _0^5 (t_1-t_0)}\log ^{2 \theta } (1/\delta _0) \le \delta _0^{\frac{3}{2}}\log ^{C_\theta } (1/\delta _0),\quad \sigma :=C\sqrt{\delta _0^5 (t_1-t_0)}
	\end{equation}
	The martingale $\beta _n$ satisfy the assumptions of Theorem~\ref{claim:bernstein} with this $M$ and $\sigma _i=\sigma $ by equations \eqref{eq:bound on increments of beta_n} and \eqref{eq:variance of increments}. Thus, taking $\lambda =\delta _0/4 $ we get
	\begin{equation}
	\mathbb P ( |\beta _{n_3}-\alpha _0| \ge \lambda ) \le 2 \exp \left( -\frac{\lambda ^2 }{n_3\sigma ^2  +\frac{1}{3}M \lambda} \right) \le 2 \exp \left(  -\frac{c \delta _0^2 }{2 \epsilon \delta _0^2 +c\delta _0^{\frac{5}{2}}\log ^C(1/\delta _0)} \right)\le C e^{-c\epsilon ^{-1} }
	\end{equation}
	Thus, using that $|\alpha _0-\delta _0| \le \delta _0^{1.1}$ we get
	\begin{equation}
	\begin{split}
	\mathbb P (\zeta \le t_0  +\epsilon \delta _0^{-3} ) &\le \mathbb P (\zeta  \le t_{ n_3} )\le \mathbb P (\zeta _3=\zeta \le t_{ n_3} ) + \mathbb P (\zeta _1 \wedge \zeta _2=\zeta \le t_{ n_3} ) \\
	&\le C n_3 \delta _0^5+ \mathbb P (|\alpha _{n_3}-\alpha _0| \ge 4 \delta _0/10) \\
	&\le C \delta _0 ^4+ \mathbb P (|\alpha _{n_3}-\beta _{n_3}| >  \delta _0/10)+\mathbb P (|\beta _{n_3}-\alpha _0| \ge  \delta _0/4) \le C\delta _0^4 +Ce^{-c \epsilon ^{-1}},
	\end{split}
	\end{equation}
	where in the third inequality we used \eqref{eq:zeta 3 happens fast} and in the fifth inequality we used \eqref{eq:beta is close to alpha}. This finishes the proof of the fourth part of the lemma when $\epsilon \ge \log ^{-2}(1/\delta _0)$. It is clear that the same inequality holds for $\epsilon <\log ^{-2}(1/\delta _0)$ as well (since the $e^{-c\epsilon ^{-1}}<\delta _0^4$ in this case).
	
	Finally, we prove the second part of the lemma. Let $n_4:=\lceil  \log ^2 (1/\delta _0) \delta _0^{-3}/(t_1-t_0) \rceil  \le \delta _0^{-1}$ and 
	\begin{equation}
	\mathcal A := \{\zeta =\zeta _1\le t_{n_4}\},\quad \mathcal B:=\{\zeta =\zeta _2\le t_{n_4}\} ,\quad \mathcal C :=\{ \zeta _1\wedge \zeta _2 >\zeta \text{ and } \zeta \le t_{n_4}  \} \cup \{  \zeta >t_{n_4}\}.
	\end{equation}
	We have
	\begin{equation}\label{eq:equality of expectations}
	\delta _0+O(\delta _0^{1.1})=\alpha _0=\mathbb E [\beta _{n_4}]= \mathbb E [\mathds 1 _{\mathcal A}\beta _{n_4}] +\mathbb E [\mathds 1 _{\mathcal B}\beta _{n_4}]+\mathbb E [\mathds 1 _{\mathcal C}\beta _{n_4}].
	\end{equation}
	Next, we estimate each one of the terms separately. We have
	\begin{equation}\label{eq:prob of C}
	\mathbb P( \mathcal C  ) \le  \mathbb P (\zeta =\zeta _3 \le t_{n_4} )+ \mathbb P (\zeta \ge t_{n_4} ) \le C n_4 \delta _0^5+C e^{-c \log ^2(1/\delta _0)} \le C \delta _0^4,
	\end{equation}
	where in the second inequality we used \eqref{eq:zeta 3 happens fast} and \eqref{eq:probability that the stopping time is larger than kn_0}. Thus,
	\begin{equation}\label{eq:expectation on C}
	\mathbb E [\mathds 1 _{\mathcal C}\beta _{n_4}] \le C \delta _0^4.
	\end{equation}
	
	On the event $\mathcal A$ we have
	\begin{equation}
	|\beta _{n_4}- \delta _0/2 |\le | \beta _{n_4}-\alpha _{n_4}| +|\alpha _{n_4}-\delta _0/2| \le C \frac{\delta _0}{\log (1/\delta _0)} +C\delta _0^{\frac{3}{2}} \log ^C (1/\delta _0) \le C \frac{\delta _0}{\log (1/\delta _0)},
	\end{equation}  
	where in the second inequality we bound the first term using \eqref{eq:beta is close to alpha} and bound the second term using the definition of $\zeta _1$ and the fact that the increments of $\alpha _n$ are small. Thus 
	\begin{equation}\label{eq:expectation on A}
	\mathbb E [\mathds 1 _{\mathcal A}\beta _{n_4} ]=\frac{\delta _0}{2} \cdot \mathbb P (\mathcal A ) +O\left(\frac{\delta _0}{ \log (1/\delta _0)}\right)
	\end{equation}
	
	Using the same arguments we have 
	\begin{equation}\label{eq:expectation on B}
	\begin{split}
	\mathbb E  [\mathds 1 _{\mathcal B}\beta _{n_4}]=2 \delta _0 \cdot \mathbb P (\mathcal B )+O\left(\frac{\delta _0}{ \log (1/\delta _0)}\right)
	&= 2\delta _0 \left(1- \mathbb P (\mathcal A ) +O(\delta _0 ^4) \right)+O\left(\frac{\delta _0}{ \log (1/\delta _0)}\right) \\
	& =2 \delta _0-2\delta _0 \cdot \mathbb P (\mathcal A ) +O\left(\frac{\delta _0}{ \log (1/\delta _0)}\right),
	\end{split}
	\end{equation}
	where in the second equality we used \eqref{eq:prob of C}.
	
	Substituting \eqref{eq:expectation on A},\eqref{eq:expectation on B} and \eqref{eq:expectation on C} into \eqref{eq:equality of expectations} we get
	\begin{equation}
	\delta _0= \frac{\delta _0}{2} \mathbb P (\mathcal A )+2 \delta _0-2 \delta _0 \cdot \mathbb P (\mathcal A )+O\left(\frac{\delta _0}{ \log (1/\delta _0)}\right)
	\end{equation}
	which means that 
	\begin{equation}
	\mathbb P (\mathcal A )= \frac{2}{3} +O\left( \log ^{-1} (1/\delta _0)\right),\quad \mathbb P (\mathcal B )= \frac{1}{3} +O\left( \log ^{-1} (1/\delta _0)\right)
	\end{equation}
	The second part of the lemma follows as the event $\{\zeta =\zeta _1\} \setminus \mathcal A \subseteq \{\zeta >t_{n_4}\}$ has a small probability.
\end{proof}

\begin{remark}
	On the event $\{\zeta <\zeta _3\}$ (which happens with high probability by Lemma~\ref{lem:stitching short intervals} ) we have that 
	\begin{equation}
	\frac{\delta _0}{3} (\zeta-t_0) \le X_{\zeta }-X_{t_0} \le 3 \delta _0 (\zeta -t_0)
	\end{equation}
\end{remark}

In the next subsection we are going to change $t_0$, $\alpha _0$ and $\delta _0$ and therefore we write more specifically $\zeta (t_0,\alpha _0, \delta _0)$ for the stopping time $\zeta $ and similarly with $\zeta _1,\zeta _2 ,\zeta _3$.

We define the random variable $\alpha '=\alpha '(t_0,\alpha _0, \delta _0)\in \mathcal F _{\zeta }$ as follows. on the event  $\{ \zeta =t_i \}$, we let $\alpha ':=\alpha _i$. 

\subsubsection{Stitching medium intervals}\label{sec:medium}
In this section we stitch the medium intervals of time together to create the large intervals of time in which the speed decreases with high probability. The idea is to repeatedly use the previous section to show that the speed behaves somewhat like a random walk with a negative drift on the dyadic scales. As in the previous section we let $t_0>0$, $\delta _0>0$ sufficiently small and  $\alpha _0 >0$ with $|\alpha _0-\delta _0| \le \delta _0^{1.1}$. We also let $M >0$ with $2\delta _0\le M \le \sqrt{\delta _0}$. Finally, let $Y _{t_0} \in \mathfrak{R}'(\alpha _0,t_0)$ and let $\Pi _S$ be the aggregate with initial condition $Y _{t_0}$. We define the processes $\delta _i, \alpha _i$   and the stopping times $t_i$ inductively (note that the $t_i$ are defined not as in the previous subsection). Suppose we defined $\delta _1,\dots ,\delta _i,\alpha _1,\dots ,\alpha _i$ and $t_1,\dots ,t_i$ (note that the definition of $t_i$ will be different this time). Define the stopping times
\begin{equation}
\begin{split}
&\zeta _1':=\min \left\{t_i\ge t_0: \delta _i\le \delta _0/2 \right\}\\
&\zeta _2' :=\inf \{t_i\ge t_0:\delta_i > M \} \\
& \zeta _3'  :=\inf \{t_i\ge t_1: \zeta (t_{i-1},\alpha _{i-1},\delta _{i-1})=\zeta _3(t_{i-1},\alpha _{i-1},\delta _{i-1})\}   
\end{split}
\end{equation}
and $\zeta '=\zeta _1'\wedge \zeta _2 ' \wedge \zeta _3' $. Let $t_{i+1}:=\zeta (t_i,\alpha _i,\delta _i)$. We define $\delta _{i+1}\in \mathcal F _{t_{i+1}}$ as follows 
\begin{equation}
\delta _{i+1}:=
\begin{cases}
\delta _{i}/2,\quad &\text{on the event }\{\zeta '>t_{i} \text{ and } \zeta '_3>t_{i+1} \text{ and } \zeta (t_i,\alpha _i,\delta _i)=\zeta _1(t_i,\alpha _i,\delta _i)  \}\\
\ \delta _{i},\quad &\text{on the event }\{\zeta '\le t_{i} \text{ or } \zeta '_3=t_{i+1}  \}\\
2\delta _{i},\quad &\text{on the event }\{\zeta '>t_{i} \text{ and } \zeta '_3>t_{i+1} \text{ and } \zeta (t_i,\alpha _i,\delta _i)=\zeta _2(t_i,\alpha _i,\delta _i)  \}
\end{cases}
\end{equation}
We also let $\alpha _{i+1}:=\alpha '(t_i,\alpha _i,\delta _i)$ where $ \alpha '$ is defined in the end of the previous section. On the event $\{\zeta '> t_i\}$ we have that $\delta _i \le M$, and there exist a random variable $\alpha _i$ with $|\alpha _i -\delta _i|\le \delta _i^{1.1}$ so that $Y_{t_i} \in \mathfrak{R}'(\alpha _i,t_i)$. Thus, on this event, we can apply Lemma~\ref{lem:stitching short intervals} with $t_i, \delta _i$ and $\alpha _i$ as $t_0, \delta _0$ and $\alpha _0$ respectively. By the lemma, with probability at least $1-C \delta _i^3$ we have that $\zeta _3 '>t_{i+1}$. The reader should think of $\delta _i$ as the sequence of different dyadic scales that the speed traveled through.
\begin{lem}\label{lem:medium}\label{lem:stitching medium}
	The stopping time $\zeta '$ satisfies the following properties:
	\begin{enumerate}
		\item 
		\begin{equation}
		\mathbb P (\zeta ' =\zeta '_3  ) \le C \delta _0  
		\end{equation}
		\item  
		\begin{equation}
		\mathbb P (\zeta ' =\zeta '_2  ) \le \left( \frac{\delta _0}{M} \right)^c  
		\end{equation}
		\item 
		For all $k\ge 1 $,
		\begin{equation}
		\mathbb P (   \zeta ' >t_0+ k \delta _0^3   ) \le C e^{-c\sqrt{k}}
		\end{equation}
		\item 
		For all $\epsilon <1$, 
		\begin{equation}
		\mathbb P (   \zeta ' <t_0+ \epsilon  \delta _0^3  ) \le C \delta _0^4+ C e^{-c\epsilon ^{-1}}
		\end{equation}
		
	\end{enumerate}
\end{lem}

\begin{proof}
	Define $W_i:=\log _2 (\delta _i/\delta _0)$. Let $n\ge 1 $ and let $\zeta _1,\zeta _2 ,\zeta _3 ,\zeta  $ be the stopping times from the previous subsction with $t_0,\alpha _0,\delta _0$ being $t_n,\alpha _n,\delta _n$ respectively. By Lemma~\ref{lem:stitching short intervals} we have that 
	\begin{equation}
	\begin{split}
	\mathbb E (W_{n+1}-W_n \ | \ \mathcal F _{t_n} )&=\mathds 1 \{ \zeta '>t_n \} \left( -\mathbb P ( \delta _{i+1}=2 \delta _i \ | \ \mathcal F _{t_n}) +\mathbb P ( \delta _{i+1}= \delta _i /2 \ | \ \mathcal F _{t_n}) \right) \\
	&= \mathds 1 \{ \zeta '>t_n \} \left( -\frac{1}{3} +O\left(\frac{1}{\log (1/\delta _n)} \right) \right)\le -\frac{1}{4} \mathds 1 \{ \zeta '>t_n \},
	\end{split}
	\end{equation}
	where the last inequality holds when $\delta _0$ is sufficiently small and as $\delta _n\le 2M \le 2 \sqrt{\delta _0}$. Thus, the process $M_n:=W_n+\frac{1}{4} \sum _{j=0}^{n-1 } \mathds 1 \{\zeta ' >t_j\}$ is a super-martingale with increments bounded in absolute value by $2$. We get, by Azuma inequality 
	\begin{equation}\label{eq:zeta' larger than t_n}
	\mathbb P (\zeta '>t_n )=\mathbb P (W_n\ge 1, \zeta '>t_n )\le \mathbb P (M_n-M_0 \ge n/4 ) \le Ce^{-cn}.
	\end{equation} 
	We can now prove the second part of the lemma.  By \eqref{eq:zeta' larger than t_n} with $n_0:=\frac{1}{2} \log _2 (\frac{M}{\delta _0})$ we have
	\begin{equation}
	\mathbb P (\zeta '=\zeta _2' ) \le \mathbb P (\zeta '\ge t_{n_0} ) \le C e^{-cn_0} \le \left(\frac{\delta _0}{M}\right) ^c.
	\end{equation}
	We turn to prove the first part of the lemma. For all $n\ge 1 $, by Lemma~\ref{lem:stitching short intervals} we have that 
	\begin{equation}
	\begin{split}
	\mathbb P (\zeta '=\zeta _3'= t_n  ) \le \mathbb E \left[ \mathds 1 \{\zeta '>t_{n-1} \} \cdot \mathbb P (\zeta _3' =t_n \ | \ \mathcal F _{t_{n-1}})   \right]  \le  \mathbb E \left[ \mathds 1 \{\zeta '>t_{n-1} \} C\delta _{n-1}^3  \right] \le C M^3 
	\end{split}
	\end{equation}
	and therefore $\mathbb P (\zeta '=\zeta _3' \le  t_n ) \le CnM^3$. Thus, using also \eqref{eq:zeta' larger than t_n} with $n_1:=\lfloor \delta _0^{-\frac{1}{2}} \rfloor $ we get
	\begin{equation}
	\mathbb P (\zeta '=\zeta '_3 ) \le \mathbb P (\zeta '=\zeta '_3\le t_{n_1} )+\mathbb P (\zeta '\ge t_{n_1} ) \le C \delta _0. 
	\end{equation}
	Next, we prove part (3). The lower bound for $\zeta '$ follows immediately from Lemma~\ref{lem:stitching short intervals}. Indeed, $\zeta '\ge \zeta (t_0,\alpha _0,\delta _0)$. We turn to prove the upper bound. By \eqref{eq:zeta' larger than t_n} we have that
	\begin{equation}
	\begin{split}
	\mathbb P (\zeta ' \ge t_0+k \delta _0^{-3} ) &\le \mathbb P ( \zeta '\ge t_n )+\sum _{i=1}^n \mathbb P (\zeta '=t_i \ge t_0+k \delta _0^{-3} ) \\
	&\le C e^{-cn} +\sum _{i=1}^n \sum _{j=1}^i \mathbb P \left( \zeta '=t_i, \ t_j-t_{j-1}>\frac{k}{n}\delta _0^{-3}   \right)\\
	&\le C e^{-cn} +\sum _{i=1}^n \sum _{j=1}^i \mathbb E \left[ \mathds 1 \{\zeta '>t_{j-1}\} \mathbb P \left( t_j-t_{j-1}>\frac{k}{n}\delta _0^{-3}  \ | \ \mathcal F _{t_{j-1}} \right)  \right]\\
	&\le C e^{-cn } + \sum _{i=1}^n \sum _{j=1}^i e^{-ck/n} \le Ce^{-cn} +n^2e^{-ck/n},
	\end{split}
	\end{equation}
	where the fourth inequality we used Lemma~\ref{lem:stitching short intervals}. The third part of the lemma follows from the last result by substituting $n=\sqrt{k}$.
\end{proof}

On the event $\{\zeta '<\zeta '_3\}$ we have that
\begin{equation}
\frac{\delta _0}{3} (\zeta '-t_0) \le X_{\zeta '} -X_{t_0} \le 2M(\zeta '-t_0).
\end{equation}

For $\epsilon _0>0$, we denote the good event by 
\begin{equation}
\mathcal A =\mathcal A (t_0,\alpha _0,\delta _0,M,\epsilon ):=\{\zeta '<\zeta '_3 \wedge \zeta '_2\}\cap \{ t_0+\epsilon _0 \delta _0^{-3}\le \zeta ' \le t_0+ \epsilon _0 ^{-1} \delta _0^{-3}\}.
\end{equation}
By Lemma~\ref{lem:stitching medium} we have  
\begin{equation}\label{eq:A is likely}
\mathbb P (\mathcal A ^c  )\le C \delta _0 +C \left(\frac{\delta _0}{M}\right)^c +Ce ^{-c\epsilon ^{-\frac{1}{2}} } \le C \left(\frac{\delta _0}{M}\right)^c +Ce ^{-c\epsilon ^{-\frac{1}{2}} }
\end{equation}
Define $\alpha''= \alpha ''(t_0,\alpha _0,\delta _0,M):=\alpha '(t_{i-1},\alpha _{i-1},\delta _{i-1})$ on the event $\{ \zeta '=t_i \}$ for $i>0$ and $\alpha ''(t_,\alpha _0,\delta _0,M):=\alpha _0$ on $\{ \zeta '=t_0\}$. Note that, on $\{\zeta '<\zeta ' _2 \wedge \zeta ' _3\}$ we have $|\alpha ''-\delta _0/2|\le (\delta _0/2)^{1.1}$. Since all the parameters are going to change again in the next section we write more specifically $\zeta '(t_0,\alpha _0,\delta _0,M)$ instead of $\zeta '$. We do the same with the other stopping times $\zeta '_1,\zeta '_2$ and $\zeta '_3$.

\subsubsection{Stitching long intervals}\label{sec:long}

In this section we stitch together the long intervals of time in order to prove that aggregate has the $t^{\frac{2}{3}}$ growth.

\begin{thm}\label{thm:stopping time of getting to speed T to the minus one third}
	Let $\epsilon , \alpha _0 >0$ such that $\alpha _0$ is sufficiently small depending on $\epsilon $. Let $Y_{t_0} \in \mathfrak{R} ' (\alpha _0, t_0)$ and let $ \Pi _S$ be the aggregate with initial condition $Y_{t_0}$. There exist $C_\epsilon >0$ such that the following holds. Let $t\ge t'(t_0, \alpha_0,\epsilon )$ and define the stopping time 
	\begin{equation}
	\zeta _t:=\inf \{s>t_0 : \exists \alpha  \text{ such that } |\alpha -t^{-\frac{1}{3}}| \le t^{-\frac{2}{5}} \text{ and } Y_s \in \mathfrak{R}' ( \alpha, s) \}.
	\end{equation}
	We have that
	\begin{equation}\label{eq:event in theorem}
	\mathbb P \left( \zeta _t \le C_\epsilon t \text{ and } X_{\zeta _t} \le C_\epsilon t^{\frac{2}{3}} \right) \ge 1-\epsilon 
	\end{equation}
\end{thm}

\begin{proof}
	Let $m\ge 1$ sufficiently large independently of $t$ and let $n=n_t:=\lceil \log _2(\alpha _0^{-1} t^{\frac{1}{3}}) \rceil $. Define for all $0\le j \le n$ 
	\begin{equation}
	\delta _j:=\alpha _0 2^{-j}, \quad \quad M_j:=\begin{cases}
	\ \sqrt{\delta _j} ,\quad &j \le \frac{1}{2} n \\
	\delta _j 2^{m+\frac{1}{4}(n-j)}, \quad &j> \frac{1}{2} n
	\end{cases}, \quad \quad \epsilon _j:=(m+n-j)^{-2}.
	\end{equation}
	It is easy to verify that $\frac{1}{2}t^{-\frac{1}{3}} \le \delta _{n}\le  t^{-\frac{1}{3}} $ and that $\delta _j \le M_j \le \sqrt{\delta _j}$ for all $j$ as long as $t$ is sufficiently large.
	
	Define the sequence of stopping times $t_i$ and the sequence of random variables $\alpha _i\in \mathcal F _{t_i}$ inductively by
	\begin{equation}
	t_{i+1}:=\zeta ' (t_i,\alpha _i,\delta _i, M_i),\quad \alpha_{i+1}:=\alpha '' (t_i,\alpha _i,\delta _i, M_i)
	\end{equation} 
	
	Let $\mathcal A _i: =\mathcal A (t_i,\alpha _i,\delta _i,M_i,\epsilon _i)$ and $\mathcal B _j :=\bigcap _{i=1}^j \mathcal A_i$. We start by showing that $\mathbb P (\mathcal B _{n})$ is large. Using \eqref{eq:A is likely} we get
	\begin{equation}\label{eq:prob of bad event}
	\begin{split}
	\mathbb P (\mathcal B _{n} ^c ) \le \mathbb P (\mathcal A _1^c) +\sum _{j=2}^{n} \mathbb P ( \mathcal B _{j-1 } \cap \mathcal A _j ^c )&= \mathbb P (\mathcal A _1^c)+\sum _{j=2}^{n} \mathbb E \left[ \mathds 1 _{\mathcal B _{j-1} }\mathbb P ( \mathcal A _j ^c \ | \ \mathcal F _{t_j-1})  \right]  \\
	&\le C \sum _{j=1}^{n} \left(\frac{\delta _j}{M_j}\right)^c +C \sum _{j=1}^{n} e^{-c \epsilon _j^{-\frac{1}{2}}}. 
	\end{split}
	\end{equation} 
	Let $S_1,S_2$ be the corresponding first and second sums in the right hand side of \eqref{eq:prob of bad event}. We bound the second term using the definition of $\epsilon _j$ by 
	\begin{equation}
	S_2\le C e^{-cm } \sum _{j=1}^{n}   Ce^{-c(n-j)} \le C e^{-cm } \sum _{k=0}^{\infty }e^{-ck} \le Ce^{-cm}.
	\end{equation}
	The first term is bounded as follows 
	\begin{equation}
	S_1 \le C \sum _{j =1}^{n/2} \delta _j ^c+Ce^{-cm} \sum _{j= \lceil n /2 \rceil  } ^{n}  e^{-c(n-j)} \le C \alpha _0^c+C e^{-cm }\sum _{k=1}^{\infty } e^{-ck} \le C \alpha _0^c +C e^{-cm}.
	\end{equation}
	Since $\alpha _0 $ is sufficiently small and $m$ is sufficiently large we get that $\mathbb P (\mathcal B _{n} ) \ge 1-\epsilon $.
	
	Next, we show that on $\mathcal B _{n}$ the event in \eqref{eq:event in theorem} holds. To this end we bound the stopping time $t_{n}$ and the aggregate size $X_{t_{n}}$ and then show that, on $\mathcal B _{n}$ we have that $\zeta _t \le t_{n}$. We have 
	\begin{equation}
	\begin{split}
	t_{n}=t_0+\sum _{j=0}^{n-1} t_{j+1}-t_j \le t_0+\sum _{j=0}^{n} \epsilon _j^{-1} \delta_j^{-3} \le t_0+ \delta _{n}^{-3}\sum _{j=0}^{n} (m+n-j)^2 2^{3(n-j)}\\
	\le t_0+C t \sum _{k=0}^{\infty } (2m+k)^2 2^{-3k} \le t_0+Cmt \le C_{t_0,m}t.
	\end{split}
	\end{equation}
	
	We turn to bound the aggregate size at time $t_n$. We have
	\begin{equation}
	\begin{split}
	X_{t_{n}} \le  \sum _{j=0}^{n-1} X_{t_{j+1}}-X_{t_{j}} &\le C\sum _{j=0}^{n-1 } M_j(t_{j+1}-t_{j}) \le  C\sum _{j=0}^{n} M_j\epsilon _j^{-1}\delta _j^{-3} \\
	&\le C\sum _{j=1}^{n/2} \epsilon _j^{-1} \delta _j^{-3} +C_{m}  \sum _{j=\lceil n/2 \rceil }^{n } (m+n-j)^2\delta _j^{-2} 2^{\frac{1}{4}(n-j)} \\
	&\le C\epsilon _0^{-1} \delta _{\lfloor n/2 \rfloor }^{-3}+C_m \delta _{n}^{-2} \sum _{j=\lceil n/2 \rceil }^{n }(m+n-j)^2 e^{-c(n-j)} \le \\
	&\le C_{m} n^2 \delta _{\lfloor n/2 \rfloor }^{-3}   + C_m \delta _{n}^{-2} \sum _{k=0}^{\infty} (m+k)e^{-ck} \\
	&\le   C_{m} n^2 \delta _{\lfloor n/2 \rfloor }^{-3} + C_m \delta _{n}^{-2} 
	\le C_{m,\alpha _0} t^{\frac{1}{2}} \log ^2 t+C_{m,\alpha _0} t^{\frac{2}{3}} \le C_{m,\alpha _0} t^{\frac{2}{3}}.
	\end{split}
	\end{equation}
	
	Thus, it suffices to prove that on $\mathcal B _{n}$ we have $\zeta _t \le t_n$. To this end, recall that the time interval $[t_0,t_{n}]$ is the union of many short intervals. On $\mathcal B _{n}$ the process is regular in the endpoints of each of these intervals with a parameter $\alpha $ that do not change by more than $\alpha ^{\frac{3}{2}} \log ^C (1/\alpha )$ between consecutive short intervals. Thus, by a discrete mean value theorem the process is regular at some time $t'<t_{n}$ with $\alpha '$ such that $| \alpha '- t^{-\frac{1}{3}}|\le t^{-\frac{2}{5}}$. It follows that $\zeta _t \le t_{n}$ on $\mathcal B _{n}$.
\end{proof}

\subsection{Convergence in distribution}\label{sec:convergence in distribution}

In this section we show how to deduce the convergence to the diffusion from the induction step theorem. To this end we'll use the following theorem due to Helland \cite{helland1981minimal}. The theorem we state here is weaker than the result in \cite{helland1981minimal} but suffices for what we need. 

\subsubsection{The result of Helland}
We start with the basic settings and notations in the paper of Helland \cite{helland1981minimal} . Let $x_0>0$, $\mu :[0,\infty ]\to \mathbb R $ and $\sigma : [0,\infty ] \to [0,\infty ]$. Let $X(t)$ for $t>0$ be a solution to the stochastic differential equation 
\begin{equation}
dX(t)= \mu (X(t)) dt +\sigma (X(t)) d B_t,\quad X(0)=x_0.
\end{equation}
such that almost all sample paths are continuous. Suppose that there are no accessible boundaries. That is $\mathbb P ( \forall t>0 ,\ 0<X(t)<\infty )=1$

For any $n\ge 1$, let $X^{(n)}(t)>0$ be a sequence of processes. Let $\Delta _n \to 0$ and for any $k \ge 0$ let $t_k=t_k^{(n)}:=k \Delta _n$. In what follows $\overset{p}{\longrightarrow }$ denotes convergence in probability. Suppose that:
\begin{enumerate}
	\item 
	$X_n(0)\overset{p}{ \longrightarrow } x_0$
	\item 
	$X_n(t)$ is fixed on the intervals $[t_k, t_{k+1})$
	\item 
	For all $t,\epsilon >0$ we have as $n \to \infty $
	\begin{equation}
	\sum _{\substack{k: \\ t_k \le t}} \mathbb P \left(  |X_n(t_{k+1})-X_n(t_k) | \ge \epsilon \ \big|  \ \mathcal F _{t_k}   \right) \overset{p}{\longrightarrow } 0
	\end{equation}
	\item 
	For all $t>0$ and a compact set $K \subseteq (0,\infty )$  we have as $n \to \infty $
	\begin{equation}
	\sum _{\substack{k: \\ t_k \le t}} \Big| \mathbb E \left(  X_n(t_{k+1})-X_n(t_k)  \ \big|  \ \mathcal F _{t_k}   \right)-\mu (X_n(t_k)) \cdot \Delta _n \Big| \cdot  \mathds 1 \{X_n(t_k) \in K\} \overset{p}{\longrightarrow } 0 
	\end{equation}
	\item 
	For all $t>0$ and a compact set $K \subseteq (0,\infty )$  we have as $n \to \infty $
	\begin{equation}
	\sum _{\substack{k: \\ t_k \le t}} \Big|  \mathbb E \left(  (X_n(t_{k+1})-X_n(t_k) )^2  \ \big|  \ \mathcal F _{t_k}   \right) -\sigma ^2 (X_n(t_k)) \Delta _n \Big| \cdot  \mathds 1 \{X_n(t_k) \in K\} \overset{p}{\longrightarrow } 0 
	\end{equation}
\end{enumerate}

\begin{thm}[Helland]\label{thm:helland}
	Suppose that the assumptions $(1)-(5)$ above hold. Then, 
	\begin{equation}
	(X_n(t))_{t>0} \overset{d}{ \longrightarrow } (X(t))_{t>0}
	\end{equation}
\end{thm}

\begin{remark}\label{remark:on convergence}
	The convergence in Theorem~\ref{thm:helland} is in the sense of the Stone topology defined in \cite{stone1963weak} and discussed in \cite{helland1981minimal}. We will not define this topology. Instead, we just note that:
	\begin{enumerate}
		\item 
		when $X$ is almost surely continuous, the convergence is equivalent to the following two conditions 
		\begin{enumerate}
			\item 
			The finite dimensional distributions of $X_n$ converge weakly to the finite dimensional distributions of $X$.
			\item 
			For all $\epsilon ,T >0$ we have  
			\begin{equation}
			\lim _{\delta \to 0} \limsup _{n\to \infty } \mathbb P \bigg( \sup _{\substack{ 0 \le s,t\le T  \\ |s-t|\le \delta  }} \left|X_n(t)-X_n(s) \right| \ge \epsilon  \bigg) =0
			\end{equation} 
		\end{enumerate}
		\item 
		It can be shown, using Skorhod representation theorem, that if $X$ is almost surely continuous then there is a coupling of the processes $X_n$ and $X$ such that for all $T>0$ almost surely we have $\sup _{t\le T} |X_n(t)-X(t)| \to 0$.
	\end{enumerate}
\end{remark}

\subsubsection{Proof of Theorem~\ref{thm:main theorem 2}}

Throughout this section $X_t$ is the size of the usual aggregate with no initial condition. Let $\epsilon  >0$ sufficiently small, $\alpha _0 >0$ sufficiently small depending on $\epsilon $ and let  $t_0:=\alpha _0^{-20}$. Recall the definition of $\overline{Y} _{t,\alpha }$ given in \eqref{eq:def:sandwichfixed}. Define the event
\begin{equation}\label{eq:def of A_0}
\mathcal A _0 := \{ X_{t_0} \le t_0 \} \cap \{\forall s >0,\ Y_{t_0}(s) \le \overline{Y} _{t_0,\alpha _0} (s)   \}  \cap \{ \overline{Y} _{t_0,\alpha _0} \in \mathfrak{R} (\alpha _0, \alpha _0^8 ;[t_0]) \}.
\end{equation} 
The first event on the right hand side of \eqref{eq:def of A_0} holds with high probability as $S(t)\le \frac{1}{2}$ almost surely and therefore $X_{t_0}\precsim \text{Poisson}(t_0/2)$. The second and third events hold with high probability by Lemma~\ref{lem:aggregate is bounded} and Theorem~\ref{thm:induction:base:main} respectively. Thus $\mathbb P (\mathcal A _0) \ge 1-\epsilon $ as long as $\alpha _0$ is sufficiently small depending on $\epsilon $. Let $( \overline{Y}_t, \overline{S}(t) ,\overline{X} _t)$ be the aggregate with initial condition $\overline{Y} _{t_0,\alpha _0}$. We note that, according to Definition~\ref{def:aggregate with initial condition} we have that $\overline{X} _{t_0}=0$ and by Claim~\ref{claim:aggregate with initial condition} we have almost surely on $\mathcal A _0$ that $X_t \le t_0+\overline{X} _t$ for all $t>t_0$.

We turn to bootstrap the regularity by applying Theorem~\ref{subsec:regoverview:overview}. Let $t_1:=t_0+\alpha _0^{-2} \log ^{10 \theta }(1/\alpha _0)$. Define 
\begin{equation}
    \mathcal A_1 := \{\overline{X}_{t_1} \le  t_0 \} \cap \{  \exists \alpha _1 \text{ with } \overline{Y} _{t_1} \in \mathfrak{R} '(\alpha _1 ,t_1) \}.
\end{equation}
Almost surely on $\mathcal A _0$ we have that $\mathbb P (\mathcal A _1 \ | \ \mathcal F _{t_0} )\ge 1- \epsilon $. Indeed, the first event holds with high probability by the same argument as above. For the second one, by Theorem~\ref{subsec:regoverview:overview}, there is a random variable $\alpha _1'\le 2 \alpha _0$ such that almost surely on $\mathcal A _0$ we have  $\mathbb P ( Y_{t_1}\in \mathfrak{R} ^\sharp (\alpha _1', \alpha _1^8,[t_1] \ | \ \mathcal F _{t_0} )\ge \alpha _0^{10}$ and therefore, by Definition~\ref{def:def of regularity prime} there is some $\alpha _1 \le 3 \alpha _0$ such that almost surely on $\mathcal A _0$, $\mathbb P ( Y_{t_1}\in \mathfrak{R} ' (\alpha _1, t_1 ) \ | \ \mathcal F _{t_0})$. We get that 
\begin{equation}
\mathbb P (\mathcal A _0 \cap \mathcal A _1) = \mathbb E (\mathds 1 _{\mathcal A _0} \mathbb P (\mathcal A _1 \ | \ \mathcal F _{t_0})) \ge (1-\epsilon ) \mathbb P (\mathcal A _0) \ge (1-\epsilon )^2 \ge 1-2 \epsilon .
\end{equation}

Next let $\delta >0$ sufficiently small such that $\delta^{\frac{2}{3}} C _\epsilon <\epsilon $ where $C_\epsilon $ is from Theorem~\ref{thm:stopping time of getting to speed T to the minus one third}. Finally, let $t>0$ sufficiently large (depending on all other parameters) such that $\delta t \ge t'(t_1, \alpha _1 ,\epsilon )$ where $t'$ is from Theorem~\ref{thm:stopping time of getting to speed T to the minus one third}. Let $t_2$ and $\alpha _2$ be the stopping time and random variable from Theorem~\ref{thm:stopping time of getting to speed T to the minus one third} with $\delta t$ instead of $t$. Define the events
\begin{equation}
\mathcal A _2 :=\{ t_2 \le \epsilon t  \} \cap \{\overline{X}_{t_2} \le \epsilon t ^{\frac{2}{3}}\}.
\end{equation}
and $\mathcal A := \mathcal A _0 \cap \mathcal A _1 \cap \mathcal A _2$. The random variable $\alpha _2$ is defined on $\mathcal A _2 $ and satisfies $|\alpha _2 -(\delta t )^{-\frac{1}{3}}| \le C _\delta t^{-\frac{2}{5}}$. By Theorem~\ref{thm:stopping time of getting to speed T to the minus one third} and the choice of $\delta $ almost surely on the event $\mathcal A _0 \cap \mathcal A _1$ we have $\mathbb P (\mathcal A _2  \ | \ \mathcal F _{t_1}) \ge 1-\epsilon  $. Thus we get $\mathbb P (\mathcal A )\ge 1-3 \epsilon $.

We let $ (\tilde{Y } ,\tilde{S}, \tilde{X})$ be the aggregate $(\overline{Y },\overline{S} , \overline{X })$ condition on the event $\mathcal A $. Starting with $t_2$ and $\alpha _2$, We define the sequence $t_i$ and the random variables $\alpha _i$ like in Section~\ref{sec:short}, with a minor difference. Instead of stopping the process when the speed changes by a factor of $2$, we stop the process when the speed changes by a factor of $\log (1/\alpha _2)$. More precisely, we let $t_3:=t_2+\alpha _2^{-2} \log ^{10 \theta -2 }(1/\alpha _2)$ and $t_{i+1}:=t_2+(i-1)(t_3-t_2)$. We define the sequence $\alpha _i$ inductively. Let
\begin{equation}
\begin{split}
&\zeta_1:=\min \left\{t_i\ge t_2:\ \alpha _i < \alpha _2 \log ^{-1} (1/\alpha _2) \right\} \\
&\zeta _2 :=\min \{t_i\ge t_2:\ \alpha _i >  \alpha _2 \log (1/\alpha _2) \} \\
&\zeta _3:=\min \{t_i \ge t_2: \tilde{Y}_{t_i} \notin \mathfrak{R} '(\alpha _i,t_i)  \}\\
& \zeta _4:=\min \{t_{i+1}\ge t_3: \left| X_{t_{i+1}}-X_{t_{i}}-\alpha_{i} (t_1-t_0) \right|  \ge  \alpha _{i} ^{1.4}(t_1-t_0) \},
\end{split}
\end{equation}
We let $\zeta :=  \zeta _1 \wedge \zeta _2 \wedge \zeta _3 \wedge \zeta _4$. On $\{\zeta >t_i\}$ we let $\alpha _{i+1}$ be the random variable $\alpha _1 $ from the Theorem~\ref{thm:induction step for repeatd use} and on $\{\zeta \le t_i\}$ we let $\alpha _{i+1}:=\alpha _i$. 

\begin{claim}\label{claim:zeta is large}
	We have that $\mathbb P (\zeta \le t_2+ \alpha _2^{-3}  \log (1/\alpha _2))\le  C \log ^{-c}(1/\alpha _2)$
\end{claim}

We'll not give all the details of the proof as it is similar to some other results in Sections~\ref{sec:short} and \ref{sec:medium}.

\begin{proof}[Sketch of Proof]
	We define the sequence of scales $\delta _2,\delta _3,\dots $ as in Section~\ref{sec:medium} with $\delta _2:=\alpha _2$ and $\delta _i \in \{\alpha _2 2^k \ | \ k\in \mathbb Z  \}$. Unlike Section~\ref{sec:medium} we do not change the steps of time every time we cross a scale. Let $W_i :=\log _2 (\delta _i/\alpha _2)$ as in the proof of Lemma~\ref{lem:stitching medium}. Since Lemma~\ref{lem:stitching short intervals} holds even when the steps of time are larger or smaller by a $\log $ factor, we still have that $W_i$ is roughly a random walk with a downward drift. We let $n_0:=\log \log (1/\alpha _2)$. By the martingale argument in the proof of Lemma~\ref{lem:stitching medium} the sequence $\delta _i$ will decrease to $4 \alpha _2/ \log (1/\alpha _2)$ before hitting $\zeta $. Once the speed is so small, the time it will take for the speed to change its scale by a factor of $2$ will be at least $\alpha _2^{-3} \log (1/\alpha _2)$ with high probability by Lemma~\ref{lem:stitching short intervals}. 
\end{proof}

We let $A(t)$ be the continuation of $\alpha _i$ that is $A(t)$ is defined by $A(t):=\alpha _i$ for any $t_i\le t <t_{i+1}$. We also define the rescaled process 
\begin{equation}
V_t(s):= t^{\frac{1}{3}} A(t_2+ s t)
\end{equation}

\begin{lem}\label{lem:speed to sde}
	We have that
	\begin{equation}
	\left( V_t(s) \right) _{s>0} \overset{d}{\longrightarrow } \left(Z_\delta (s)\right)_{s>0}, \quad t \to \infty 
	\end{equation}
	where $Z_\delta $ is the solution to $dZ=2Z^{\frac{5}{2}}dB_t$ with $Z_\delta (0)=\delta ^{-\frac{1}{3}}$. The convergence is in the sense of Remark~\ref{remark:on convergence}.
\end{lem}

\begin{proof}
	We check all the conditions of Theorem~\ref{thm:helland} with $\mu =0$, $\sigma (x)=2x^{\frac{5}{2}}$ and $x_0:=\delta ^{-\frac{1}{3}}$. Let
	\begin{equation}
	\Delta _t:=\frac{(t_3-t_2)}{t}  \le C_\delta  t^{-\frac{1}{3}} \log ^Ct \to  0, \quad t \to \infty 
	\end{equation}
	Finally, we let $s_k:=k \Delta _t$ so that $V_t(s_k)=A(t_{k+2})=\alpha _{k+2}$. The first condition holds since 
	\begin{equation}
	V_t(0)=t^{\frac{1}{3}} A(t_2)= t^{\frac{1}{3}} \alpha _2  \to \delta ^{-\frac{1}{3}}.
	\end{equation}
	It is clear that the second condition holds. Next, we check the third condition. For sufficiently large $t$ we have that
	\begin{equation}
	\mathbb P ( | V_t(s_{k+1})-V_t(s_k) |\ge \epsilon ' \ | \ \mathcal F _{t_k}) \le \mathbb P(  | \alpha _{k+3}-\alpha _{k+2} | \ge \epsilon ' t^{-\frac{1}{3}} \ | \  \mathcal F _{t_k} ) =0
	\end{equation}
	so in particular the third condition holds.
	We turn to prove the fourth condition. Let $s>0$ and $K \subseteq (0,\infty )$ a compact set. We have that 
	\begin{equation}
	\begin{split}
	\sum _{\substack{k: \\ s_k \le s}} \left| \mathbb E \left[  V_t(s_{k+1})-V_t(s_{k})  \ \big|  \ \mathcal F _{t_{k+1}}   \right] \right|  \mathds 1 \{V_t(t_k) \in K \} =t^{\frac{1}{3}} \sum _{\substack{k: \\ s_k \le s}}  \left| \mathbb E \left[ \alpha _{k+3}-\alpha _{k+2}  \ \big|  \ \mathcal F _{t_{k+1}}   \right] \right|  \mathds 1 \{ t^{\frac{1}{3}}\alpha _{k+2} \in K\} \\
	\le C  t^{\frac{1}{3}} \sum _{\substack{k: \\ t_k \le st}} \frac{(t_{k+3}-t_{k+2}) \alpha _{k+2}^4}{\log ^3(1/\alpha _{k+2})} \mathds 1 \{t^{\frac{1}{3}} \alpha _{k+2} \in K\} \le C _{K,\delta } t^{\frac{1}{3}}  \frac{\alpha _2 ^4 }{\log ^3 (1/\alpha _2 )} \sum _{\substack{k: \\ t_k \le st}} (t_{k+2}-t_{k+1}) \le \frac{C_{K,\delta }}{\log ^3 t} \to 0
	\end{split}
	\end{equation}
	
	Lastly, we check the fifth condition 
	\begin{equation}
	\begin{split}
	&\sum _{\substack{k: \\ s_k \le s}}  \left| \mathbb E \left(  (V_t(s_{k+1})-V_t(s_k) )^2  \ \big|  \ \mathcal F _{t_{k+1}}   \right) -4 V_t(t_k) ^5 \Delta _t \right|\mathds 1 \{V_t(t_k) \in K \} \\
	&=t^{\frac{2}{3}}\sum _{\substack{k: \\ s_k \le s}} \left| \mathbb E \left(  (\alpha _{k+3}-\alpha _{k+2} )^2  \ \big|  \ \mathcal F _{t_{k+2}}   \right) - 4 \alpha _{k+2} ^5 (t_3 -t_2) \right| \mathds 1 \{V_t(t_k) \in K \} \\
	&\le  C t^{\frac{2}{3}}\sum _{\substack{k: \\ s_k \le s}} \alpha _{k+2}^{5+c} (t_1-t_0)   \mathds 1 \{V_t(t_k) \in K, \ \zeta > t_{k+2} \}	+Ct^{\frac{2}{3}}\sum _{\substack{k: \\ s_k \le s}}   \alpha _{k+2} ^5 (t_3 -t_2) \mathds 1 \{V_t(t_k) \in K , \ \zeta \le t_{k+2}\}\\
	&\le C t^{\frac{5}{3}}\alpha _2 ^{5+c} +Ct^{\frac{5}{3}}\alpha _1^{5} \mathds 1 \{ \zeta \le t_2 +\alpha _2^{-3} \log (1/\alpha _2) \} \overset{p}{\longrightarrow } 0.
	\end{split}
	\end{equation}
	Where the convergence follows from Claim~\ref{claim:zeta is large}. The result now follows from Theorem~\ref{thm:helland}.
\end{proof} 

\begin{cor}\label{cor:integrate both sides}
	We have the following convergence in finite dimensional distributions 
	\begin{equation}
	t^{-\frac{2}{3}} \left( \tilde{X}_{t_2+st}-\tilde{X}_{t_2} \right) \overset{d}{\longrightarrow } \intop _0^s Z_\delta (x)dx
	\end{equation}
\end{cor}

\begin{proof}
	On the event $\{\zeta > t_2 + \alpha _2^{-3} \log (1/\alpha _2)\}$ we have for all $k \ge 2 $ with $t_k \le st$ that 
	\begin{equation}
	\left| \tilde{X} _{t_{k+1}}-\tilde{X}_{t_k} -\alpha _k (t_{k+1}-t_k) \right| \le \alpha _k^{1.4} (t_3-t_2) \le \alpha _2^{1.3}(t_3-t_2)\le t^{-\frac{2}{5}}(t_3-t_2).
	\end{equation}
	Thus
	\begin{equation} \label{eq:aggregate to integral}
	\left| \tilde{ X} _{t_2+st}-\tilde{X}_{t_2 } - \intop _{t_2}^{t_2+st} A(t')dt'\right| \le C_st^{\frac{3}{5}}
	\end{equation}
	Next, we have that 
	\begin{equation}\label{eq:integral to sde}
	t^{-\frac{2}{3}}\intop _{t_2}^{t_2+st} A(t')dt'=t^{\frac{1}{3}} \intop _0^s A(t_2+xt)dx =\intop _0^s V_t(x)dx \overset{d}{\longrightarrow } \intop _0^s Z_\delta (x)dx,
	\end{equation}
	where the convergence is in finite dimensional distributions and it follows from Lemma~\ref{lem:speed to sde} and by part (2) of Remark~\ref{remark:on convergence}. Indeed, there is a coupling of the processes such that the convergence is almost sure and uniform on compact subsets of $(0,\infty )$ and therefore the integrals of the processes converge as well. The statement of the corollary follows from \eqref{eq:aggregate to integral} and \eqref{eq:integral to sde}.
\end{proof}

We need one last claim for the proof of Theorem~\ref{thm:main theorem 2}.

\begin{claim}\label{claim:convegence of integrals of sdes}
	Let $Z$ and $Z_M$ for $M \ge 1$, be the solutions of $dZ(t)=2 Z(t)^{\frac{5}{2}}d B _t$ with $Z(0)=\infty $ and $Z_M(0)=M$. We have 
	\begin{equation}
	\left(\int _0^s Z_M(x) dx\right)_{s\ge 0} \overset{d}{\longrightarrow } \left(\int _0^s Z(x) dx\right)_{s\ge 0},
	\end{equation} 
	as $M \to \infty $, where the convergence is in finite dimensional distributions.
\end{claim}

\begin{proof}
    Let $V_a (x)$ and $V(x)$ be a $\frac{8}{3}$-Bessel processes with initial condition $V_a (0)=a $ and $V(0)=0$ By Remark~\ref{remark:on bessel} it suffices to show that
    \begin{equation}\label{eq:convergence of Bessel}
	\left(\int _0^s V_a (x) dx\right)_{s\ge 0} \overset{d}{\longrightarrow } \left(\int _0^s V(x) dx\right)_{s\ge 0},
	\end{equation} 
	as $a \to 0$. To this end we couple the two processes. Recall that $V_a $ and $V$ are solutions to the equation
	\begin{equation}
	    dV(x)=dB_x +\frac{5}{6}\frac{dx}{V(x)}.
	\end{equation}
	Suppose that both $V_a$ and $V$ are driven by the same Brownian motion. The drift term is larger as long as the process is smaller and therefore, under this coupling, $V_a (x)-V(x)$ is positive and decreasing. Thus $|V_a (x)-V(x)|\le a $ for all $x>0$. Therefore, for all $s>0$
	\begin{equation}
	    \left| \intop _0^s V_a(x)dx - \intop _0^s  V(x)dx   \right| \le \intop _0^s |V_a(x) -V(x)|dx\le as\to 0 ,\quad \text{a.s.},
	\end{equation}
	as $a\to 0$. Equation \eqref{eq:convergence of Bessel} follows from this.
\end{proof}

\begin{proof}[Proof of Theorem~\ref{thm:main theorem 2}]
	On the event $\mathcal A $ we have almost surely 
	\begin{equation}\label{eq:bound onX}
	X_t \le X_{t_0}+\overline{X} _t \le t_0+\overline{X}_{t_2}+\overline{X}_{t_2+t}-\overline{X}_{t_2}\le 2 \epsilon t^{\frac{2}{3}}+ \overline{X}_{t_2+t}-\overline{X}_{t_2},
	\end{equation}
	where the last inequality holds for sufficiently large $t$. Thus, for all $z>0$ we have  
	\begin{equation}
	\begin{split}
	\liminf _{t\to \infty } \mathbb P (X_t \le zt^{\frac{2}{3}} )  &\ge \liminf _{t\to \infty } \mathbb P (X_t \le zt^{\frac{2}{3}} , \mathcal A ) \\
	&\ge \liminf _{t\to \infty } \mathbb P (\overline{X}_{t_2+t}-\overline{X}_{t_2} \le (z-2 \epsilon )t^{\frac{2}{3}} , \mathcal A ) \\
	&\ge \liminf _{t\to \infty } \mathbb P ( \mathcal A )\cdot \mathbb P (\overline{X}_{t_2+t}-\overline{X}_{t_2} \le (z-2 \epsilon ) t^{\frac{2}{3}} \ | \ \mathcal A )\\
	&\ge (1-2 \epsilon ) \cdot \liminf _{t\to \infty } \mathbb P (\tilde{X}_{t_2+t}-\tilde{X}_{t_2} \le (z-2 \epsilon ) t^{\frac{2}{3}})  \\
	&\ge (1-2 \epsilon ) \cdot  \mathbb P \Big( \int _0^1 Z_\delta (x) dx \le (z-2\epsilon ) \Big) ,
	\end{split}
	\end{equation}
	where in the last inequality we used Corollary~\ref{cor:integrate both sides}.	Recall that the following inequality holds for all $\epsilon, \delta $ as long as $\delta $ is sufficiently small depending on $\epsilon $. Taking $\delta $ to zero and using Claim~\ref{claim:convegence of integrals of sdes} we get
	\begin{equation}
	\liminf _{t\to \infty } \mathbb P (X_t \le zt^{\frac{2}{3}} ) \ge (1-2 \epsilon ) \cdot  \mathbb P \Big( \int _0^1 Z (x) dx \le (z-2\epsilon ) \Big) ,
	\end{equation}
	where $Z$ is the solution to the same SDE with $Z(0)=\infty$. Next, taking $\epsilon $ to zero we get
	\begin{equation}
	\liminf _{t\to \infty } \mathbb P (X_t \le zt^{\frac{2}{3}} ) \ge   \mathbb P \Big( \int _0^1 Z (x) dx \le z \Big) .
	\end{equation}
	By the same arguments for all $z_1 ,\dots ,z_k>0$ we have 
	\begin{equation}
	\liminf _{t\to \infty } \mathbb P (X_{s_1t} \le z_1t^{\frac{2}{3}}, \dots , X_{s_kt} \le z_kt^{\frac{2}{3}} ) \ge \mathbb P \Big( \int _0^{s_1} Z (x) \le z_1, \dots , \int _0^{s_k} Z (x) \le z_k \Big) .
	\end{equation}
	
	We turn to prove a matching upper bound on the last probability. As the arguments are similar we will not give all the details. To this end, we change the definition of $\mathcal A _0$ slightly. We take a sufficiently small $\alpha _0$ and $t_0=100 \alpha _0^{-2}$ and use the same definition as in \eqref{eq:def of A_0} with $\underline{Y}_{t_0,\alpha _0}$ instead of $\overline{Y}_{t_0,\alpha _0}$. We also let $( \underline{Y}_t, \underline{S}(t) ,\underline{X} _t)$ be the aggregate with initial condition $\underline{Y}_{t_0,\alpha _0}$ and work with $( \underline{Y}_t, \underline{S}(t),\underline{X} _t)$ instead of $( \overline{Y}_t ,\overline{S}(t),\overline{X} _t)$. The rest of the definitions such as $\mathcal A_1 , \mathcal A _2, \mathcal A $ and $\tilde{X}$ stay the same. As in \eqref{eq:bound onX} we have on $\mathcal A $ that
	\begin{equation}
	X_t\ge \underline{X} _t-\underline{X}_{t_2}\ge \underline{X}_{t_2+(1-2\epsilon )t}-\underline{X}_{t_2}.
	\end{equation}  
	Thus
	\begin{equation}
	\begin{split}
	\limsup _{t\to \infty } \mathbb P (X_t \le zt^{\frac{2}{3}} )  &\le \limsup _{t \to \infty } \mathbb P (\mathcal A^c) +\limsup _{t\to \infty } \mathbb P (\overline{X}_{t_2+(1-2\epsilon )t}-\overline{X}_{t_2} \le z t^{\frac{2}{3}} , \mathcal A ) \\
	&\le 3 \epsilon  + \limsup _{t\to \infty } \mathbb P (\tilde{X}_{t_2+(1-2\epsilon )t}-\tilde{X}_{t_2} \le z t^{\frac{2}{3}})  \\
	&\le 3 \epsilon + \cdot  \mathbb P \Big( \int _0^{1-2\epsilon } Z_\delta (x) dx \le z \Big) .
	\end{split}
	\end{equation}
	Taking $\delta \to 0$ and then $\epsilon \to 0$ we get
	\begin{equation}
	\limsup _{t\to \infty } \mathbb P (X_t \le zt^{\frac{2}{3}} ) \le   \mathbb P \Big( \int _0^1 Z (x) dx \le z \Big) .
	\end{equation}
	By the same arguments 
	\begin{equation}
	\limsup _{t\to \infty } \mathbb P (X_{s_1t} \le z_1t^{\frac{2}{3}}, \dots , X_{s_kt} \le z_kt^{\frac{2}{3}} ) \le \mathbb P \Big( \int _0^{s_1} Z (x) \le z_1, \dots , \int _0^{s_k} Z (x) \le z_k \Big) .
	\end{equation}
	Thus
	\begin{equation}
	\lim _{t\to \infty } \mathbb P (X_{s_1t} \le z_1t^{\frac{2}{3}}, \dots , X_{s_kt} \le z_kt^{\frac{2}{3}} ) = \mathbb P \Big( \int _0^{s_1} Z (x) \le z_1, \dots , \int _0^{s_k} Z (x) \le z_k \Big) .
	\end{equation}
	This finishes the proof of Theorem~\ref{thm:main theorem 2}.
\end{proof}

	\section*{Acknowledgements}
	We are deeply grateful to Vladas Sidoravicius who, with great enthusiasm, introduced us to the MDLA model and wish we could have shared these results with him.
	DE thanks Ryan Alweiss, Ofir Gorodetsky, Oren Yakir, Sang Woo Ryoo, Elad Zelingher and Jonathan Zung for useful discussions. DN is supported by a Samsung Scholarship. AS thanks Amir Dembo for his thoughtful discussions on the mildly supercritical case.  AS is supported by  NSF grants DMS-1352013 and DMS-1855527, Simons Investigator grant and a MacArthur Fellowship.
	\bibliographystyle{amsplain}
	\bibliography{biblicomplete}
	
	\newpage

 \appendix

\section{The renewal process, Fourier transform and singularity analysis}\label{sec:fourier and renewal}

In this section we study the asymptotic behavior of functions $K$ and the density of the renewal process $K^*$ as $\alpha \to 0$. We also estimate a key integral that appears in the drift term of the speed increment from Section \ref{sec:double int}. The main tools we use are Fourier transforms and complex analysis.

\subsection{Fourier transform of $\tau $ and $K$}

Let $0< \alpha <1$. Let $Y(t)$ be a rate $\alpha $ Poisson process and $W(t)$ be a continuous time random walk with $W(0)=0$. Let $\tau := \inf \{t>0 : W(t)>Y(t)\}$. Recall that $K=K_\alpha $ is given by $K(t)=\alpha (1+2\alpha ) \mathbb P (t < \tau <\infty )$. We also let $X$ be a random variable with density~$K$.

\begin{lem}
	For all $s\in \mathbb C $ with $\text{Re} (s)<0$ we have
	\begin{enumerate}
		\item 
		\begin{equation}
		\mathbb E \left[ e^{s\tau }\mathds{1}_{\{\tau <\infty\}}\right]= \frac{1+\alpha -s - \sqrt{(1+\alpha -s )^2 -(1+2 \alpha )}}{1+2 \alpha }
		\end{equation}
		\item 
		\begin{equation}
		\mathbb E [e^{sX}]= \intop _0^{\infty } e^{st}K(t)dt=\frac{\alpha }{s} \left( \alpha -s - \sqrt{(1+\alpha -s )^2 -(1+2 \alpha )}\right).
		\end{equation}
	\end{enumerate}
\end{lem}

\begin{proof}
	Let $U_t:=Y(t)-W(t)$. Using that $Y(t)\sim \text{Poisson}(\alpha t)$ and that the number of steps taken by the random walk up to time $t$ is distributed $\text{Poisson}(t)$ we get that for any $\theta >0$  
	\begin{equation}
	\begin{split}
	\mathbb E [\theta ^{U_t}]=\mathbb \mathbb E [\theta ^{Y(t)}]E[\theta ^{W(t)}]&= \Big( e^{-\alpha t} \sum _{k=0}^{\infty } \frac{(\alpha t)^k}{k!} \theta ^{k} \Big)\Big( e^{- t} \sum _{k=0}^{\infty } \frac{ t^k}{k!} \Big( \frac{\theta +\theta ^{-1}}{2} \Big) ^{k} \Big)\\
	&= \exp \left( t\left( \alpha (\theta -1) + \frac{\theta +\theta ^{-1}}{2}-1\right) \right).
	\end{split}
	\end{equation}
	Thus, if we let 
	\[s=s(\theta ):=-\alpha (\theta -1)-\frac{\theta +\theta ^{ -1}}{2}+1\]
	we get that the process $Z_t:=\theta ^{U_t} e^{st}$ is a martingale. Note that for $\theta _0:=1/(1+2 \alpha )$ we have $s(\theta _0)=0$ and that $s(\theta )<0$ for $\theta <\theta _0$. Thus for $0< \theta <\theta _0$ we have $Z_t \to 0$ almost surely and therefore $Z_{t\wedge \tau } \to  \theta ^{-1} e^{s\tau }\mathds{1}_{\{\tau <\infty\}} $. Using the bounded convergence theorem and inverting the function $s(\theta )$ we obtain
	\begin{equation}
	\mathbb E \left[ e^{s\tau }\mathds{1}_{\{\tau <\infty\}}\right]=\theta (s) \cdot \mathbb E (Z_0)=\theta (s)=\frac{1+\alpha -s - \sqrt{(1+\alpha -s )^2 -(1+2 \alpha )}}{1+2 \alpha }.
	\end{equation} 
	We proved the last equality for all $s<0$ however, since both sides are analytic functions in $\{\text{Re} (s)<0 \}$, by the uniqueness theorem the equality holds for all $s$ in this domain.
	
	We turn to prove the second part. Using the definition of $K$ and part (1) we obtain
	\begin{equation}
	\begin{split}
	\mathbb E e^{sX}& = \intop _{0}^{\infty} e^{st} K(t)dt=\alpha (1+2 \alpha ) \intop _0^{\infty} e^{st}\cdot  \mathbb P (t \le \tau < \infty)dt\\
	&=\alpha (1+2 \alpha )  \mathbb E \left[ \intop _0^{\infty} e^{st} \mathds{1}_{\{\tau <\infty\}} \mathds{1}_{\{\tau  \ge t\}} dt \right]=\alpha (1+2 \alpha )  \mathbb E \left[ \mathds{1}_{\{\tau <\infty\}} \intop _0^{\tau } e^{st}  dt \right]\\
	&= \frac{\alpha (1+2 \alpha )}{s}  \left(\mathbb E \left[e^{s\tau }\mathds{1}_{\{\tau <\infty\}} \right] -\mathbb P (\tau <\infty )\right)= \frac{\alpha }{s} \left( \alpha -s - \sqrt{(1+\alpha -s )^2 -(1+2 \alpha )}\right)
	\end{split}
	\end{equation}
	as needed.
\end{proof}

In the next corollary we compute the moments of the random variable $X$.

\begin{cor}\label{cor:moments of K}
	We have that 
	\begin{equation}
	\begin{split}
	&\intop _0^\infty u K(u)du=\frac{1+2 \alpha }{2 \alpha ^2 }=\frac{1}{2}\alpha ^{-2}+O(\alpha ^{-1}),\quad \intop _0^\infty u^2 K(u)du=\frac{(1+\alpha )(1+2 \alpha )}{\alpha ^4}=\alpha ^{-4}+O(\alpha ^{-3}) \\ 
	&\intop _0^\infty u^2 K'(u)=-\frac{1+2 \alpha }{\alpha ^2 }=-\alpha ^{-2}+O(\alpha ^{-1}),\quad \intop _0^\infty uK'(u)du=-1
	\end{split}
	\end{equation}
\end{cor}

\begin{proof}
	the first two identities follows from differentiating at $s=0$ both sides of the equation
	\begin{equation}
	\intop _{0}^{\infty} e^{st} K(t)dt=\frac{\alpha }{s} \left( \alpha -s - \sqrt{(1+\alpha -s )^2 -(1+2 \alpha )}\right).
	\end{equation}
	the last two identities follows from integration by parts.
\end{proof}

Next, we find the behavior of the Fourier transform in different regions of the complex plain. For $s>0$ define
\begin{equation}
\varphi (s):= \mathbb E e^{-s X}=-\frac{\alpha }{s} \left(\alpha +s -\sqrt{(1+\alpha +s)^2-(1+2 \alpha ) }  \right).
\end{equation}
We think of $\varphi $ as a complex function defined on a certain domain $\Omega \subseteq \mathbb C $. See part (1) of Claim~\ref{claim:basic properties of phi} for more details.

Define also 
\begin{equation}
\tilde{\varphi }(s):=-\frac{\alpha }{s} (\alpha -\sqrt{\alpha ^2 +2 s}),
\end{equation}
We'll see that $\varphi $ can be approximated by $\tilde{\varphi}$ for an appropriate range of $s$.

\begin{claim}\label{claim:basic properties of phi}
	The function $\varphi$ satisfies the following properties
	\begin{enumerate}
		\item 
		The functions $\varphi $ and $\tilde{\varphi}$ extend to analytic and bounded functions in 
		\begin{equation}
		\Omega :=\mathbb C \setminus \{ s \in \mathbb R : s<-1 -\alpha +\sqrt{1+2\alpha }\}.
		\end{equation}
		\item 
		We have that $|\varphi(s)|\le C \alpha  /|s|$ for sufficiently large $|s|$ (independently of $0< \alpha <1$).
		\item 
		For any $s \in \Omega $ with $|s|\ge \frac{1}{5} \alpha ^2 $ we have
		\begin{equation}
		|\tilde{\varphi}(s)|, |\varphi (s)|\le  \frac{C \alpha }{\sqrt{|s|}} 
		\end{equation} 
		\item 
		For any  $s \in \Omega $ with  $ |s|\ge \frac{1}{5} \alpha ^2 $ and $|\alpha ^2 +2 s |\ge \frac{1}{2} |s| $ we have 
		\begin{equation}\label{eq:bound on phi and tilde phi}
		|\varphi (s)-\tilde{\varphi }(s)| \le C  \alpha 
		\end{equation}
		\item 
		For any  $s\in \Omega $ with $ |s|\ge \frac{1}{5} \alpha ^2  $ and $|\alpha ^2 +2s|\ge \frac{1}{2} |s| $ we have 
		\begin{equation}
		|1-\varphi(s)|\ge c,\quad  |1-\tilde{\varphi}(s)|\ge c
		\end{equation}
	\end{enumerate}
\end{claim}

\begin{proof}
	The proof of the first part is standard and follows from 
	\begin{equation}\label{eq:square root}
	\sqrt{(1+\alpha +s)^2-(1+2 \alpha )}=\sqrt{(s+1+\alpha +\sqrt{1+2\alpha })}\sqrt{(s+1+\alpha -\sqrt{1+2\alpha })}
	\end{equation}
	where this equality holds when $s>0$ and the right hand side is analytic in $\Omega $ if we let $\sqrt{\ \cdot \ }$ be the standard square root defined on $\mathbb C \setminus \mathbb R _{\le 0}$ (note that in fact $\varphi$ is analytic in a larger domain but we will not use it). We turn to prove that $\tilde{\varphi}$ is analytic in $\Omega $. In the definition of $\tilde{\varphi}$ we let $\sqrt{\ \cdot \ }$ be the standard square root and so $\tilde{\varphi}$ is analytic in 
	\begin{equation}
	\mathbb C \setminus \{ s \in \mathbb R : s<- \frac{1}{2} \alpha ^2 \}
	\end{equation}
	which contains $\Omega$.
	
	The second part follows as the right hand side of \eqref{eq:square root} behaves like $s+O(1)$ as $|s|\to \infty$ uniformly in $0<\alpha <1$.
	
	We turn to prove the third part. Suppose that $\frac{1}{5} \alpha ^2 \le |s|\le C$ for some $C>0$. We have 
	\begin{equation}
	| \varphi (s)| \le \frac{\alpha }{|s|} (\alpha +|s|+|\sqrt{\alpha ^2 +2s+2 \alpha s+s^2 }|)\le \frac{\alpha }{|s|}(\alpha +|s|+C \sqrt{|s| }) \le \frac{C\alpha }{\sqrt{|s|}}.
	\end{equation}
	The bound for any $s\in \Omega $ follows from the last estimate together with the second part of the claim. By the same arguments we get the same bound for $\tilde{\varphi}$.
	
	We turn to prove part (4). When $s \in \Omega $ and $|s|\ge \frac{1}{5}$, by part (3) of the claim we have that $|\tilde{\varphi}(s)|,|\varphi(s)|\le C \alpha $ and \eqref{eq:bound on phi and tilde phi} follows. Let $s \in \Omega $ with  $\frac{1}{5} \alpha ^2 \le |s|\le \frac{1}{5} $ and $|\alpha ^2 +2s |\ge \frac{1}{2} |s|$. In this case we have   
	\begin{equation}
	\varphi (s)=-\frac{\alpha }{s}(\alpha +s -\sqrt{\alpha ^2 +2 s +2\alpha s +s^2}),
	\end{equation}
	where the square root is the standard square root (note that when $\re (s) \le -1-\alpha $, the square root is no longer the standard one). 
	Let $I$ be the line segment connecting $\alpha ^2 +2 s+2 \alpha s+s^2$ and $\alpha ^2 +2 s$. For any $z \in I$ we have 
	\begin{equation}\label{eq:bound in interval}
	|z|\ge |\alpha ^2 +2s|-|2 \alpha s +s^2 | \ge \frac{1}{2}|s|-\frac{1}{4}|s|=\frac{1}{4}|s|. 
	\end{equation}
	Next, we claim that $I \subseteq \mathbb C \setminus \mathbb R _{\le 0}$. indeed, $\left| \im (2 \alpha s +s^2)\right| \le |\im (s)|$ and therefore if $\im (s)\neq 0$ then $I$ is contained entirely inside the upper or lower half planes. If $\im (s)=0$  and $I\nsubseteq \mathbb C \setminus \mathbb R _{\le 0 }$ then, since $s\in \Omega $ we have that $0\in I$ contradicting \eqref{eq:bound in interval}. Thus, using \eqref{eq:bound in interval} again we obtain  
	\begin{equation}
	\left| \sqrt{\alpha ^2 +2s+2 \alpha s +s^2 }-\sqrt{\alpha ^2 +2s} \right| \le \intop _{I } \frac{1}{2\sqrt{|z|}}|dz| \le C \frac{\alpha |s|+|s|^2 }{\sqrt{|s|}} \le C  (\alpha \sqrt{|s|}+|s|^{\frac{3}{2}}). 
	\end{equation}
	We get that 
	\begin{equation}
	|\tilde{\varphi}(s)-\varphi (s)|\le  C  \frac{\alpha }{|s|} (|s|+\alpha \sqrt{|s|}) \le C \left( \alpha + \frac{\alpha ^2 }{\sqrt{|s|}} \right) \le C \alpha.
	\end{equation}
	
	Finally, we prove the last part of the claim. The  inequality in the case $|s|\ge C \alpha ^2  $ follows from part (3) of the claim. Thus suppose that $\frac{1}{5} \alpha ^2 \le |s|\le C \alpha ^2 $. $z=0$ is the unique solution to the equation $(1-\sqrt{1+2z})/z=1$ in $\mathbb C \setminus \mathbb R _{\le 0}$ and therefore if we let $z:=\alpha ^{-2}s$ we have that $|z|\ge \frac{1}{5} $ and  
	\begin{equation}
	\left|\tilde{\varphi} (s)-1\right|=\left| 1-\frac{1-\sqrt{1+2 z}}{z}\right| \ge c. 
	\end{equation}
	The same inequality holds when we replace $\tilde{\varphi}$ with $\varphi$ by part (4) of the claim.
\end{proof}

Throughout this section we'll use the complex contours $\gamma =\gamma ^{(t)}:=\gamma _1+\gamma _2+\gamma _3$ and $\delta :=\delta _1 +\delta _2 +\delta _3$ where
\begin{equation}
\begin{split}
\gamma _1(x):=x-\frac{i}{t},\quad x \in (-\infty ,0],\quad \quad \quad \quad \quad   &\delta _1(x):=x-i \alpha ^2 ,\quad x \in \left( -\infty ,\frac{1}{4}\alpha ^2 \right]\\
\gamma _2(x):=\frac{1}{t}e^{ix},\quad x \in \left[ -\frac{\pi }{2} ,\frac{\pi }{2}\right], \quad \quad \quad \quad \quad &\delta _2(x):=-\frac{1}{4}\alpha ^2 +ix,\quad x \in \left[ -\alpha ^2  ,\alpha ^2 \right] \\
\gamma _3(x):=-x+\frac{i}{t},\quad x \in [0,\infty ),\quad \quad \quad \quad \quad &\delta _3(x):=-x+\alpha ^2 i,\quad x \in \left[ \frac{1}{4}\alpha ^2 ,\infty \right). \\
\end{split}
\end{equation} 
See Figure \ref{critical1Dfigure}.

\subsection{Estimates for $K$}

\begin{figure}[!tbp]
	\centering
	\begin{minipage}[b]{0.49\textwidth}
		\begin{tikzpicture}
		\begin{axis}[ticks=none,axis lines=middle,xmin=-3,xmax=1,ymin=-2,ymax=2,axis equal image]
		\addplot[domain=0:-2.5]{0.5};
		\addplot[dash pattern=on 1pt off 2pt,domain=-2.5:-3]{0.5};
		\addplot[dash pattern=on 1pt off 2pt,domain=-2.5:-3]{-0.5};
		\addplot[domain=0:-2.5]{-0.5};
		\draw [-,decorate,decoration={snake,amplitude=0.7mm,segment length=1mm,post length=1mm}] (-0.25,0) -- (-3.5,0) node [pos=0.66,above]{} ;
		\draw [samples =70, domain=-90:90] plot({0.5*cos(\x)},{0.5*sin(\x)});
		\node[text width=0cm] at (-1.2,0.7)
		{$\gamma ^{(t)} $};
		\draw[->,thick] (-0.7,0.5) -- (-0.7-0.001,0.5) node [pos=0.66,above]{} ;
		\draw[->,thick] (-0.7,-0.5) -- (-0.7+0.001,-0.5) node [pos=0.66,above]{} ;
		\node at (0,0)[circle,fill,inner sep=1.5pt]{};
		\end{axis}
		\end{tikzpicture}
	\end{minipage}
	\begin{minipage}[b]{0.29\textwidth}
		\begin{tikzpicture}
		\begin{axis}[ticks=none,axis lines=middle,xmin=-3,xmax=1,ymin=-2,ymax=2,axis equal image]
		\addplot[domain=-0.5:-2.5]{1};
		\addplot[dash pattern=on 1pt off 2pt,domain=-2.5:-3]{1};
		\addplot[dash pattern=on 1pt off 2pt,domain=-2.5:-3]{-1};
		\addplot[domain=-0.5:-2.5]{-1};
		\draw [-,decorate,decoration={snake,amplitude=0.7mm,segment length=1mm,post length=1mm}] (-1,0) -- (-3.5,0) node [pos=0.66,above]{} ;
		\draw [-] (-0.5,-1) -- (-0.5,1) node [pos=0.66,above]{} ;
		\node[text width=0cm] at (-2,-0.85)
		{$\delta _1 $};
		\node[text width=0cm] at (-0.8,0.4)
		{$\delta _2 $};
		\node[text width=0cm] at (-2,0.8)
		{$\delta _3 $};
		\draw[->,thick] (-1.5,1) -- (-1.5-0.001,1) node [pos=0.66,above]{} ;
		\draw[->,thick] (-1.5,-1) -- (-1.5+0.001,-1) node [pos=0.66,above]{} ;
		\draw[->,thick] (-0.5,0.2) -- (-0.5,0.2+0.001) node [pos=0.66,above]{} ;
		\node at (0,0)[circle,fill,inner sep=1.5pt]{};
		\end{axis}
		\end{tikzpicture}
	\end{minipage}
	\caption{The contour $\gamma ^{(t)} $ and the contour $\delta =\delta _1+\delta _2 +\delta _3$. }
	\label{critical1Dfigure}
\end{figure}
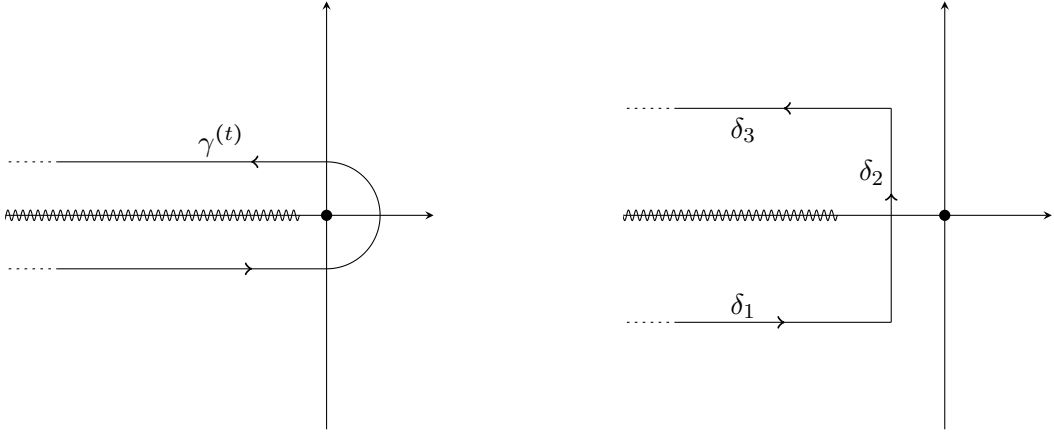

\begin{lem}\label{lem:estimate on K}
	For any $t>0$ we have 
	\begin{equation}
	K(t) \le \frac{C\alpha }{\sqrt{t+1}}e^{-c \alpha ^2 t}.
	\end{equation}
\end{lem}

\begin{proof} 
	Consider the function $g(t):=\int _{0}^t K(x) dx$. It is easy to check that the Laplace transform of $g$ is $\varphi (s)/s$. Thus, using the Laplace inversion theorem and the fact that $\varphi (s)/s$ is integrable on $1+i\mathbb R $ by part (2) of Claim~\ref{claim:basic properties of phi}, we get that 
	\begin{equation}
	\intop _{0}^t K(x) dx=\frac{1}{2 \pi i}\intop _{1+i\mathbb R} \frac{e^{ts}}{s} \varphi (s) ds=\frac{1}{2 \pi i}\intop _{\gamma ^{(1)}} \frac{e^{ts}}{s} \varphi (s) ds,
	\end{equation}
	where in the last equality we deformed the contour using again the fact that the function inside the integral decays sufficiently fast by part (2) of Claim~\ref{claim:basic properties of phi}. 
	
	Using the dominated convergence theorem one can differentiate under the integral sign to obtain 
	\begin{equation}
	K(t) =\frac{1}{2 \pi i}\intop _{\gamma ^{(1)} } e^{ts} \varphi (s) ds.
	\end{equation}
	
	When $t\le \alpha ^{-2}$ we deform the contour again to $\gamma ^{(t)}$ and use part (3) of Claim~\ref{claim:basic properties of phi} to obtain 
	\begin{equation}
	K(t)\le C \alpha \intop _{\gamma ^{(t)} } |s|^{-\frac{1}{2}} e^{t \re (s)} \le \frac{C\alpha }{\sqrt{t}} \intop _{\gamma ^{(1)}} |\omega |^{-\frac{1}{2}}e^{\re (\omega )} \le \frac{C\alpha }{\sqrt{t}} \le \frac{C\alpha }{\sqrt{t}} e^{-c \alpha ^2 t },
	\end{equation}
	where in the second inequality we changed variables by $\omega =ts$.
	
	Suppose next that $t\ge \alpha ^{-2}$. In this case we use the contour $\delta $ instead of $\gamma$. By part (3) of Claim~\ref{claim:basic properties of phi} and using the fact that $\re (s)<-c\alpha ^2 $ on $\delta $ we get 
	\begin{equation}
	\begin{split}
	K(t) &\le C \alpha \intop _{\delta }|s|^{-\frac{1}{2}}e^{t\re (s) }\le C\alpha  e^{-c \alpha ^2 t} \intop _{\delta }|s|^{-\frac{1}{2}}e^{\frac{1}{2}t\re (s) }\le \frac{C\alpha}{\sqrt{t}}  e^{-c \alpha ^2 t} \intop _{\delta '}|\omega |^{-\frac{1}{2}}e^{\frac{1}{2}\re (\omega ) }\\
	&= \frac{C\alpha}{\sqrt{t}}  e^{-c \alpha ^2 t} \intop _{\delta _1'+\delta _3' }|\omega |^{-\frac{1}{2}}e^{\frac{1}{2}\re (\omega ) } + \frac{C\alpha}{\sqrt{t}}  e^{-c \alpha ^2 t} \intop _{\delta _2' }|\omega |^{-\frac{1}{2}} \le \frac{C\alpha}{\sqrt{t}}  e^{-c \alpha ^2 t} +\sqrt{t}\alpha ^3  e^{-c \alpha ^2 t} \le \frac{C \alpha }{\sqrt{t}} e^{-c \alpha ^2 t},
	\end{split}
	\end{equation} 
	where $\delta '=\delta _1'+\delta _2'+\delta _3'$ and where $\delta ', \delta _1', \delta _2 ' , \delta _3 '$ are the image under the change of variables $\omega = t s$ of the contours $\delta ,\delta _1 , \delta _2 , \delta _3 $ respectively. Note that in the last inequality we changed the value of $c$ slightly and used that $t \ge \alpha ^{-2}$. 
\end{proof}

\subsection{Estimates for the renewal process}

Define the functions
\begin{equation}
K^*(t):=\sum _{n=1}^{\infty } K^{*n}(t),\quad \tilde{K}=K^*-\frac{2\alpha ^2}{1+2 \alpha }, 
\end{equation}
where $K^{*n}$ is the $n$-fold convolution of $K$ with itself. The function $K^*$ is the density of the renewal process with inter-arrival times that have density $K$.

\begin{lem}\label{lem:estimat for K tilde}
	For any $t>0$ we have 
	\begin{equation}\label{eq:what we need for K tilde}
	|\tilde{K}(t)| \le \frac{C\alpha }{\sqrt{t+1}} e^{-c \alpha ^2 t }.
	\end{equation}
\end{lem}

\begin{proof}
	By the same arguments as in the proof of Lemma~\ref{lem:estimate on K} we have for any $n\ge 1$
	\begin{equation}
	K^{*n}(t)=\frac{1}{2 \pi i } \intop _{\gamma ^{(1)}} e^{ts} \varphi(s)^n ds.
	\end{equation}
	Thus,
	\begin{equation}\label{eq:integral form of K*}
	K^*(t)=\sum _{n=1}^{\infty } K^{*n} (t) =\frac{1}{2 \pi i } \intop _{\gamma ^{(1)}} e^{ts} \frac{\varphi (s)}{1-\varphi(s )} ds.
	\end{equation}
	Indeed, by part (3) of Claim~\ref{claim:basic properties of phi} we have that $|\varphi(s)|\le C \alpha \le 1/2$ on the contour $\gamma ^{(1)}$ and therefore the sum $\sum _{n=1}^{\infty } \varphi (s)^n$ converges absolutely and we can change the order of summation and integration.
	Note that the function we integrate in \eqref{eq:integral form of K*} is meromorphic in $\Omega $ with a single pole at $s=0$. Indeed, it is easy to see that $\varphi (s)=1$ if and only if $s=0$. 
	
	Suppose that $t\le \alpha ^{-2}$. Deforming $\gamma ^{(1)}$ to $\gamma ^{(t)}$ and using part (5) of Claim~\ref{claim:basic properties of phi} we obtain
	\begin{equation}\label{eq:K tilde when t is small}
	|\tilde{ K}(t)| \le 2 \alpha ^2 + K^*(t) \le 2\alpha ^2 + C\intop _{\gamma ^{(t)}}e^{t\re (s)} |\varphi(s)| \le 2 \alpha ^2 +\frac{C \alpha }{\sqrt{t}} e^{-c \alpha ^2 t} \le \frac{C \alpha }{\sqrt{t}} e^{-c \alpha ^2 t},
	\end{equation}
	where in the third inequality we repeat the same arguments as in the proof of Lemma~\ref{lem:estimate on K}.
	
	Suppose next that $t \ge \alpha ^{-2}$. In this case we use the contour $\delta $. In order to deform $\gamma ^{(1)}$ to $\delta $ we need to compute the residue at $s=0$. We have 
	\begin{equation}
	\res \left( \frac{e^{ts} \varphi(s) }{1-\varphi(s)} , 0\right)=\varphi(0) \lim _{s\to 0} \frac{s}{\varphi(0)-\varphi(s)}= -\frac{1}{\varphi'(0)}=\frac{2\alpha ^2 }{1+ 2\alpha }.
	\end{equation}
	Thus
	\begin{equation}
	K^*(t)= \frac{2 \alpha ^2 }{1+ 2 \alpha }+\frac{1}{2 \pi i } \intop _{\delta } e^{ts} \frac{\varphi (s)}{1-\varphi(s )} ds.
	\end{equation}
	Therefore, using part (5) of Claim~\ref{claim:basic properties of phi} and the same bounds as in the proof of Lemma~\ref{lem:estimate on K} we get
	\begin{equation}\label{eq:K tilde when t is large}
	|\tilde{ K}(t)| \le C \intop _{\delta } e^{t\re (s)} |\varphi(s)| \le \frac{C \alpha }{\sqrt{t}} e^{-c \alpha ^2 t}.
	\end{equation}
	This finishes the proof of \eqref{eq:what we need for K tilde} for all values of $t$.
\end{proof}

The following is an immediate corollary
\begin{cor}\label{cor:bound on Kstar}
	For any $t>0$ we have that 
	\begin{equation}
	K^*(t)\le \frac{C \alpha }{\sqrt{t+1}}+2 \alpha ^2 
	\end{equation}
\end{cor}

Repeating exactly the same arguments as in the proof of Lemma~\ref{lem:estimate on K} and Lemma~\ref{lem:estimat for K tilde} we obtain

\begin{cor}\label{cor:bound on K'}
	We have 
	\begin{equation}
	| K'(t) | \le \frac{C\alpha }{t^{\frac{3}{2}}}e^{-c \alpha ^2 t}, \quad |\tilde{ K}'(t)|=|(K^*)'(t)| \le \frac{C\alpha }{t^{\frac{3}{2}}}e^{-c \alpha ^2 t}
	\end{equation}
\end{cor}

\begin{remark}
	We do not prove it but the functions $K=K^{(\alpha )}$ and $\tilde{ K}=\tilde{K}^{(\alpha )}$ have a scaling limit of the form
	\begin{equation}
	\alpha ^{-2}K( \alpha ^{-2} x) \to k_1(x), \quad \alpha ^{-2}\tilde{K}( \alpha ^{-2} x) \to k_2(x), \quad \alpha \to 0.
	\end{equation}
	Moreover the Laplace transform of $k_1$ is $\tilde{\varphi}(\alpha ^2x)$ which is independent of $\alpha $. The following lemma shows that, due to a fortuitous cancellation, the limiting functions $k_1$ and $k_2$ are equal. 
\end{remark}

\begin{lem}\label{lem:K tilde minus K}
	For any $t>0$ we have  
	\begin{equation}
	|\tilde{K}(t)-K(t)| \le \frac{C \alpha }{t+1}. 
	\end{equation}
\end{lem}

\begin{proof}
	By the proof of Lemma~\ref{lem:estimate on K} and Lemma~\ref{lem:estimat for K tilde} we have  
	\begin{equation}
	K(t) =\frac{1}{2 \pi i}\intop _{\delta } e^{ts} \varphi (s) ds,\quad \tilde{ K}(t)=\frac{1}{2 \pi i } \intop _{\delta  } e^{ts} \frac{\varphi (s)}{1-\varphi(s )} ds.
	\end{equation}
	In order to prove the lemma we show first that one can replace $\varphi$ with $\tilde{\varphi }$ in the last two integrals. Then we prove that 
	\begin{equation}\label{eq:identity of ontegrals}
	\intop _{\delta } e^{ts} \tilde{\varphi} (s) ds= \intop _{\delta  } e^{ts} \frac{\tilde{\varphi } (s)}{1-\tilde{\varphi }(s )} ds.
	\end{equation}
	
	By part (4) of Claim~\ref{claim:basic properties of phi} we have
	\begin{equation}
	\begin{split}
	\left| \intop _{\delta } e^{ts} \varphi (s) ds-\intop _{\delta } e^{ts} \tilde{\varphi } (s) ds\right| &\le \intop _{\delta } e^{t\re (s)} |\varphi (s)-\tilde{\varphi}(s)| \le C \alpha \intop _{\delta } e^{t\re (s)}\\
	&\le \frac{C\alpha }{t } e^{-c \alpha ^2 t }\intop _{\delta _2'}1+ \frac{C\alpha }{t }  \intop _{\delta _1'+\delta _3 '} e^{ \re (\omega )}\le C \alpha ^3  e^{-c \alpha ^2 t}+\frac{C\alpha }{t} \le \frac{C\alpha }{t}. 
	\end{split}
	\end{equation}
	Similarly we have
	\begin{equation}
	\left| \intop _{\delta  } e^{ts} \frac{\varphi (s)}{1-\varphi(s )} -\intop _{\delta  } e^{ts} \frac{\tilde{\varphi } (s)}{1-\tilde{\varphi }(s )} ds\right|\le \intop _{\delta }e^{t \re (s)} \frac{|\varphi(s)-\tilde{ \varphi}(s)|}{|1-\tilde{\varphi }(s)||1-\varphi(s)|} \le C\alpha \intop _{\delta } e^{t \re (s)}\le \frac{C \alpha }{t}.
	\end{equation}
	
	Next, we prove the identity \eqref{eq:identity of ontegrals}.
	\begin{equation}
	\begin{split}
	\intop _{\delta  } e^{ts} \frac{\tilde{\varphi } (s)}{1-\tilde{\varphi }(s )} ds- \intop _{\delta } e^{ts} \tilde{\varphi} (s) ds =\intop _{\delta  } e^{ts} \frac{\tilde{\varphi } (s)^2}{1-\tilde{\varphi }(s )} ds&= \intop _{\delta } e^{ts} \frac{\alpha ^2 (\alpha - \sqrt{\alpha ^2 +2s} )^2}{s^2 + \alpha s( \alpha - \sqrt{\alpha ^2 +2s} )}\\
	=4 \alpha ^2\intop _{i \mathbb R }e^{t\frac{y^2 -\alpha ^2}{2}} \frac{ (\alpha -y  )^2y}{(y^2 -\alpha ^2)^2+2\alpha (y^2 -\alpha ^2 )(\alpha -y)}dy&=4 \alpha ^2e^{-\frac{t \alpha ^2 }{2}}\intop _{i \mathbb R }e^{\frac{ty^2}{2}} \frac{ y}{(y+\alpha )^2-2\alpha (y+\alpha )}dy\\
	&=4 \alpha ^2e^{-\frac{t \alpha ^2 }{2}}\intop _{i \mathbb R }e^{\frac{ty^2}{2}} \frac{ y}{y^2 -\alpha ^2}dy=0,
	\end{split}
	\end{equation}
	where in the third equality we changed variables by $y=\sqrt{\alpha ^2 +2s}$ and also deformed the contour to the imaginary axis (this is possible as the function is analytic except for two poles at $\alpha $ and $-\alpha $). The last equality holds as we integrate an odd function.
\end{proof}

The last lemma gives a tight bound when $t$ is large. In the following corollary we improve the estimate when $t$ is small.

\begin{cor}\label{cor:K tilde minus K}
	For any $t>0$ we have  
	\begin{equation}
	|\tilde{K}(t)-K(t)| \le C \min \{\alpha ^2 ,\frac{ \alpha }{t+1} \}. 
	\end{equation}
\end{cor}

\begin{proof}	
	By Lemma~\ref{lem:estimate on K} we have 
	\begin{equation}\label{eq:convolution of 2}
	\begin{split}
	K*K(t)= \intop _{0}^t K(s)K(t-s) ds &= 2\intop _{0}^{t/2} K(s)K(t-s)ds   \\
	&\le C \intop _{0}^{t/2} \frac{\alpha ^2 }{\sqrt{s} \sqrt{t-s}}dt 
	\le \frac{C\alpha ^2 }{\sqrt{t}} \intop _0 ^{t/2} \frac{1}{\sqrt{s}}ds \le C \alpha ^{2} .
	\end{split}
	\end{equation}	
	Thus, letting $X_j$ be an i.i.d. sequence with density $K$, if $t \le \alpha ^{-2}$ we have
	\begin{equation}
	\begin{split}\label{eq:convolution of 3}
	\sum _{j=3}^{\infty} K^{*j}(t) &= \sum _{j=1}^{\infty} \intop _0 ^t K^{*j}(s)K^{*2}(t-s)ds \le C \alpha ^2 \sum _{j=1}^{\infty} \intop _0 ^t K^{*j}(s)ds \\
	& \le C \alpha ^2 \sum _{j=1}^{\infty} \mathbb P \left( \sum _{i=1}^j X_i \le t \right)  \le C \alpha ^2 \sum _{j=1}^{\infty} \mathbb P \left( \forall i \le j , \  X_i \le  \alpha ^{-2}  \right) \le C \alpha ^2 \sum _{j=1}^{\infty} e^{-cj} \le C \alpha ^2 .
	\end{split}
	\end{equation}
	By \eqref{eq:convolution of 2} and \eqref{eq:convolution of 3} we have for $t \le \alpha ^{-2}$
	\begin{equation}
	|\tilde{ K}(t)-K(t) | \le C \alpha ^2+K*K(t) +\sum _{j=3}^{\infty } K^{*j}(t) \le C \alpha ^2 .
	\end{equation}
	The corollary easily follows from the last bound and Lemma~\ref{lem:K tilde minus K}.
\end{proof}

\subsection{The double integral in the drift term}

In Theorem \ref{thm:double int:main}, we show that the drift term of the aggregate speed can be expressed in terms of a certain integral involving the functions $K^*$ and $J$. In this subsection we show how the Fourier transform estimates we developed in this section can be used in order to estimate this integral.  

\begin{lem}\label{lem:integral computation}
	Let $  \alpha ^{-2} \log ^{2}(1/\alpha ) \le T  \le \alpha ^{-\frac{5}{2}}$ and let 
	\begin{equation}
	I_1:=\intop _0 ^ T \intop _0 ^u J_{s,u}K^* (u-s) ds \, du \quad I_2:=\intop _0^{T }  \intop _0 ^x   \intop _0^u J_{s,u} K^*(x-u)K^* (x-s) ds \, du \,  dx 
	.
	\end{equation}
	We have that $I_1+I_2=2  +O(\sqrt{\alpha })$.
\end{lem}

We start by showing that the function $J$ can be expressed in terms of $K$
\begin{claim}\label{claim:formula for J}
	We have that
	\begin{equation}
	J_{s,u}=-\frac{1}{\alpha ^2 (1+2 \alpha )^2}\intop _s^u K'(x) K(u-x)dx -\frac{2  }{(1+2 \alpha )^2}K(u).
	\end{equation}
\end{claim}

\begin{proof}
	Let $\tau '$ be an independent copy of $\tau $ and recall the definition of $J$ in \eqref{eq:def:J:expected val}. As $\tau _s \ge \tau $ we have 
	\begin{equation}
	\begin{split}
	J_{s,u}&=\mathbb P (u \le \tau _s <\infty)-\mathbb P (u\le \tau < \infty )\\
	&=\mathbb P ( \tau <u , \ u \le \tau _s <\infty  )-\mathbb P ( u \le \tau <\infty , \ \tau _s=\infty )\\
	&= \mathbb P ( s< \tau <u , \ u \le \tau +\tau '<\infty  )-\mathbb P ( u \le \tau <\infty , \ \tau '=\infty )\\
	&= -\frac{1}{\alpha ^2 (1+2 \alpha )^2}\intop _s^u K'(x) K(u-x)dx -\frac{2  }{(1+2 \alpha )^2}K(u)
	\end{split}
	\end{equation}
	where in the last equality we used the definition of $K$ and the fact that $-\alpha ^{-1}K'(x)$ is the density of $\tau \mathds 1 \{ \tau <\infty  \}$.
\end{proof}

\begin{cor}\label{cor:bound on deterministic J}
    For all $u>s>0$ we have 
    \begin{equation}
        |J_{s,u}|\le \frac{Ce^{-c \alpha ^2 u}}{\sqrt{u+1}\sqrt{s+1}}.
    \end{equation}
\end{cor}

For the proof of the corollary we will need the following simple bound.

\begin{claim}\label{claim:some claim}
    For all $s<u$ we have 
    \begin{equation}
         \intop _s^u \frac{1}{(x+1)^{\frac{3}{2}} (u-x+1)^{\frac{1}{2}}}dx \le \frac{C}{ \sqrt{s+1}\sqrt{u+1} }.
    \end{equation}
\end{claim}

\begin{proof}
    Suppose first that $u>2s$. In this case
\begin{equation}
    I_{s,u}\le \intop _s^{u/2} \frac{C}{(x+1)^{\frac{3}{2}}(u+1)^{\frac{1}{2}}}dx +\intop _{u/2}^u \frac{C}{(u+1)^{\frac{3}{2}}(u-x)^{\frac{1}{2}}}dx\le \frac{C}{\sqrt{s+1}\sqrt{u+1}}+\frac{C}{u+1} \le \frac{C}{\sqrt{s+1}\sqrt{u+1}}.
\end{equation}
In the case $u\le 2s$ we have 
\begin{equation}
    I_{s,u}\le  \intop _{s}^u \frac{1}{ (s+1)^{\frac{3}{2}}(u-x)^{\frac{1}{2}}}dx \le \frac{\sqrt{u}}{(s+1)^{\frac{3}{2}}} \le  \frac{C}{\sqrt{s+1}\sqrt{u+1}},
\end{equation}
as needed.
\end{proof}

\begin{proof}[Proof of Corollary~\ref{cor:bound on deterministic J}]
By Claim~\ref{claim:formula for J}, Lemma~\ref{lem:estimate on K} and Corollary~\ref{cor:bound on K'} we have 
\begin{equation}\label{eq:some bound on |J|}
\begin{split}
    |J_{s,u}| \le \alpha ^{-2} \intop _s^u |K'(x)|K(u-x)& dx +2 K(u) 
    \le  \intop _s^u \frac{Ce^{-c \alpha ^2 u }}{(x+1)^{\frac{3}{2}}(u-x+1)^{\frac{1}{2}}}dx +\frac{C\alpha e^{-c \alpha ^2 u}}{ \sqrt{u+1} } \\
    &\le \frac{Ce^{-c \alpha ^2 u}}{\sqrt{u+1}\sqrt{s+1}}+\frac{C\alpha e^{-c \alpha ^2 u}}{ \sqrt{u+1} } \le \frac{Ce^{-c \alpha ^2 u}}{\sqrt{u+1}\sqrt{s+1}},
\end{split}
\end{equation}
where the third inequality is by Claim~\ref{claim:some claim} and the last inequality follows by slightly changing the values of $c$ and $C$.
\end{proof}

\begin{proof}[Proof of Lemma~\ref{lem:integral computation}]
	First, note that by Corollary~\ref{cor:bound on deterministic J} one can change the upper limit $T$ in the definition of $I_1$ to $\infty $ without changing the value of $I_1$ by more that $O(\alpha ^{10})$. Using this observation and the definition of $\tilde{K}$ we obtain
	\begin{equation}\label{eq:I_1}
	I_1= O(\alpha ^{10}) + \intop _0^{\infty } \intop _0^u J_{s,u} \tilde{K}(u-s) ds \, du + \frac{2\alpha ^2 }{1+2 \alpha }\intop _0^{\infty } \intop _0^u J_{s,u} ds \, du.
	\end{equation}
	Denote the last two terms by $I_3$ and $I_4 $ respectively. By Claim~\ref{claim:formula for J} we have
	\begin{equation}\label{eq:I_4}
	\begin{split}
	I_4&=-\frac{2}{(1+2 \alpha )^3 } \intop _0^\infty   \intop _0^u \intop _0^x K'(x)K(u-x)   ds \, dx \, du -\frac{4\alpha ^2 }{(1+2 \alpha )^3 } \intop _0^{\infty } \intop _0^u K(u) ds \, du \\ 
	&=-\frac{2}{(1+2 \alpha )^3 } \intop _0^\infty   \intop _0^u x K'(x)K(u-x)     dx \, du -\frac{4\alpha ^2 }{(1+2 \alpha )^3 } \intop _0^\infty  u K(u)  du\\
	&=-\frac{2}{(1+2 \alpha )^3 } \intop _0^\infty   K(b) db \intop _0^{\infty } a K'(a)     da  -2 +O(\alpha ) =O(\alpha ),
	\end{split}
	\end{equation}
	where in the third equlity we changed variables by $a=x$ and $b=u-x$ and used Corollary~\ref{cor:moments of K}. In the last equality we used Corollary~\ref{cor:moments of K} again.
	
	We turn to estimate $I_3$. Let $F(u):=-1+\int _0^u K(x)dx$ and note that $F$ decays exponentially fast. Using Claim~\ref{claim:formula for J} we get
	\begin{equation}\label{eq:I_3}
	\begin{split}
	\alpha ^2 &(1+2\alpha )^2 I_3= -\intop _0^{\infty } \intop _0^u \intop _0^x K'(x) K(u-x) \tilde{K} (u-s)  -2\alpha ^2 \intop _0^{\infty } \intop _0^u K(u)\tilde{K}(u-s)\\
	&=-\intop _0^{\infty } \intop _0^c \intop _0^\infty  K'(a+b) K(c-b) \tilde{K}(c)  da \, db \, dc -2\alpha ^2 \intop _0^{\infty } \intop _0^\infty  K(a+b)\tilde{K}(b)da \, db \\
	&=\intop _0^{\infty } \tilde{K}(c)  \intop _0^c  K(b) K(c-b)    +2\alpha ^2 \intop _0^{\infty }   F(b)\tilde{K}(b)   =\intop _0^{\infty }   (K*K) \cdot \tilde{K}   +2\alpha ^2   F\cdot \tilde{K},
	\end{split}
	\end{equation}
	where in the second equality we made the change of variablles $a=s,\, b=x-s,\, c=u-s$. Note that the right hand side of \eqref{eq:I_3} is of order $O( \alpha ^2 )$ by Lemma~\ref{lem:estimat for K tilde}, equation \eqref{eq:convolution of 2} and the definition of $F$. Thus, 
	\[I_3=O(\alpha ) +  \intop _0^{\infty }   \alpha ^{-2}(K*K) \cdot \tilde{K}   +2 F\cdot \tilde{K}.\]
	
	Next, we estimate the integral $I_2$. Using the definition of $\tilde{K}$ we obtain
	\begin{equation}\label{eq:I_2}
	\begin{split}
	I_2&= \intop _0^{T }  \intop _0 ^x   \intop _0^u J_{s,u} \left( \tilde{K}(x-u)+\frac{2 \alpha ^2 }{1+2 \alpha } \right)\left( \tilde{K}(x-s)+\frac{2 \alpha ^2 }{1+2 \alpha } \right)\\
	&= \intop _0^{\infty }  \intop _0 ^x   \intop _0^u J_{s,u} \tilde{K}(x-u)\tilde{K} (x-s)+\frac{2 \alpha ^2 }{1+2 \alpha }\intop _0^{\infty }  \intop _0 ^x   \intop _0^u J_{s,u}  \tilde{K}(x-u) \\
	&+\frac{2 \alpha ^2 }{1+2 \alpha }\intop _0^{\infty }  \intop _0 ^x   \intop _0^u J_{s,u} \tilde{K}(x-s) +\frac{4\alpha ^4 }{(1+2 \alpha )^2}\intop _0^{T }  \intop _0 ^x   \intop _0^u J_{s,u}+O(\alpha ^{10}),
	\end{split}
	\end{equation}
	where in the second inequality we also changed the upper limit $T$ to $\infty $ in the first three integrals using the same arguments as in \eqref{eq:I_1}. Note that we could not change the upper limit in the fourth integral as the corresponding infinite integral doesn't converge. 
	We denote the four terms on the right hand side of \eqref{eq:I_2} by $I_5, I_6, I_7$ and $I_8$ respectively. 
	
	We start with $I_5$. Using Claim~\ref{claim:formula for J} we get 
	\begin{equation}\label{eq:I_5}
	\begin{split}
	&\alpha ^2 (1+2 \alpha )^2 I_5 =  \alpha ^2 (1+2 \alpha )^2\intop _0^{\infty }  \intop _0 ^x   \intop _0^u J_{s,u} \tilde{K}(x-u)\tilde{K} (x-s)\\
	&=-\intop _0^\infty \intop _0^x \intop _0^u \intop _0^y  K'(y)K(u-y)\tilde{K}(x-u)\tilde{K}(x-s)-2\alpha ^2 \intop _0^{\infty } \intop _0^x  \intop _0^u K(u)\tilde{K}(x-u)\tilde{K}(x-s)\\
	&=-\intop _0^\infty \intop _0^d  \intop _0^c  \intop _0^\infty   K'(a+b)K(c-b)\tilde{K}(d-c)\tilde{K}(d)-2\alpha ^2 \intop _0^{\infty } \intop _0^c \intop _0^\infty  K(a+b)\tilde{K}(c-b)\tilde{K}(c)\\
	&=\intop _0^\infty \tilde{K}(d) \intop _0^d  \tilde{K}(d-c) \intop _0^c   K(b)K(c-b)+2\alpha ^2 \intop _0^{\infty }\tilde{K}(c) \intop _0^c   F(b)\tilde{K}(c-b)\\
	&=\intop _0^\infty \tilde{K}\cdot  (\tilde{K}*K*K) +2\alpha ^2 \tilde{K} \cdot (F*\tilde{K}),\\
	\end{split}
	\end{equation}
	where third equality we applied the change of variables $a=s,\, b=y-s,\, c=u-s,\, d=x-s$ to the first integral and the change of variables $a=s,\, b=u-s,\, c=x-s$ to the second integral. As before, the right hand side of the last equation is of order $O(\alpha ^2 )$ and therefore 
	\[I_5=O(\alpha ) +\intop _0^\infty \alpha ^{-2} \tilde{K}\cdot  (\tilde{K}*K*K) +2 \tilde{K} \cdot (F*\tilde{K}).\]
	
	We turn to estimat $I _6$. Using Claim~\ref{claim:formula for J} again we get
	\begin{equation}
	\begin{split}
	&\frac{(1+2 \alpha )^3 }{2}I_6=\alpha ^2 (1+2 \alpha )^2 \intop _0^{\infty }  \intop _0 ^x   \intop _0^u J_{s,u} \tilde{K}(x-u)\\
	&=-\intop _0^\infty \intop _0^x \intop _0^u \intop _0^y  K'(y)K(u-y)\tilde{K}(x-u)-2\alpha ^2 \intop _0^{\infty } \intop _0^x  \intop _0^u K(u)\tilde{K}(x-u)\\
	&=-\intop _0^\infty \intop _0^\infty  \intop _0^\infty  \intop _0^\infty   K'(a+b)K(c)\tilde{K}(d)-2\alpha ^2 \intop _0^{\infty } \intop _0^\infty \intop _0^\infty  K(a+b)\tilde{K}(c)\\
	&=\intop _0^\infty \intop _0^\infty  \intop _0^\infty   K(b)K(c)\tilde{K}(d)+2\alpha ^2 \intop _0^{\infty } \intop _0^\infty   F(b)\tilde{K}(c)=\intop _0^\infty  \tilde{K}-\intop _0^\infty  \tilde{K} +O(\alpha )=O(\alpha ), \\
	\end{split}
	\end{equation}
	where in the third equality we changed variables by $a=s,\, b=y-s,\, c=u-y, \, d=x-u$ and in the fifth equality we used that, by Corollary~\ref{cor:moments of K} we have 
	\begin{equation}
	\intop _0^\infty F(x) dx=-\intop _0^\infty \intop _x^\infty K(y)dy dx=-\intop _0^ \infty yK(y)dy = -\frac{1}{2} \alpha ^{-2} +O(\alpha ^{-1}).
	\end{equation}
	
	Next, we estimate $I_7$. We have
	\begin{equation}
	\begin{split}
	\frac{(1+2 \alpha )^3 }{2}I_7&=\alpha ^2 (1+2 \alpha )^2 \intop _0^{\infty }  \intop _0 ^x   \intop _0^u J_{s,u} \tilde{K}(x-s)\\
	&=-\intop _0^\infty \intop _0^x \intop _0^u \intop _0^y  K'(y)K(u-y)\tilde{K}(x-s)-2\alpha ^2 \intop _0^{\infty } \intop _0^x  \intop _0^u K(u)\tilde{K}(x-s)\\
	&=-\intop _0^\infty \intop _0^d  \intop _0^c  \intop _0^\infty    K'(a+b)K(c-b)\tilde{K}(d) -2\alpha ^2 \intop _0^{\infty } \intop _0^c \intop _0^\infty  K(a+b)\tilde{K}(c)\\
	&=\intop _0^\infty \tilde{K}(d) \intop _0^d   \intop _0^c    K(b)K(c-b)+2\alpha ^2 \intop _0^{\infty }\tilde{K}(c) \intop _0^c   F(b)\\
	&=\intop _0^\infty \tilde{K}\cdot (K*K*\mathds{1})+2 \alpha ^2 \tilde{K}\cdot  (F*\mathds{1}),
	\end{split}
	\end{equation}
	where in the third equality we used the same change of variables as before. The right hand side of the last equation is of order $O(1)$ and therefore
	\[I _7 =O(\alpha ) + \intop _0^\infty 2\tilde{K}\cdot (K*K*\mathds{1})+4 \alpha ^2 \tilde{K}\cdot  (F*\mathds{1})\] 
	
	Finally, we estimate $I_8$. We have
	\begin{equation}
	\frac{(1+2 \alpha )^2}{4 \alpha ^4} I_8=\intop _0^{T }  \intop _0 ^x   \intop _0^u J_{s,u}=  \intop _0 ^{\infty }   \intop _0^u \intop _{u}^{T} J_{s,u}=T\intop _0 ^{\infty }   \intop _0^u  J_{s,u}-\intop _0 ^{\infty }   \intop _0^u  uJ_{s,u}.
	\end{equation}
	We denote the last two terms by $I_9$ and $I_{10}$ respectively. We have that $I_9=O( \alpha ^{-\frac{7}{2}})$ by the assumption on $T$ and by the bound on $I_4$ in \eqref{eq:I_4}. 
	
	We turn to estimate $I_{10}$. We have 
	\begin{equation}
	\begin{split}
	\alpha ^2& (1+2 \alpha )^2 I _{10} =-\alpha ^2 (1+2 \alpha )^2 \intop _0^\infty \intop _0 ^u  uJ_{s,u}=\intop _0^\infty \intop _0^u \intop _0^x uK'(x)K(u-x) +2\alpha ^2 \intop _0^{\infty} \intop _0^u u K(u)\\
	&=\intop _0^\infty \intop _0^u  ux K'(x)K(u-x) +2\alpha ^2 \intop _0^{\infty}  u^2 K(u) =\intop _0^\infty \intop _0^\infty  a(a+b) K'(a)K(b) +2 \alpha ^{-2}+O(\alpha ^{-1})\\
	&= \intop _0^\infty  a^2 K'(a)da \intop _0^\infty  K(b)db + \intop _0^\infty  a K'(a)da \intop _0^\infty  bK(b)db +2\alpha ^{-2} +O(\alpha ^{-1})=\frac{1}{2} \alpha ^{-2}+O(\alpha ^{-1}),
	\end{split}
	\end{equation}
	where in the fourth and last equality we used Corollary~\ref{cor:moments of K}. Thus, $I_{10}=\frac{1}{2} \alpha ^{-4} +O (\alpha ^{-3})$ and therefore $I _8 =2+O(\sqrt{\alpha })$.
	
	Adding all the estimates for $I_3,I_4,I_5,I_7$ and $I_8$ we obtain 
	
	\begin{equation}
	\begin{split}
	I_1+I_2&= I_3+I_4+I_5+I_6+I_7+I_8=2 +O(\sqrt{\alpha }) +\intop _0^{\infty }   \alpha ^{-2} (K*K) \cdot \tilde{K}   +2 F\cdot \tilde{K}  \\
	&+\intop _0^\infty \alpha ^{-2}\tilde{K}\cdot  (\tilde{K}*K*K) +2 \tilde{K} \cdot (F*\tilde{K})+ \intop _0^\infty 2\tilde{K}\cdot (K*K*\mathds{1})+4 \alpha ^2 \tilde{K}\cdot (F*\mathds{1}). \\
	\end{split}
	\end{equation}
	Next, using that $\tilde{K}=K^*-2 \alpha ^2+O(\alpha ^3)$ we see that the the third integral is completely cancled by the $\tilde{K}$ inside the brackets of the second integral. Thus,
	\begin{equation}\label{eq:final computation in the integral}
	\begin{split}
	I_1 +I_2&=2+ O(\sqrt{\alpha } )+\alpha ^{-2}\intop _0^\infty \tilde{K}\cdot \left( K^* *K*K+K*K+2 \alpha ^2 (F*K^*+F) \right)\\
	&=2+O(\sqrt{\alpha }) +\alpha ^{-2}\intop _0^\infty \tilde{K}\cdot \left( K*K^* + 2\alpha ^2 (\mathds{1}*K*K^* -\mathds{1}*K^*+\mathds{1}*K-1 ) \right)\\
	&=2 +O(\sqrt{\alpha }) +\alpha ^{-2} \intop _0^\infty \tilde{K}\cdot \left( K*K^* - 2\alpha ^2 ) \right)=2 +O(\sqrt{\alpha }) +\alpha ^{-2} \intop _0^\infty \tilde{K}(\tilde{K}-K),
	\end{split}
	\end{equation}
	where in the secnd equality we udsed that $F=K*\mathds{1}-1$, in the third equality we used that $K^*=K+K*K^*$ and in the last equality we used that $\tilde{K}=K^*-2 \alpha ^2+O(\alpha ^3)$.
	
	Next, we estimate the integral in the right hand side of \eqref{eq:final computation in the integral}. We have 
	\begin{equation}
	\intop _0^\infty \tilde{K}(\tilde{K}-K) \le C\alpha ^2  \intop _0^ {\alpha ^{-1}}  \tilde{ K}(t) dt +C \alpha \intop _{\alpha ^{-1}} ^{\infty } \frac{\tilde{ K} (t)}{t} dt  \le C \alpha ^3 \intop _0^{\alpha ^{-1}} \frac{1}{\sqrt{t}} dt + C \alpha ^2 \intop _{\alpha ^{-1}}^{\infty } \frac{1}{t^{\frac{3}{2}}} dt \le C \alpha ^{\frac{5}{2}},
	\end{equation} 
	where in the first inequality we used Corollary~\ref{cor:K tilde minus K} and in the second inequality we used Lemma~\ref{lem:estimat for K tilde}. Substituting the last bound into \eqref{eq:final computation in the integral} finishes the proof of the lemma.
\end{proof}

 \section{Martingale concentration lemmas}\label{subsec:app:mgconcen}

	We provide the details of the deferred proofs of martingale concentration lemmas. We begin with establishing Lemma \ref{lem:concen of int:conti:forwardtime}. It will be clear from the proof that Corollary \ref{lem:concentrationofint:continuity} follows from the same proof as well.
	
	\begin{proof}[Proof of Lemma \ref{lem:concen of int:conti:forwardtime}]
		Since we have a somewhat weakened assumption of the upper bound on $f_t$, we split the integral by two parts and control them separately.
		Moreover, as the error bound suggest, we use union bound over a discretized interval of $[0,h]$, and then complete the argument by showing a certain type of continuity of the integral. Without loss of generality, we assume that $\tau_0\ge \tau_-$ almost surely (otherwise the integral is zero). 
		
		For each $t\in[0,h]$, we write
		\begin{equation}\label{eq:concenofint:conti:mainsplit}
		\intop _{\tau_-}^{t\wedge \tau_0} f_t(x) d \widetilde{\Pi }_g(x)
		=\intop _{\tau_-}^{(\Delta\vee \tau_-)\wedge \tau_0} f_t(x) d \widetilde{\Pi }_g(x)
		+
		\intop _{(\Delta\vee \tau_-)\wedge \tau_0}^{t\wedge \tau_0} f_t(x) d \widetilde{\Pi }_g(x).
		\end{equation}
		The first integral can be controlled using the definitions of $\tau', \tau''$, namely,
		\begin{equation}
	N\underline{f}_t(p_1) - \textbf{A} \le	\intop _{\tau_-}^{(\Delta\vee \tau_-)\wedge \tau_0} f_t(x) d \widetilde{\Pi }_g(x) \leq N \bar{f}_t(p_1) + \textbf{A}.
		\end{equation}
		On the other hand, the second integral can be controlled using Corollary \ref{cor:concentration of integral}, where we obtain
		\begin{equation}
		\PP\left(\left|\intop _{\Delta \wedge \tau_0}^{t\wedge \tau_0} f_t(x) d \widetilde{\Pi }_g(x) \right|\ge a\sqrt{\textbf{M}} \right) \leq Ce^{-a}.
		\end{equation}
		Then, we define $\mathcal{T}:= \{t\in[0,h]: t=k\delta\Delta, \ k\in\mathbb{Z} \}$, and combine the above two estimates to deduce
		\begin{equation}\label{eq:concentration:conti:unionoverdiscrete}
		\PP\left(\intop _{\tau_-}^{t\wedge\tau_0} f_t(x) d \widetilde{\Pi }_g(x) \le Nf_t(p_1)+\textbf{A}+ a\sqrt{\textbf{M}},\ \forall t\in\mathcal{T} \right) \geq 1- \left(\frac{Ch}{\delta \Delta}\right)e^{-a},
		\end{equation}
		and an analogous bound holds for the lower bound.
		
		Next, for any $t\in[0,h]$, let $t_\delta$ be such that $t_\delta\le t$, $t_\delta \in \mathcal{T}$.Write $t' = (t\wedge \tau_0)\vee \tau_-$, $t_\delta' =( t_\delta \wedge \tau_0) \vee \tau_-$, and express that 
		\begin{equation}\label{eq:concenint:conti:contidiffform}
		\intop _{\tau_-}^{t'} f_t(x) d \widetilde{\Pi }_g(x)
		-\intop _{\tau_-}^{t'_\delta} f_{t_\delta}(x) d \widetilde{\Pi }_g(x) =
		\intop_{t_\delta'}^{t'} f_t(x)d\widetilde{ \Pi}_g(x)
		+ \intop_{\tau_-}^{t_\delta'} (f_t(x)-f_{t_\delta}(x)) d\widetilde{ \Pi}_g(x).
		\end{equation}
		The first integral can again be bounded by using the definitions of $\tau', \tau''$:
		\begin{equation}\label{eq:concentration:conti:diff1}
		 N\underline{f}_t(p_1)-\textbf{A} \le \intop_{t_\delta'}^{t'} f_t(x)d\widetilde{ \Pi}_g(x)  \leq N\bar{f}_t(p_1) + \textbf{A},
		\end{equation}
		where we used the fact that $\bar{f}_t(x)$ (resp. $\underline{f}_t(x)$) is a decreasing (resp. increasing) function. On the other hand, we control the second integral by using
		\begin{equation}\label{eq:fandfdeldiff}
		|f_t(x)- f_{t_\delta}(x)| \leq D(t-t_\delta)\le D\delta\Delta.
		\end{equation}
		Moreover, due to $\tau_0\le \tau'$, the total number of points $|\Pi_g[0,t'_{ \delta}]|$ is bounded by $hN\Delta^{-1}$. Thus,
		\begin{equation}\label{eq:concentration:conti:diff2}
		\left|\intop_{\tau_-}^{t_\delta'} (f_t(x)-f_{t_\delta}(x)) d\widetilde{ \Pi}_g(x)
		\right|
		\leq
		D\delta \Delta ( hN\Delta^{-1} + h\eta)= D\delta ( hN+h\Delta\eta). 
		\end{equation}
		
		Therefore, under the event described inside \eqref{eq:concentration:conti:unionoverdiscrete}, we have by combining \eqref{eq:concentration:conti:diff1} and \eqref{eq:concentration:conti:diff2} that
		\begin{equation}\label{eq:concentration:conti:generalt}
		\begin{split}
		\intop _{\tau_-}^{t'} f_t(x) d \widetilde{\Pi }_g(x) \le&  N\bar{f}_t (p_1) +2\textbf{A} + D\delta (hN+h\Delta \eta)\\
		&+
		N\bar{f}_{t_\delta}(p_1) + a\sqrt{\textbf{M}}, 
		\end{split}
		\end{equation}
		and also the corresponding lower bound.
		Using \eqref{eq:fandfdeldiff}, we can simplify \eqref{eq:concentration:conti:generalt}  to
		\begin{equation}
		\begin{split}
		\intop _{\tau_-}^{t'} f_t(x) d \widetilde{\Pi }_g(x) \le&
		2N \bar{f}_t(p_1) + 2\textbf{A} + D\delta (2N\Delta + hN + h\Delta \eta)+ a\sqrt{\textbf{M}}\\
		\le& 2N \bar{f}_t(p_1) + 3\textbf{A} + a\sqrt{\textbf{M}},
		\end{split}\end{equation}
		where the last inequality is due to the assumption on $\delta$. Again, we have the corresponding bound for the lower bound as well. Therefore, we obtain conclusion by noticing that the above holds deterministically for all $t\in[0,h]$ under the event given in  \eqref{eq:concentration:conti:unionoverdiscrete}. 
	\end{proof}
	
	Note that the proof of Corollary \ref{lem:concentrationofint:continuity} follows from the exact same argument, except that we split the integrals by $[0,t-\Delta]$ and $[t-\Delta, t]$ instead of \eqref{eq:concenofint:conti:mainsplit}, namely,
		\begin{equation}\label{eq:concenofint:conti:mainsplit1}
		\intop _{0}^{t'} f_t(x) d \widetilde{\Pi }_g(x)
		=\intop _{0}^{(t-\Delta)\wedge \tau_0} f_t(x) d \widetilde{\Pi }_g(x)
		+
		\intop _{(t-\Delta)\wedge  \tau_0}^{t\wedge \tau_0} f_t(x) d \widetilde{\Pi }_g(x).
		\end{equation}
The details are omitted due to similarity.

	\section{Negligibility of the third order contributions}\label{subsec:app:Qbound}
	
	In this section, we establish Proposition \ref{prop:Qbound:quenched}.
		The proof is based on a similar approach as Proposition \ref{prop:Jbound:quenched}, but requires more delicate analysis on the coupled events of $W$ and $Y$. Recall the definition \eqref{eq:def:Q:basic form}, and let
		\begin{eqnarray}
		&A_{v,u} := \mathds{1}{\{s\le T_{v,u}<\infty \}}, &A_{v} := \mathds{1}{\{s\le T_v<\infty \}},\\
		&A_u:= \mathds{1}\{s\le T_u <\infty \},  & A_0 := \mathds{1}\{s\le T<\infty\},\\
		&{A}:= A_{v,u}-A_v-A_u +A_0. &
		\end{eqnarray}
		Then, $\textbf{Q}_{v,u,s} = \mathbb{E}_\alpha[{A}\,|\,Y_{\le v} ]$.
		
		 Recalling that $T\le T_u\le T_v \le T_{v,u}$, $A$ can take nonzero values only in the seven cases as illustrated in the Table \ref{table:1} below.

		\begin{table}[h!]
			\centering
			\begin{tabular}{||c || c | c | c | c||} 
				\hline 
				Event & $\in (0,s)$ &  $\in [s,\infty)$ & $=\infty$ &  Value of $A$ \\ [0.5ex] 
				\hline\hline
				$\textbf{A}_1$ & none & $T$ & $T_u,\,T_v,\,T_{v,u}$ & $+1$ \\ [0.3ex]
				\hline
				$\textbf{A}_2$ & none & $T,\,T_u,\,T_v$ & $T_{v,u}$ & -1 \\[0.3ex]
				\hline
				$\textbf{A}_3$ & $T$ & $T_u$ & $T_v,\,T_{v,u}$ & $-1$ \\ [0.3ex]
				\hline
				$\textbf{A}_4$ & $T$ & $T_u,\,T_v$ & $T_{v,u}$ & $-2$ \\ [0.3ex]
				\hline
				$\textbf{A}_5$ & $T,\,T_u$ & $T_v$ & $T_{v,u}$ & $-1$ \\  [0.3ex]
				\hline
				$\textbf{A}_6$ & $T$ & $T_u,\,T_v,\,T_{v,u}$ & none & $-1$\\ [0.3ex]
				\hline
				$\textbf{A}_7$ & $T,\,T_u,\,T_v$ & $T_{v,u}$ & none & $+1$ \\  [0.3ex]
				\hline 
			\end{tabular}
			\caption{The cases that give nonzero values of $A$. Each row describes the regimes that $T,\,T_u,\,T_v$, and $T_{v,u}$ belong to. For instance, the second row denotes that $\textbf{A}_2 = \mathds{1}\{t\le T\leq T_u \leq T_v <\infty = T_{v,u} \}$, and that $A=-1$ on $\textbf{A}_2$.}
			\label{table:1}
		\end{table}
		
		From Table \ref{table:1}, we see that
		\begin{equation}
		A = \mathds{1} \{\textbf{A}_1 \} -\mathds{1} \{\textbf{A}_2 \} -\mathds{1} \{\textbf{A}_3 \} -2\cdot\mathds{1} \{\textbf{A}_4 \} -\mathds{1} \{\textbf{A}_5 \} -\mathds{1} \{\textbf{A}_6 \} +\mathds{1} \{\textbf{A}_7 \}.
		\end{equation}
		Therefore, our strategy is to estimate each $\PP_\alpha(\textbf{A}_i \,|\, \mathcal{F}_v)$. This is done by the following lemmas, and we obtain the conclusion from combining them together.

	The following lemmas provide estimates on $\PP_\alpha (\textbf{A}_i \,|\,\mathcal{F}_v)$. Recall the parameter $\hat{h}$ from \eqref{eq:def:horizon}. Moreover, in the following statements, ``an event holds \textit{with very high probability}'' means that for all sufficiently small $\alpha>0$, it holds with probability at least 
	\begin{equation}
	1-e^{-\alpha^{-c_\epsilon}},
	\end{equation}
	for some constant $c_\epsilon>0$.
	
	\begin{lem} \label{lem:A1andA2}
		We have	$0\leq \mathbb{P}_\alpha(\textbf{A}_1\,|\,\mathcal{F}_v )- \mathbb{P}_\alpha (\textbf{A}_2 \,|\, \mathcal{F}_v) = \frac{4\alpha^2}{ (1+2\alpha)^2} \PP_\alpha (s\leq T <\infty \,|\,\mathcal{F}_v)$. Moreover, with very high probability,
		\begin{equation} \label{eq:A1andA2main}
		\textnormal{for all } 0\leq v<u<s\leq \hat{h}, \quad 
		\mathbb{P}_\alpha(\textbf{A}_1\,|\,\mathcal{F}_v )- \mathbb{P}_\alpha (\textbf{A}_2 \,|\, \mathcal{F}_v)
		\leq  
		\frac{\alpha^{-2 \epsilon }}{\sqrt{(v+1)(u+1)(s+1)} }.
		\end{equation}
	\end{lem}
	
	\begin{lem} \label{lem:A3}
		$\textbf{A}_3$ is an empty event. In particular,  $\PP_\alpha (\textbf{A}_3\,|\, \mathcal{F}_v) =0 $.
	\end{lem}
	
	\begin{lem} \label{lem:A4}
		We have $\PP_\alpha (\textbf{A}_4 \,|\, \mathcal{F}_v) = \frac{2\alpha}{1+2\alpha} \PP_\alpha (T \leq s < T_u <\infty \,|\, \mathcal{F}_v)$. Moreover, with very high probability,
		\begin{equation}
		\textnormal{for all } 0\leq v<u<s\leq \hat{h}, \quad \PP_\alpha (\textbf{A}_4 \,|\, \mathcal{F}_v) \leq 
		\frac{\alpha^{-2 \epsilon }}{\sqrt{(v+1)(u+1)(s+1)} }.
		\end{equation}
	\end{lem}
	
	\begin{lem} \label{lem:A5}
		We have $\PP_\alpha (\textbf{A}_5 \,|\,\mathcal{F}_v) \leq \frac{2\alpha}{1+2\alpha}\PP_\alpha (T <s \leq T_v<\infty \,|\,\mathcal{F}_v)$. Moreover, with very high probability,
		\begin{equation}
		\textnormal{for all } 0\leq v<u<s\leq \hat{h}, \quad
		\PP_\alpha (\textbf{A}_5 \,|\,\mathcal{F}_v) \leq
		\frac{\alpha^{-2\epsilon}}{\sqrt{(v+1)(u+1)(s+1)}}.
		\end{equation}
	\end{lem}
	
	\begin{lem} \label{lem:A6andA7}
		With very high probability, we have
		\begin{equation}
		\textnormal{for all }0\le v<u<s\le \hat{h}, \quad 
		|\PP_\alpha(\textbf{A}_6 \,|\, \mathcal{F}_v) - \PP_\alpha (\textbf{A}_7 \,|\, \mathcal{F}_v)| \leq 
		\frac{\alpha^{-2\epsilon}}{ \sqrt{(v+1)(u+1)(s+1)}}.
		\end{equation}
	\end{lem}
	
	Although Lemmas \ref{lem:A1andA2}--\ref{lem:A5} follow from straight-forward generalizations of Proposition \ref{prop:Jbound:quenched}, the proof of Lemma  \ref{lem:A6andA7} requires more work. Indeed, we will later see that the probabilities $\PP_\alpha(\textbf{A}_i\,|\, \mathcal{F}_v),$ $i\in\{1,2,6,7\}$ do not satisfy the estimate \eqref{eq:A1andA2main} individually, but their leading orders cancel out under the setting of Lemmas \ref{lem:A1andA2} and \ref{lem:A6andA7}. Though such cancellation is apparent in Lemma \ref{lem:A1andA2}, a refined analysis is needed in Lemma \ref{lem:A6andA7} to observe such phenomenon.

	In the remaining of this subsection, we present the proofs of Lemmas \ref{lem:A3}--\ref{lem:A5}, and then prove Lemmas \ref{lem:A1andA2} and \ref{lem:A6andA7}. 	Also, for $Y=\{Y(x)\}$, we define 
	$Y^p$ and $Y^{p,q}$ as
	\begin{equation}
	Y^p(x) := Y(x) + \mathds{1}\{x>p \}, \quad Y^{p,q}(x) := Y(x) +\mathds{1} \{x>p \} + \mathds{1} \{ x> q\}.
	\end{equation}
	
	\begin{proof}[Proof of Lemma \ref{lem:A3}]
		Note that $T_u > s$ implies $T_v = T_u$, since $Y^v(t') = Y^u(t')$ for $t'>s$. Therefore, $\{s<T_u <\infty = T_v \}$ described in the definition of $\textbf{A}_3$ is an invalid event.
	\end{proof}
	
	\begin{proof}[Proof of Lemmas \ref{lem:A4} and \ref{lem:A5}]
		We present the proof of Lemma \ref{lem:A4}. Proof of Lemma \ref{lem:A5} follows from the same argument and is left to the reader.
		
		On $\textbf{A}_4$, $T <s\leq T_u$ implies that the random walk $W$ hits $Y$ at $T \in [u,s)$. Then, since $s\leq T_u <\infty$, it hits $Y^{u}$ at $T_u=T_v \in[s,\infty)$. Finally, since $T_{v,u}=\infty$, it never hits $Y^{v,u}$. Therefore,
		\begin{equation}
		\begin{split}
		\PP_\alpha (\textbf{A}_4 \,|\, \mathcal{F}_v ) &= \PP_\alpha (T<s\leq T_u <\infty \,|\,\mathcal{F}_v) \\
		&\quad \times \PP_\alpha (W(t') < Y^{v,u}(t'),  \; \forall t' \geq T_u \,|\, W(T_u)=Y^{v,u}(T_u)-1 )\\
		&= \frac{2\alpha}{1+2\alpha} \PP_\alpha (T<s\leq T_u <\infty \,|\,\mathcal{F}_v) ,
		\end{split}
		\end{equation}
		which proves the first part of the lemma. 
		
		Then, we can use Proposition \ref{prop:Jbound:quenched} to estimate the probability in the RHS of the above. This gives that with very high probability, we have for all $0\leq v<u<s\leq \hat{h}$  (which gives in particular, $v\leq \alpha^{-2-\epsilon}$) that
		\begin{equation}
		\quad \PP_\alpha (\textbf{A}_4 \,|\, \mathcal{F}_v) \leq 
		\frac{\alpha^{1- \epsilon }}{\sqrt{(u+1)(s+1)} }
		\leq 
		\frac{\alpha^{-2 \epsilon }}{\sqrt{(v+1)(u+1)(s+1)} }.
		\end{equation}
	\end{proof}
	
	\begin{proof}[Proof of Lemma \ref{lem:A1andA2}]
		On $\textbf{A}_1 = \{s\leq T <\infty = T_u=T_v=T_{v,u} \}$, the random walk $W$ hits $Y$ at time $T\in[s,\infty)$ but never hits $Y^u$ after that. Thus,
		\begin{equation} \label{eq:A1formula}
		\begin{split}
		\PP_\alpha (\textbf{A}_1 \,|\, \mathcal{F}_v) &= \PP_\alpha(s\leq T <\infty \,|\,\mathcal{F}_v) \\
		&\quad \times \PP_\alpha(W(t') <Y^u(t'), \;\forall t'>T\,|\, W(T)=Y^u(T)-1 )\\
		&=  \PP_\alpha(s\leq T <\infty \,|\,\mathcal{F}_v) \times \frac{2\alpha}{1+2\alpha} .
		\end{split}
		\end{equation}
		On the other hand, on $\textbf{A}_2 = \{s\leq T \leq T_u\leq T_v <\infty =T_{v,u}  \}$, the random walk $W$ hits $Y^u$ as well at $T_u \in(T,\infty)$ after hitting $Y$ at $T\in[s,\infty)$ (One can also see that $T_u=T_v$), but never hits $Y^{v,u}$. Therefore,
		\begin{equation} \label{eq:A2formula}
		\begin{split}
		\PP_\alpha (\textbf{A}_2 \,|\,\mathcal{F}_v) &= \PP_\alpha (s\leq T <\infty \,|\, \mathcal{F}_v)\\
		&\quad  \times
		\{\PP_\alpha(T_{u}<\infty\,|\, T\in (s,\infty) )
		-
		\PP_\alpha( T_{v,u}<\infty \,|\, T\in (s,\infty) )  \}\\
		&= \PP_\alpha (s\leq T<\infty \,|\,\mathcal{F}_v) \times
		\left\{\frac{1}{1+2\alpha} - \frac{1}{(1+2\alpha)^2}  \right\}
		\end{split}
		\end{equation}
		
		Combining \eqref{eq:A1formula} and \eqref{eq:A2formula}, we obtain that
		\begin{equation}
		0\leq \PP_\alpha (\textbf{A}_1 \,|\,\mathcal{F}_v) -\PP_\alpha (\textbf{A}_2 \,|\,\mathcal{F}_v)
		= \frac{4\alpha^2}{(1+2\alpha)^2} \times \PP_\alpha (s\leq T<\infty \,|\,\mathcal{F}_v) .
		\end{equation}
		This proves the first part of Lemma \ref{lem:A1andA2}. Then, \eqref{eq:A1andA2main} follows from the same argument as in the proof of Lemma \ref{lem:A4}, by using Lemma \ref{lem:bound on H}.
	\end{proof}
	
	What remains is to establish Lemma \ref{lem:A6andA7}. To this end, it will be useful to divide $\textbf{A}_7$ into two disjoint events $\textbf{A}_7^1$ and $\textbf{A}_7^2$ as follows:
	\begin{equation} \label{eq:def:A71andA72}
	\begin{split}
	\textbf{A}_7^1 := \{T\geq u \}\cap \textbf{A}_7, \quad \textnormal{and} \quad 
	\textbf{A}_7^2 := \{T<u \} \cap \textbf{A}_7.
	\end{split}
	\end{equation}
	Note that $\textbf{A}_7^1, \textbf{A}_7^2$ are equivalent to the following events:
	\begin{equation}
	\begin{split}
	&\textbf{A}_7^1 = \{u \leq T<T_u = T_v <s\leq T_{v,u}<\infty \};\\
	&\textbf{A}_7^2 = \{v\leq T=T_u <u \leq T_v <s \leq T_{v,u}<\infty \}.
	\end{split}
	\end{equation}
	Then, we bound $|\PP_\alpha(\textbf{A}_6\,|\,\mathcal{F}_v ) - \PP_\alpha(\textbf{A}_7^1 \, |\,\mathcal{F}_v)|$ and $\PP_\alpha (\textbf{A}_7^2\,|\,\mathcal{F}_v)$ as the follows.
	
	\begin{lem} \label{lem:A6andA71}
		With very high probability, we have for all $0\leq v<u<s\leq \hat{h}$ that
		\begin{equation}
		|\PP_\alpha(\textbf{A}_6\,|\,\mathcal{F}_v ) - \PP_\alpha(\textbf{A}_7^1 \, |\,\mathcal{F}_v)|
		\leq \frac{\alpha^{-2\epsilon}}{(s+1)\sqrt{u+1} } 
		.
		\end{equation}
	\end{lem}
	
	\begin{lem} \label{lem:A72}
		With very high probability, we have for all $0\leq v<u<s\leq \hat{h}$ that
		\begin{equation}
		\PP_\alpha(\textbf{A}_7^2\,|\,\mathcal{F}_v) \leq \frac{\alpha^{-\epsilon}}{\sqrt{(v+1)(u+1)(s+1)}}.
		\end{equation}
	\end{lem}
	
	\begin{proof}[Proof of Lemma \ref{lem:A6andA7}]
		Lemma \ref{lem:A6andA7} follows straight-forwardly from Lemmas \ref{lem:A6andA71} and \ref{lem:A72}.
	\end{proof}
	
	\begin{proof}[Proof of Lemma \ref{lem:A6andA71}]
		Let $f_v$ denote the probability density function of $T$, conditioned on $\mathcal{F}_v$. Namely,
		\begin{equation}
		f_v(t') =- \frac{\textnormal{d}}{\textnormal{d}t'} 
		\PP_\alpha (t'\leq T \,|\,\mathcal{F}_v ).
		\end{equation}
		Under the notation of Lemma \ref{lem:bound on density}, $f_v = \mathbb{E}_\alpha [f_Y \,|\, \mathcal{F}_v ]$. Similarly, we write
		\begin{equation}
		f(t') = - \frac{\textnormal{d}}{\textnormal{d}t'} 
		\PP_\alpha (t'\leq T ) = \mathbb{E}_\alpha [ f_Y(t')].
		\end{equation}
		For simplicity, we also introduce
		\begin{equation}
		h(t') :=  \PP_\alpha (t' \leq T <\infty ).
		\end{equation}
		Here, note that $f(t')= -h'(t')$.
		
		On $\textbf{A}_6$, we have $u\leq T<s$, since if $T<u$, then $T=T_u$ which is a contradiction. Also, since both $T_u$ and $T_v$ are at least $s$, they are equal. Therefore, we can write
		\begin{equation} \label{eq:A6intform}
		\begin{split}
		\PP_\alpha(\textbf{A}_6 \,|\,\mathcal{F}_v )
		&= \intop_u^s f_v(x)  \PP_\alpha( s\leq T_u <\infty \,|\,T=x)\\
		&\quad\qquad 
		\times \PP_\alpha ( W(t')<Y^{v,u}(t'),\;\forall t'>T_u\,|\, W(T_u)=Y^{v,u}(T_u)-1  ) \;\textnormal{d}x\\
		&=
		\intop_u^s
		f_v(x) h(s-x) \times	\frac{2\alpha}{1+2\alpha} \;\textnormal{d}x	
		=\intop_u^s f_v(x)h(s-x) h(0) \textnormal{d}x,
		\end{split}
		\end{equation}
		where the last equality follows from $h(0) = \PP_\alpha (0\leq T<\infty) = \frac{2\alpha}{1+2\alpha}$.
		
		On the other hand, $\PP_\alpha(\textbf{A}_7^1 \,|\,\mathcal{F}_v)$ can be written as
		\begin{equation} \label{eq:A71intform}
		\begin{split}
		\PP_\alpha(\textbf{A}_7^1 \,|\,\mathcal{F}_v )
		&= \intop_u^s \intop_x^s f_v(x) \PP_\alpha(T_v \in [y,y+\textnormal{d}y) \,|\, T=x )\, \PP_\alpha( s\le T_{v,u}<\infty \,|\, T_v=y ) \;\textnormal{d}x \\
		&=\intop_u^s\intop_x^s f_v(x)f(y-x) h(s-y) \textnormal{d}y\,\textnormal{d}x\\
		&=-\intop_u^s\intop_x^s f_v(x) h(s-y) h'(y-x) \textnormal{d}y\,\textnormal{d}x.
		\end{split}
		\end{equation}
		Using the identity
		\begin{equation}
		\intop_{\frac{s+x}{2}}^s h'(s-y) h(y-x) \textnormal{d} y
		=
		\intop_{x}^{\frac{s+x}{2}} h(s-y) h'(y-x)\textnormal{d}y,
		\end{equation}
		note that
		\begin{equation}
		\begin{split}
		h(s-x) h(0) - h\left(\frac{s-x}{2} \right)^2
		&=
		\intop_{\frac{s+x}{2}}^s \frac{\textnormal{d}}{\textnormal{d}y} \left\{h(s-y) h(y-x) \right\} \textnormal{d}y\\
		&=
		-\intop_x^s h(s-y) h'(y-x) \textnormal{d}y
		+2\intop_{\frac{s+x}{2}}^s h(s-y)h'(y-x) \textnormal{d}y.
		\end{split}
		\end{equation}
		This gives that
		\begin{equation}
		\begin{split}
		\PP_\alpha(\textbf{A}_6\,|\,\mathcal{F}_v) -\PP_\alpha(\textbf{A}_7^1\,|\,\mathcal{F}_v)
		&=
		\intop_u^s f_v(x) h\left(\frac{s-x}{2} \right)^2 \textnormal{d}x
		+
		2\intop_u^s \intop_{\frac{s+x}{2}}^t f_v(x) h(s-y) h'(y-x) \textnormal{d}y\\
		&=: \textbf{I}_1 + \textbf{I}_2.
		\end{split}
		\end{equation}
		Then, our next task is to estimate $\textbf{I}_1$ and $\textbf{I}_2$. Recall Corollary~\ref{cor:conditioning on less information} and the bounds on $h$ and $f=-h'$ (Lemma \ref{lem:estimate on K:intro}, noting that $K_\alpha(t') = \alpha(1+2\alpha) h(t')$). This gives that there exists a constant $C>0$ such that with very high probability,
		\begin{equation} \label{eq:A6andA71 integral form}
		\begin{split}
		&\textbf{I}_1\leq
		\intop_u^s \frac{C\alpha^{-\epsilon}\, \textnormal{d}x}{(x+1)^{3/2} (s-x+1) } ;
		\\
		&\textbf{I}_2 \leq
		\intop_u^s \intop_{\frac{s+x}{2}}^{t} \frac{C \alpha^{-\epsilon}\, \textnormal{d}y\,\textnormal{d}x}{(x+1)^{3/2} (y-x+1)^{3/2}(s-y+1)^{1/2} },
		\end{split}
		\end{equation}
		hold for all $0<u<s\leq \hat{h}$. 
		
		Both integrals in the RHS of \eqref{eq:A6andA71 integral form} can be estimated based on the idea used in Claim \ref{claim:some claim}.   Thus, the result of Lemma \ref{lem:bound on I1I2:inA6A7} tells us that
		\begin{equation}
		\textbf{I}_1+\textbf{I}_2 \leq  \frac{C \alpha^{-\epsilon} \log s}{(s+1)\sqrt{u+1} }
		\leq \frac{\alpha^{-2\epsilon}}{\sqrt{(v+1)(u+1)(s+1)}},
		\end{equation}
		where $C>0$ may take a different value from that in \eqref{eq:A6andA71 integral form} but is still an absolute constant. The last inequality follows from the assumption $s\leq\hat{h}\le  \alpha^{-2-\epsilon}$. 
	\end{proof}
	
	\begin{lem} \label{lem:bound on I1I2:inA6A7}
		There exists an absolute constant $C>0$ such that for any $0\leq u<s$,
		\begin{equation}
		\begin{split}
		&\qquad \qquad \textbf{I}_1':=\intop_u^s \frac{ \textnormal{d}x}{(x+1)^{3/2} (s-x+1) }\leq
		\frac{C \log s}{(s+1)\sqrt{u+1}};
		\\
		&\textbf{I}_2':=\intop_u^s \intop_{\frac{s+x}{2}}^{s} \frac{ \textnormal{d}y\,\textnormal{d}x}{(x+1)^{3/2} (y-x+1)^{3/2}(s-y+1)^{1/2} }
		\leq
		\frac{C \log s}{(s+1)\sqrt{u+1}}.
		\end{split}
		\end{equation} 
	\end{lem}
	
	\begin{proof}
		We give a proof for the second inequality. The first inequality can be obtained by the same approach and is simpler.
		
		We can rewrite the LHS of the second line by
		\begin{equation} \label{eq:I2int 2nd form:A6A7}
		\textbf{I}_2'=\intop_{\frac{s+u}{2}}^s \intop_u^{2y-s} \frac{\textnormal{d}x \, \textnormal{d}y}{(x+1)^{3/2}(y-x+1)^{3/2} (s-y+1)^{1/2} }.
		\end{equation}
		As done in Claim \ref{claim:some claim}, we divide into two cases, when $t\leq 2s$ and $t>2s$. In what follows, $C>0$ denotes an absolute constant that may take different values at different places.
		
		\vspace{2mm}
		\noindent \textbf{Case 1.} $s\leq 2u$.
		\vspace{2mm}
		
		In this case, we can bound the integral in a rather straight-forward way using the expression \eqref{eq:I2int 2nd form:A6A7}:
		\begin{equation}
		\begin{split}
		\textbf{I}_2' \leq \intop_{ \frac{s+u}{2} }^s \frac{C \,\textnormal{d}y}{(u+1)^{3/2}(s-y+1)}
		\leq \frac{C \log s}{(s+1)^{3/2}} \leq \frac{C \log s}{(s+1)\sqrt{v+1}}.
		\end{split}
		\end{equation}
		
		\vspace{2mm}
		\noindent \textbf{Case 2.} $s>2u$.
		\vspace{2mm}
		
		Noting that $\frac{u+(2y-s)}{2} = y-{\frac{u-s}{2}}$, we divide the inner integral into two parts, from $u$ to $y- \frac{s-u}{2}$ and from $y-\frac{u-s}{2}$ to $2y-s$. Using \eqref{eq:I2int 2nd form:A6A7}, the frist half can be bounded as
		\begin{equation} \label{eq:I2int case2:1st half}
		\begin{split}
		\intop_{\frac{s+u}{2}}^s \intop_u^{y- \frac{s-u}{2}} \frac{\textnormal{d}x \, \textnormal{d}y}{(x+1)^{3/2}(y-x+1)^{3/2} (s-y+1)^{1/2} } &\leq
		\intop_{ \frac{s+u}{2} }^s
		\frac{2^{3/2}\textnormal{d}y}{(u+1)^{1/2} (s-u+2)^{3/2} (s-y+1)^{1/2}}\\
		&\leq \frac{4}{(s-u+2)\sqrt{u+1}} \leq \frac{C}{(s+1)\sqrt{u+1}}.
		\end{split}
		\end{equation}
		
		To control the second half, we first note that
		\begin{equation} 
		\begin{split}
		\intop_{y- \frac{s-u}{2}}^{2y-s} \frac{\textnormal{d}x }{(x+1)^{3/2}(y-x+1)^{3/2} }
		\leq
		\frac{2}{(y-(\frac{s-u}{2})+1)^{3/2}(s-y+1)^{1/2} }.
		\end{split}
		\end{equation}
		Thus,
		\begin{equation} \label{eq:I2int case2:2nd half:1}
		\begin{split}
		\intop_{\frac{s+u}{2}}^s \intop_{y- \frac{s-u}{2}}^{2y-s} \frac{\textnormal{d}x \, \textnormal{d}y}{(x+1)^{3/2}(y-x+1)^{3/2} (s-y+1)^{1/2} }
		\leq
		\intop_{ \frac{s+u}{2} }^s \frac{2\,\textnormal{d}y}{(y-(\frac{s-u}{2})+1)^{3/2}(s-y+1) }.
		\end{split}
		\end{equation}
		Then, we again split the right integral into two parts, from $\frac{s+u}{2}$ to $\frac{1}{2}(s+\frac{s+u}{2})$ and from $\frac{1}{2}(s+\frac{s+u}{2})$ to $s$. This gives that
		\begin{equation} \label{eq:I2int case2:2nd half:2}
		\begin{split}
		&\intop_{ \frac{s+u}{2} }^{\frac{1}{2}(s+\frac{s+u}{2})} \frac{2\,\textnormal{d}y}{(y-(\frac{s-u}{2})+1)^{3/2}(s-y+1) }
		\leq \frac{C}{(s+1)\sqrt{u+1}};\\
		&\intop_{\frac{1}{2}(s+\frac{s+u}{2})}^s \frac{2\,\textnormal{d}y}{(y-(\frac{s-u}{2})+1)^{3/2}(s-y+1) }
		\leq
		\frac{C\log s}{(s+1)^{3/2}}, 
		\end{split}
		\end{equation}
		and we obtain the conclusion for Case 2: $s>2u$ by combining the equations \eqref{eq:I2int case2:1st half}, \eqref{eq:I2int case2:2nd half:1} and \eqref{eq:I2int case2:2nd half:2}.
	\end{proof}
	
	\begin{proof}[Proof of Lemma \ref{lem:A72}]
		Recall the definitions of $f_v, f$ and $h$ introduced in the beginning of the proof of Lemma \ref{lem:A6andA71}. Similarly as \eqref{eq:A6intform} and \eqref{eq:A71intform}, we can write $\PP_\alpha (\textbf{A}_7^2 \,|\,\mathcal{F}_v)$ as
		\begin{equation}
		\begin{split}
		\PP_\alpha (\textbf{A}_7^2 \,|\,\mathcal{F}_v) 
		=
		\intop_v^u f_v(x) \intop_u^s f(y-x) h(s-y) \,\textnormal{d}y\textnormal{d}x. 
		\end{split}
		\end{equation}
		As in the previous proof, recall the estimates  on $f_v$ (Lemma \ref{lem:bound on density}) and  on $f$ and $h$ (Lemma \ref{lem:estimate on K:intro}). Then, we obtain that the following holds with very high probability: 
		\begin{equation}
		\PP_\alpha (\textbf{A}_7^2 \,|\,\mathcal{F}_v) 
		\leq 
		\intop_v^u \intop_u^s \frac{C\alpha^{-\epsilon}\,\textnormal{d}y\textnormal{d}x}{(x+1)^{3/2}(s-x+1)^{3/2} (s-y+1)^{1/2} },
		\end{equation}
		for all $0\leq v<u<s\leq \hat{h}$. . Then, Lemma \ref{lem:bound on I3:inA71} below gives the correct estimate on the double integral and hence we obtain the conclusion.
	\end{proof}
	
	\begin{lem} \label{lem:bound on I3:inA71}
		There exists an absolute constant $C>0$ such that for any $0\leq v< u< s\le \hat{h}$,
		\begin{equation}
		\intop_v^u \intop_u^s\frac{\textnormal{d}y\textnormal{d}x}{(x+1)^{3/2}(t-x+1)^{3/2} (s-y+1)^{1/2} } 
		\leq  \frac{C\,\textnormal{d}y\textnormal{d}x}{\sqrt{(v+1)(u+1)(s+1)} }.
		\end{equation}
	\end{lem}
	
	\begin{proof}
		This follows from the same argument as Claim \ref{claim:some claim} and Lemma \ref{lem:bound on I1I2:inA6A7}. Namely, one may the following four cases separately: 
		\begin{equation}
		\{u\leq 2v \} \cap \{s\leq 2u \}; \;\; \{u\leq 2v \} \cap \{s < 2u \}; \;\; \{u > 2v \} \cap \{s\leq 2u \}; \;\; \textnormal{and}\;\; \{u> 2v \} \cap \{s > 2u \}.
		\end{equation}
		Then, when $u>2v$ (resp. $s>2u$), split the integral $\int_v^u$ (resp. $\int_u^s$) into two parts, $\int_v^{u/2}$ and $\int_{u/2}^u$ (resp. $\int_u^{s/2}$ and $\int_{s/2}^s$), as done in Claim \ref{claim:some claim}. The rest of the argument goes the same as Lemma \ref{lem:bound on I1I2:inA6A7} and we omit the details.
	\end{proof}

\end{document}